\documentclass[a4paper, 11pt, reqno]{amsart}

\usepackage{amssymb}
\usepackage{amsfonts}
\usepackage{amsmath}
\usepackage{bbm}
\usepackage[margin=2.3cm]{geometry}
\usepackage{hyperref}

\newtheorem{thm}{Theorem}[section]

\newtheorem{prop}[thm]{Proposition}
\newtheorem{lemma}[thm]{Lemma}
\newtheorem{lem}[thm]{Lemma}
\newtheorem{cor}[thm]{Corollary}
\newtheorem{claim}[thm]{Claim}
\newtheorem{conj}[thm]{Conjecture}

\newtheorem{observation}[thm]{Observation}

\theoremstyle{definition}

\renewcommand{\epsilon}{\varepsilon}
\numberwithin{thm}{section}
\numberwithin{equation}{section}

\theoremstyle{remark}
\newtheorem{remark}[thm]{Remark}

\newcommand{\Prob}{\mathbb{P}}

\newcommand{\E}{\mathbb{E}}

\newcommand{\COMMENT}[1]{}

\begin{document}

\title{A blow-up lemma for approximate decompositions}

\author{Jaehoon Kim}

\address{School of Mathematics, University of Birmingham, 
Edgbaston, Birmingham, B15 2TT, United Kingdom}
\email{j.kim.3@bham.ac.uk, d.kuhn@bham.ac.uk, d.osthus@bham.ac.uk}

\author{Daniela K\"uhn}

\author{Deryk Osthus}

\author{Mykhaylo Tyomkyn}
\address{
School of Mathematical Sciences, Tel Aviv University, Tel Aviv 69978, Israel}  
\email{tyomkynm@post.tau.ac.il}

\thanks{The research leading to these results was partially supported by the EPSRC, grant no. EP/M009408/1 (D.~K\"uhn and D.~Osthus), 
and by the Royal Society and the Wolfson Foundation (D.~K\"uhn).
The research was  also partially supported by the European Research Council
under the European Union's Seventh Framework Programme (FP/2007--2013) / ERC Grant 306349 (J.~Kim, D.~Osthus and M.~Tyomkyn). }

\date{\today}
\begin{abstract}
We develop a new method for constructing approximate decompositions of dense graphs into sparse graphs and apply it to longstanding decomposition problems. For instance, our results imply the following. Let $G$ be a quasi-random $n$-vertex graph and suppose $H_1,\dots,H_s$ are bounded degree $n$-vertex graphs with $\sum_{i=1}^{s} e(H_i) \leq (1-o(1)) e(G)$. Then $H_1,\dots,H_s$ can be packed edge-disjointly into $G$. The case when $G$ is the complete graph $K_n$ implies an approximate version of the tree packing conjecture of Gy\'arf\'as and Lehel for bounded degree trees, and of the Oberwolfach problem.

We provide a more general version of the above approximate decomposition result which can be applied to super-regular graphs and thus can be combined with Szemer\'edi's regularity lemma. In particular our result can be viewed as an extension of the classical blow-up lemma of Koml\'os, S\'ark\H{o}zy and Szemer\'edi to the setting of approximate decompositions.
\end{abstract}

\maketitle

\section{Introduction}\label{sec:intro}
\subsection{Packings and decompositions}
Questions on packings and decompositions have a long history, going back to the 19th century.
For instance, the existence of Steiner triple systems (proved by Kirkman in 1847) corresponds to a decomposition of the edge set of the complete graph $K_n$
on $n$ vertices into triangles (if $n\equiv 1 \text{ or } 3 \mod 6$).
A fundamental theorem of Wilson~\cite{Wilson} generalizes this to decompositions of $K_n$ into arbitrary graphs $H$ of fixed size:
for any graph $H$, if $n$ is sufficiently large and satisfies trivially necessary divisibility conditions, then $K_n$ has a decomposition into edge-disjoint copies of $H$.
Similarly, Walecki's theorem goes back to 1892 and gives a decomposition of $K_n$ into Hamilton cycles (if $n$ is odd).
Recently, there has been some exciting progress in this area,
often involving the use of probabilistic techniques. 

Here, a \emph{decomposition} of a graph $G$ into graphs $H_1,\dots,H_s$ is a set of pairwise edge-disjoint copies of $H_1,\dots,H_s$
in $G$ which together cover all edges of $G$.
Conversely, we say that graphs $H_1,\dots, H_s$ \emph{pack into $G$} if there are copies of $H_1,\dots, H_s$ in $G$ so that these copies
are pairwise edge-disjoint. We informally refer to an `approximate decomposition' or a `near-optimal packing' 
if there is a packing which leaves only a small proportion 
of the edges of $G$ uncovered.

There are several beautiful conjectures which have driven a large amount of research in the area.
A prime example is the tree packing conjecture of Gy\'arf\'as and Lehel, which 
would guarantee a decomposition of a complete graph into a suitable given collection of trees.
\begin{conj}[Gy\'arf\'as and Lehel~\cite{GLtrees}] \label{gyarfaslehel}
Given trees $T_1,\dots, T_n$, where for each $i\in [n]$, the tree $T_i$ has $i$ vertices,
the complete graph $K_n$ has a decomposition into copies of $T_1,\ldots, T_n$. 
\end{conj}
A related conjecture, made by Ringel, concerns decompositions of complete graphs into copies of a single tree:
for every tree $T$ on $n+1$ vertices, $K_{2n+1}$ has  a decomposition into $2n+1$ copies of $T$.
There are a large number of partial results on Conjecture~\ref{gyarfaslehel}, some focusing on special classes of trees and some on embedding a (small) proportion of the trees (see e.g.~\cite{BP, B, D2, GLtrees, HB, R,Zak}).

Possibly the  most striking results towards Conjecture~\ref{gyarfaslehel} have been obtained for
the case of bounded degree trees. In particular, a recent result by
B\"ottcher, Hladk\`y, Piguet and Taraz~\cite{BHPT} allows for approximate decompositions of $K_n$ into bounded degree trees that are
permitted to be almost spanning. More precisely, their main result states that for all $\epsilon>0$ and $\Delta\in\mathbb{N}$
there exists $n_0\in\mathbb{N}$ such that whenever $n\ge n_0$ and $T_1,\dots,T_s$ is any family of trees with $|T_i| \le n$, $\Delta(T_i)\le \Delta$
and $\sum_{i=1}^s e(T_i) \le \binom{n}{2}$, then $T_1,\dots,T_s$ pack into $K_{(1+\epsilon)n}$.
Note that this implies an approximate version of Conjecture~\ref{gyarfaslehel} for bounded degree trees (it is approximate both in the sense that none of the trees is spanning and that they do not form a decomposition).
The result in~\cite{BHPT} was strengthened by  Messuti, R\"odl and Schacht~\cite{MRS}
to approximate decompositions of complete graphs into almost spanning graphs of bounded degree which are `separable'
(roughly speaking, a graph is separable if it can be split into bounded size components by removing a small proportion of its vertices).
Independently to us, Ferber, Lee and Mousset~\cite{FLM} generalized this further by proving that one can obtain an approximate decomposition 
of $K_n$ into separable graphs of bounded degree which are allowed to be spanning.
In particular, this means that one can always obtain an approximate decomposition of $K_n$ into spanning trees of bounded degree.

Our first result is in fact more general and guarantees an approximate decomposition of a dense quasi-random graph $G$ 
into arbitrary bounded degree graphs.
More precisely, we say that a graph $G$ on $n$ vertices is {\em $(\epsilon,p)$-quasi-random} if all vertices
$v$ of $G$ have degree $d_G(v)=(1\pm \epsilon)pn$ and all pairs of distinct vertices  $u,v$ have $(1\pm \epsilon)p^2n$ common neighbours. (In fact, the latter condition is only required for almost all pairs, see Section~\ref{sec:concl}.)
\begin{thm}\label{main 1b} 
For all $\Delta\in \mathbb{N}$ and $0<p_0,\alpha \leq 1$ there exist $\epsilon >0$ and $n_0 \in \mathbb{N}$ such that
the following holds for all $n \ge n_0$ and $p\ge p_0$.
Suppose $H_1,\dots, H_s$ are $n$-vertex graphs with $\Delta(H_\ell)\leq\Delta$ for all $\ell \in [s]$ and
$\sum_{\ell=1}^{s} e(H_\ell) \leq (1-\alpha) \binom{n}{2}p$. Suppose that $G$ is an $(\epsilon,p)$-quasi-random graph on $n$ vertices.
Then $H_1,\dots, H_s$ pack into $G$. 

Moreover, if in addition $\sum_{\ell=1}^{s} e(H_\ell)\ge (1-2\alpha)\binom{n}{2}p$, then writing $\phi(H_\ell)$ for the copy of $H_\ell$ in this packing of $H_1,\dots,H_s$ in $G$ and writing $J:=G -(\phi(H_1)\cup\dots\cup \phi(H_s))$, we have $\Delta(J)\leq 4\alpha p n$. 
\end{thm} 
This immediately implies the corresponding result (asymptotically almost surely) for the binomial random graph $G_{n,p}$ where $p$ is constant
and for the complete graph $K_n$.
Note that the case of $G=K_n$ in Theorem~\ref{main 1b} extends the previous results in~\cite{BHPT,FLM,MRS}
mentioned above.

The case $G=K_n$ also yields an approximate version of the longstanding `Oberwolfach problem' (proposed by Ringel in 1967).
Given an arbitrary union $F$ of vertex-disjoint cycles altogether spanning $n$ vertices where $n$ is odd, the Oberwolfach problem asks 
for a decomposition of $K_n$ into copies of $F$.
Bryant and Scharaschkin~\cite{BS} have recently provided an affirmative answer for infinitely many $n$.

The famous Bollob\'as-Eldridge-Catlin conjecture states
that if $ G_1 $ and $ G_2 $ are graphs on $n$ vertices
and $ (\Delta(G_1) + 1) (\Delta(G_2) + 1) < n + 1 $, then $ G_1 $ and $ G_2 $ can be packed into $K_n$.
Bollob\'as, Kostochka and Nakprasit~\cite{BKN}  investigated versions of this conjecture for packing many graphs of bounded degeneracy.
The case $G=K_n$ of Theorem~\ref{main 1b} can be viewed as an asymptotically optimal version of the  
conjecture for packing many graphs of bounded degree.

It would be very interesting to know whether an analogue of Theorem~\ref{main 1b} holds for sparse (quasi-)random graphs, 
i.e.~when the density tends to zero. Very recently, Ferber and Samotij~\cite{FS} were able to obtain very strong results
for the case of spanning trees. In particular, rather than requiring $\Delta$ to be bounded, $\Delta$
is allowed to be a polynomial in $pn$.

\subsection{The blow-up lemma}
Combined with Szemer\'edi's regularity lemma~\cite{S regularity}, the blow-up lemma of Koml\'os, S\'ark\"ozy and Szemer\'edi~\cite{KSSblowup} has had a 
major impact on extremal graph theory.
Roughly speaking, Szemer\'edi's regularity lemma guarantees a partition of any dense graph into a bounded number of
random-like bipartite subgraphs, while the blow-up lemma 
allows to embed bounded degree graphs $H$ into such random-like host graphs $G$.
(Here $H$ is allowed to have the same number of vertices as~$G$, i.e.~the blow-up lemma can be used to find spanning subgraphs.)
These two tools can be combined to find spanning structures in dense graphs, and within the last 20 years, they
have lead to a series of very strong results. Early striking results concern 
spanning trees~\cite{KSSztrees} and powers of Hamilton cycles~\cite{KSSz98}, more recent results include
$H$-factors~\cite{KSSz01,KOmatch} and embeddings of bounded degree graphs of small bandwidth~\cite{bandwidth}.
For further results, see e.g.~the survey~\cite{BCCsurvey}.

Here we develop a `blow-up lemma for approximate decompositions'. This version of the blow-up lemma will not only guarantee a single copy of $H$ inside the host
graph~$G$, but will guarantee a collection of pairwise edge-disjoint copies of $H$ in $G$ such that together all these
copies of $H$ cover almost all edges of~$G$.
In fact, our result is even more general -- we show that, essentially, 
any not too large collection of graphs $H_1,\dots,H_s$ of uniformly bounded maximum degree pack into $G$.

As with the classical blow-up lemma, for this it is natural to consider a partite setting and to assume that
both the graphs $H_\ell$ and the host graph $G$ `share' a common equipartition. 
We need the following standard definitions.
A bipartite graph $G$ with vertex classes $A$ and $B$ is called \emph{$(\epsilon,d)$-regular} if 
for all $A'\subseteq A, B'\subseteq B$ with $|A'| > \epsilon|A|$ and $|B'|> \epsilon|B|$ we have
$$
\frac{e(G[A',B'])}{|A'||B'|} = d\pm \epsilon. 
$$
We say that $G$ is \emph{$(\epsilon,d)$-super-regular} if it is $(\epsilon,d)$-regular and $d_G(v) = (d\pm \epsilon)|A|$ for all $v\in B$ and $d_G(v)= (d\pm \epsilon)|B|$ for all $v\in A$.
It is well known and easy to see that super-regularity is a weaker requirement than (a bipartite version of) quasi-randomness. We can now state our second result. 
\begin{thm}\label{main 2}
For all $0<d_0, \alpha \leq 1$ and $\Delta,r\in \mathbb{N}$ there exist $\epsilon>0$ and $n_0\in \mathbb{N}$ such that the following holds
for all $n\geq n_0$ and $d\ge d_0$.
Suppose $H_1, \dots, H_s$ are  $r$-partite graphs such that each $H_\ell$ has 
vertex classes $X_1,\dots,X_r$ of size $n$ 
and $\Delta(H_\ell)\leq \Delta$. Suppose that $G$ is an $r$-partite graph with vertex classes $V_1,\dots, V_r$ of size $n$,
where $G[V_i,V_j]$ is  $(\epsilon,d)$-super-regular  for all $1 \le i \neq j \le r$.
If $\sum_{\ell=1}^{s} e(H_\ell) \leq (1-\alpha)e(G)$, then $H_1,\dots,H_s$ pack into $G$.
\end{thm}
The case when $s=1$ corresponds to the classical blow-up lemma by Koml\'os, S\'ark\"ozy and Szemer\'edi~\cite{KSSblowup}.
An application of Szemer\'edi's regularity lemma to an arbitrary dense graph $G$ is naturally associated with a `reduced graph' $R$,
whose vertices correspond to the clusters of the regularity partition and whose edges correspond to those pairs of clusters which
in $G$ induce an $\epsilon$-regular graph
of significant density. This reduced graph may not be complete and the number of clusters may be relatively large.
The following result is designed with such a situation in mind.
(A corresponding extension of the classical blow-up lemma to this setting was proved by Csaba~\cite{Csaba} in the context of the
Bollob\'as-Eldridge-Catlin conjecture, see Lemma~\ref{thm:oldblowup}.)

\begin{thm}\label{thm:main3}
For all  $0<\alpha, \eta, d_0 \le 1$ and $\Delta, \Delta_R \in \mathbb{N}$ there exists $\epsilon>0$ so that for all $r\in \mathbb{N}$
there exists $n_0=n_0(\epsilon,r)\in \mathbb{N}$
such that the following holds for all $n \ge n_0$ and $d\ge d_0$.
Let $s\in \mathbb{N}$ be such that $s \leq \eta^{-1} n$. Suppose that $R$ is a graph on $[r]$ with $\Delta(R)\leq \Delta_R$.
Suppose that $H_1,\dots, H_s$ are $r$-partite graphs such that each $H_\ell$ has 
vertex classes $X_1,\dots,X_r$ of size $n$ and satisfies $\Delta(H_\ell)\leq \Delta$. Further, suppose that 
$\sum_{\ell=1}^s e(H_\ell[X_i,X_{j}])\leq (1-\alpha)dn^2$ for all $ij\in E(R)$ and $H_\ell[X_i,X_{j}]$ is empty if $ij \notin E(R)$. 
Suppose finally that $G$ is an $r$-partite graph with vertex classes $V_1,\dots, V_r$ of size $n$,
where $G[V_i,V_j]$ is  $(\epsilon,d)$-super-regular for all $ij \in E(R)$.
Then $H_1,\dots, H_{s}$ pack into $G$.
\end{thm}

In Section~\ref{sec:concl}, we will formulate even more general versions: we for instance allow arbitrary densities for the pairs $G[V_i,V_j]$ and do not require that the vertex classes have exactly equal size. Such a setting allows us to derive an approximate version of the bipartite analogue of the tree packing conjecture, formulated by 
Hobbs, Bourgeois and Kasiraj \cite{HB} in 1986 (see Conjecture~\ref{bipartite tree packing conjecture} and Corollary~\ref{bipartite version}). 

Furthermore, in the main result of Section~\ref{sec:main} (Theorem~\ref{main lemma}), we can also require that the embeddings of the 
$H_\ell$ satisfy additional restrictions: for instance, for the vertices of $H_\ell$ we can specify certain `target sets' in $G$ into which these vertices will be embedded.
In a subsequent paper, Joos, Kim, K\"uhn and Osthus \cite{JKKO} use this to prove several `exact' packing results, including both the tree packing conjecture (Conjecture~\ref{gyarfaslehel}) as well as Ringel's conjecture for all bounded degree trees. 

\subsection{Related results, open questions and further applications}
Kirkman's result on Steiner triple systems was recently generalized by 
Keevash~\cite{Kdesign}, who showed that every
sufficiently large quasi-random graph $G$ of sufficient density has a decomposition into $K_r$ for fixed $r$, provided the
obvious necessary divisibility conditions hold
(here the quasi-randomness assumption is stronger than the one in Theorem~\ref{main 1b}).
Similarly, it is natural to consider decompositions of graphs of large minimum degree into fixed subgraphs  (see e.g.~\cite{BKLO}).

K\"uhn and Osthus~\cite{Kelly,KellyII} extended Walecki's theorem on Hamilton decompositions of complete graphs to the setting of a
`robustly expanding' regular host graph $G$.
(The robust expansion condition is considerably weaker than that of quasi-randomness or $\epsilon$-regularity
and also applies to all graphs of degree $cn$ for $c>1/2$.)
In~\cite{CKLOT}, this result is used as a tool to prove several decomposition and packing conjectures involving Hamilton cycles
and perfect matchings. 
Also, in~\cite{GKO}, it is used to derive optimal decomposition results for dense quasi-random graphs into other structures, including linear forests.
A different generalization of Walecki's theorem is given by the Alspach conjecture, which states that for odd $n$, the complete graph $K_n$ should have a decomposition into cycles $C^1,\dots,C^s$,
provided that $\sum_{\ell=1}^s |C^\ell| = \binom{n}{2}$. This was recently confirmed by Bryant, Horsley and Pettersson~\cite{BHP}.

Of particular interest are decompositions into $H$-factors
(also referred to as `resolvable designs').
Here an \emph{$H$-factor} in a graph $G$ is a set of disjoint copies of $H$ which together cover all vertices of $G$.
A classical result of Ray-Chaudhuri and Wilson~\cite{RCW} states that if $H = K_k$, if $n$ 
is a sufficiently large multiple of $k$ and if $n-1$
is a multiple of $k-1$, then $K_n$ admits a decomposition into $H$-factors. 
Dukes and Ling~\cite{DukesLing} resolved the more general problem of decomposing a complete graph $K_n$ into $H$-factors for arbitrary but 
fixed graphs $H$ (subject to divisibility conditions -- note that, unlike in Wilson's theorem~\cite{Wilson}, 
for some graphs $H$, such as $H=K_{1,2t+1}$, those are never fulfilled). Related results (which avoid this divisibility issue) were obtained by Alon and Yuster~\cite{AY}.
It would be very interesting if one could extend these results on resolvable designs to non-complete host graphs using the results of this paper.
Moreover, the new developments on hypergraph designs in~\cite{GKLO,GKLO:Fs,Kdesign} raise the question whether one can also obtain such resolvable designs in the hypergraph setting.
A related challenge would be to extend some or all of the results of this paper to hypergraphs.

The Hamilton-Waterloo problem (which in general is wide open and generalizes the Oberwolfach problem)
asks for a decomposition of $K_n$ into $C_\ell$-factors and $C_{\ell'}$-factors,
where $\ell$ and $\ell'$ as well as the number of  cycle factors of a given type are given (and $n$ is odd).
Several special cases have received considerable attention, 
such as triangle factors versus Hamilton cycles (see e.g.~\cite{DL}).
Theorem~\ref{main 1b} immediately implies an approximate solution to this problem.
Moreover, we can combine Theorem~\ref{main 1b} with the results in~\cite{Kelly,KellyII} to obtain a general decomposition result 
into factors,
which solves the Hamilton-Waterloo problem if a significant proportion of the factors are Hamilton cycles.

\begin{cor}\label{Hamilton-Waterloo for many cycles}
For all $\Delta\in \mathbb{N}$ and $0<p_0,\beta \leq 1$ there exist $\epsilon >0$ and $n_0 \in \mathbb{N}$ such that
the following holds for all $n \ge n_0$ and $p\ge p_0$.
Suppose $r_\ell \le \Delta$ for all $\ell \in [s]$ and $\sum_{\ell=1}^{s} r_\ell=pn$.
Suppose $H_1,\dots, H_s$ are $n$-vertex graphs so that $H_\ell$ is $r_\ell$-regular for all $\ell \in [s]$ and that
$H_1,\dots,H_{\beta n}$ are Hamilton cycles.
Suppose further that $G$ is an $(\epsilon,p)$-quasi-random graph on $n$ vertices
which is $pn$-regular.
Then $G$ has a decomposition into $H_1,\dots, H_s$.
\end{cor}

\subsection{Organization of the paper}
In the next section, we give a sketch of the main ideas of the proof.
In Section \ref{sec:tools} we then collect the tools that we need later.
In Section \ref{sec:matching} we establish some properties of a typical matching in a super-regular bipartite graph. 
These will be used in Section \ref{sec:RR} in order to prove a `uniform blow-up lemma' (Lemma~\ref{modified blow-up}).
This lemma forms the core of the paper and embeds a `near-regular' graph $H$ into $G$ in a sufficiently `random-like' fashion.
In Section \ref{sec:main} we use Lemma~\ref{modified blow-up} in order to construct the desired packing
when each $H_\ell$ is near-regular (Theorem~\ref{main lemma}). 
In Section \ref{sec:stacking} we will reduce the problem of finding a packing of bounded degree graphs $H_1,\dots,H_s$
to the case when each $H_\ell$ is near-regular.
In Section \ref{sec:concl} we combine Theorem~\ref{main lemma} with the results from Section~\ref{sec:stacking}. 
In particular, we will deduce Theorems~\ref{main 1b}--\ref{thm:main3} as well as Corollary~\ref{Hamilton-Waterloo for many cycles}.
  
\section{Sketch of the method} \label{sec:sketch}

In this section, we highlight some of the main ideas of the argument towards Theorem~\ref{main 2}.
Suppose we are given a balanced $r$-partite graph $G$ with vertex classes $V_1,\dots, V_r$ of size $n$
such that the bipartite graph $G[V_i,V_j]$ is $(\epsilon,d)$-super-regular for all $1 \le i \neq j \le r$.
For simplicity, suppose $H$ is a balanced $r$-partite graph with vertex classes $X_1,\dots, X_r$ of size $n$, such that the bipartite graph $H[X_i,X_j]$
is $k$-regular for all $1 \le i \neq j \le r$ (where $k$ is a large constant).
Our aim is to find an approximate $H$-decomposition of $G$,
i.e.~a set of $(1-o(1))dn/k$ edge-disjoint copies of $H$ in $G$ so that $X_i$ is mapped to $V_i$ for all $i \in [r]$.

The classical blow-up lemma guarantees one such copy of $H$ in $G$. In fact, one can repeatedly apply the blow-up lemma to obtain 
$\epsilon^2 n$ (say) edge-disjoint copies of $H$ in $G$, but after this, we can no longer guarantee that the remaining subgraph of $G$
is $\epsilon'$-regular for small enough $\epsilon'$.

We overcome this as follows:
we will prove a `uniform blow-up lemma' which returns a random copy $\phi(H)$ of $H$ in $G$ so that
$\phi(H)$ behaves like a uniformly distributed random subgraph of $G$ (see Lemma~\ref{modified blow-up}). 
Then the hope would be that with high probability, the graph $G-\phi(H)$ is still $\epsilon$-regular.
One desirable  property towards this goal would be that for almost all edges $e$ of $G$, the probability that $e$ lies in $\phi(H)$
is close to $p:=e(H)/e(G) \sim k/(dn)$.
(Clearly, one cannot achieve this for \emph{every} edge of $G$, as for example $G$ may have an  edge which does not lie in 
a triangle.)

Towards this, we will use the characterization of $(\epsilon, d)$-super-regular graphs in terms of co-degrees:
for all $j \neq i$ and all vertices $u \in V_i$, we have that $|N_G(u) \cap V_j|=(d\pm \epsilon) n$ .
Also, for almost all pairs $u,v\in V_i$, the common neighbourhood $N_G(u,v)\cap V_j$  has size  $(d^2 \pm 3 \epsilon) n$.
Assume for simplicity that we have equality everywhere, i.e.~for all pairs $u,v \in V_i$, we have 
$|N_G(u,v) \cap V_j|=d^2 n$ for all $j \neq i$.
Assume also that for every edge of $G$ the probability that $e$ lies in $\phi(H)$ is exactly  $p=k/(dn)$, independently of all other edges.
Note that for any copy $\phi(H)$ in $G$, the bipartite pairs of $G-E(\phi(H))$ have density $d_1:=d-k/n$.
Our uniform blow-up lemma would then find a copy $\phi(H)$ of $H$ in $G$ where we expect that
\begin{equation} \label{sketch}
|N_{G-E(\phi(H))}(u,v)| = d^2 n(1-p)^2 =d_1^2 n.
\end{equation}
This is of course exactly what one would expect if $H$ were a completely random subgraph of $G$
(of the same density). It turns out that property (B1) of our uniform blow-up lemma can be used to
prove a surprisingly accurate approximation to the idealized formula in~(\ref{sketch}) (see Claim~\ref{succession}).

Removing $E(\phi(H))$ from $G$ now leaves a graph whose regularity parameters are much better than if we had removed this copy greedily.
So ideally one would keep applying this uniform blow-up lemma, maintaining (after the removal of $i$ copies)
an $(\epsilon_i,d_i)$-super-regular graph $G_i$ with $d_i=d-ik/n$ and $\epsilon_i \sim \epsilon \ll d_i$.
However, due to the fact that we need to allow $i$ to be linear in $n$,
this seems to be extremely challenging, if not infeasible.

So instead, we pursue a `nibble-based' approach: we remove our copies of $H$ in a large (but constant) number $T$ of `rounds'.
For each $t \in [T]$, at the beginning of Round $t$, let $G^t$ be the graph of currently available edges 
(i.e.~those which do not lie in a copy of $H$
selected in a previous round) and suppose that $G^t$ is $(\epsilon_t,d_t)$-super-regular for $\epsilon_t \ll d_t$.
Let $\gamma \sim d/(kT)$ (then $\gamma$ will be a small constant).
Within each round, we remove $\gamma n$ copies $\phi_1(H),\dots,\phi_{\gamma n}(H)$ 
of $H$ from $G^t$, each chosen randomly in $G^t$ (and independently of the other copies) 
according to the uniform blow-up lemma.
So the edge sets of $\phi_i(H)$ will usually have a small but significant overlap for different $i$.
On the other hand, if after Round $t$ we let $G^{t+1}$ be the graph obtained by removing the edges of all the $\phi_i(H)$, then one can show that
$G^{t+1}$ is still $(\epsilon_{t+1},d_{t+1})$-super-regular for $\epsilon_{t+1} \ll d_{t+1}$.
This means we can continue with a new embedding round for $G^{t+1}$.
It turns out that we can in fact carry on with this approach until we have the required number of copies $\phi_i(H)$.
Note that these copies of $H$ will be edge-disjoint if they are constructed in different rounds.

However, we still need to resolve overlaps between the edge sets of
the $\phi_i(H)$ which are constructed within the same round.
In other words, we need to modify  $\phi_i(H)$ into an embedding $\phi'_i(H)$ such that all the $\phi'_i(H)$ are pairwise edge-disjoint.
For this, call any edge of $\phi_i(H)$ which also belongs to some other $\phi_j(H)$ a conflict edge.
Let $W_i \subseteq V(G)$ be the vertices which have distance at most one to an endvertex of a conflict edge
in $\phi_i(H)$. For technical reasons, we also enlarge $W_i$ by adding a small random set of vertices.
Then we still have $|W_i|/n\ll 1$. Remove any edges from $\phi_i(H)$ which are incident to $W_i$.
We now patch up the resulting partial embeddings by using edges from a
sparse `patching graph' $P\subseteq G$ which we set aside at the beginning of the proof.
For  this patching process to work, we need each $\phi_i(H)$ to be `compatible' with $P$.
For instance, this means that if $w \in W_i$, if $v_1,v_2,v_3 \notin W_i$ 
and if $v_\ell w \in E(\phi_i(H))$ for each $\ell \in [3]$, then 
in the graph $P$,  $v_1,v_2,v_3$ need to have many common neighbours in $W_i$
(these are then potential candidates for the new image of $w$ in $\phi'_i(H)$).
This compatibility will already be ensured during the construction of $\phi_i(H)$ 
-- in the proof of the uniform  blow-up lemma, 
we will disregard any embeddings $\phi_i(H)$ which are not compatible with $P$.
Lemma~\ref{patching} formalizes the above description
(we will deduce it from the uniform blow-up lemma).

The core of the current paper is the uniform  blow-up lemma (Lemma~\ref{modified blow-up}) described above.
To prove this lemma, we develop an approach by R\"odl and Ruci\'nski~\cite{RR} (who designed it to give an alternative 
proof of the classical blow-up lemma).
We will find a copy of $H$ as the union of a bounded number of matchings.
For this, we first apply the Hajnal-Szemer\'{e}di theorem to the square $H^2$ of $H$ in order to
obtain a refined partition of $V(H)$ into classes $Y_i$, 
where the bipartite subgraph $H[Y_i,Y_j]$ is a (possibly empty) matching for each pair $Y_i,Y_j$ of classes.
We also find a corresponding partition of each $V_i$ into subclasses $U_j$.
We will embed each class $Y_j$ into $U_j$ in a single round.
For this, at each round, we have a `candidacy graph' $A_i^j$ with vertex classes $Y_j$ and $U_j$, where a vertex $x \in Y_j$
is adjacent to $v \in U_j$ if after Round $i$, $v$ is still a good candidate for $\phi(x)$.
Initially, we may take $A_0^j$ to be the complete bipartite graph, and as the embedding progresses, the candidacy graphs grow sparser,
as additional constraints come from the partial embedding.
For instance, if we embed a neighbour $y$ of $x$ in Round $i$, then the number of
candidates for $\phi(x)$ is expected to shrink by a factor of~$d$. In particular, we always  have $A_{i+1}^j \subseteq A_i^j$.
As indicated in the previous paragraph, we also make sure that only embeddings which are compatible with the patching
graph $P$ are permitted by the candidacy graph $A_{i+1}^j$.  
Crucially, we will be able to show that each $A_i^j$ is super-regular -- in particular, this means that  $A_i^j$ contains a perfect matching.
When embedding $Y_i$ in Round $i$, we will choose such a perfect matching $\sigma$ in  $A_{i-1}^i$ 
uniformly at random and embed $x \in Y_i$ to $v=\sigma(x)$.
The key difficulty in proving Lemma~\ref{modified blow-up} is in proving property (B1), 
which allows us to derive an approximation to~(\ref{sketch}).  
 
The above nibble process together with the patching is carried out in Section~\ref{sec:main}, 
to prove our main decomposition theorem (Theorem~\ref{main lemma}).
Theorem~\ref{main lemma} however assumes that $H$ has the property that each graph $H[X_i,X_j]$ is (extremely close to being) regular.
Of course, in general we do not want to assume that $H$ has this approximate regularity property.
So in Section~\ref{sec:stacking} we show that even if $H$ does not satisfy this approximate regularity property, then we can
pack together a large but bounded number of copies of $H$ (in a suitably 
random fashion) into a new graph $H'$ to which we can apply Theorem~\ref{main lemma}. 
In Section~\ref{sec:concl} we use this to deduce (from Theorem~\ref{main lemma}) several further results about packing arbitrary graphs 
$H$ of bounded degree.

\section{Notation and tools}\label{sec:tools}

\subsection{Notation}\label{subsec:notation}
For $s\in \mathbb{N}$ let $[s]:=\{1,\dots,s\}$. For a graph $G$ and an edge set $E$, let $G-E$ denote the graph $G'$ with $V(G')=V(G), E(G')= E(G)\setminus E$. Given another graph $H$ we let $G-H:=G-E(H)$.
As usual, $|G|$, $e(G)$ and $\Delta(G)$ will denote the number of vertices, edges and the maximum degree in $G$, respectively. Given a set $W\subseteq V = V(G)$, we let $N_G(W)=\bigcap_{v\in W}N_G(v)$, where $N_G(v)$ is the neighbourhood of $v$ in $G$. We write $G[W]$ for the induced subgraph of $G$ on $W$; 
when $G$ is $r$-partite with partition classes $V_1,\dots,V_r$ and $W_i\subseteq V_i$ for all $i\in [r]$ we also write $G[W_1,\dots,W_r]:=G[W_1\cup \dots \cup W_r]$.
We say an $rn$-vertex graph $G$ is \emph{$r$-equipartite} if $G$ admits an $r$-partition into independent sets $V_1,\dots, V_r$ with $|V_i|=n$.

Let $R$ be a graph on $[r]$. We say that a graph $G$ admits a vertex partition $(R,V_1,\dots, V_r)$ if $V_1,\dots, V_r$ form a partition of $V(G)$ into independent sets and $G[V_i,V_j]$ is empty if $ij \notin E(R)$. If $\mathcal{V}:=\{V_1,\dots,V_r\}$ then we also say that $G$ admits the vertex partition $(R,\mathcal{V})$. Given a symmetric $r\times r$ matrix $\vec{k}$ with entries $k_{i,j}\in \mathbb{N}$ and $C\in \mathbb{N}$, we say that $G$ is \emph{$(R,\vec{k},C)$-near-equiregular} with respect to $(V_1,\dots,V_r)$ if 
\begin{itemize}
\item $G$ admits the vertex partition $(R,V_1,\dots,V_r)$,
\item $\left||V_i|-|V_j|\right|\leq C$ for all $i\neq j \in [r]$,
\item for each $ij\in E(R)$ all vertices in $G[V_i,V_j]$ have degree $k_{i,j}$ except for at most $Ck_{i,j}$ vertices having degree $k_{i,j}+1$. 
\end{itemize}
If $R$ is the complete graph on $[r]$, $k_{i,j}=k$ for all $i\neq j\in [r]$ and $G$ is $(R,\vec{k},C)$-near-equiregular then we also say that $G$ is \emph{$(r,k,C)$-near-equiregular}.

We write $R_K$ for the $K$-fold blow-up of $R$. So if $R$ is a graph on $[r]$, then $R_K$ is the graph on $[Kr]$ obtained by replacing each vertex $i$ of $R$ by the set of vertices $\{(i-1)K+1,\dots, iK\}$ and replacing each edge of $R$ by a copy of $K_{K,K}$. We record the following observation for future reference, where $(R_K)^2$ denotes the square of $R_K$.
\begin{observation}\label{J scatter}
Suppose $r,K\in \mathbb{N}$ and $R$ is a graph on $[r]$ with $\delta(R)\geq 1$. Let $J_i$ be the set of vertices of $R_K$ which corresponds to the vertex $i$ of $R$. Let $I$ be an independent set of $(R_K)^2$. Then $|J_i\cap I| \leq 1$.
\end{observation}

For $a,b,c\in \mathbb{R}$ we write $a = b\pm c$ if $b-c \leq a \leq b+c$. Equations containing $\pm$ are always to be interpreted from left to right, e.g. $b_1\pm c_1=b_2\pm c_2$ means that $b_1-c_1\ge b_2-c_2$ and $b_1+c_1\le b_2+c_2$. In order to simplify the presentation, we omit floors and ceilings and treat large numbers as integers whenever this does not affect the argument. The constants in the hierarchies used to state our results have to be chosen from right to left. More precisely, if we claim that a result holds whenever $0 < 1/n \ll a \ll b \ll c \leq 1$ (where $n$ is typically the order of a graph), then this means that there are non-decreasing
functions $f^* : (0, 1] \rightarrow (0, 1]$, $g^* : (0, 1] \rightarrow (0, 1]$ and $h^* : (0, 1] \rightarrow (0, 1]$ such that the result holds for all $0 < a, b, c \leq1 $ and all $n \in \mathbb{N}$ with $b \leq f^*(c)$, $a \leq g^*(b)$ and $1/n \leq h^*(a)$. We will not calculate these functions explicitly. Hierarchies with more constants are
defined in a similar way.

The following functions $h',h,g',g,q_*,f,q$ will be used in the course of the proof. The values of $q_*, f, q$ depend on a further constant $w\in \mathbb{N}$, and we shall only use these after the value of $w$ has been fixed. Let  
\begin{align}\label{eq:functions}
&h'(a):= a^{1/10}, & &h(a):= a^{1/20}, & &g'(a) := a^{1/120}, & & g(a):= a^{1/300},\nonumber\\ 
&q_*(a):= a^{(1/300)^{w+1}}, & &f(a):= a^{(1/300)^{w+2}}, & &q(a):= a^{(1/300)^{w+3}}.
\end{align}
Here and elsewhere we denote by $g^b(a)$ the $b$-fold iteration of $g(a)$ (and similarly for the other functions). 
Note that for $a<1$ we have $h'(a)<h(a) < g'(a) < g(a) < q_*(a) < f(a) < q(a)$. 

\subsection{Probabilistic estimates}

We shall need the concentration inequalities of Azuma and Chernoff-Hoeffding. $X_0,\dots, X_N$ is a {\em martingale} if for all $n\in [N]$, $\mathbb{E}[X_{n}\mid X_0,\dots,X_{n-1}] = X_{n-1}$. We say it is {\em $c$-Lipschitz} if $|X_{n}-X_{n-1}| \leq c$ holds for all $n\in[N]$.

Our applications of Azuma's inequality will mostly involve \emph{exposure martingales} (also known as Doob martingales). These are martingales of the form $X_i:=\E[X\mid Y_1,\dots, Y_i]$, where $X$ and $Y_1,\dots,Y_i$ are some previously defined random variables. 

\begin{thm}[Azuma's inequality]\label{Azuma} 
Suppose that $\lambda, c >0$ and that $X_0,\dots, X_N$ is a $c$-Lipschitz martingale. 
Then 
\begin{align}
\Prob[\left|X_N-X_0\right|\geq \lambda]\leq 2e^{\frac{-\lambda^2}{2Nc^2}}.
\end{align}
\end{thm}

For $n\in \mathbb{N}$ and $0\leq p\leq 1$ we write $Bin(n,p)$ to denote the binomial distribution with parameters $n$ and $p$. For $m,n,N\in \mathbb{N}$ with $m,n<N$ the \emph{hypergeometric distribution} with parameters $N$, $n$ and $m$ is the distribution of the random variable $X$ defined as follows. Let $S$ be a random subset of $\{1,2, \dots, N\}$ of size $n$ and let $X:=|S\cap \{1,2,\dots, m\}|$. We will use the following bound, which is a simple form of Chernoff-Hoeffding's inequality.

\begin{lemma}[see {\cite[Remark 2.5 and Theorem 2.10]{JLR}}] \label{Chernoff Bounds}
Let $X\sim Bin(n,p)$ or let $X$ have a hypergeometric distribution with parameters $N,n,m$.
Then
$\mathbb{P}[|X - \mathbb{E}(X)| \geq t] \leq 2e^{-2t^2/n}$.
\end{lemma}

We shall need the following two inequalities for bounding tails of random variables in terms of the binomial distribution.

\begin{prop}[Jain, see {\cite[Lemma 8]{super-chernoff}}] \label{generalised-chernoff}
Let $B \sim Bin(n,p)$, and let $X_1, \ldots, X_n$ be Bernoulli random variables such that, for any $s \in [n]$ and any $x_1, \ldots, x_{s-1}\in \{0,1\}$ we have
\begin{align*}
\Prob\left[X_s = 1 \mid X_1 = x_1, \ldots, X_{s-1} = x_{s-1}\right] \leq p.
\end{align*}
Then $\Prob[\sum_{i=1}^{n} X_i \geq a] \leq \Prob[B \geq a]$ for any~$a\ge 0$.

\noindent
Likewise, if for any $s \in [n]$ and any $x_1, \ldots, x_{s-1}\in \{0,1\}$ we have
\begin{align*}
\Prob\left[X_s = 1 \mid X_1 = x_1, \ldots, X_{s-1} = x_{s-1}\right] \geq p,
\end{align*}
then $\Prob[\sum_{i=1}^{n} X_i \leq a] \leq \Prob[B \leq a]$ for any~$a\ge 0$.
\end{prop}

\subsection{Graph Theory tools}
In the preparation stages of our proof we shall apply the Hajnal-Szemer\'{e}di theorem. Given a set  $X$ of size $n$, an \emph{equitable $r$-partition} of $X$ is a partition of $X$ into sets of size $\lfloor n/r \rfloor$ and $\lceil n/r \rceil$ (note that the number of sets of each size is uniquely determined). An \emph{equitable $k$-colouring} of a graph $G$ is an equitable partition of $V(G)$ into $k$ independent sets. 
\begin{thm}[Hajnal-Szemer\'{e}di~\cite{HSz}]\label{thm:HS} 
Every graph $G$ with $\Delta(G)\leq k$ admits an equitable $(k+1)$-colouring. 
\end{thm}

Let $R$ be a graph on $[r]$ and suppose that $G$ is a graph with vertex partition $(R,V_1,\dots,V_r)$. Let $\vec{d}$ be a symmetric $r \times r$ matrix with entries $d_{i,j}$. We say that $G$ is \emph{$(\epsilon,\vec{d})$-super-regular with respect to $(R,V_1,\dots,V_r)$} if $G[V_i,V_j]$ is $(\epsilon,d_{i,j})$-super-regular whenever $ij\in E(R)$. We say that an $r$-partite graph $G$ with partition classes $V_1,\dots,V_r$ is \emph{$(\epsilon,d)$-super-regular with respect to $V_1,\dots,V_r$} if each $G[V_i,V_j]$ is  $(\epsilon,d)$-super-regular for all distinct $i,j\in [r]$.

The next statements are standard facts about graph regularity.
\begin{prop}\label{typical degree}
Suppose $G[A,B]$ is $(\epsilon,d)$-regular and $B'\subseteq B$ with $|B'|\geq \epsilon |B|$. Then all but at most $2\epsilon |A|$ vertices in $A$ have degree $(d\pm 2\epsilon)|B'|$ in $B'$.
\end{prop}
\begin{prop}\label{restriction}
Suppose $G[A,B]$ is $(\epsilon,d)$-regular, and $A'\subseteq A, B'\subseteq B$ with $|A'|/|A|,|B'|/|B| \geq \delta$. Then $G[A',B']$ is $({\epsilon}/{\delta},d)$-regular.
\end{prop}
\begin{prop}\label{regularity after edge deletion}
Let $k\ge 4$ and let $0<1/n\ll \epsilon \ll 1/k,d,1/(C+1)$. Suppose that $G[A,B]$ is $(\epsilon,d)$-super-regular with $|A|=|B|\pm C$ and $|A|,|B|\geq n$. If $F$ is a spanning subgraph of $G$ such that for each $v\in V(G)$, $d_G(v)-d_F(v)<k\epsilon n$, 
then $F$ is $(3\sqrt{k\epsilon}/2,d)$-super-regular. In particular, $F$ is $(k\sqrt{\epsilon},d)$-super-regular.
\end{prop}

The following lemma states that a super-regular graph can be, at the cost of increasing $\epsilon$, edge-decomposed into two sparser graphs, each of which is also super-regular.
\begin{lemma}\label{preparing patching graph}
Suppose $0<\beta\leq d\leq 1$ and $0<1/n \ll \epsilon \ll \beta, d, d-\beta$, and that $G[A,B]$ is an $(\epsilon,d)$-super-regular graph with $|A|,|B|\geq n$. 
Then there exists a spanning subgraph $P$ of $G$ such that $P$ is $(2\epsilon,\beta)$-super-regular and $G-P$ is $(2\epsilon,d-\beta)$-super-regular.
\end{lemma} 
\begin{proof}
\COMMENT{\begin{proof}
For each edge $e\in E(G)$, we select $e$ with probability ${\beta}/{d}$, all choices being independent. Then we let $P$ be the spanning subgraph formed by the selected edges. 

For any vertex $v\notin V_i$, let $B^i_{v}$ be the event that $|N_{P}(v) \cap V_i| \neq (\beta\pm 2\epsilon)n$. Then a vertex $w \in N_{G}(v)\cap V_i$ is in $N_{P}(v)\cap V_i$ with probability $\beta/d$.
Hence, 
$$\mathbb{E}[|N_{P}(v)\cap V_i|] = \left(\frac{\beta}{d}\right)|N_{G}(v)\cap V_i| = (\beta\pm\epsilon) n.$$
So by Chernoff, we can show that 
$$\mathbb{P}[B^i_{v}] < 2^{-\epsilon n/10} .$$

Also, for two sets $S_1\subset V_i, S_2\subset V_2$ with $|S_1|,|S_2|\geq \epsilon n$, let $B(S_1,S_2)$ be the event that $\frac{|E_P(S_1,S_2)|}{|S_1||S_2|} \neq \beta\pm 2\epsilon$.
Then an edge $uv \in E(S_1,S_2)$ is in $E_P(S_1,S_2)$ with probability $\beta/d$. Hence
$$\mathbb{E}[E_P(S_1,S_2)|] = (\beta\pm \epsilon)|S_1||S_2|.$$ 
So by Chernoff, we can show that
$$\mathbb{P}[B(S_1,S_2)] < 2^{-\epsilon^3 n^2/100} .$$

Therefore 
$$\mathbb{P}[ P \text{ is not }(2\epsilon,\beta)\text{-super-regular} ] \leq \sum_{1\leq i\leq r}\sum_{v\in V(G)\setminus V_i} \mathbb{P}[B^i_{v}] + \sum_{1\leq i\neq j\leq r} \sum_{S_1\in \binom{V_i}{\geq \epsilon n}, S_2 \in \binom{V_j}{\geq \epsilon n}}\mathbb{P}[B(S_1,S_2)] = c^n$$ for a constant $0<c<1$.

Thus with probability $1-c^n$ for a constant $0<c<1$, $P$ is $(2\epsilon,\beta)$-super-regular and $G-P$ is $(2\epsilon,d-\beta)$-super-regular.
Thus there exists a subgraph $P$ satisfying the required conditions.
\end{proof}}
For each edge $e\in E(G)$, we select $e$ with probability ${\beta}/{d}$, all choices being independent. Let $P$ be the spanning subgraph of $G$ formed by the selected edges. It is a straightforward exercise to check that with nonzero probability the above conditions are indeed satisfied.
\end{proof}

The following two statements establish a link between codegree and graph regularity. 
The first one is due to Duke, Lefmann and R\"odl~\cite{DLR} (a similar result is proved in~\cite{Aetal}), the converse provided by 
Proposition~\ref{regularity implies codegree} follows immediately from the definitions.
\begin{thm} \cite{DLR} \label{codegree implies regularity}
Suppose $0<\epsilon <2^{-200}$. Suppose $G=G[A,B]$ is a bipartite graph with $|A|>2/\epsilon$, and let $d:=\frac{e(G)}{|A||B|}$. Let $D$ be the collection of all pairs $\{x,x'\}$ of vertices of $A$ for which 

\begin{itemize}
\item[(i)]$d(x), d(x') > (d-\epsilon)|B|$, and 

\item[(ii)]$|N_G(x)\cap N_G(x')| < (d+\epsilon)^2|B|$.
\end{itemize}
Then if $|D|> \frac{1}{2}(1-5\epsilon)|A|^2$, the graph $G$ is $(\epsilon^{1/6},d)$-regular.
\end{thm}

\begin{prop}\label{regularity implies codegree}
\COMMENT{$(d\pm \epsilon)^2 = d^2 \pm 3\epsilon$.
Given $x\in A$ we have $|B'|=|N_G(x)|=(d\pm \epsilon)|B|$ by super-regularity. Let $A'\subseteq A$ be the set of all vertices with more than $(d+\epsilon)^2|B|$ or less than $(d-\epsilon)^2|B|$ joint neighbours with $x$. Then $e(A',B')/|A'||B'|\neq d\pm \epsilon$, thus we must have $|A'|<2\epsilon |A|$. So we have at most $2\epsilon |A|$ 'bad' pairs for any given vertex $x\in A$, yielding at most $\epsilon |A|^2$ such pairs in total}
Suppose $0<1/n\ll \epsilon \ll d$ and that $G[A,B]$ is $(\epsilon,d)$-super-regular with $|A|,|B|\geq n$. Then all but at most $\epsilon|A|^2$ vertex pairs $\{x,x'\}\subseteq A$ satisfy $|N_G(x)\cap N_G(x')| = (d^2\pm 3\epsilon)|B|$.
\end{prop}

The following observation states that if we delete some edges incident to a small number of vertices from a super-regular graph $G$, then the resulting graph still contains a 
dense super-regular graph $G'$.

\begin{lemma}\label{lem: deletion of some edges still contain super-regular}
Suppose $0<1/n \ll \epsilon \ll d_0, d$ and $G[A,B]$ is $(\epsilon,d)$-super-regular with $|A|=|B|=n$. Let $A'\subseteq A$ with $|A'| \leq \epsilon n$.
Suppose that for each $v\in A'$, we have a set $B'(v)\subseteq N_{G}(v)$ with $|B'(v)| \geq d_0 n.$ Then there exists an $(\epsilon^{1/3},d_0 )$-super-regular spanning subgraph $G'$ of $G$ such that for all $v\in A'$, we have $N_{G'}(v) \subseteq B'(v)$.
\end{lemma}
\proof
Note that the conditions imply that $d_0 \le d+\epsilon$.
Without loss of generality, we may further assume that $d_0 \le d-\epsilon$.
Now consider the subgraph $H$ of $G$ such that $N_{H}(v)= B'(v)$ for all $v\in A'$ and $N_{H}(v)=N_{G}(v)$ for all $v\in A\setminus A'$.
Then a random subgraph $G'$ of $H$ such that $d_{G'}(v) = d_0 n$ for all $v\in A$ is $(\epsilon^{1/3},d_0)$-super-regular with high probability.%
\COMMENT{Consider $U,W$ with $|U|,|W| \ge \epsilon^{1/3} n$.
Then 
$$
\mathbb{E}  [ e_{G'}(U,W)] \ge \frac{d_0}{d+\epsilon}e_G(U,W) -|A'| n \ge \frac{d-\epsilon}{d+\epsilon} d_0 |U| |W| -\epsilon n^2
$$}
\endproof


\section{Random Perfect Matchings} \label{sec:matching}

It is well-known that every $(\epsilon,d)$-super-regular balanced bipartite graph $G$ contains a perfect matching. In this section we show that the number of the perfect matchings containing a given edge $e$ is roughly the same for all $e\in E(G)$, i.e. each edge has the same likelihood of appearing in a random perfect matching.

Given a bipartite graph $G[U, W]$, we write $M(G)$ for the set of all perfect matchings of $G$, $M_e(G)$ for the set of all perfect matchings containing a given edge $e\in E(G)$, and set $M'_e(G):=M(G)\setminus M_e(G)$. 
For $\sigma \in M(G)$, we often abuse the notion and think of $\sigma$ as a bijection from $U$ to $W$ (where $\sigma(u)=w$ if and only
if $uw$ is an edge in the perfect matching $\sigma$). 

We will use the following result of R\"{o}dl and Ruci\'{n}ski~\cite{RR}, which implies that the edges of a random perfect matching are uniformly distributed with respect to large sets of vertices. 

\begin{thm} \cite{RR} \label{M0}
Suppose $0<1/n \ll {c} \ll \epsilon \ll d\leq 1$ and $h'(a)=a^{1/10}$ is as defined in $(\ref{eq:functions})$.
Let $G[U,W]$ be an $(\epsilon,d)$-super-regular bipartite graph with $|U|=|W|=n$ and let $S\subseteq U, T\subseteq W$ with
$|S|=sn, |T|=tn$ and $s,t\geq h'(\epsilon)$.%
  \COMMENT{the original version states $|\sigma(S)\cap T | = (st\pm h'(\epsilon))n$ without the assumption
that $s,t\geq h'(\epsilon)$ and with $h'(\epsilon)=a^{1/3}$ or so. Adding the lower bound on $s,t$ yields
the multiplicative version $|\sigma(S)\cap T | = (1\pm h'(\epsilon))stn$}
Then at least $(1-(1-{c})^n)|M(G)|$ perfect matchings $\sigma$ of $G$ satisfy
$|\sigma(S)\cap T | = (1\pm h'(\epsilon))stn$.
\end{thm}

We also use the following result of Alon, R\"odl and Ruci\'nski~\cite{ARR} on the number of perfect matchings in (super-)regular graphs (which was also a tool in the proof of Theorem~\ref{M0}).

\begin{thm}\cite{ARR}\label{thm for M3}
Let $0<1/n\ll\epsilon \ll d\leq 1$, and let $G[U,W]$ be an $(\epsilon,d)$-super-regular bipartite graph with $|U|=|W|=n$. Then 
$$ (d-2\epsilon)^n n! \leq|M(G)|\leq (d+2\epsilon)^n n!.$$
If $G[U,W]$ is $(\epsilon,d)$-regular bipartite graph, then 
$$|M(G)| \leq (d+3\epsilon)^n n!.$$
\end{thm}

We now use Theorem \ref{M0} to prove a `localised' version of it. 

\begin{thm}\label{MM}
Suppose $0<1/n \ll \epsilon \ll d\leq 1$ and $h(a)=a^{1/20}$ is as defined in $(\ref{eq:functions})$.
Let $G[U,W]$ be an $(\epsilon,d)$-super-regular bipartite graph with $|U|=|W|=n$.
Then 
$$\frac{|M_e(G)|}{|M(G)|} = (1\pm h(\epsilon)) \frac{1}{dn}. $$
In other words, if we choose a perfect matching $\sigma$ of $G$ uniformly at random, then for any edge $uv\in E(G)$ with $u\in U$ and $v\in W$,
$$\mathbb{P}[ \sigma(u)=v ] = (1\pm h(\epsilon)) \frac{1}{dn}.$$ 
\end{thm}
\begin{proof}
For $\sigma\in M(G)$ and distinct vertices $u_1,u_2,u_3\in U$ such that
$\sigma(u_1)u_2, \sigma(u_2)u_3, \sigma(u_3)u_1 \in E(G)$, we define the {\em $(u_1,u_2,u_3)$-switch} $S_{u_1,u_2,u_3}(\sigma)$
of the matching $\sigma$ to be a new matching $\sigma'$ where
$$\sigma'(u) :=  \left\{ \begin{array}{ll}
\sigma(u) & \text{ if }u\notin \{u_1,u_2,u_3\},\\
\sigma(u_3) & \text{ if } u=u_1, \\
\sigma(u_1) & \text{ if } u=u_2,\\
\sigma(u_2) & \text{ if } u=u_3.\\
\end{array} \right.$$
If $\sigma'\in M(G)$ is the $(u_1,u_2,u_3)$-switch of $\sigma\in M(G)$ for some $u_1,u_2,u_3\in U$, we also say that
$\sigma'$ is a \emph{switch of $\sigma$}.

Our aim is to estimate $|M_e(G)|/ |M'_e(G)|$ for a given edge $e=uv$. To do so, we consider the auxiliary bipartite graph
$H$ with bipartition $A_1:=M_e(G)$, $A_2:=M'_e(G)$ such that $\sigma \sigma' \in E(H)$ if and
only if $\sigma\in A_1, \sigma'\in A_2$ and $\sigma'$ is a switch of $\sigma$. So $V(H)=M(G)$.
Let
\begin{align*}
A' := \{\sigma\in M(G) : \sigma(N_G(v'))\cap N_G(u') = (d^2 \pm 2h'(\epsilon))n \text{ for all } u'\in U, v'\in W\}.
\end{align*}
Choose a new constant $c$ such that $1/n \ll c \ll \epsilon$. Since $G$ is $(\epsilon,d)$-super-regular, $|N_G(u')| = (d\pm \epsilon)n$ for all $u'\in V(G)$. This together with Theorem \ref{M0} implies
that
\begin{align}\label{U' size}
|A'| \geq (1-n^2 (1-{c})^n)|V(H)|.
\end{align}
Note that for all $\sigma_1 \in A_1$ 
\begin{align}\label{degree def}
d_H(\sigma_1) = |\{ S_{u,u_1,u_2}(\sigma_1) : v u_1,\sigma_1(u_1)u_2, \sigma_1(u_2)u \in E(G), u_1,u_2 \in U\setminus \{u\} \text{ with } u_1\neq u_2\}|,
 \end{align}
and for all $\sigma_2 \in A_2$ 
\begin{align}\label{degree def2} 
d_H(\sigma_2) = |\{ S_{u,u_1,\sigma^{-1}_2(v)}(\sigma_2) :   u_1\sigma_2(u),\sigma_2(u_1)\sigma^{-1}_2(v)  \in E(G), u_1 \in U\setminus \{u,\sigma_2^{-1}(v)\}\}|.
 \end{align}
Thus for any $\sigma_1 \in A_1, \sigma_2 \in A_2$, we have
\begin{align}\label{degree U}
d_H(\sigma_1) \leq  n^2 \ \ \ \ \text{and}\ \ \ \ d_H(\sigma_2) \leq  n. 
 \end{align}
Let us now estimate the degree of vertices of $A'$ in the graph $H$.  
For $\sigma_1 \in A_1$, 
$$d_H(\sigma_1) \stackrel{(\ref{degree def})}{=} |\{(u_1,u_2)\in (U\setminus\{u\})\times (U\setminus \{u\}): vu_1, \sigma_1(u_1)u_2, \sigma_1(u_2)u \in E(G), u_1\neq u_2\}|.$$
Since $u_1 \in N_{G}(v)$ we have $ |N_{G}(v)| = (d\pm \epsilon) n$ choices for $u_1$.
Once $u_1$ is fixed, we have to choose $u_2$ such that $u_2 \in N_{G}(\sigma_1(u_1))$ and $\sigma_1(u_2) \in N_{G}(u)$.
So if $\sigma_1 \in A'$, there are $(d^2 \pm 2h'(\epsilon)) n$ choices for $u_2$ once $u_1$ is fixed.
Hence we obtain that
\begin{align}\label{deg 1}
\text{if }\sigma_1 \in A_1\cap A', \text{ then } d_H(\sigma_1) = (d\pm \epsilon) n \cdot (d^2 \pm 2h'(\epsilon) ) n =  (d^3 \pm 3h'(\epsilon))n^2.
\end{align}
For $\sigma_2 \in A_2$, 
$$d_H(\sigma_2) \stackrel{(\ref{degree def2})}{=} |\{u_1\in U\setminus \{u,\sigma_2^{-1}(v)\}: u_1\sigma_2(u), \sigma_2(u_1) \sigma_2^{-1}(v) \in E(G) \}|.$$
Thus we count the number of choices of $u_1$ such that $u_1 \in N_G(\sigma_2(u))$ and $\sigma_2(u_1) \in N_G(\sigma^{-1}_2(v))$.
Similarly as before, if $\sigma_2 \in A'$, there are $(d^2 \pm 2h'(\epsilon)) n$ choices for~$u_1$. 
Hence we obtain that 
\begin{align}\label{deg 2}
\text{if }\sigma_2 \in A_2\cap A',\text{ then } d_H(\sigma_2) = (d^2 \pm 2h'(\epsilon))n.
\end{align}
Thus
\begin{eqnarray*}
|E(H)| &=& \sum_{\sigma_1\in A_1\cap A'} d_H(\sigma_1) + \sum_{\sigma_1\in A_1\setminus A'} d_H(\sigma_1)  
\stackrel{(\ref{degree U}),(\ref{deg 1})}{=} (d^3 \pm 3h'(\epsilon))n^2 |A_1\cap A'| \pm n^2 |A_1\setminus A'|\nonumber \\
& = & (d^3 \pm 3h'(\epsilon))n^2 |A_1| \pm 2n^2 |A_1\setminus A'| 
\stackrel{(\ref{U' size}) }{=} (d^3 \pm 3h'(\epsilon))n^2 |A_1| \pm 2n^4 (1-{c})^n |V(H)|,
\end{eqnarray*}
and
\begin{eqnarray*}
|E(H)|&=& \sum_{\sigma_2\in A_2\cap A'} d_H(\sigma_2) + \sum_{\sigma_2\in A_2\setminus A'} d_H(\sigma_2) 
\stackrel{(\ref{degree U}),(\ref{deg 2})}{=} (d^2 \pm 2h'(\epsilon))n |A_2\cap A'| \pm n |A_2\setminus A'| \nonumber\\
& \stackrel{(\ref{U' size}) }{=} & (d^2 \pm 2h'(\epsilon))n |A_2| \pm 2 n^3 (1-{c})^n |V(H)|.
\end{eqnarray*}
This in turn implies that
\begin{align*}
(d^3 \pm 3h'(\epsilon))n^2 |A_1| = (d^2 \pm 2h'(\epsilon))n |A_2| \pm 3n^4 (1-{c})^n (|A_1|+|A_2|).
\end{align*}
Since $1/n \ll {c}$, we conclude that 
$|A_2| = (d\pm 7h'(\epsilon)/d^2) n |A_1|$.
Since $h(\epsilon) = \sqrt{h'(\epsilon)}$ by~(\ref{eq:functions}), this in turn implies that
$$\frac{|M_e(G)|}{|M(G)|}=\frac{|A_1|}{|A_1|+|A_2|} = 
\frac{|A_1|}{|A_1|+ (d\pm 7h'(\epsilon)/d^2) n |A_1|}   = (1\pm h(\epsilon)) \frac{1}{dn}.$$
\end{proof}

From Theorems~\ref{M0}--\ref{MM} we now deduce further properties of a random perfect matching which we will 
make frequent use of in Section~\ref{sec:RR}.

\begin{lem}\label{lem:matchings}
Suppose $0< 1/n\ll \epsilon \ll d'<d\leq 1$ and $0<c\ll d'\leq d/9$.\COMMENT{What we use in the proof is that $8d'/d\leq 1-\sqrt{\epsilon}$.} Let $h(a)=a^{1/20}$ be as defined in $(\ref{eq:functions})$
and let $G[U,W]$ be an $(\epsilon,d)$-super-regular bipartite graph with $|U|=|W|=n$.
Then a perfect matching $\sigma$ of $G$ chosen uniformly at random satisfies the following.
\begin{itemize}

\item[(M1)] For every $u\in U$ and every $S\subseteq N_G(u)\subseteq W$ with $|S|=sn$,
$$
\Prob\left[\sigma(u)\in S\right]=(1\pm h(\epsilon))\frac{s}{d} .
$$
\item[(M2)] 
For every $w\in W$ and every $S\subseteq N_G(w)\subseteq U$ with $|S|=sn$,
$$
\Prob\left[w\in \sigma(S)\right]=(1\pm  h(\epsilon))\frac{s}{d} .
$$
\item[(M3)] 
Let $G'\subseteq G$ be a subgraph of $G$ with $V(G')=V(G)$ such that $\Delta(G')\leq d'n$. Then 
$$\Prob\left[|\{u\in U: u\sigma(u) \in E(G') \}| >  \frac{8d'}{d}n\right] < (1-c)^n .$$
\end{itemize}
\end{lem}
\begin{proof}
To show (M1) and (M2), we use Theorem \ref{MM} to conclude that
$$\mathbb{P}[\sigma(u)\in S]= \sum_{w\in S} \mathbb{P}[ \sigma(u)=w] = (1\pm h(\epsilon)) \frac{|S|}{dn},$$
and 
$$\mathbb{P}[w \in \sigma(S)]= \sum_{u\in S} \mathbb{P}[ \sigma(u)=w] = (1\pm h(\epsilon)) \frac{|S|}{dn}.$$

To prove (M3), 
let $U_\sigma := \{ u\in U: u\sigma(u) \in E(G')\}$ and let $m:=8 d'n/d$.
For given $U'\subseteq U$ with $|U'| = m$, we shall now find an upper bound for the number of perfect matchings $\sigma$ in $G$ such that $U'\subseteq U_{\sigma}$.
For each vertex $u\in U'$, there are at most $d'n$ ways to choose $\sigma(u)$ such that $u\sigma(u)\in E(G')$. After having chosen $\sigma(u)$ for all $u\in U'$, by Proposition~\ref{restriction},
we are left with a $(\sqrt{\epsilon},d)$-regular graph $G[U\setminus U',W\setminus\sigma(U')]$, which contains at most $(d+3\sqrt{\epsilon})^{n-m} (n-m)!$
perfect matchings by Theorem \ref{thm for M3}. Thus for any set $U'\subseteq U$ with $|U'|= m$, the number of perfect matchings $\sigma$ in $G$ which satisfy $U'\subseteq U_{\sigma}$ is at most 
$(d'n)^{m} (d+3\sqrt{\epsilon})^{n-m} (n- m)!.$
Since there are at most $\binom{n}{m}$ choices for $U'$, the number of perfect matchings $\sigma$ with $|U_\sigma| \geq m$ is at most 
$$\binom{n}{m} (d'n)^{m} (d+3\sqrt{\epsilon})^{n-m} (n- m)!.$$
By Theorem \ref{thm for M3}, we know $|M(G)| \geq (d-2\epsilon)^n n!$.
 Thus, 
\begin{align*}
\Prob\left[|\{u\in U: u\sigma(u) \in E(G') \}| >  \frac{8d'}{d}n\right]  &\leq 
\frac{ \binom{n}{m} (d'n)^{m} (d+3\sqrt{\epsilon})^{n-m} (n- m)!}{(d-2\epsilon)^n n!} \\
&= \left(\frac{d+3\sqrt{\epsilon}}{d-2\epsilon}\right)^{n-m}\frac{(d'n)^{m} }{ (d-2\epsilon)^m m!} \\
& \leq (1 + \epsilon^{1/3})^n \frac{d^m (m/8)^m}{(d/2)^m (m/3)^m}\\
& \leq (1 + \epsilon^{1/3})^n \left(\frac{3}{4}\right)^{8d'n/d} \leq (1-c)^n.
\end{align*}
To obtain the third line we use that $8d'n = dm$ and $m! > (m/3)^m$ for large $m$.
To obtain the final inequality we use that $\epsilon,c \ll d', d$.
\end{proof}




\section{A uniform blow-up lemma}\label{sec:RR}
\subsection{Statement and discussion}
Our goal in this section is to establish a probabilistic version of the classical blow-up lemma, which finds a `uniformly distributed' copy of $H$ in $G$. While the classical version asserts that, informally, a super-regular graph $G$ contains a copy of any bounded degree graph $H$, here we show that such a copy of $H$ can be selected to have some additional `random-like' properties. Ideally this would take a similar form as Theorem~\ref{MM}: for a randomly chosen copy of $H$ in $G$, each edge of $G$ has the same probability of being in $H$. Unfortunately this is false. (Suppose $H$ is a triangle factor and that $e\in E(G)$ does not lie in a triangle.) However we will be able to show that a suitable randomised algorithm gives a copy of $H$ in $G$ which is randomly distributed in a very strong sense (see Lemma \ref{modified blow-up}).

First we need the following definitions. Let $\phi$ be an embedding of a graph $H$ into a graph $G$ and let $G'\subseteq G$. When applying Lemma~\ref{modified blow-up}, $G'$ will be a `bad' graph which we would like to avoid as much as possible when embedding $H$. We define
\begin{align}
\begin{split}
&\phi_E(H,G,G'):=\phi(E(H))\cap E(G'), \\
&\phi(H,G,G'):= \{v\in V(G): \text{ there exists } e\in \phi_E(H,G,G') \text{ such that }v\in e\}, \\
&\phi_2(H,G,G'):=\{v\in V(G): \text{ there exists } u\in V(G) \text{ such that } uv\in \phi(E(H)), u\in \phi(H,G,G')\}.
\end{split}
\end{align}
Note that $\phi(H,G,G')\subseteq \phi_2(H,G,G')$.

We would also like our embedding of $H$ to be well behaved with respect to additional given graphs $A_0$ and $P$ (where $A_0$ encodes the initial candidates for the images of each vertex of $H$, and $P$ is the patching graph which will be used to adjust the embeddings later on). To formalise this, let $R$ be a graph on $[r]$.
Let $H$ be a graph and let $\mathcal{X}=\{X_1,\dots, X_{r}\}$ be a partition of $V(H)$. We say that a (not necessarily uniform) hypergraph $N$ with vertex set $V(H)= X_1\cup \dots \cup X_{r}$ and edge set $\{N_x : x\in V(H)\}$ is an {\em $(H,R,\mathcal{X})$-candidacy hypergraph} if all $x,y\in V(H)$ satisfy the following properties:
\begin{itemize}
\item if $x\in X_i$, then $|N_x\cap X_{j}|\leq 1$ for any $j\in [r]$, and $|N_x\cap X_{j}|=0$ if $j\notin N_R(i)$,
\item $N_H(x)\subseteq N_x$, 
\item $y\in N_x$ if and only if $x\in N_y$.
\end{itemize}

Notice that the above definition is only applicable to graphs $H$ which admit the vertex partition $(R,X_1,\dots, X_r)$ with $\Delta(H[X_i,X_j])\leq 1$ for all $i,j \in [r]$ (and thus $\Delta(H)\leq \Delta(R)$).
Let $P$ be a graph admitting vertex partition $(R,\mathcal{V})$ where $\mathcal{V}=( V_1,\dots, V_{r})$. Suppose that $|X_i|=|V_i|$ for all $i\in [r]$ and that $\phi: V(H)\rightarrow V(P)$ is a bijection, mapping each $X_i$ to $V_i$. Let $A_0$ be a graph on $V(P)\cup V(H)$.
We call a bipartite graph $F$ on  $V(F)=(V(H),V(P))$ an \emph{$(H,P,R,A_0,\phi,\mathcal{X},\mathcal{V},N)$-candidacy bigraph} if it satisfies the following conditions.
\begin{itemize}
\item [(CB1)] $N$ is an $(H,R,\mathcal{X})$-candidacy hypergraph.
\item [(CB2)] $N_{F}(x) \subseteq N_{A_0}(x) \cap \bigcap_{y\in N_x} N_{P}(\phi(y))\cap V_i$ for all $x\in X_i$.
\end{itemize}
Note that in particular $E(F[X_i,V_j])=\emptyset$ for all $i\neq j$. A candidacy bigraph $F$ will always encode permissible images for embeddings in the patching step,
i.e.~in the patching step of the proof of Theorem~\ref{main lemma} we may only embed $x$ to $v$ if $xv\in E(F)$.

For a given graph $R$ on $[r]$ 
and a symmetric $r\times r$ matrix $\vec{\beta}$ with entries $\beta_{i,j}$, we denote 
\begin{align}\label{p R def}
p(R,\vec{\beta},i):= \prod_{\ell \in N_R(i)} \beta_{i,\ell}.
\end{align}
This will be convenient for measuring the densities of the bipartite subgraphs of the candidacy bigraph $F$.
Recall that $R_K$ denotes the $K$-fold blow-up of $R$.
Now we can state our main result in this section.

\begin{lemma}[Uniform blow-up lemma]\label{modified blow-up}
Suppose 
$$0< 1/n \ll c \ll \epsilon\ll \gamma \ll \beta,d,d_0,1/k,1/\Delta_R, 1/(C+1) \text{ and } 1/n \ll 1/r.$$ Let $K:=(k+1)^2\Delta_R$,
let $w:=K^2\Delta_R^2(\Delta_R+1)$ and let $f$ be the function defined in \eqref{eq:functions}. 
 Suppose that $R$ is a graph on $[r]$ with $\Delta(R)=\Delta_R$. Let $\vec{d}, \vec{\beta},\vec{k}$ be symmetric $r \times r$ matrices such that $d=\min_{ij\in E(R)} d_{i,j} , \beta=\min_{ij\in E(R)}\beta_{i,j}, k=\max_{ij \in E(R)} k_{i,j}$ and $k_{i,j} \in \mathbb{N}$ for all $ij\in E(R)$. Suppose that the following hold.
\begin{itemize}
\item $G$ is an $(\epsilon, \vec{d})$-super-regular graph with respect to $(R,\mathcal{V})$, where $\mathcal{V}=(V_1,\dots, V_r)$,
$\max_{i \in [r]} |V_i|=n$ and $n -C \leq |V_i| \leq n$ for all $i\in [r]$. 
\item $P$ is an $(\epsilon,\vec{\beta})$-super-regular graph with respect to $(R,\mathcal{V})$.
\item $H$ is an $(R,\vec{k},C)$-near-equiregular graph admitting the vertex partition $(R,\mathcal{X})$ with $\mathcal{X}=(X_1,\dots, X_r)$ where $|X_i|=|V_i|$. 
\item $A_0$ is a bipartite graph with bipartition $(V(H),V(G))$ such that $N_{A_0}(X_i)=V_i$ and $A_0[X_i,V_i]$ is $(\epsilon,d_0)$-super-regular for each $i\in [r]$.
\end{itemize}
Then there exists a randomised algorithm (the `uniform embedding algorithm') which succeeds with probability at least $1-(1-c)^n$ in finding an embedding $\phi$ of $H$ into $G$ such that $\phi(X_i)=V_i$ for each $i\in [r]$ and $\phi(x) \in N_{A_0}(x)$ for each $x\in V(H)$. Conditional on being successful, this algorithm returns $(\phi,\mathcal{Y}, \mathcal{U}, F,N)$ with the following properties. 
\begin{itemize}
\item[(B1)] For all $ij\in E(R)$, $v\in V_i$ and $S\subseteq V_j\cap N_G(v)$ with $|S| > f(\epsilon) n$, 
$$\mathbb{E}\left[ |N_{\phi(H)}(v)\cap S| \right] = (1\pm f(\epsilon)) \frac{k_{i,j}|S|}{d_{i,j}n}. $$
\item[(B2)] $\mathcal{U}= \{U_1,\dots, U_{Kr}\}$ is a partition refining $V_1,\dots, V_r$ and $\mathcal{Y}= \{Y_1,\dots, Y_{Kr}\}$ is a partition refining $X_1,\dots, X_r$ such that  $U_i\subseteq V_{\lceil i/K \rceil}$, $Y_i\subseteq X_{\lceil i/k \rceil}$, $|Y_i|=|U_i|$,
$\max_{i \in [Kr]}|Y_i|=\lceil n/K \rceil$ and $|Y_i|-|Y_j| \leq C$ for all $i\neq j$. 
$F$ is a disjoint union of bipartite graphs $F_1,\dots,F_{Kr}$ such that each $F_j$ has bipartition $(Y_j,U_j)$. $N$ is a hypergraph with vertex set $V(H)$ and a hyperedge $N_x$ for each $x\in V(H)$.  Moreover, the following conditions hold. 
\begin{itemize}
\item[(B2.1)] $N$ is an $(H,R_K,\mathcal{Y})$-candidacy hypergraph with $\max\{|N_x|: N_x\in N\} \leq K\Delta_R$. 
\item[(B2.2)] For all $j\in [Kr]$ and $x\in Y_j$, $N_{F_j}(x)\subseteq U_j\cap N_{A_0}(x)\cap \bigcap_{y\in N_x} N_{P}(\phi(y))$, thus $F$ is a $(H,P,R_K,A_0,\phi,\mathcal{Y},\mathcal{U},N)$-candidacy bigraph.
\item[(B2.3)] $P$ is $(\epsilon^{1/3},\vec{\beta'})$-super-regular with respect to $(R_K,\mathcal{U})$, where $\vec{\beta'}$ is the symmetric $Kr\times Kr$ matrix with entries $\beta'_{\ell,\ell'} :=\beta_{i,j}$ whenever $i=\lceil \ell/K \rceil$ and $j=\lceil \ell'/K \rceil.$
\item[(B2.4)] For each $j\in [Kr]$ the graph $F_j$ is $(f(\epsilon),d_0 p(R_K,\vec{\beta'},j))$-super-regular.
\end{itemize}

\item[(B3)] For all $u\neq v\in V(G)$ and $Z\subseteq V(H)$ with $|Z|\leq \gamma^3 n$,
\begin{itemize}
\item[(B3.1)] 
$\mathbb{P}[ N_H(\phi^{-1}(u))\cap N_H(\phi^{-1}(v)) \neq\emptyset] \leq 1/\sqrt{n}.$
\item[(B3.2)]  $\mathbb{P}[ \phi^{-1}(v) \in Z] \leq \gamma^{2}$. 
\end{itemize}

\item[(B4)] Let $G''$ be a subgraph of $G$ with $V(G'')=V(G)$ such that $\Delta(G'')\leq \gamma n$. 
\begin{itemize}
\item[(B4.1)] For any vertex $v\in V(G)$, $\mathbb{P}[v\in \phi_2(H,G,G'')] \leq \gamma^{1/2}$.
\item[(B4.2)] $\mathbb{P}[|\{u : u \in U_{j}, u\in \phi_2(H,G,G'') \}| \leq \gamma^{3/5} n \text{ for all }j \in [Kr]] \geq 1-(1-2 c)^n.$
\end{itemize}

\item[(B5)]
 Suppose $v_1,\dots, v_s\in V_i$ where $i\in [r]$ with $s\leq K$.
For all $j\in N_{R}(i)$ and $v\in V_j$, let $B_{v}$ be the random variable such that 
$$
B_{v}:= \left\{\begin{array}{ll}
1 & \text{ if there exists }\ell \in [s] \text{ with } v_\ell v \in \phi(E(H)), \\
0 & \text{ otherwise.} 
\end{array}\right.
$$
Then for all $j\in N_R(i)$ and all but at most $2 f(\epsilon) n$ vertices $v\in N_{G}(v_1,\dots, v_s) \cap V_j$ we have that 
$$\mathbb{P}[B_{v}=1] = \left(1\pm 2f(\epsilon)\right)\frac{k_{i,j}s}{d_{i,j} n}.
$$

\item[(B6)]
For all $i\in[r]$ and all sets $Q\subseteq X_i$ and $W\subseteq V_i$ with $|Q|,|W|> f(\epsilon) n$, $$\mathbb{P}\left[|\phi(Q)\cap W| = \frac{(1\pm f(\epsilon))|Q||W|}{n} \right] \geq 1-(1-2c)^n.$$
\end{itemize}

\end{lemma}

Note that it is property (B1) that most intuitively encapsulates the `randomness' of the embedding $\phi$. Informally speaking, it says that a sufficiently large set $S$ in the neighbourhood of $v$ will contain approximately as many neighbours of $v$ in the embedded copy of $H$ as expected. 
As mentioned earlier, the graphs $F_i$ in (B2) can be viewed as `candidacy graphs'. If they are super-regular, this means that the embedding $\phi$ is well behaved with respect to the patching graph $P$. We will need this when replacing some edges of $\phi(E(H))$ by edges of $P$. (B3.1) says that for any pair of vertices of $G$, the preimages are unlikely to be joined by a path of length two in $H$. (B3.2) says that for a very small subset $Z$ of $V(H)$ and a vertex $v$ in $G$, $v$ is unlikely to be contained in $\phi(Z)$. When applying Lemma \ref{modified blow-up} in Section~\ref{sec:main}, $G''$ in (B4) will play the role of the union of certain previously embedded graphs. (B4) shows that the overlap of these previous embeddings with the new embedding $\phi(H)$ is small.
 Note in (B5), $N_{\phi(H)}(v)\cap \{v_1,\dots, v_s\} \neq \emptyset$ if and only if $B_v=1$. So (B5) can be viewed as a `localized' version of (B1) which applies to most vertices $v$ and small sets $S=\{v_1,\dots, v_s\}$. 
The case $s=1$ of (B5) can also be viewed as a generalization of Theorem~\ref{MM} from matchings to arbitrary bounded degree graphs $H$ which holds for most edges $e$ of $G$. We will only use the case $s\leq 2$ of (B5).
(B6) will only be used in order to prove property (T3) of the packing in Theorem~\ref{main lemma}.

To prove Lemma \ref{modified blow-up} we shall introduce first the \emph{Slender graph algorithm} and then the \emph{Uniform embedding algorithm}. Roughly speaking the Slender graph algorithm finds an embedding of $H$ into $G$ where $H$ is `slender' in the sense that it is the union of suitable perfect matchings. The Uniform embedding algorithm then takes $H$ as in Lemma~\ref{modified blow-up}, and after suitable preprocessing, it embeds the resulting graph $H'$ into $G$ via the Slender graph algorithm.

\subsection{The Four graphs lemma}
Here we state the Four graphs lemma of R\"odl and Ruci\'nski \cite{RR}, which is required for the analysis of the Slender graph algorithm and was a key tool in their proof of the classical blow-up lemma. Roughly speaking, it asserts the following: given three super-regular graphs forming a `blown-up triangle', choosing a random perfect matching $M$ in one of the pairs and then contracting the edges of $M$ yields a super-regular pair. 

More precisely, suppose that $W_1, W_2,W_3$ are disjoint sets of size $n$, and that for each pair $i,j$ with $1\leq i<j\leq 3$ we have a bipartite graph $F_{ij}$ with partition $(W_i,W_j)$. We say that the triple of graphs $(F_{12},F_{13},F_{23})$ is \emph{$(\epsilon,d_{12},d_{13},d_{23})$-super-regular} if 

\begin{itemize}
\item[(R1)] each $F_{ij}$ is $(2\epsilon,d_{ij})$-super-regular,
\item[(R2)] every edge $uv\in E(F_{12})$ (where $u\in W_1$, $v\in W_2$) satisfies
$$|N_{F_{13}}(u) \cap N_{F_{23}}(v)| = (d_{13}d_{23}\pm 2\epsilon) n.$$
\end{itemize}
Moreover, given a bijection (which might be inducing a perfect matching) $\sigma: W_1\rightarrow W_2$ in $F_{12}$ we define a fourth graph $A_\sigma$ with bipartition $(W_1,W_3)$, where $uv \in E(A_\sigma)$ with $u\in W_1$ if and only if both $uv \in E(F_{13})$ and $\sigma(u)v \in E(F_{23})$.

\begin{lem}[Four graphs lemma~\cite{RR}]\label{The four graphs lemma}
Suppose $0<1/n\ll c\ll \epsilon \ll d_{12},d_{13},d_{23},1/(C+1)\leq 1$. Let $g'$ be the function defined in \eqref{eq:functions}. Let $(F_{12},F_{13},F_{23})$ be a $(\epsilon,d_{12},d_{13},d_{23})$-super-regular triple of graphs with vertex sets $W_1, W_2, W_3$, each of size $n$. Then if we choose a perfect matching $\sigma: W_1\rightarrow W_2$ in $F_{12}$ uniformly at random, then $A_\sigma$ is $(g'(\epsilon), d_{13}d_{23})$-super-regular with probability at least $1-(1-c)^n$. \newline
Moreover, suppose some $x_1,\dots,x_C\in W_1, y_1,\dots, y_C\in W_2$ are distinct and for each $i\in [C]$ they satisfy
\begin{align}\label{eq:codensities}
|N_{F_{13}}(x_i) \cap N_{F_{23}}(y_i)| = (d_{13}d_{23}\pm2\epsilon) n.
\end{align}
Then assigning $\sigma(x_i):=y_i$ and choosing a perfect matching in $F_{12}\setminus \{x_1,\dots, x_C,y_1,\dots,y_C\}$ uniformly at random, we still obtain, with probability at least $1-(1-c)^n$, a bijection $\sigma: W_1\rightarrow W_2$ for which $A_\sigma$ is $(g'(\epsilon),d_{13}d_{23})$-super-regular. 
\end{lem}

Note that we do not require $y_i\in N_{F_{12}}(x_i)$.

\begin{proof}

The original proof (based on Theorem~\ref{M0}) from~\cite{RR} carries over. It is straightforward to check that artificially assigning $\sigma(x_i):=y_i$ does not affect the proof other than slightly increasing the value of $g'$.
\end{proof}

\subsection{The Slender graph algorithm}
The Slender graph algorithm embeds a `slender' graph $H$ into a graph $G$, where the reduced graph is given by $R_*$. We now formally describe the graphs and parameters forming part of the input. After this, we describe the algorithm itself. Lemma~\ref{slender lemmas} and Lemma~\ref{Slendered blow-up} then show the algorithm has the desired properties. \newline

{\noindent \bf Valid input.}
We say $\mathcal{S}= (G,P,H,R_*,A_0,\mathcal{U},\mathcal{Y},\mathcal{I},c,\epsilon,d_0,\vec{d},\vec{\beta},K,\Delta_R,C)$ is a valid input for the Slender graph algorithm if the following holds, where $m$ denotes the size of the largest partition class in $\mathcal{U}$ and $\mathcal{Y}$, and $r:=|R_*|/K$.
\begin{itemize}
\item[(V1)] $\mathcal{I} = (I_1,\dots, I_w)$ is a partition of $[Kr]$ with $w:=K^2\Delta_R^2(\Delta_R+1)$. Here we allow some of the $I_{i'}\in \mathcal{I}$ to be empty.
\item[(V2)] $R_*$ is a graph on $[Kr]$ such that each $I_{i'} \in \mathcal{I}$ is an independent set of $R_*$. Moreover, $\Delta(R_*)\leq K\Delta_R$ and $\Delta(R_*[I_{i'},I_{j'}])\leq 1$ for all $i',j'\in [w]$.
\item[(V3)] Both $\vec{d}$ and $\vec{\beta}$ are symmetric $Kr \times Kr$ matrices where $d:=\min_{ij \in E(R_*)}d_{i,j}$ and $\beta:= \min_{ij \in E(R_*)}\beta_{i,j}$. Moreover, $1/m \ll c \ll \epsilon \ll \beta,d,d_0, 1/K, 1/\Delta_R, 1/(C+1) \leq 1$, and $1/m \ll 1/r$.

\item[(V4)] $G$ is a $(\epsilon,\vec{d})$-super-regular graph with respect to $(R_*,\mathcal{U})$ where $\mathcal{U}=(U_1,\dots, U_{Kr})$ and $m-C\leq |U_i|\leq m$ for all $i\in [Kr]$.
\item[(V5)] $P$ is a $(\epsilon,\vec{\beta})$-super-regular graph with respect to $(R_*,\mathcal{U})$.
\item[(V6)] $H$ is a graph admitting vertex partition $(R_*,\mathcal{Y})$ with $\mathcal{Y}=(Y_1,\dots, Y_{Kr})$ such that $|Y_i|=|U_i|$ for $i\in[Kr]$ and $H[Y_i,Y_j]$ is a matching of size $\min\{|Y_i|,|Y_j|\}$ if $ij\in E(R_*)$ and empty otherwise.
\item[(V7)] $A_0$ is a disjoint union of bipartite graphs $A_0(1),\dots,A_0(Kr)$ such that each $A_0(i)$ has bipartition $(Y_i,U_i)$
and $A_0(i)=A_0[Y_i,U_i]$ is $(\epsilon,d_0)$-super-regular for each $i\in [Kr]$.

\end{itemize} \vspace{0.2cm}

We can now formally introduce the Slender graph algorithm for graphs $G,P,H,R_*,A_0$, partitions $\mathcal{U},\mathcal{Y},\mathcal{I}$ and constants $c,\epsilon, d_0, K, \Delta_R, C$ and $Kr \times Kr$ matrices $\vec{\beta},\vec{d}$ forming a valid input ($R_*$ will play the role of $R_K$ in Lemma \ref{modified blow-up}). As indicated by (V6) it views $H$ as a union of perfect matchings. Their vertex classes are embedded into $G$ in $w$ rounds, where $w$ is defined as in (V1). The vertex classes of $H$ to be embedded in a given round correspond to an independent set of $R_{*}$, so their embeddings do not `interfere' with each other. We will choose the images of the vertices  in $x \in Y_i$ by choosing a random perfect matching in an auxiliary graph $A(i)$, where $xv$ is an edge in $A(i)$ if $v$ is a suitable candidate for the image $\phi(x)$ of $x$ (see \eqref{def A(i)}).

\vspace{0.2cm}
\noindent {\bf The Slender graph algorithm on $(G,P,H,R_*,A_0,\mathcal{U},\mathcal{Y},\mathcal{I},c,\epsilon,d_0,\vec{d},\vec{\beta},K,\Delta_R,C)$.}
\vspace{0.2cm}

\noindent {\bf Preparation Round.} The goal of the preparation round is to reduce the problem to the setting where the vertex classes all have equal size by adding additional artificial vertices to the smaller classes. 

For all $i\in [Kr]$ and $j\in [m-|Y_i|]$ we introduce additional `artificial' vertices $y_{i,j}, u_{i,j}$. Let $Y'_i$ be the union of $Y_i$ together with the $y_{i,j}$ and define $U'_i$ similarly. We now define graphs $G',P',H'$ and $A'_0$ with $G\subseteq G', P\subseteq P', H\subseteq H', A_0\subseteq A'_0$ and  $V(G'):=V(P'):= \bigcup_{i=1}^{Kr} U'_i$, $V(H'):= \bigcup_{i=1}^{Kr} Y'_i$ and $V(A'_0):= V(H')\cup V(G')$ as follows. For all vertices $u_{i,j} \in U'_i$ and $u\in U_{\ell}$ with $i\ell \in E(R_{*})$ we include the edge $uu_{i,j}$ in $E(G')$ independently with probability $d_{i,\ell}$ and include the edge $uu_{i,j}$ in $E(P')$ independently with probability $\beta_{i,\ell}$. 
Also, for all vertices $x\in Y_i$ and $u_{i,j} \in U'_i$, we include the edge $xu_{i,j}$ in $E(A'_0)$ independently with probability $d_0$.  Similarly for all $y_{i,j} \in Y'_i, u\in U_i$, we include the edge $uy_{i,j}$ in $E(A'_0)$ independently with probability $d_0$. For all $1\leq i\neq j \leq Kr$ with $ij\in E(R_*)$ add at most $C$ missing edges to extend the matching $H[Y'_i,Y'_j]$ to a perfect matching $H'[Y'_i,Y'_j]$. We define the following events.
\begin{itemize}
\item[(P1)] $G'[V(G)]=G$, $P'[V(P)]=P$ and $A'_0[V(G)\cup V(H)]=A_0$.
\item[(P2)] $G'$ is $(2\epsilon,\vec{d})$-super-regular with respect to $(R_*,U'_1,\dots, U'_{Kr})$ and $P'$ is $(2\epsilon,\vec{\beta})$-super-regular with respect to $(R_*,U'_1,\dots, U'_{Kr})$. Also, $A'_0[Y'_i,U'_i]$ is $(2\epsilon,d_0)$-super-regular for all $i\in [Kr]$.
\item[(P3)] For all $i\in [Kr]$ the following holds for any set $Q_1=\{q_1,q_2,\dots, q_{s}\}$ of $s\leq K\Delta_R$ vertices such that $q_\ell \in U'_{j_{\ell}}\setminus U_{j_{\ell}}$ and $j_{\ell} \in N_{R_*}(i)$ for every $\ell\in[s]$, any set $Q_2\subseteq V(G)$ with $|Q_2| \leq K\Delta_R$ and any vertex $y\in Y'_i$:
\begin{enumerate}
\item[(i)] $\displaystyle |N_{G'}(Q_1) \cap N_{G'}(Q_2) \cap N_{A'_0}(y)| = (\prod_{\ell=1}^{s} d_{i,j_{\ell}}) | N_{G'}(Q_2)\cap N_{A'_0}(y)| \pm 2\epsilon m.$ 
\item[(ii)] $\displaystyle |N_{P'}(Q_1) \cap N_{P'}(Q_2)\cap N_{A'_0}(y) | = (\prod_{\ell=1}^{s} \beta_{i,j_{\ell}}) | N_{P'}(Q_2)\cap N_{A'_0}(y)| \pm 2\epsilon m.$ 
\end{enumerate}
\item[(P4)] $H'[V(H)]=H$ and $H'[Y'_i,Y'_j]$ is a perfect matching for all $ij\in E(R_*)$.
\end{itemize}
Note that (P1) and (P4) always hold. If (P2) or (P3) does not hold, then we end the algorithm with \textbf{failure of type 1}. So if failure of type 1 does not occur, then (P1)--(P4) hold.\newline

\noindent {\bf Round $0$.} Here we initialise the settings for the algorithm. Let $f_0$ be an empty partial embedding which sends no vertex in $H'$ to no vertex in $G'$. Recall from (V1) that $\mathcal{I}=(I_1,\dots, I_w)$ and for all $i'\in [w]$ let $I_1^{i'}:= \bigcup_{j'=1}^{i'} I_{j'}$.
For all $x\in V(H')$ and $i'\in [w]$ let $$N_{i'}(x):= \bigcup_{j \in I_1^{i'}} Y'_j \cap N_{H'}(x).$$ 
So $N_{i'}(x)\subseteq \bigcup_{j\in I_1^{i'} \cap N_{R_*}(\ell)} Y'_{j}$ for all $x\in Y'_\ell$. Also since $\Delta(R_*[I_{i'},I_{j'}])\leq 1$ by (V2) and $H'[Y'_{i},Y'_{j}]$ is empty if $ij\notin E(R_*)$ and a perfect matching if $ij\in E(R_*)$, we have $|\bigcup_{j\in I_{j'}} Y_j\cap N_{i'}(x)|\leq 1$ for $i',j'\in[w]$. Thus $|N_{i'}(x)|=|N_{i'}(y)| \leq i'$ for all $x,y\in Y'_{j}$, $i'\in [w]$ and $j \in [Kr]$. 

Let $g$ be the function defined in \eqref{eq:functions}. So $g(a) = a^{1/300} \geq (4K\Delta_R )^{1/120}a^{1/240}= g'(4K\Delta_R \sqrt {a})$ for $a\ll 1/K,1/\Delta_R$. For all $t\in [w]$ let 
\begin{align}\label{eq:xirecursion}
\xi_t:= g^t(2\epsilon) \text{ and let } \xi_0 := 2\epsilon.
\end{align} 
So 
\begin{align}\label{xi property}
\xi_{t+1} \geq g'(4K\Delta_R \sqrt{\xi_t})\text{ } ~\text{ for any } t \in [w].
\end{align}
For all $j \in [Kr]$, let $A^j_0:=A'_0[Y'_j,U'_j]$ and $B^j_0:=A^j_0$ and define 
\begin{align}\label{definition p 0}
p(\vec{d},j,0):=d_0 \text{ and } p(\vec{\beta},j,0):=d_0.
\end{align}

So (P2) implies that $A^j_0$ is $(\xi_0,p(\vec{d},j,0))$-super-regular and $B^j_0$ is $(\xi_0,p(\vec{\beta},j,0))$-super-regular for each $j\in [Kr]$. $A^j_0$ and $B^j_0$ encode the initial constraints (or candidates) for the embedding. Move to Round 1. \newline

\noindent {\bf Round $i'$.} Given an embedding $f_{i'-1}$ of $H'[\bigcup_{j\in I_1^{i'-1}} Y'_j ]$ into $G'[\bigcup_{j\in I_1^{i'-1}} U'_j]$ constructed in Round $i'-1$, we aim to extend it to $f_{i'}:H'[\bigcup_{j\in I_1^{i'}} Y'_j ]\rightarrow G'[\bigcup_{j\in I_1^{i'}} U'_j ] $.

Assume that for each $j\in [Kr]$ the bipartite graphs $A^j_{i'-1}$ and $B^j_{i'-1}$ on $(Y'_j,U'_j)$ constructed in Round $i'-1$ are such that 
\begin{itemize}
\item $A^j_{i'-1}$ is $(\xi_{i'-1} ,p(\vec{d},j,i'-1))$-super-regular,
\item $B^j_{i'-1}$ is $(\xi_{i'-1},p(\vec{\beta},j,i'-1))$-super-regular. 
\end{itemize}
Here the candidacy graph $A^j_{i'-1}$ encodes the constraints accumulated until Round $i'-1$ that we have for embedding $Y'_j$ into $U'_j$ in order to extend $f_{i'-1}$ to $f_{i'}$. The candidacy graph $B^j_{i'-1}$ tracks analogous constraints for the patching graph $P$ which we will use in Section~\ref{sec:main} to resolve overlaps between embeddings of distinct graphs $H$.

Recall that $|N_{R_*}(j)\cap I_{i'}|\leq 1$ for $i'\in[w]$, and $j\in [Kr]$ by (V2). So for any symmetric $Kr \times Kr$ matrix $\vec{z}$ we can let
\begin{align} \label{P relation}
p(\vec{z},j,i') := \left\{\begin{array}{ll}
z_{j,\ell} p(\vec{z},j,i'-1) & \text{ if } N_{R_*}(j)\cap I_{i'}=\{\ell\},\\
p(\vec{z},j,i'-1) & \text{ if } N_{R_*}(j)\cap I_{i'}=\emptyset.
\end{array}\right.
\end{align}
We will apply this with $\vec{d}$ and $\vec{\beta}$ playing the role of $\vec{z}$. \eqref{P relation} will be used to update the densities of the candidacy graphs after each round. For $i\in I_{i'}$, we define a spanning subgraph $A(i)$ of $A^{i}_{i'-1}$ which contains all those edges $xv \in E(A_{i'-1}^{i})$ with $x\in Y'_i$ and $v\in U'_i$ satisfying 
\begin{align}\label{def A(i)}
&| N_{A^{j}_{i'-1}}(x_j) \cap N_{G'}(v) | = (p(\vec{d},j,i') \pm 2\xi_{i'-1} )m ~\text{ and } \nonumber \\ &| N_{B^{j}_{i'-1}}(x_j) \cap N_{P'}(v) | = (p(\vec{\beta},j,i') \pm 2\xi_{i'-1} )m
\end{align}
for every $j\in N_{R_*}(i)$, where $\{x_j\} := N_{H'}(x)\cap Y'_j$. Excluding the edges not satisfying \eqref{def A(i)} will enable us to apply the four graphs lemma (Lemma~\ref{The four graphs lemma}) with $A(i)$ playing the role of $F_{12}$.

We now choose an embedding $\sigma_i(x)$ for all $x\in Y'_i$ with $i\in I_{i'}$. Let $\sigma_{i}(y_{i,j}):=u_{i,j}$ for all $i\in I_{i'}$ and $j\in [m-|Y_i|]$ 
and extend $\sigma_i$ to a bijection from $Y'_i$ to $U'_i$ by choosing a perfect matching of $A(i)[Y_i, U_i]$ uniformly at random (such a perfect matching will be shown to exist in Lemma~\ref{slender lemmas}(i)). 
We now extend the current embedding. Let
\begin{align}\label{embedding extension}
 f_{i'}(x)  := \left\{ \begin{array}{ll}
f_{i'-1}(x) &\text{ if } x\in Y'_i \text{ for } i\in I_1^{i'-1},\\ \sigma_{i}(x)   &\text{ if } x\in Y'_{i} \text{ for } i\in I_{i'}. 
\end{array}\right. \end{align}
\noindent
We now update the candidacy graphs to ensure that they incorporate the new constraints arising from the embeddings in the current round. For each $j\in [Kr]$, we let $A^j_{i'}$ be the bipartite graph on $(Y'_j,U'_j)$ with the edge set
\begin{align}\label{Aji x def}
E(A^j_{i'}) := \left\{xv : x\in Y'_j, v\in N_{A^j_0}(x) \text{ and } uv\in E(G') \text { for each } u\in f_{i'}(N_{i'}(x))\right\}.
\end{align}
So for all $x\in Y'_j$,
\begin{align}\label{eq:updatecandgraphs}
N_{A^j_{i'}}(x) = N_{A^j_0}(x) \cap \bigcap_{y\in N_{i'}(x)} N_{G'}(f_{i'}(y)).
\end{align}
For each $j\in [Kr]$, we let $B^j_{i'}$ be the bipartite graph on $(Y'_j, U'_j)$ with the edge set
\begin{align}\label{Bji def}
E(B^j_{i'}) := \{xv : x\in Y'_j, v\in N_{B^j_0}(x) \text{ and } uv\in E(P') \text{ for each } u\in f_{i'}(N_{i'}(x))\}.
\end{align}
Note that $A_{i'}^j \subseteq A_{i'-1}^j$ and $B_{i'}^j \subseteq B_{i'-1}^j$ for all $j\in [Kr]$.
Note also that 
\begin{align}\label{stay same}
A_{i'}^j=A_{i'-1}^j\text{ and }B_{i'}^j=B_{i'-1}^j\text{ if }N_{R_*}(j)\cap I_{i'} =\emptyset.
\end{align}
If for some $j \in [Kr]$, $A^j_{i'}$ is not  $(\xi_{i'} ,p(\vec{d},j,i'))$-super-regular or $B^j_{i'}$ is not $(\xi_{i'} ,p(\vec{\beta},j,i'))$-super-regular, then we abort the algorithm with \textbf{failure of type 2}. If $i'<w$, then proceed to Round $i'+1$. If $i'=w$,  then we return $\phi(x) := f_{w}(x)$ for every $x\in V(H)$ and $F_j:= B^j_{w}[Y_j, U_j]$ for every $j\in [Kr]$, and we are done with success. This completes the description of the Slender graph algorithm. 

Note that the description would be slightly simpler if we embed only one class in each round. However, then the number of rounds would depend on $|R_*|$ rather than on $K$ and $\Delta_R$, in which case our errors would accumulate too much. \newline

A straightforward application of the Chernoff-Hoeffding bound in Lemma \ref{Chernoff Bounds} shows that failure of type 1 is very unlikely. We omit the proof.%
\COMMENT{
\begin{proof}
By Lemma \ref{Chernoff Bounds}, for $u_{i,j}\in U'_{i}\setminus U_i$ and $\ell \in N_{R_*}(i)$,
$$\mathbb{P}\left[ | N_{G'}(u_{i,j})\cap U'_\ell|=(d_{i,\ell} \pm \epsilon)m \right] \geq 1 - (1-4c)^{Km} \text{ and } \mathbb{P}\left[ | N_{P'}(u_{i,j})\cap U'_\ell|=(\beta_{i,\ell} \pm \epsilon)m \right] \geq 1 - (1-4c)^{Km}.$$
Also adding at most $C$ artificial vertices on each $G[U_i,U_j]$ or $P[U_i,U_j]$ does not significantly affect the $\epsilon$-regularity. So $P'$ is $(2\epsilon,\vec{\beta})$-super-regular with probability at least $1-CKr\Delta_R(1-4c)^{Km}$, and $G'$ is $(2\epsilon,\vec{d})$-super-regular with probability at least $1-CKr\Delta_R(1-4c)^{Km}$. Also for fixed set $Q_1=\{q_1,\dots, q_s\}, Q_2$, and vertex $y$ satisfying the conditions in (P3), $|N_{G'}(Q_1) \cap N_{G'}(Q_2) \cap N_{A_0}(y)| = (\prod_{\ell=1}^{s} d_{i,j_{\ell}})| N_{G'}(Q_2)\cap N_{A_0}(y)| \pm 2\epsilon m$ holds with probability at least $1-(1-4c)^{Km}$. Since there are at most $m(Kr m)^{2K\Delta_R}$ choices for $Q_1$ and $Q_2$ and $y$, (P3)(i) holds with probability at least $1- m(Krm)^{2K\Delta_R}(1-4c)^{Km}$. By the same argument, (P3)(ii) also holds with probability at least $1-m(Krm)^{2K\Delta_R} (1-4c)^{Km}$.
Thus failure of type 1 happens with probability at most $(2m(Krm)^{2K\Delta_R} + CKr\Delta_R)(1-4c)^{Km}\leq (1-3c)^{Km}$. \end{proof}}

\begin{prop}\label{preparation}
Suppose $\mathcal{S}=(G,P,H,R_*,A_0,\mathcal{U},\mathcal{Y},\mathcal{I},c,\epsilon,d_0,\vec{d},\vec{\beta},K,\Delta_R,C)$ is a valid input.
Then in the Slender graph algorithm, failure of type 1 occurs with probability at most $(1-3c)^{Km}$.
\end{prop}

For each $0\leq i' \leq w$ let $\mathcal{A}_{i'}$ be the event that the following two statements hold. 

\begin{itemize}
\item[(A$1_{i'}$)] For each $j\in [Kr]$, $A^j_{i'}$ is $(\xi_{i'},p(\vec{d},j,i'))$-super-regular.

\item[(A$2_{i'}$)] For each $j\in [Kr]$, $B^j_{i'}$ is $(\xi_{i'},p(\vec{\beta},j,i'))$-super-regular.
\end{itemize}
So we have failure of type $2$ if and only if $\mathcal{A}_{i'}$ fails to hold for some $i'$.
Note that $\mathcal{A}_0$ always holds for the Slender graph algorithm on $\mathcal{S}$ if $\mathcal{S}$ is a valid input (as observed after \eqref{definition p 0}).
We also define $\mathcal{A}_0^{i'}:= \bigwedge_{j'=0}^{i'} \mathcal{A}_{j'}$.

The next lemma states that the super-regularity of the candidacy graphs is preserved throughout the algorithm and thus that properties (M1)--(M3) from Lemma~\ref{lem:matchings} carry over into the current setting. 

\begin{lemma}\label{slender lemmas}
Suppose $\mathcal{S}=(G,P,H,R_*,A_0,\mathcal{U},\mathcal{Y},\mathcal{I},c,\epsilon,d_0,\vec{d},\vec{\beta},K,\Delta_R,C)$ is a valid input. Assume that in the Slender graph algorithm on input $\mathcal{S}$, there was no failure prior to Round $i'$. Then the following holds.
\begin{itemize}
\item[(i)] For all $i\in I_{i'}$ we have $\Delta( A^i_{i'-1} - A(i)) \leq 4K\Delta_R \xi_{i'-1} m$. In particular, this means that $A(i)$ is $(4K\Delta_R\sqrt{\xi_{i'-1}},p(\vec{d},i,i'-1))$-super-regular for all $i\in I_{i'}$.
\item[(ii)] Let $\mathcal{B}_{i'-1}$ be an event which depends only on the history of the Slender graph algorithm prior to Round $i'$
and such that $\Prob[\mathcal{A}_{0}^{i'-1}, \mathcal{B}_{i'-1}]>0$. Then 
$$\mathbb{P}[\mathcal{A}_{i'}\mid \mathcal{A}_{0}^{i'-1}, \mathcal{B}_{i'-1}] \geq 1 - (1-3c)^{Km}.$$
\item[(iii)] Let $\mathcal{B}_{i'-1}$ be as in (ii), and let $h'(a)=a^{1/10}$ and $h(a)=a^{1/20}$ be as defined in \eqref{eq:functions}. Then the following statements hold.
\begin{itemize}
\item[$(\text{M}'1)_{i'}$] For all $i\in I_{i'}$, $x\in Y_{i}$ and every $S\subseteq N_{A(i)}(x)\cap U_{i}$,
$$\Prob[\sigma_{i}(x)\in S\mid \mathcal{A}_0^{i'-1},\mathcal{B}_{i'-1}]=(1\pm h(4K\Delta_R\sqrt{\xi_{i'-1}}))\frac{|S|}{p(\vec{d},i,i'-1)m} .
$$
\item[$(\text{M}'2)_{i'}$] For all $i\in I_{i'}$, $v\in U_{i}$ and every $S\subseteq N_{A(i)}(v)\cap Y_{i}$,
$$
\Prob[v\in \sigma_{i}(S)\mid \mathcal{A}_0^{i'-1},\mathcal{B}_{i'-1}]=(1\pm  h(4K\Delta_R\sqrt{\xi_{i'-1}}))\frac{|S|}{p(\vec{d},i,i'-1)m} .
$$
\item[$(\text{M}'3)_{i'}$] For all $i\in I_{i'}$, all $d'$ with $c \ll d' \leq d^{K\Delta_R}/9$ and all $A'\subseteq A(i)$ with $V(A')=U_i\cup Y_i$ and $\Delta(A')\leq d'm$,
$$\Prob[|\{x\in Y_i: x\sigma_{i}(x) \in E(A') \}| >  8d'm/{ p(\vec{d},i,i'-1)}\mid \mathcal{A}_0^{i'-1},\mathcal{B}_{i'-1}] < (1-4Kc)^m.$$
\item[$(\text{M}'4)_{i'}$] For all $i\in I_{i'}$, all $S\subseteq Y'_i$ and all $T\subseteq U'_i$ with $|S|,|T|\ge h'(4K\Delta_R\sqrt{\xi_{i'-1}})m$,
$$
\Prob[|\sigma_{i}(S)\cap T|=(1\pm h'(4K\Delta_R\sqrt{\xi_{i'-1}}))|S||T|/m \mid \mathcal{A}_0^{i'-1},\mathcal{B}_{i'-1}]\ge 1-(1-4Kc)^m.
$$
\end{itemize}
\end{itemize}
\end{lemma}
\begin{proof}
First, we show (i). Let us fix $i\in I_{i'}$, $v\in U'_{i}$ and $j \in N_{R_*}(i)$, and define
$$N^{\leq}_{v,j}(G'): = \left\{ x\in Y'_{i} : |N_{A^j_{i'-1}}(x_j)\cap N_{G'}(v)| \leq (p(\vec{d},j,i')- 2\xi_{i'-1})m\right\},$$
where $\{x_j\} = N_{H'}(x)\cap Y'_j$. 
Since $G'[U'_{i},U'_j]$ is $(2\epsilon,d_{i,j})$-super-regular by (P2), $|N_{G'}(v)\cap U'_j| \geq (d_{i,j}-2\epsilon)m$. 
Together with the fact that $A^j_{i'-1}$ is $(\xi_{i'-1},p(\vec{d},j,i'-1))$-regular (since there was no failure prior to Round $i'$) and 
$p(\vec{d},j,i'-1)(d_{i,j}-2\epsilon)-\xi_{i'-1}>p(\vec{d},j,i')-2\xi_{i'-1}$
by \eqref{P relation}, it follows that $|N_{v,j}^{\leq }(G')|\leq \xi_{i'-1} m$. 
Similarly we get 
$$|N_{v,j}^{\leq}(P')|\leq \xi_{i'-1} m, ~ |N_{v,j}^{\geq}(G')|\leq \xi_{i'-1} m ~ \text{ and } ~|N_{v,j}^{\geq}(P')|\leq \xi_{i'-1} m,$$
where
$$N^{\leq}_{v,j}(P'): = \left\{ x\in Y'_i : |N_{B^j_{i'-1}}(x_j)\cap N_{P'}(v)| \leq (p(\vec{\beta},j,i')- 2\xi_{i'-1})m\right\},$$
$$N^{\geq}_{v,j}(G'): = \left\{ x\in Y'_i : |N_{A^j_{i'-1}}(x_j)\cap N_{G'}(v)| \geq (p(\vec{d},j,i')+ 2\xi_{i'-1})m\right\},$$
$$N^{\geq}_{v,j}(P'): = \left\{ x\in Y'_i : |N_{B^j_{i'-1}}(x_j)\cap N_{P'}(v)| \geq (p(\vec{\beta},j,i')+ 2\xi_{i'-1})m\right\}.$$

When creating $A(i)$, we have removed an edge $xv$ from $A^{i}_{i'-1}$ if and only if $x\in N^{\leq}_{v,j}(G') \cup N_{v,j}^{\geq}(G') \cup N^{\leq}_{v,j}(P') \cup N_{v,j}^{\geq}(P')$ for at least one $j$ with $j\in N_{R_*}(i)$.
Since $|N_{R_*}(i)| \leq K\Delta_R$ by (V2) this implies that
\begin{align}\label{deg diff u}
d_{A^i_{i'-1}}(v) - d_{A(i)}(v) \leq 4K\Delta_R \xi_{i'-1} m.
\end{align}
In the same way, we can show that for each $x\in Y'_{i}$,
$$d_{A^i_{i'-1}}(x) - d_{A(i)}(x) \leq 8K\Delta_R \epsilon m \leq 4K\Delta_R \xi_{i'-1} m.$$   
Together with (\ref{deg diff u}) this shows that for any $z\in Y'_{i}\cup U'_{i}$, 
\begin{align*}
d_{A^i_{i'-1}}(z) - d_{A(i)}(z) \leq 4K\Delta_R \xi_{i'-1} m.
\end{align*}
Thus, by Proposition~\ref{regularity after edge deletion} and (A$1_{i'-1}$) it follows that (i) holds.

Now we show (ii). Assume $i'\in [w]$ and that both $\mathcal{A}_0^{i'-1}$ and $\mathcal{B}_{i'-1}$ hold. Recall from \eqref{P relation} and \eqref{stay same} that for all $j\notin \bigcup_{i\in I_{i'}} N_{R_*}(i)$ we have $A^j_{i'} = A^j_{i'-1}$, $B^j_{i'} = B^j_{i'-1}$,  $p(\vec{d},j,i')=p(\vec{d},j,i'-1)$ and $p(\vec{\beta},j,i')=p(\vec{\beta},j,i'-1)$. So these $A^j_{i'}$ are $(\xi_{i'},p(\vec{d},j,i'))$-super-regular and these $B^j_{i'}$ are $(\xi_{i'},p(\vec{\beta},j,i'))$-super-regular. Thus we can restrict our attention to
all the $j\in \bigcup_{i\in I_{i'}} N_{R_*}(i)$. Consider any $i\in I_{i'}$.
Note from (V2) that $R_{*}[I_{i'}]$ is an empty graph and $\Delta(R_*[I_{i'},I_{j'}])\leq 1$ for $i'\neq j'$. Property (ii) will follow from several applications of the four graphs lemma (Lemma~\ref{The four graphs lemma}) with $A(i)$ playing the role of $F_{12}$. To prepare for this, we introduce the following graphs.
\begin{itemize}
\item For all $j\ell \in E(R_*)$ let $\psi_{j,\ell} : Y'_{j} \rightarrow Y'_{\ell}$ be a bijection such that $\{\psi_{j,\ell}(x)\} = N_{H'}(x) \cap Y'_{\ell}$ for all $x\in Y'_j$.
%

\item For all $j\in N_{R_*}(i)$, let $F^j_{13}(i):= (Y'_{i},U'_j;E)$, where $xv \in E$ if and only if  $x\in Y'_i$, $v\in U'_j$ and $\psi_{i,j}(x)v \in E(A_{i'-1}^j)$.

\item For all $j\in N_{R_*}(i)$, let $F'^j_{13}(i) := (Y'_{i},U'_j;E')$, where $xv \in E'$ if and only if $x\in Y'_i$, $v\in U'_j$ and $\psi_{i,j}(x)v \in E(B_{i'-1}^j)$.

\item For all $j \in N_{R_*}(i)$, let $F^j_{23}(i):= G'[U'_{i},U'_j]$ and $F'^j_{23}(i) := P'[U'_{i},U'_j]$.
\end{itemize}

Observe that, since $H'[Y'_{i},Y'_j]$ is a perfect matching for each $j\in N_{R_*}(i)$, the graph $F_{13}^j(i)$ is isomorphic to $A_{i'-1}^j$ under the isomorphism $\psi'_{i,j}$ which keeps the elements of $U'_j$ fixed while mapping $x$ to $\psi_{i,j}(x)$ for all $x\in Y'_{i}$.
Since we are conditioning on $\mathcal{A}_0^{i'-1}$, properties (A$1_{i'-1})$ and (A$2_{i'-1})$ hold and so for all $j\in N_{R_*}(i)$ we have
\begin{itemize}
\item[(a)] $F_{13}^j(i)$ is $(\xi_{i'-1},p(\vec{d},j,i'-1))$-super-regular, 
\item[(a$'$)] $F'^{j}_{13}(i)$ is $(\xi_{i'-1},p(\vec{\beta},j,i'-1))$-super-regular, 
\item[(b)] $F_{23}^j(i)$ is $(2\epsilon,d_{i,j})$-super-regular, 
\item[(b$'$)] $F'^{j}_{23}(i)$ is $(2\epsilon,\beta_{i,j})$-super-regular, 
\end{itemize}
where (b) and (b$'$) follow by (P2). For $x\in Y'_{i}$ and $v\in U'_{i}$, note that  
\begin{align}\label{eq:psycodegrees}
N_{F_{13}^j(i)}(x) \cap N_{F_{23}^j(i)}(v) =  N_{A^j_{i'-1}}(\psi_{i,j}(x))\cap N_{G'}(v), \\
N_{F'^j_{13}(i)}(x) \cap N_{F'^j_{23}(i)}(v) =  N_{B^j_{i'-1}}(\psi_{i,j}(x))\cap N_{P'}(v).
\end{align}
Moreover, by \eqref{P relation} for all $j\in N_{R_*}(i)$ we have \begin{align}\label{pp relation}
p(\vec{d},j,i') = d_{i,j} p(\vec{d},j,i'-1) \enspace \text{ and } \enspace p(\vec{\beta},j,i')=\beta_{i,j} p(\vec{\beta},j,i'-1).
\end{align}
Together with \eqref{def A(i)} this  means that we obtain $A(i)$ from $A^{i}_{i'-1}$ 
by deleting every edge $xv$ which does not satisfy one of
\begin{align}\label{A(i) edge deletion}
&|N_{F_{13}^j(i)}(x) \cap N_{F_{23}^j(i)}(v)| = (d_{i,j} p(\vec{d},j,i'-1)\pm 2\xi_{i'-1}) m, \nonumber\\
&|N_{F'^{j}_{13}(i)}(x) \cap N_{F'^{j}_{23}(i)}(v)| = (\beta_{i,j} p(\vec{\beta},j,i'-1)\pm 2\xi_{i'-1}) m
\end{align}
for some index $j \in N_{R_*}(i)$. This will be used to verify condition (R2) for Lemma~\ref{The four graphs lemma}.
To apply Lemma~\ref{The four graphs lemma} we also need to check \eqref{eq:codensities} for the artificial vertices introduced in the preparation round. For this, note that two applications of (P3) imply the following property.
\begin{itemize}
\item[(P$3'$)] For all $ij \in E(R_*)$, all $Q_1 =\{q_1,\dots, q_s\}$ with $|Q_1|\leq K\Delta_R-1$ where $q_\ell \in U'_{j_{\ell}}\setminus U_{j_\ell}$ and $j_{\ell} \in N_{R_*}(j)$, all $Q_2 \subseteq V(G)$ with $|Q_2|\leq K\Delta_R$, all $v'\in U'_{i}\setminus (U_i\cup Q_1)$ and all $y\in Y'_j$, we have
\begin{itemize}
\item[(i)] $\displaystyle |N_{G'}(v') \cap N_{G'}(Q_1)\cap N_{G'}(Q_2)\cap  N_{A'_0}(y) | = d_{i,j} | N_{G'}(Q_1)\cap N_{G'}(Q_2)\cap  N_{A'_0}(y) | \pm 4\epsilon m,$

\item[(ii)] $\displaystyle|N_{P'}(v') \cap N_{P'}(Q_1)\cap N_{P'}(Q_2)\cap  N_{A'_0}(y)| = \beta_{i,j} | N_{P'}(Q_1)\cap N_{P'}(Q_2) \cap  N_{A'_0}(y) | \pm 4\epsilon m.$
\end{itemize}
\end{itemize}

For all vertices $x'\in Y'_{i}\setminus Y_{i}, v' \in U'_{i}\setminus U_{i}$ with $i\in I_{i'}$ and $ij\in E(R_*)$, we apply (P$3'$)(i) with $Q_1:= f_{i'}(N_{i'-1}(\psi_{i,j}(x'))\setminus V(H))$, $Q_2:= f_{i'}(N_{i'-1}(\psi_{i,j}(x'))\cap V(H))$, $v'$ and $\psi_{i,j}(x')$ to obtain
\begin{eqnarray}\label{artificial vertex fine}
|N_{F_{13}^j(i)}(x') \cap N_{F_{23}^j(i)}(v') |&\stackrel{\eqref{eq:psycodegrees}}{=}&  |N_{A^j_{i'-1}}(\psi_{i,j}(x'))\cap N_{G'}(v')| \nonumber\\
&\stackrel{\eqref{eq:updatecandgraphs}}{=}& |N_{A_0^j}(\psi_{i,j}(x'))\cap N_{G'}(Q_1)\cap N_{G'}(Q_2)\cap N_{G'}(v')| \nonumber\\
&\stackrel{(\text{P}3')(i)}{=}& d_{i,j}|N_{A_0^j}(\psi_{i,j}(x'))\cap N_{G'}(Q_1)\cap N_{G'}(Q_2)| \pm 4\epsilon m \nonumber\\
&\stackrel{\eqref{eq:updatecandgraphs}}{=} & d_{i,j}| N_{A^j_{i'-1}}(\psi_{i,j}(x'))| \pm 4\epsilon m \nonumber\\
&\stackrel{(\text{A}1)_{i'-1}}{=} & d_{i,j}(p(\vec{d},j,i'-1)\pm \xi_{i'-1} )m \pm 4\epsilon m \nonumber\\
&=& (d_{i,j} p(\vec{d},j,i'-1) \pm 2\xi_{i'-1} )m.
\end{eqnarray}
Similarly, we use (P$3'$)(ii) to obtain
\begin{eqnarray*}
 |N_{F'^j_{13}(i)}(x') \cap N_{F'^j_{23}(i)}(v')| &=&  |N_{B^j_{i'-1}}(\psi_{i,j}(x'))\cap N_{P'}(v')| = (\beta_{i,j} p(\vec{\beta},j,i'-1)\pm 2\xi_{i'-1} )m.
\end{eqnarray*}
Note also that for all $j\in N_{R_*}(i)$, $x\in Y'_j$ and $v\in U'_j$,
\begin{eqnarray} \label{equivalence}
 xv \in E(A^{j}_{i'}) 
&\stackrel{(\ref{Aji x def})}{\Leftrightarrow}& xv \in E(A^j_0) \text{ and }uv \in E(G') \text{ for each } u\in f_{i'}(N_{i'}(x)) \nonumber \\
&\Leftrightarrow& xv \in E(A^{j}_{i'-1}) \text{ and } f_{i'}(\psi^{-1}_{i,j}(x))\in N_{G'}(v) \nonumber \\
&\Leftrightarrow& \psi^{-1}_{i,j}(x)v \in E(F^j_{13}(i)) \text{ and } f_{i'}(\psi^{-1}_{i,j}(x))v \in E(F^j_{23}(i)).
\end{eqnarray}

We now wish to apply Lemma \ref{The four graphs lemma} for each $i\in I_{i'}$ and $j\in N_{R_*}(i)$ with the following graphs and parameters. \newline

{
\begin{tabular}{c|c|c|c|c|c|c|c}
object/parameter &  $A(i)$ & $F^j_{13}(i)$ & $F^j_{23}(i)$ & $A^j_{i'}$ & $5Kc$ & $m$ & $4K\Delta_R\sqrt{\xi_{i'-1}}$ \\ \hline
playing the role of &  $F_{12}$ & $F_{13}$ & $F_{23}$ & $A_{\sigma_i}$ & $c$ & $n$ & $\epsilon$
\\ 
\end{tabular}
}\newline\vspace{0.2cm}

\noindent Note that Lemma~\ref{slender lemmas}(i), as well as the properties (a) and (b) (which were defined before \eqref{eq:psycodegrees}) imply that the graphs satisfy the regularity condition (R1) stated before Lemma~\ref{The four graphs lemma}. Note also that \eqref{artificial vertex fine} implies that \eqref{eq:codensities} is satisfied. Recall from the description of the Slender graph algorithm that $\sigma_{i}: Y'_i\rightarrow U'_i$ is the bijection obtained in Round $i'$ for $A(i)$ and $f_{i'}(\psi_{i,j}^{-1}(x)) = \sigma_{i}(\psi_{i,j}^{-1}(x))$ for all $x\in Y'_j$ with $j \in N_{R_*}(i)$ by \eqref{embedding extension}. So \eqref{equivalence} implies that the graph $A_{\sigma_i}$ defined before Lemma \ref{The four graphs lemma} is indeed isomorphic to $A_{i'}^j$ via the
isomorphism $\psi'_{i,j}$ defined by $\psi'_{i,j}(x)=\psi_{i,j}(x)$ for $x\in Y'_i$ and $\psi'_{i,j}(v)=v$ for $v\in U'_j$. Also \eqref{A(i) edge deletion} means that condition (R2) before Lemma~\ref{The four graphs lemma} is satisfied. So altogether this means we can indeed apply Lemma~\ref{The four graphs lemma}.
Recall that $\mathcal{A}_0^{i'-1}, \mathcal{B}_{i'-1}$ only depend on the history prior to Round $i'$ and that $\xi_{i'} \geq g'(4K\Delta_R\sqrt{\xi_{i'-1}})$ by \eqref{xi property}. So Lemma~\ref{The four graphs lemma} together with \eqref{pp relation} and a union bound over all $i\in I_{i'}$ and $j\in N_{R_*}(i)$ implies that 
\begin{align}\label{A1i}
\mathbb{P}[(\text{A}1_{i'}) \mid \mathcal{A}_0^{i'-1},\mathcal{B}_{i'-1} ] \geq 1-K^2r\Delta_R(1-5Kc)^m\geq 1-(1-4c)^{Km}.
\end{align}
\noindent
Similarly as in \eqref{equivalence} it follows that
\begin{align} \label{equivalence2}
\begin{split}
xv \in E(B^{j}_{i'})
&\Leftrightarrow \psi^{-1}_{i,j}(x)v \in E(F'^j_{13}(i)) \text{ and } f_{i'}(\psi^{-1}_{i,j}(x))v \in E(F'^j_{23}(i)).
\end{split}
\end{align}
So similarly as before, for each $i\in I_{i'}$ and $j\in N_{R_*}(i)$ we can apply Lemma~\ref{The four graphs lemma} to $A(i), F'^{j}_{13}(i), F'^{j}_{23}(i)$ to obtain
\begin{align}\label{A2i}
\mathbb{P}[(\text{A}2_{i'}) \mid \mathcal{A}_0^{i'-1},\mathcal{B}_{i'-1} ] \geq 1-(1-4c)^{Km}.
\end{align}
\noindent
By (\ref{A1i}) and (\ref{A2i}), we obtain
\begin{align*} \mathbb{P}[\mathcal{A}_{i'} \mid \mathcal{A}_0^{i'-1}, \mathcal{B}_{i'-1}] &= \mathbb{P}[ (\text{A}1_{i'}), (\text{A}2_{i'}) \mid \mathcal{A}_0^{i'-1},\mathcal{B}_{i'-1}] \geq  1 - 2(1-4c)^{Km} \geq 1-(1-3c)^{Km}.
\end{align*} 
Property (iii) follows immediately from (i) as well as Lemma~\ref{lem:matchings} and Theorem~\ref{M0} applied to $A(i)[Y_i,U_i]$.
Indeed, to check $(\text{M}'3)_{i'}$, note that $p(\vec{d},i,i'-1)\geq d^{\Delta(R_*)} \geq d^{K\Delta_R} \geq 9 d'$
by~\eqref{P relation}, (V1) and (V2).
(Actually, we get an additional error arising from the existence of the artificial vertices (i.e.~from considering
$A(i)[Y_i,U_i]$ instead of $A(i)$) as well as the fact that we only have
$m-C\le |Y_i|=|U_i|\le m$. However, this error is insignificant. Alternatively,
one could have used that Proposition~\ref{regularity after edge deletion} actually implies that
$A(i)$ is $(3\sqrt{K\Delta_R \xi_{i'-1}},p(\vec{d},i,i'-1))$-super-regular.)
\end{proof}

We can now deduce from Proposition~\ref{preparation} and Lemma~\ref{slender lemmas} that the Slender graph algorithm succeeds with high probability, and that the candidacy graphs $F_j$ for the patching process have strong regularity properties.

\begin{lemma}\label{Slendered blow-up}
Suppose  $\mathcal{S}=(G,P,H,R_*,A_0,\mathcal{U},\mathcal{Y},\mathcal{I},c,\epsilon,d_0,\vec{d},\vec{\beta},K,\Delta_R,C)$ is a valid input. Let $q_*$ be the function defined in \eqref{eq:functions}. Then the Slender graph algorithm applied to $\mathcal{S}$ succeeds with probability at least $1-(1-2c)^{Km}$. Moreover, if it succeeds, then it returns an embedding $\phi$ of $H$ into $G$ and bipartite graphs $F_j$ on $(Y_j,U_j)$ (for each $j\in [Kr]$) such that $\phi(Y_i)=U_i$ for all $i\in[r]$, $\phi(x) \in N_{A_0}(x)$ for all $x\in V(H)$ and such that $\phi$ and the graphs $F_j$ satisfy the following condition.
\begin{itemize}
\item[(i)] Each $F_j$ is a $(q_*(\epsilon), d_0 p(R_*,\vec{\beta},j))$-super-regular bipartite graph on $(Y_j,U_j)$ such that $ N_{F_j}(x)\subseteq U_j \cap N_{A_0}(x) \cap \bigcap_{u\in \phi(N_H(x))} N_P(u)$ for all $x\in Y_j$. 
\end{itemize}
\end{lemma}
\begin{proof}
To show that the Slender graph algorithm succeeds with probability at least $1-(1-2c)^{Km}$, recall from Proposition \ref{preparation} that the failure of type 1 occurs with probability at most $(1-3c)^{Km}$.
Once failure of type 1 does not occur, the event $\mathcal{A}_0^{w}$ is equivalent to the algorithm being successful. By applying Lemma \ref{slender lemmas}(ii) with $\mathcal{B}_{i'-1} = \mathcal{A}_0^{i'-1}$, we obtain 
\begin{align*}
\mathbb{P}[\mathcal{A}_0^{w}] = \prod_{i'=1}^{w} \mathbb{P}[\mathcal{A}_{i'} \mid \mathcal{A}_0^{i'-1}] \geq \left(1-(1- 3c)^{Km} \right)^{w} \geq 1 - w (1-3c)^{Km}. \end{align*} 
Thus the algorithm succeeds with probability at least $1 - w(1-3c)^{Km} -(1-3c)^{Km} \geq 1- (1-2c)^{Km}$. Also
$\phi(Y_i)=U_i$ and $\phi(x) \in N_{A^j_{0}}(x)\cap U_j \subseteq N_{A_0}(x)$ for $x\in Y_j$ are trivial from the description of the algorithm.

Now let us assume that the algorithm succeeds, i.e. $\mathcal{A}_0^{w}$ occurs. Note that $\xi_w=g^w(2\epsilon)$ by \eqref{eq:xirecursion}. 
In particular, this means that for all $j\in [Kr]$ we obtain graphs $B_{w}^j$ on  $(Y'_j, U'_j)$, each of which is $(q_*(\epsilon)/2,p(\vec{\beta},j,w))$-super-regular. 
Thus $F_j= B_{w}^j[Y_j, U_j]$ is $(q_*(\epsilon),p(\vec{\beta},j,w))$-super-regular (as $|Y'_j\setminus Y_j|+|U'_j\setminus U_j|\leq 2C$).
But $p(\vec{\beta},j,w) = d_0 p(R_*,\vec{\beta},j)$ by \eqref{p R def}, \eqref{definition p 0} and \eqref{P relation}, so $F_j$ is $(q_*(\epsilon), d_0 p(R_*,\vec{\beta},j))$-super-regular. 
Also, by the definition of $B_{w}^j$, for any $x\in Y_j$ 
\begin{eqnarray*}
N_{F_j}(x) &\stackrel{(\ref{Bji def})}{=}&  U_j\cap N_{B^j_0}(x) \cap \bigcap_{u\in f_{w} (N_{H'}(x))} N_{P'}(u)\subseteq  U_j\cap N_{A_0}(x) \cap \bigcap_{u\in f_{w}(N_H(x))} N_{P}(u) \\
&=&U_j\cap N_{A_0}(x) \cap \bigcap_{u\in \phi(N_H(x))} N_{P}(u)
\end{eqnarray*}
as $U_j\cap N_{P'}(u) = U_j\cap N_{P}(u)$ for all $u\in V(G)=V(P)$, and $\phi$ is the restriction of $f_{w}$ to $V(H)$. Thus (i) holds.
\end{proof}

\subsection{The Uniform embedding algorithm}
We will find the required embedding $\phi$ in Lemma \ref{modified blow-up} via the following random algorithm, which first preprocesses the input to satisfy the requirements of the Slender graph algorithm and then runs the Slender graph algorithm.
So suppose that $(G,P,H,R,A_0,\mathcal{V},\mathcal{X},c,\epsilon,d_0,\vec{d},\vec{\beta},k,\Delta_R,C)$ satisfies the assumption in Lemma \ref{modified blow-up}. Recall that $H^2$ denotes the square of the graph $H$. Recall $R_K$ denotes the $K$-fold blow-up of $R$. In Steps 1--3 below we will refine the partitions given by Lemma~\ref{modified blow-up} in a suitable way to ensure that we have a valid input for the Slender graph algorithm satisfying (V1)--(V7). So in Step 4 we can then apply the Slender graph algorithm to obtain the desired embedding. \newline

\noindent
{\bf The Uniform embedding algorithm on $(G,P,H,R,A_0,\mathcal{V},\mathcal{X},c,\epsilon,d_0,\vec{d},\vec{\beta},k,\Delta_R,C)$.}
\vspace{0.2cm}

\noindent {\bf Step 1}.
Recall that we assume that $H$ is $(R,\vec{k},C)$ near-equiregular with respect to $\mathcal{X}$
(which was defined in Section~\ref{subsec:notation}). We apply the Hajnal-Szemer\'edi theorem (Theorem \ref{thm:HS})
to $H^2[X_i]$ for each $i\in [r]$ to find an equitable partition of $X_i$ into $K:=(k+1)^2 \Delta_R$ sets which are independent in $H^2$. (Note that this is possible since $\Delta(H^2[X_i])\leq k(k+1)\Delta_R$.)
Let $\mathcal{Y}:=\{ Y_1,\dots, Y_{Kr}\}$ be the resulting partition of $V(H)$ such that for all $i\in[r]$
 $$X_i = \bigcup_{j\in J_i} Y_j,$$  
where $J_i := \{(i-1)K+1,\dots, iK\}$ and $m-C \leq |Y_j| \leq m$ with $m:=\lceil n/K \rceil$.

Since $\Delta(R)\leq \Delta_R$, we can find a vertex partition of $R$ into independent sets $W^*_1,\dots, W^*_{\Delta_R+1}$. Let $w:= (K\Delta_R)^2(\Delta_R+1).$ We now view the $K$-fold blow-up $R_K$ of $R$ as being obtained from $R$ by replacing each vertex $i$ of $R$ with the set of vertices $J_i$. Note that $\Delta(R_K)\leq K \Delta_R$. Let $W^{**}_1,\dots, W^{**}_{\Delta_R+1}$ be the vertex partition of $V(R_K)=[Kr]$ such that $i\in W^*_{j'}$ if and only if $J_i \subseteq W^{**}_{j'}$. For each $i'\in [\Delta_R+1]$, we choose a vertex partition of $(R_K)^2[W^{**}_{i'}]$ into $(K\Delta_R)^2$ sets which are independent in $(R_K)^2$.\COMMENT{it's ok since $\Delta((R_K)^2[W^{**}_{i'}]) \leq (K\Delta_R-1)K \Delta_R$}
Denote the classes of the resulting vertex partition of $W^{**}_{i'}$ by $I_{j'}$ where $(i'-1)(K\Delta_R)^2+1 \leq j' \leq i'(K\Delta_R)^2$. Note that for all $i',j' \in [w]$ we have
\begin{align}\label{max deg R_K I}
\Delta(R_K[I_{i'},I_{j'}]) \leq 1.
\end{align}

For all $i\in [r]$, let $i'_{\max}$ be the largest index $i'$ such that $J_i\cap I_{i'}\neq \emptyset$, and let $i'_{\min}$ be the smallest such index. Then for any $i,j \in [r]$ with $ij\in E(R)$, one of the following holds:\COMMENT{Because, $ij \in E(R)$ implies that $R_K[J_i,J_j]$ is a complete bipartite graph. Furthermore, $J_i \subset W^{**}_{i'}, J_j\subset W^{**}_{j'}$ for some $i'\neq j'$. WLOG, assume $i'< j'$. Then $i'_{\max} \leq i'(K\Delta)^2$ while $j'_{\min} \geq (j'-1)(K\Delta)^2+1$. Thus we get one of these. }
\begin{align}\label{round index}
i'_{\max} < j'_{\min}, ~\text{ or }\enspace j'_{\max} < i'_{\min}.
\end{align}
We let $\mathcal{I}:= \{ I_1,\dots, I_{w}\}$.
Let $\vec{d'}, \vec{\beta'}$ be $Kr\times Kr$ matrices such that $d'_{\ell,\ell'} := d_{i,j}$ and $\beta'_{\ell,\ell'} := \beta_{i,j}$ whenever $\ell\in J_i, \ell'\in J_j$. So we have now satisfied (V1)--(V3) in the definition of a valid input (with $R_K$ playing the role of $R_*$ and $d'_{\ell,\ell'},\beta'_{\ell,\ell'}$ playing the roles of $d_{i,j}$ and $\beta_{i,j}$ respectively).
\vspace{0.1cm}

\noindent {\bf Step 2}. Note that for all $1\leq i \neq j \leq Kr$, the graph $H[Y_i,Y_j]$ is a matching if $ij\in E(R_K)$ and empty if $ij\notin E(R_K)$. For all $ij\in E(R_K)$ we add edges to $H[Y_i,Y_j]$ to obtain a graph $H_*$ on $V(H)$ so that each $H_*[Y_i,Y_j]$ is a matching of size $\min\{|Y_i|,|Y_j|\}$. So $H_*$ satisfies (V6) (with $H_*$ playing the role of $H$). \vspace{0.1cm}

\noindent {\bf Step 3}. For each $i\in [r]$ we now choose an equitable partition $U_{(i-1)K+1}, \dots, U_{iK}$ of $V_i$.
In the case when $A_0[X_i,V_i]$ is complete bipartite, we could simply consider an equitable partition chosen uniformly at random.
In general, there are a few vertices we need to be more careful about. More precisely, we consider%
\COMMENT{below we have $2\epsilon$ instead of $\epsilon$ since $|Y_j|$ might be a little smaller than $m$}
$$
W_i := \{v \in V_i: \exists j \in J_i \text{ such that } |N_{A_0}(v)\cap Y_j|\neq (d_0\pm 2\epsilon) m \}.
$$ 
Since $A_0[X_i,V_i]$ is $(d_0,\epsilon)$-super-regular, 
\begin{align}\label{eq:K2epsm}
|W_i|\leq K\epsilon n \leq K^2 \epsilon m.
\end{align}
Note that for each $v\in W_i$, there exists some $j \in J_i$ such that $|N_{A_0}(v)\cap Y_j| > (d_0-2\epsilon)m$.\COMMENT{Since $\sum_{j\in J_i} |N_{A_0}(v)\cap Y_j| = |N_{A_0}(v)| = (d_0\pm \epsilon)n$, use pigeonhole principle.}
Let $(W'_{(i-1)K+1},\dots, W'_{iK})$ be a partition of $W_i$ such that $v \in W'_{j}$ implies that $|N_{A_0}(v)\cap Y_j| > (d_0-2\epsilon) m$.
Now we partition $V_i\setminus W_{i}$ into $W''_{(i-1)K+1}, \dots, W''_{iK}$ uniformly at random subject to the condition $|W''_j| = |Y_j|-|W'_j|$. Let $U_j := W'_j\cup W''_j$.
For each $j\in J_i$ and each $v\in W'_j$ we delete some arbitrary edges in $A_0$ between $v$ and $Y_i$ to obtain a spanning subgraph $A^\diamond_0$ of $A_0$ such that $|N_{A^\diamond_0}(u)\cap Y_j| = (d_0\pm 2\epsilon)m$ for all $u\in U_j$ and $j\in J_i$.
Let $A^*_0:=\bigcup_{j\in [Kr]} A^\diamond_0[Y_j,U_j]$.%
\COMMENT{So $A^*_0$ is obtained from $A^\diamond_0$ by deleting all (cross-) edges between $Y_j$ and $U_i$ with $i\neq j\in J_\ell$ for some
$\ell\in [r]$.}
So $A^*_0\subseteq A_0$.

Let $\mathcal{U}:= \{U_1,\dots, U_{Kr}\}$. Note that $G$ and $P$ admit vertex partition $(R_K,\mathcal{U})$ and $H_*$ admits vertex partition $(R_K,\mathcal{Y})$. We define the following events.
\begin{itemize}
\item[($\mathcal{G}$)] $G[U_{i},U_{j}]$ is $(\epsilon^{1/3},d'_{i,j})$-super-regular for all $ij \in E(R_K)$.
\item[($\mathcal{P}$)] $P[U_{i},U_{j}]$ is $(\epsilon^{1/3},\beta'_{i,j})$-super-regular for all $ij \in E(R_K)$.
\item[($\mathcal{A}_0^*$)] $A^*_0[Y_j,U_j]$ is $(\epsilon^{1/3},d_0)$-super-regular for all $j\in [Kr]$. 
\end{itemize}
If one of these events does not hold, we end the algorithm with \textbf{failure of type 1}. \vspace{0.1cm}

\noindent {\bf Step 4}. We apply the Slender graph algorithm on 
$$\mathcal{S}=(G,P,H_*,R_K,A^*_0,\mathcal{U},\mathcal{Y},\mathcal{I},c,\epsilon^{1/3},d_0,\vec{d}',\vec{\beta}',K,\Delta_R,C).$$
\COMMENT{Note that $c,m$ here are exactly playing the role of $c,m$ in the slender graph algorithm, respectively.} Note that if we have no failure of type 1, then $\mathcal{S}$ satisfies (V1)--(V7). Recall that in the preparation round of the Slender graph algorithm we define 
$G', H_*', P'$ and $(A^*_0)'$ by adding artificial vertices and we have $Y'_j \supseteq Y_i$ and $U'_j \supseteq U_j$ such that $|Y'_j|=|U'_j|=m$ and
\begin{itemize}
\item[($\mathcal{G}'$)]$G'$ is $(2\epsilon^{1/3},\vec{d}')$-super-regular with respect to $(R_K,U'_1,\dots, U'_{Kr})$, 
\item[($\mathcal{P}'$)] $P'$ is $(2\epsilon^{1/3},\vec{\beta}')$-super-regular with respect to $(R_K,U'_1,\dots, U'_{Kr})$, 
\item[($\mathcal{A}'_0$)] $(A^*_0)'[Y'_j,U'_j]$ is $(2\epsilon^{1/3},d_0)$-super-regular for each $j\in [Kr]$ 
\end{itemize}
(provided that the Slender graph algorithm does not abort with failure of type 1). If the Slender graph algorithm fails, we abort the Uniform embedding algorithm with {\bf failure of type 2}. Otherwise, we obtain an embedding $\phi$ and bipartite graphs $F_j$ on $(U_j, Y_j)$ for each $j\in [Kr]$. For each $x\in V(H)$, let $N_x:= N_{H_*}(x)$. Return $(\phi,\mathcal{Y},\mathcal{U},F,N)$, where $F:= \bigcup_{j=1}^{Kr} F_j$ and $N:=\{N_x : x\in V(H)\}$. \vspace{0.1cm} \newline

\begin{claim}
If $\mathcal{S}=(G,P,H,R,A_0,\mathcal{V},\mathcal{X},c,\epsilon, d_0, \vec{d},\vec{\beta},k,\Delta_R,C)$ satisfies the assumption in Lemma \ref{modified blow-up}, then
the Uniform embedding algorithm on $\mathcal{S}$ 
fails with probability at most $(1-c)^n$. 
\end{claim}
\begin{proof}
We first consider the event $(\mathcal{A}^*_0)$. Consider any $i\in [r]$. For all $j\in J_i$ and each $u \in U_{j}$ we have $|N_{A^*_0}(u)\cap Y_j|=(d_0\pm 2\epsilon) m = (d_0\pm \epsilon^{1/3}) m$ by construction. 
Also for $x\in Y_{j}$, 
\begin{align}\label{A*0 deg}
|N_{A^*_0}(x)\cap U_j| = |N_{A^*_0}(x)\cap W''_j| + |N_{A^*_0}(x)\cap W'_j| \stackrel{\eqref{eq:K2epsm}}{=} |N_{A^*_0}(x)\cap W''_j| \pm K^2\epsilon m.
\end{align}
Using the Chernoff-Hoeffding bound in Lemma~\ref{Chernoff Bounds}, it is easy to check that with probability at least $1-(1-3c)^n$ we have
\begin{align*}
|N_{A^*_0}(x)\cap W''_j| &= (1\pm \epsilon) \frac{|W''_j|}{|V_i\setminus W_i|}|N_{A_0}(x)\cap (V_i\setminus W_i)| \stackrel{\eqref{eq:K2epsm}}{=} (1\pm \epsilon) \frac{m\pm K\epsilon n}{n\pm K\epsilon n}(d_0\pm 2K\epsilon)n \\
&=(d_0\pm \epsilon^{1/3}/2)m.
\end{align*}
Together with \eqref{A*0 deg} this implies that
$$\mathbb{P}[ |N_{A^*_0}(x)\cap U_j| = (d_0\pm \epsilon^{1/3})m ] \geq 1 - (1-3c)^n.$$
Moreover, it is easy to check that $A^*_0[Y_j,U_j]$ is $(\epsilon^{1/3},d_0)$-regular.%
\COMMENT{$A_0[Y_j,U_j]$ is trivially $(K\epsilon,d_0)$-regular. To get $A^*_0$, we only delete edges incident to $W'_i\cap U_j$ which has size at most $K\epsilon n$. Thus for any set $A\subset Y_j, B\subset U_j$ with $|A|,|B| \geq \epsilon^{1/3} m$, $\frac{e_{A^*_0}(A,B)}{|A||B|} = \frac{e_{A_0}(A,B) \pm |W'_i \cap U_j||A|}{|A||B|} = (d_0\pm K\epsilon) + \frac{K\epsilon n|A|}{|A||B|} = (d_0\pm K\epsilon) + K^2\epsilon^{2/3} = (d_0\pm \epsilon^{1/3}).$ Thus $A^*_0$ is $(\epsilon^{1/3},d_0)$-regular. } Together with a union bound over all $x\in V(H)$ this shows that $A^*_0[Y_j, U_j]$ is $(\epsilon^{1/3},d_0)$-super-regular for all $j\in [Kr]$ with probability at least $1-rn(1-3c)^n$. 

Second, we consider the events ($\mathcal{G}$) and ($\mathcal{P}$). For this consider $G[U_{j}, U_{j'}]$ for $j\in J_i$ and $j'\in J_{i'}$ with $ii'\in E(R)$.  
We have that for all $u\in U_j$
$$|N_{G}(u)\cap U_{j'}| = |N_{G}(u)\cap W''_{j'}| + |N_{G}(u)\cap W'_{j'}| \stackrel{\eqref{eq:K2epsm}}{=} |N_{G}(u)\cap W''_{j'}|\pm K^2\epsilon m.$$ 
Again using Lemma~\ref{Chernoff Bounds}, with probability at least $1-(1-3c)^n$ we have 
\begin{align}\label{eq:longandboring}
|N_{G}(u)\cap W''_{j'}| = (1\pm \epsilon) \frac{|W''_{j'}|}{|V_{i'}\setminus W_{i'}|}|N_{G}(u)\cap (V_{i'}\setminus W_{i'})| =(d'_{j,j'}\pm \epsilon^{1/3}/2) m.
\end{align}
(Here we use that $d'_{j,j'}=d_{i,i'}$.) Hence, if \eqref{eq:longandboring} holds, then $|N_{G}(u)\cap U_{j'}| =(d'_{j,j'}\pm \epsilon^{1/3}/2) m \pm K^2\epsilon m = (d'_{j,j'}\pm \epsilon^{1/3}) m$. Also, $G[U_j, U_{j'}]$ is $(\epsilon^{1/3},d'_{j,j'})$-regular by Proposition \ref{restriction}. Together with a union bound over all $2n$ vertices of $G[U_j,U_{j'}]$ and all $|E(R_K)|\leq K^2\Delta_Rr$ edges of $R_K$, this shows that ($\mathcal{G}$) holds 
with probability at least $1-2nK^2\Delta_R r(1-3c)^n$. Similarly, ($\mathcal{P}$) holds 
with probability at least $1-2nK^2\Delta_R r(1-3c)^n$. 
Thus failure of type 1 in the uniform embedding algorithm occurs with probability at most $(1-2 c)^n$.

If failure of type 1 does not occur, then $\mathcal{S}$ 
is a valid input for the Slender graph algorithm.
 Thus Lemma \ref{Slendered blow-up} implies that the Slender graph algorithm on input $\mathcal{S}$ fails with probability at most $(1-2c)^{Km}$. Therefore, the uniform embedding algorithm succeeds with probability at least $ 1 - (1-2c)^n- (1-2c)^{Km} \geq 1 - (1-c)^n.$
\end{proof}

We now proceed to the main part of the proof of Lemma \ref{modified blow-up}, which is establishing (B1).

\begin{claim}\label{claim:B1}
(B1) holds.
\end{claim}
\begin{proof}
Suppose that $ij\in E(R)$, $v\in V_i$ and $S\subseteq V_j \cap N_G(v)$ satisfies $|S|> f(\epsilon)n$. Recall from Step 1 of the Uniform embedding algorithm that $J_a= \{(a-1)K+1,\dots, a K\}$ for each $a \in [r]$. Recall that in Step 3 of this algorithm, we take a partition of $V_a\setminus W_a$ (chosen uniformly at random) into parts $W''_\ell$ of size $|Y_\ell|-|W'_\ell|$ with $\ell\in J_a$ and that $|W'_\ell|\leq |W_a| \leq K\epsilon n$. Then $U_\ell = W''_\ell\cup W'_\ell$. For every $\ell \in J_j$ let $$S_\ell:=S\cap U_\ell.$$  
So $S$ is partitioned into $S_{(j-1)K+1},\dots,S_{jK}$. We define the following event.
\begin{itemize}
\item[(BB1)] $|S_{\ell}| = (1\pm \epsilon^{1/3})\frac{|S|}{K} \text{ for all } \ell \in J_j.$
\end{itemize}
Since $|S_\ell \cap W''_{\ell}|$ has the hypergeometric distribution, and $|S_\ell\cap W'_{\ell}| \leq K\epsilon n$, Lemma~\ref{Chernoff Bounds} and the fact that $c\ll \epsilon$ imply that \COMMENT{ 
$\mathbb{E}[|S_\ell \cap W''_{\ell}|] = \frac{|W''_j|}{|V_j\setminus W_j|}|S \setminus W_j| = (1\pm \epsilon^{1/3}/2) \frac{|S|}{K}$. Thus $|S_\ell| = |S_\ell\cap W''_{\ell}| + |S_\ell\cap W'_{\ell}| = (1\pm \epsilon^{1/3}/2)\frac{|S|}{K} + K\epsilon n = (1\pm \epsilon^{1/3})\frac{|S|}{K}$ with high probability.}
\begin{align}\label{mathcal B prob 1}
\mathbb{P}[(\text{BB}1) \text{ holds }] \geq 1- (1-2c)^n.
\end{align}
Recall that in Step 4 of the Uniform embedding algorithm we apply the Slender graph algorithm on $(G,  P, H_*,R_K,A^*_0, \mathcal{U}, \mathcal{Y},\mathcal{I}, c,\epsilon^{1/3}, d_0, \vec{\beta'}, \vec{d'},K,\Delta_R,C)$. In the Preparation round of the Slender graph algorithm we obtain equipartite graphs $G',H'_*, P'$ as well as the graph $(A_0^*)'$ by adding artificial vertices.  
Let $U'_1,\dots,U'_{Kr}$ denote the partition classes of $G'$ such that each $U'_{i'}$ is obtained from $U_{i'}$ by adding at most $C$ 
artificial vertices $u_{i',j'}$. 
Analogously, let $Y'_1,\dots,Y'_{Kr}$ denote the partition classes of $H'_*$ such that each $Y'_{i'}$ is obtained from $Y_{i'}$ by adding at most $C$ artificial vertices $y_{i',j'}$. 
We define the following event. It implies that the neighbourhoods of the artificial vertices are `well behaved' with respect to $S$.
Recall that $A_0^\ell$ is defined in terms of $(A_0^*)'$ before~(\ref{definition p 0}).
\begin{itemize}
\item[(BB2)] For all $\ell\in J_j$, all $Q_1 = \{q_1,\dots, q_s\} \subseteq V(G')\setminus V(G)$ with $|Q_1|\leq K\Delta_R$, $q_{\ell'} \in U'_{j_{\ell'}}$ and $j_{\ell'} \in N_{R_K}(\ell)$, for all $x\in Y'_{\ell}$ and all $Q_2 \subseteq  V(G)\setminus U_{\ell}$ with $|Q_2|\leq K\Delta_R$ we have
$$ |S_{\ell}\cap N_{A_0^\ell}(x) \cap N_{G'}(Q_1) \cap N_{G'}(Q_2)| = (\prod_{\ell'=1}^{s}d'_{\ell,j_{\ell'}}) |S_{\ell}\cap N_{A_0^\ell}(x)\cap N_{G'}(Q_2)| \pm 2\epsilon n.$$ 
\end{itemize}
So the set $Q_1$ consists only of artificial vertices outside $U'_\ell$. 
Recall from the Preparation round of the Slender graph algorithm that  $|N_{G'}(u_{i',j'})\cap Q|$ has the binomial distribution for each artificial vertex $u_{i',j'} \in U'_{i'} \setminus U_{i'}$ and each $Q\subseteq U_{i''}$ with $i'i''\in E(R_K)$. So%
\COMMENT{Also because the choice of $Q_1,Q_2$ are at most $n^{2Kr}$.} Lemma~\ref{Chernoff Bounds} and the fact that $c\ll \epsilon$ imply that
\begin{align}\label{mathcal B prob 2}
\mathbb{P}[(\text{BB}2) \text{ holds }] \geq 1- (1-2c)^n.
\end{align}
\noindent
We define $\mathcal{B}_S := (\text{BB}1)\wedge (\text{BB}2)$, then (\ref{mathcal B prob 1}) and (\ref{mathcal B prob 2}) imply
\begin{align}\label{mathcal B prob}
\mathbb{P}[\mathcal{B}_S] \geq 1- 2(1-2c)^n.
\end{align}
Note that it is easy to see that two applications of (BB2) imply the following.
\begin{itemize}
\item[(BB$2'$)] For all $\ell \in J_j$, any $Q_1\subseteq \bigcup_{j'\in N_{R_K}(\ell)}(U'_{j'}\setminus U_{j'})$ with $|Q_1| \leq K\Delta_R-1$, any $Q_2 \subseteq V(G)\setminus U_{\ell}$ with $|Q_2|\leq K\Delta_R$, any $x\in Y'_{\ell}$, and any $z\in U'_{\ell'} \setminus (V(G)  \cup Q_1)$ with $\ell'\in N_{R_K}(\ell)$,
$$|N_{G'}(z) \cap S_\ell \cap N_{A^{\ell}_0}(x) \cap N_{G'}(Q_1)\cap N_{G'}(Q_2)| = d'_{\ell,\ell'} |S_\ell \cap N_{A^{\ell}_0}(x) \cap N_{G'}(Q_1)\cap N_{G'}(Q_2)| \pm 4\epsilon n.$$
\end{itemize}
(BB$2'$) will be used to check that after each round the artificial vertices have the `right' number of neighbours in $S_\ell$ in the current candidacy graph. Recall that in Step $1$ of the Uniform embedding algorithm we have chosen a partition $\mathcal{I}=(I_1,\dots, I_w)$ of $V(R_K)$ such that each $I_b$ is an independent set in $(R_K)^2$. Together with Observation \ref{J scatter} this means that there are
$a, b_1,\dots,b_K \in [Kr]$ and $a',b'_1,\dots, b'_K\in [w]$ satisfying $b'_1< \dots< b'_K$ as well as
\begin{align}\label{aa'bb'}
v\in U'_{a},\enspace a\in I_{a'},\enspace a\in J_i, \enspace J_j\subseteq \bigcup_{s=1}^{K}I_{b'_{s}} \enspace \text{ and } \enspace J_j\cap I_{b'_s} =\{b_s\},  
\end{align}
where $v\in V_i$ is as defined in $(\text{B}1)$. So $J_j=\{b_1,\dots, b_K\}$.

Note that by \eqref{eq:xirecursion} for all $t\in[w]$ we have 
\begin{align}\label{eq:xitexplicit}
\xi_{t}=g^{t}(2\epsilon^{1/3}) ~\text{ and }~ \xi_0= 2\epsilon^{1/3}.
\end{align}
(Indeed, recall that we apply the Slender graph algorithm 
with $\epsilon^{1/3}$ playing the role of $\epsilon$.\COMMENT{while other constants like $c,m,\beta,d,K,C,w$ are playing the role of itselves where $m$ is defined as below.})
Moreover,
\begin{align}
m=|U'_\ell|=|Y'_\ell|= \left\lceil \frac{n}{K}\right\rceil.
\end{align}
So $m$ is playing the role of $m$ in the Slender graph algorithm, which is the size of a largest partition class in $\mathcal{U}$ and $\mathcal{Y}$. Let
\begin{align}
&x:= \phi^{-1}(v), \nonumber\\
\label{tau def}&\tau_{\ell,\ell'}: Y'_{\ell} \rightarrow Y'_{\ell'} \text{ with } \{ \tau_{\ell,\ell'}(z)\} = N_{H'_*}(z)\cap Y'_{\ell'} \text{ for all } z\in Y'_{\ell} \text{ with } \ell\ell' \in E(R_K), \\
&y'_{b} := \tau_{a,b}(x) ~\text{ and }~ y'_{b,\ell}:= \tau_{b,\ell}(y'_b) \text{ whenever } ab, b\ell  \in E(R_K).\nonumber
\end{align}
Note that $x, y'_b$ and $y'_{b,\ell}$ are random variables, where $x$ is fixed in Round $a'$, and that $\tau_{\ell,\ell'}$ is a bijection whenever $\ell \ell' \in E(R_K)$.
Recall that $S\subseteq V_j$, so $S\cap U'_{b}\neq \emptyset$ only for $b \in J_j=\{b_1,\dots, b_K\}$ (cf.~\eqref{aa'bb'}). By (\ref{round index}), either $a' < b'_1$ or $b'_K < a'$ holds. We treat these two cases separately. Roughly speaking, in the first case we condition on an embedding of a vertex $x^*$ onto $v$ in Round~$a'$ and show that with high probability the `right' number of vertices of $N_H(x^*)$ is embedded to $S$. In the second case we condition on an embedding of $H$ restricted to $S$; (B1) is then determined by the choice of $x=\phi^{-1}(v)$.
\vspace{0.2cm}

\noindent CASE 1. $a'< b'_1$.

For $a'\leq \ell' \leq w$, and $b\in J_j$ we define the following events.
\begin{itemize}
\item[(B$1_{\ell'}^b$)] $|N_{A^{b}_{\ell'}}(y'_{b})\cap S_{b}| = (d_{i,j}^{-1}p(\vec{d'},b,\ell')\pm \xi_{\ell'})|S_{b}|$.
\item[(B$1_{\ell'}$)] $\displaystyle \bigwedge_{b_s \in J_j: b'_s>\ell'} (\text{B}1_{\ell'}^{b_s})$ holds.
\end{itemize}
Here $p(\vec{d'},b,\ell')$ is as defined in \eqref{P relation} with $R_K$ playing the role of $R_*$, and $A_{\ell'}^b$ is as defined in \eqref{Aji x def}.
Note that (B$1_{\ell'}$) is vacuously true if $\ell'\ge b'_K$. Moreover, by construction, $ab\in E(R_K)$ for every $b\in J_j$,
and hence the vertex $y'_b$ in (B$1_{\ell'}^b$) is well-defined.

Recall that the event $\mathcal{A}_{\ell'}$ was defined before Lemma~\ref{slender lemmas}. To motivate (B$1_{\ell'}^b$), note that conditional on $\mathcal{A}_{\ell'}$, $p(\vec{d'},b,\ell')$ is the density of $A_{\ell'}^b$. As one would expect from the algorithm, in Round $\ell'>a'$ (i.e~after the pre-image $x$ of $v$ has been chosen), 
the probability that $y_b'$ is embedded onto a vertex in $S_b$ will turn out to be close to
$|N \cap S_b|/|N|$, where $N:=N_{A_{\ell'}^b}(y_b')$. This is because the set of candidates for $\phi(y_b')$ will be
almost all of $N$ and every candidate will be (almost)
equally likely. Since $S_b \subseteq N_{G'}(v) \cap U_b$, and since $|N_{G'}(v) \cap U_b| \sim d_{i,j}m$, for $\ell'>a'$
we expect this probability to be close to $|S_b|/d_{i,j}m$, which is equivalent to (B1$_{\ell'}^b$).

We define events $\mathcal{B}_{\ell'}$ in the following way:
$$ \mathcal{B}_{\ell'} := \left\{ \begin{array}{ll}
\mathcal{A}_{\ell'} & \text{ for }0\leq \ell' \leq a'-1,\\
\mathcal{A}_{\ell'} \wedge (\text{B}1_{\ell'}) & \text{ for }a'\leq \ell'\leq w.\\
\end{array} \right.$$
We will see that $\mathcal{B}_{\ell'}$ typically holds for any $0\leq \ell'\leq w$. 
We let $\mathcal{B}_{0}^{\ell'}:=\bigwedge_{j'=0}^{\ell'} \mathcal{B}_{j'}$ and let $\mathcal{B}_0^{-1}$ be the event which always occurs. Thus $\mathcal{B}_0^0 = \mathcal{B}_0^{-1}$ since $\mathcal{A}_0$ always holds. 
Note that $\mathcal{B}_0^{\ell'-1}$ and $\mathcal{B}_S$ are events which depend only on the history of the Slender graph algorithm prior to Round $\ell'$. Thus Lemma \ref{slender lemmas}(ii) implies that for all $0\leq\ell' \leq w$,\COMMENT{Note that we are on slender graph algorithm on $(G,P,H_*,R_K,A^*_0,\mathcal{U},\mathcal{Y},\mathcal{I},c,\epsilon^{1/3},d_0,\vec{d}',\vec{\beta}',K,\Delta_R,C)$. Thus $c,m$ are playing the role of $c,m$ in Lemma~\ref{slender lemmas}(ii), respectively. Thus we get $1-(1-3c)^{Km}$.}
\begin{align}\label{A ell B B}
\mathbb{P}[\mathcal{A}_{\ell'} \mid \mathcal{B}_0^{\ell'-1} ,\mathcal{B}_S]\geq 1- (1-3c)^{Km} \geq 1 - (1-3c)^n.
\end{align}
Together with (\ref{mathcal B prob}) this implies that
\begin{align}\label{B 1 a-1 B}
\mathbb{P}[\mathcal{B}_0^{a'-1},\mathcal{B}_S] &= \mathbb{P}[\mathcal{B}_S]\prod_{\ell'=0}^{a'-1}\mathbb{P}[\mathcal{B}_{\ell'} \mid \mathcal{B}_0^{\ell'-1} ,\mathcal{B}_S] = \mathbb{P}[\mathcal{B}_S]\prod_{\ell'=0}^{a'-1}\mathbb{P}[\mathcal{A}_{\ell'} \mid \mathcal{B}_0^{\ell'-1} ,\mathcal{B}_S] \geq  1-3(1-2c)^n. 
\end{align}

\noindent {\bf Subclaim 1.}
$\mathbb{P}[\mathcal{B}_{a'}\mid \mathcal{B}_0^{a'-1},\mathcal{B}_S] \geq 1- \xi_{a'}.$

To prove Subclaim 1, assume that $\mathcal{B}_S$ and $\mathcal{B}_0^{a'-1}$ hold and we are in Round $a'$ of the Slender graph algorithm. For each $b\in J_j$ consider the graph 
$F^b_{13}(a) := (Y'_{a},U'_b;E)$, as described in the proof of Lemma \ref{slender lemmas}(ii), i.e. $yv \in E$ if and only if $y\in Y'_a, v\in U'_b$ and $\tau_{a,b}(y)v \in E(A_{a'-1}^b)$. Let $A(a)$ be the subgraph of $A^{a}_{a'-1}$ as defined in \eqref{def A(i)} of the Slender graph algorithm.

Since $\tau_{a,b}$ is a bijection, 
the graph $F_{13}^b(a)$ is isomorphic to $A_{a'-1}^b$ under the isomorphism which keeps the elements of $U'_b$ fixed while mapping $Y'_a$ to $Y'_b$ via $\tau_{a,b}$.
Thus $F_{13}^b(a)$ is  $(\xi_{a'-1} ,p(\vec{d'},b,a'-1))$-super-regular using $\mathcal{B}_{a'-1}=\mathcal{A}_{a'-1}$. Also Lemma~\ref{slender lemmas}(i) implies that 
\begin{align}\label{A(a) super-regular}
A(a)\text{ is }(4K\Delta_R\sqrt{\xi_{a'-1}},p(\vec{d'},a,a'-1))\text{-super-regular}.
\end{align}
Moreover, by Proposition \ref{restriction}, 
\begin{align}\label{Fb13a}
F^b_{13}(a)[N_{A(a)}(v) , S_b]\text{ is } (\xi_{a'-1}^{1/3},p(\vec{d'},b,a'-1))\text{-regular} 
\end{align}
since $|S_b|> f(\epsilon)m/2$ by $\mathcal{B}_S$. (Here we also use that $2\xi_{a'-1}/f(\epsilon)\leq \xi^{1/3}_{a'-1}$ by \eqref{eq:xitexplicit}). 
Let 
$$T_b:=\{y\in N_{A(a)}(v)\cap Y_a: |N_{F_{13}^b(a)}(y)\cap S_b|=p(\vec{d'},b,a'-1)(1\pm \xi_{a'-1}^{1/4})|S_b|\}.$$
Recall that the vertex $x=\phi^{-1}(v)$ is determined in Round $a'$ of the Slender graph algorithm.
So $T_b$ is the set of candidates for $x=\phi^{-1}(v)$ so that $\tau_{a,b}(x)$ is `typical' with respect to $S_b$.
Then \eqref{A(a) super-regular} and \eqref{Fb13a} together imply that 
\begin{align}\label{size of Tb}
|T_b| \geq (1-\xi_{a'-1}^{1/4})p(\vec{d'},a,a'-1) m.
\end{align}
\noindent
Since $N_{A^b_{a'}}(y'_b) = N_{A^b_{a'-1}}(y'_b) \cap N_{G'}(v)$ (see (\ref{Aji x def})) and $S_b\subseteq N_{G'}(v)$, it follows that $N_{A^b_{a'}}(y'_b)\cap S_b  = N_{A^b_{a'-1}}(y'_b)\cap S_b$. This implies that if $x\in T_b$, then 
\begin{eqnarray*}
|N_{A^b_{a'}}(y'_b)\cap S_b|  &=& |N_{A^b_{a'-1}}(y'_b)\cap S_b|=|N_{F^b_{13}(a)}(x)\cap S_b| = p(\vec{d'},b,a'-1) (1\pm \xi_{a'-1}^{1/4})|S_b|\\ &\stackrel{\eqref{xi property}}{=}& (p(\vec{d'},b,a'-1)\pm \xi_{a'})|S_b|= (d^{-1}_{i,j} p(\vec{d'},b,a')\pm \xi_{a'})|S_b|.
\end{eqnarray*}
We get the final equality by (\ref{P relation}) and the fact that $d'_{a,b} = d_{i,j}$ as $a\in J_i, b\in J_j$ and $N_{R_K}(b)\cap I_{a'} = \{a\}$ as $ij\in E(R)$ and so $ab\in E(R_K)$. Thus by $(\text{M}'2)_{a'}$, 
\begin{align*} 
\mathbb{P}\left[ |N_{A^b_{a'}}(y'_b)\cap S_b| = ( d^{-1}_{i,j} p(\vec{d'},b,a') \pm \xi_{a'})|S_b| \mid \mathcal{B}_0^{a'-1}, \mathcal{B}_S\right] \geq \mathbb{P}\left[\phi^{-1}(v) \in T_b \mid \mathcal{B}_0^{a'-1},\mathcal{B}_S\right]
\end{align*}
\begin{eqnarray}\label{Nbay Sb}
 &\geq& (1-h(4K\Delta_R\sqrt{\xi_{a'-1}})) \frac{|T_b|}{p(\vec{d'},a,a'-1)m} \hspace{7.3cm} \nonumber \\ 
&\stackrel{\eqref{size of Tb},\eqref{eq:functions}}{\geq}& 1-\frac{\xi_{a'}}{2K}.
\end{eqnarray}
Thus taking a union bound over all $b\in J_j$ in (\ref{Nbay Sb}), we obtain
\begin{align}\label{B1a}
\mathbb{P}[(\text{B}1_{a'}) \mid \mathcal{B}_0^{a'-1},\mathcal{B}_S ]\geq 1 - \frac{\xi_{a'}}{2}.
\end{align}
By (\ref{A ell B B}) and (\ref{B1a}) we conclude
$$\mathbb{P}[\mathcal{B}_{a'}\mid \mathcal{B}_0^{a'-1},\mathcal{B}_S]  =\mathbb{P}[ \mathcal{A}_{a'}, (\text{B}1_{a'})\mid \mathcal{B}_0^{a'-1},\mathcal{B}_S ] \geq 1-(1-3c)^n- \frac{\xi_{a'}}{2} \geq 1-\xi_{a'}.$$
This completes the proof of Subclaim 1. \vspace{0.2cm}

\noindent {\bf Subclaim 2.}
Suppose that $a'< \ell' \leq w$ and $x^*\in X_i$ are such that $\mathbb{P}[\mathcal{B}_S,\mathcal{B}_0^{\ell'-1},\phi^{-1}(v)=x^*]>0$.
Then $\mathbb{P}[\mathcal{B}_{\ell'} \mid \mathcal{B}_S, \mathcal{B}^{\ell'-1}_0, \phi^{-1}(v)=x^*] \geq 1 - \xi_{\ell'}.$ 

To prove Subclaim 2, assume that both $\mathcal{B}_S$, $\mathcal{B}_0^{\ell'-1}$ hold, that $\phi^{-1}(v)=x^*$ and that we are in Round $\ell'$ of the Slender graph algorithm. Recall that Round $a'$ of the Slender graph algorithm determines $\phi^{-1}(v)$, which in turn determines $y'_b$ and $y'_{b,\ell}$ for all $b \in N_{R_K}(a)$ and $\ell \in N_{R_K}(b)$.
Note that for fixed $\ell' > a'$, each of $\mathcal{B}_0^{\ell'-1}$, $\mathcal{B}_S$, and $\phi^{-1}(v)=x^*$ is an event which depends only on the history prior to Round $\ell'$. Thus by Lemma \ref{slender lemmas}(ii),
\begin{align}\label{mathcal A prob under xv}
\mathbb{P}[\mathcal{A}_{\ell'} \mid \mathcal{B}_0^{\ell'-1} ,\mathcal{B}_S,\phi^{-1}(v)=x^*]\geq 1- (1-3c)^{Km} \geq 1 - (1-3c)^n.
\end{align}

Consider some $b\in J_j$ with $b \in I_{b'}$ for some $b'> \ell'$. First, assume that $b \notin N_{R_K}(I_{\ell'})$. Then $A^{b}_{\ell'} = A^{b}_{\ell'-1}$ by \eqref{stay same}. So $(\text{B}1_{\ell'-1}^{b})$ implies $(\text{B}1_{\ell'}^{b})$ since in this case $p(\vec{d'},b,\ell'-1) = p(\vec{d'},b,\ell')$ by \eqref{P relation} and since $\xi_{\ell'-1} \leq \xi_{\ell'}$. So $(\text{B}1_{\ell'}^b)$ holds in this case.

So let us next assume that $b \in N_{R_K}(I_{\ell'})$.
Then by \eqref{max deg R_K I} there is a unique index $\ell \in N_{R_K}(b) \cap I_{\ell'}$. 

First we verify (B1$^b_{\ell'}$) in the case when $y'_{b,\ell}$ is an artificial vertex, i.e. $y'_{b,\ell} = y_{\ell,j'} \in Y'_{\ell}\setminus Y_{\ell}$ for some $j'\in [C]$. Then, as described after \eqref{def A(i)}, we have $f_{\ell'}(y'_{b,\ell})=u_{\ell,j'} \in V(G')\setminus V(G)$. 
Note that 
\begin{align}\label{NAbellybS 1}
N_{A^b_{\ell'}}(y'_b)\cap S_b \stackrel{\eqref{Aji x def}}{=} N_{G'}(f_{\ell'}(y'_{b,\ell})) \cap N_{A^b_{\ell'-1}}(y'_b)\cap S_b
\end{align}
 and 
 \begin{align}\label{NAbellybS 2}
 N_{A^b_{\ell'-1}}(y'_b)\cap S_b \stackrel{\eqref{eq:updatecandgraphs}}{=} N_{A_0^b}(y'_b)\cap N_{G'}(f_{\ell'-1}(N_{\ell'-1}(y'_{b}))) \cap S_b.
 \end{align}
  Let 
$$Q_1 := (V(G')\setminus V(G))\cap f_{\ell'-1}(N_{\ell'-1}(y'_{b}))
\enspace \text{ and } \enspace Q_2 :=V(G)\cap f_{\ell'-1}(N_{\ell'-1}(y'_{b})) .$$ 
Since $f_{\ell'}( y'_{b,\ell}) \in V(G')\setminus (V(G)\cup Q_1)$, similarly as in \eqref{artificial vertex fine}, (BB$2'$)
(with $b$, $\ell$ playing the roles of $\ell$, $\ell'$) implies that 
\begin{eqnarray} \label{artificial vertex fine 2}
|N_{A^b_{\ell'}}(y'_b)\cap S_b | &\stackrel{\eqref{NAbellybS 1},\eqref{NAbellybS 2}}{=}&  |N_{G'}(f_{\ell'}(y'_{b,\ell}) ) \cap S_b \cap N_{A_0^b}(y'_b)\cap N_{G'}(Q_1)\cap N_{G'}(Q_2)| \nonumber\\
&=& d'_{\ell,b} |S_b \cap N_{A_0^b}(y'_b)\cap N_{G'}(Q_1)\cap N_{G'}(Q_2)| \pm 4\epsilon n \nonumber\\
&\stackrel{\eqref{NAbellybS 2}}{=}& d'_{\ell,b} |N_{A^b_{\ell'-1}}(y'_b)\cap S_b| \pm 4\epsilon n \nonumber\\
&\stackrel{(\text{B}1_{\ell'-1}^b)}{=}& d'_{\ell,b}(d^{-1}_{i,j} p(\vec{d'},b,\ell'-1) \pm \xi_{\ell'-1})|S_b| \pm 4\epsilon n  \nonumber\\
&\stackrel{\eqref{P relation}}{=}& (d^{-1}_{i,j} p(\vec{d'},b,\ell')\pm \xi_{\ell'})|S_b|.
\end{eqnarray}
In the final equality we use that $|S_b| > f(\epsilon) n/(2K)$ since we conditioned on $\mathcal{B}_S$.\COMMENT{So need $f(\epsilon)\geq \frac{8K\epsilon}{\xi_{\ell'-1}}=\frac{8K\epsilon}{g^{\ell'-1}(2\epsilon^{1/3})}$, similarly in \eqref{A(b) S_b}.}
Thus \eqref{artificial vertex fine 2} implies that 
\begin{align}\label{artificial yb}
\mathbb{P}[(\text{B}1_{\ell'}^{b}) \mid \mathcal{B}_S,\mathcal{B}_0^{\ell-1},\phi^{-1}(v)=x^*, y'_{b,\ell} = y_{\ell,j'}] =1.
\end{align}
Now assume that $y'_{b,\ell}$ is not an artificial vertex, i.e. $y'_{b,\ell} \in V(H_*)$.
Note that 
\begin{align}\label{eq:Naly'bl}
|N_{A(\ell)}(y'_{b,\ell})|= (p(\vec{d'},\ell,\ell'-1)\pm 4K\Delta_R\sqrt{\xi_{\ell'-1}})m
\end{align}
since $A(\ell)$ is $(4K\Delta_R\sqrt{\xi_{\ell'-1}},p(\vec{d'},\ell,\ell'-1))$-super-regular by Lemma \ref{slender lemmas}(i) and $\mathcal{B}_{\ell'-1}$. Moreover, $\mathcal{B}_{\ell'-1}$ implies that
\begin{equation} \label{recall}
|N_{A^b_{\ell'-1}}(y'_b)\cap S_b|=(d^{-1}_{i,j} p(\vec{d'},b,\ell'-1) \pm \xi_{\ell'-1})|S_b|.
\end{equation}
Let%
\COMMENT{in the def of $N^{\ell}(S_b)$ we intersect with $U_\ell$ to avoid artificial vertices for later application of $(\text{M}'1)_{\ell'}$}
$$
N^{\ell}(S_b) := \{ u\in N_{A(\ell)}(y'_{b,\ell})\cap U_\ell : |N_{G'}(u)\cap N_{A^b_{\ell'-1}}(y'_b)\cap S_b|=d'_{b,\ell}(d^{-1}_{i,j} p(\vec{d'},b,\ell'-1) \pm \xi_{\ell'})|S_b|\}.
$$ 
Note that~\eqref{NAbellybS 1} implies that (B$1_{\ell'}^b$) holds if $\phi(y'_{b,\ell}) \in N^{\ell}(S_b)$ (recall that $d'_{b,\ell}p(\vec{d'},b,\ell'-1)=p(\vec{d'},b,\ell')$). 
We will show that almost all current candidates $ u\in N_{A(\ell)}(y'_{b,\ell})$ lie in $N^{\ell}(S_b)$. 

Since $G'[U'_\ell,U'_b]$ is $(2\epsilon^{1/3},d'_{b,\ell})$-super-regular by ($\mathcal{G}'$) in Step 4, 
$$
|N^\ell(S_b)| \stackrel{(\ref{recall})}{\geq} (1- \epsilon^{1/6})|N_{A(\ell)}(y'_{b,\ell})|\stackrel{\eqref{eq:Naly'bl}}{\geq} ( p(\vec{d'},\ell,\ell'-1)-5K\Delta_R\sqrt{\xi_{\ell'-1}}) m.$$
Then, by $(\text{M}'1)_{\ell'}$, 
\begin{align}\label{N ell S b}
\mathbb{P}[\phi(y'_{b,\ell}) \in N^{\ell}(S_b)\mid  \mathcal{B}_0^{\ell'-1},\mathcal{B}_S, \phi^{-1}(v)&=x^*,y'_{b,\ell} \in V(H_*)] \geq (1-h(4K\Delta_R\sqrt{\xi_{\ell'-1}})) \frac{|N_\ell(S_b)|}{ p(\vec{d'},\ell,\ell'-1) m } \nonumber\\ 
&\geq 1-2h(4K\Delta_R\sqrt{\xi_{\ell'-1}}). 
\end{align}
So (\ref{N ell S b}) implies that
 \begin{align}\label{B2ell}
\mathbb{P}[(\text{B}1_{\ell'}^b) \mid  \mathcal{B}_0^{\ell'-1},\mathcal{B}_S, \phi^{-1}(v)=x^*,y'_{b,\ell} \in V(H_*) ] \geq 1 -  2h(4K\Delta_R\sqrt{\xi_{\ell'-1}}).
  \end{align}

Let $B:=\{b\in J_j \cap N_{R_K}(I_{\ell'}) : b\in I_{b'} \text{ with } b'> \ell'\}$ and for $b\in B$ let $\ell_b$ be the unique index such that $\{\ell_b\} = N_{R_K}(b)\cap I_{\ell'}$. So $B$ keeps track of those cluster indices in $J_j$ where the future candidates will be affected by the embeddings in the current round. Note that $|B|\leq |J_j|=K$. Thus by (\ref{mathcal A prob under xv}), (\ref{artificial yb}), (\ref{B2ell}) and our previous observation that $(\text{B}1_{\ell'}^{b})$ holds for all $b\in J_j \setminus N_{R_K}(I_{\ell'})$ with $b\in I_{b'}$ for some $b'>\ell'$, we have
 \begin{align*} &\mathbb{P}[\overline{\mathcal{B}_{\ell'}} \mid \mathcal{B}_S, \mathcal{B}^{\ell'-1}_0, \phi^{-1}(v)=x^*] 
\leq \mathbb{P}[\overline{\mathcal{A}_{\ell'}}\mid \mathcal{B}_S,\mathcal{B}_0^{\ell'-1},\phi^{-1}(v)=x^*]  \\ 
&+ \sum_{b \in B }\mathbb{P}[ \overline{(B1_{\ell'}^b)} \mid \mathcal{B}_S, \mathcal{B}^{\ell'-1}_0, \phi^{-1}(v)=x^*, y'_{b,\ell_b} \in V(H_*)]\mathbb{P}[  y'_{b,\ell_b} \in V(H_*) \mid \mathcal{B}_S, \mathcal{B}^{\ell'-1}_0, \phi^{-1}(v)=x^*]\\
 &+\sum_{b\in B } \mathbb{P}[ \overline{(B1_{\ell'}^b)} \mid \mathcal{B}_S, \mathcal{B}^{\ell'-1}_0, \phi^{-1}(v)=x^*,  y'_{b,\ell_b} \notin V(H_*)]\mathbb{P}[  y'_{b,\ell_b} \notin V(H_*) \mid \mathcal{B}_S, \mathcal{B}^{\ell'-1}_0, \phi^{-1}(v)=x^*] \\ &\leq  (1-3c)^n + 2 K\cdot h(4K\Delta_R\sqrt{\xi_{\ell'-1}}) +0 \leq \xi_{\ell'},
\end{align*} 
 which proves Subclaim 2. \vspace{0.2cm}
 
We are now ready to conclude the proof of Claim \ref{claim:B1} in Case 1. Consider any $b\in J_j$ and let $b'$ be such that $b\in I_{b'}$.
Then $b'>a'$ since we are in Case 1. Let $X_i^{nice}$ be the set of all those $x^*\in X_i$ with
$\mathbb{P}[\mathcal{B}_S, \mathcal{B}_0^{a'},\phi^{-1}(v)=x^*]>0$.%
    \COMMENT{In particular, $X_i^{nice}\subseteq Y_a$. Moreover, if $x^*\in X_i^{nice}$ then by Subclaim~2
$\mathbb{P}[\mathcal{B}_S, \mathcal{B}_0^{i'},\phi^{-1}(v)=x^*]>0$ for all $i'>a'$, and so the calculations
in (\ref{assume b x v}) and (\ref{pro fixed x}) below make sense.}
Consider any $x^*\in X_i^{nice}$.
We first calculate the probability that $\phi(y'_b) \in S_b$ conditional on $\mathcal{B}_S,\mathcal{B}_0^{b'-1}, \phi^{-1}(v)=x^*$. Note that Lemma \ref{slender lemmas}(i) implies that
\begin{eqnarray}\label{A(b) S_b}
|N_{A(b)}(y'_b)\cap S_b|&=& |N_{A^b_{b'-1}}(y'_b)\cap S_b| \pm 4K\Delta_R \xi_{b'-1} m \nonumber \\
&\stackrel{(\text{B}1_{b'-1})}{=}&(d^{-1}_{i,j} p(\vec{d'},b,b'-1) \pm 4K\Delta_R \sqrt{\xi_{b'-1}})|S_b|.
\end{eqnarray}
(In the last equality we also use that $|S_b| \geq f(\epsilon) m /2$ by $\mathcal{B}_S$.)
 Suppose that $y'_b \in V(H_*)$. Then by $(\text{M}'1)_{b'}$,
\begin{eqnarray}\label{assume b x v}
\mathbb{P}[\phi(y'_b) \in S_b \mid \mathcal{B}_S, \mathcal{B}_0^{b'-1},\phi^{-1}(v)=x^*]  &=& (1\pm h(4K\Delta_R\sqrt{\xi_{b'-1}}))\frac{|N_{A(b)}(y'_b)\cap S_b|}{p(\vec{d'},b,b'-1) m} \nonumber\\& \stackrel{\eqref{A(b) S_b},(\text{BB}1)}{=}& \left(1\pm \xi_{b'}\right) \frac{|S|}{d_{i,j}n}.
\end{eqnarray}
By Subclaim 2, (\ref{assume b x v}), and the fact that $2 w \xi_{w} < f(\epsilon)^2/3$ and $\xi_{i'}< \xi_{i''}$ for $i'< i''$ (cf. \eqref{eq:functions} and \eqref{eq:xitexplicit}), we obtain
\begin{align}\label{pro fixed x}
&\mathbb{P}[\phi(y'_b) \in S_b \mid \mathcal{B}_S, \mathcal{B}_0^{a'},\phi^{-1}(v)=x^*]  \nonumber \\
&= \mathbb{P}[\phi(y'_b)\in S_b \mid \mathcal{B}_S,\mathcal{B}_0^{b'-1},\phi^{-1}(v)=x^*]\prod_{i'=a'}^{b'-2}\mathbb{P}[\mathcal{B}_{i'+1} \mid \mathcal{B}_S,\mathcal{B}_0^{i'} , \phi^{-1}(v)=x^* ]\nonumber \\
& \ \ \  \pm \sum_{i'=a'}^{b'-2}\mathbb{P}[\overline{\mathcal{B}_{i'+1}} \mid \mathcal{B}_S,\mathcal{B}_0^{i'} , \phi^{-1}(v)=x^* ] \nonumber \\
& = \left(1\pm \xi_{b'}\right) \frac{|S|}{d_{i,j}n} (1\pm \xi_{w})^{b'-a'-1} \pm (b'-a'-1)\xi_{w}= \left(1 \pm \frac{f(\epsilon)}{3}\right) \frac{|S|}{d_{i,j}n}.
\end{align}

Note that once $x=\phi^{-1}(v)$ has been determined, the set $N_H(x)\cap X_j$ is also determined. Recall that $|N_H(x)\cap X_j| \in \{k_{i,j}, k_{i,j}+1\}$ since $H$ is $(R,\vec{k},C)$-near-equiregular. Let $X_i^{good}$ be the set of all those vertices $x^*\in X_i^{nice}$ that satisfy $|N_H(x^*)\cap X_j|=k_{i,j}$. First assume that $x\in X^{good}_i$. Write $N_H(x)\cap X_j = \{x_1,\dots, x_{k_{i,j}}\}$ and for $\ell \in [k_{i,j}]$ define $b''_\ell$ by $x_\ell \in Y_{b''_\ell}$.  So $b''_{\ell}\in J_j$ and $x_{\ell}=y'_{b''_\ell}$. Thus by (\ref{pro fixed x}), for each $x^*\in X^{good}_i$ we have
\begin{align}\label{Expect B, B1a, xv}
\mathbb{E}\left[ |\phi( N_H(x^*) )\cap S| \mid \mathcal{B}_S,\mathcal{B}_0^{a'}, \phi^{-1}(v)=x^*\right] &= \sum_{\ell=1}^{k_{i,j}} \mathbb{P}[\phi(y_{b''_\ell})\in S_{b''_\ell} \mid  \mathcal{B}_S,\mathcal{B}_0^{a'},\phi^{-1}(v)=x^*] \nonumber \\
&= \left(1\pm \frac{f(\epsilon)}{3}\right) \frac{k_{i,j}|S|}{d_{i,j}n}.
\end{align}

We now show that the contribution from the vertices in $X_i^{nice}\setminus X^{good}_i$ is insignificant (see \eqref{k+1 neighbours}).
Note that \eqref{B 1 a-1 B} and Subclaim 1 together imply that
\begin{align}\label{B 1 to a}
\mathbb{P}[\mathcal{B}_S,\mathcal{B}_0^{a'}] \geq 1 - 2\xi_{a'}.
\end{align}

Since  $A(a)$ is $(4K\Delta_R\sqrt{\xi_{a'-1}},p(\vec{d'},a,a'-1))$-super-regular by Lemma \ref{slender lemmas}(i) and $\mathcal{B}_0^{a'-1}$, we can apply Theorem \ref{MM} to see that for any $x^*\in X_i^{nice}$,
\begin{align}\label{Prob xv mid B B1a-1}
\mathbb{P}[ \phi^{-1}(v)=x^* \mid \mathcal{B}_S,\mathcal{B}_0^{a'-1} ] \leq \frac{2}{ p(\vec{d'},a,a'-1) m}.
\end{align}
Note that
\begin{align*} &\mathbb{P}[ \phi^{-1}(v)=x^* \mid \mathcal{B}_S,\mathcal{B}_0^{a'-1} ] = \mathbb{P}[ \phi^{-1}(v)=x^*\mid \mathcal{B}_S,\mathcal{B}_0^{a'} ] \mathbb{P}[\mathcal{B}_{a'}\mid \mathcal{B}_S,\mathcal{B}_0^{a'-1} ] 
\pm \mathbb{P}[\overline{\mathcal{B}_{a'}}\mid \mathcal{B}_S,\mathcal{B}_0^{a'-1} ].
\end{align*}
Together with Subclaim 1 and the fact that $3kC\xi_{a'} < q_*(\epsilon^{1/3})$\COMMENT{$\epsilon^{1/3}$ here because $\epsilon^{1/3}$ plays the role of $\epsilon$ for Slender graph algorithm, and $\xi_{t}= g^t(2\epsilon^{1/3})$} (by (\ref{eq:functions}) and
the assumption $\epsilon \ll 1/k,1/(C+1)$) this yields
\begin{eqnarray}\label{eq:case1eqnarray}  
\mathbb{P}[ \phi^{-1}(v)=x^* \mid \mathcal{B}_S,\mathcal{B}_0^{a'} ] &=&
\frac{\mathbb{P}[ \phi^{-1}(v)=x^* \mid \mathcal{B}_S,\mathcal{B}_0^{a'-1} ] \pm \mathbb{P}[\overline{\mathcal{B}_{a'}}\mid \mathcal{B}_S,\mathcal{B}_0^{a'-1} ]}{ \mathbb{P}[\mathcal{B}_{a'}\mid \mathcal{B}_S,\mathcal{B}_0^{a'-1} ]} \nonumber\\
&\stackrel{(\ref{Prob xv mid B B1a-1})}{\leq}& \frac{2/(p(\vec{d'},a,a'-1) m)\pm \xi_{a'}}{1\pm \xi_{a'}} \leq \frac{q_*(\epsilon^{1/3})}{Ck}.
\end{eqnarray}
Now \eqref{eq:case1eqnarray} together with the fact that $|X_i^{nice}\setminus X_i^{good}|\leq Ck$ 
implies
\begin{align}\label{k+1 neighbours}
\sum_{x^*\in X_i^{nice}\setminus X_i^{good}} \mathbb{P}[x^*=\phi^{-1}(v) \mid \mathcal{B}_S,\mathcal{B}_0^{a'} ] \leq q_*(\epsilon^{1/3}).
\end{align}
Therefore,
\begin{flalign*}
&\mathbb{E}[|N_{\phi(H)}(v )\cap S|] &&
\end{flalign*}
\begin{eqnarray}\label{Case1 conclusion}
&=& \sum_{x^*\in X_i^{nice}} \mathbb{P}[\mathcal{B}_S,\mathcal{B}_0^{a'}] \mathbb{P}[\phi^{-1}(v)=x^*\mid \mathcal{B}_S,\mathcal{B}_0^{a'}] \mathbb{E}[ |\phi( N_H(x^*) )\cap S| \mid \mathcal{B}_S,\mathcal{B}_0^{a'}, \phi^{-1}(v)=x^*] \nonumber\\
& & \pm (k_{i,j}+1) \mathbb{P}[\overline{\mathcal{B}_S } \vee \overline{\mathcal{B}_0^{a'}}] \nonumber \\
&\stackrel{(\ref{B 1 to a})}{=}& \sum_{x^*\in X_i^{nice}} \left(1 \pm 2\xi_{a'} \right) \mathbb{P}[\phi^{-1}(v)=x^*\mid \mathcal{B}_S,\mathcal{B}_0^{a'}]  \mathbb{E}[ |\phi( N_H(x^*) )\cap S| \mid \mathcal{B}_S,\mathcal{B}_0^{a'}, \phi^{-1}(v)=x^*] \nonumber \\
& & \pm 2(k_{i,j}+1) \xi_{a'} \nonumber \\
&\stackrel{(\ref{Expect B, B1a, xv},\ref{k+1 neighbours})}{=}& (1\pm2\xi_{a'}) \left[ \sum_{x^*\in X_i^{good}} \mathbb{P}[\phi^{-1}(v)=x^*\mid \mathcal{B}_S,\mathcal{B}_0^{a'}] \left(1\pm \frac{f(\epsilon)}{3}\right) \frac{k_{i,j}|S|}{d_{i,j}n} 
\pm (k_{i,j}+1)q_*(\epsilon^{1/3}) \right] \nonumber \\ 
& & \pm 2(k_{i,j}+1)\xi_{a'} \nonumber \\
& \stackrel{\eqref{k+1 neighbours}}{=} &(1\pm 2\xi_{a'}) \left[ (1\pm q_*(\epsilon^{1/3})) \left(1\pm \frac{f(\epsilon)}{3}\right) \frac{k_{i,j}|S|}{d_{i,j}n} 
\pm (k_{i,j}+1)q_*(\epsilon^{1/3}) \right] \pm 2(k_{i,j}+1)\xi_{a'} \nonumber \\
&=& (1\pm f(\epsilon))\frac{k_{i,j}|S|}{d_{i,j}n}.
\end{eqnarray}
To obtain the final equality we used that $|S| > f(\epsilon)n$ and $$3(k_{i,j}+1)(q_*(\epsilon^{1/3})+\xi_{a'}) \leq 6(k+1)\epsilon^{(1/3)(1/300)^{w+1}} < \epsilon^{3(1/300)^{w+2}}=(f(\epsilon))^3$$ by (\ref{eq:functions}) and the assumption $\epsilon \ll 1/k$. This proves Claim \ref{claim:B1} in Case 1.
\vspace{0.2cm}

\noindent CASE 2. $b'_K < a'$.

So in this case $\phi^{-1}(S)$ is determined before $\phi^{-1}(v)$.
For each $y\in Y_{a}$, we consider 
\begin{align}\label{Ix define}
L_y := \{ \ell \in J_j : N_H(y)\cap Y_{\ell}\neq \emptyset\} = \{j_1,\dots, j_{k'}\},
\end{align}
where $k'\in \{k_{i,j},k_{i,j}+1\}$. For each $L\in \binom{J_j}{k_{i,j}}$ define 
\begin{align}\label{XL def}
X(L):=\left\{y\in Y_{a}\cap N_{A_0^a}(v): L_y = L\right\} \text{ and } x_L:= \frac{|X(L)|}{m}
\end{align}
So the $X(L)$ partition most of the initial set $Y_a\cap N_{A_0^a}(v)$ of candidates $y$ for $\phi^{-1}(v)$ according to the set $L_y$ that records the clusters $Y_\ell\subseteq X_j$ which contain the neighbours of $y$. Let 
\begin{align}\label{mathcal I define}
\mathcal{L}:=\left\{L \in \binom{J_j}{k_{i,j}}: x_L \geq (f(\epsilon))^3 n\right\}, \enspace X':= \{y\in Y_a \cap N_{A^a_0}(v) : L_y \notin \mathcal{L}\}.
\end{align}
Since $|J_j|=K$ we have $|\mathcal{L}| \leq \binom{K}{k_{i,j}} \leq \binom{K}{k}$ and 
 \begin{align}\label{size of X'}
 |X'|\leq (f(\epsilon))^3n\binom{K}{k} + Ck \leq (f(\epsilon))^{11/4}m. 
 \end{align}\COMMENT{$Ck$ here because at most $Ck$ vertices have degree $k_{i,j}+1$ in $G$.}
 We will calculate $\E[|N_{\phi(H)}(v)\cap S|]$ by conditioning on $\phi^{-1}(v)\in X(L)$ for each fixed $L$.
For this, we consider the following sets: for $b\in J_j$, let 
 \begin{align}\label{Na def}
 N^a(S_b):= N_{H'_*}(\phi^{-1}(S_b))\cap Y'_a = \tau_{b,a}(\phi^{-1}(S_b)).
 \end{align}
These sets are relevant for the following reason:
 the definitions of $X(L)$ in \eqref{XL def} and $N^a(S_{b})$ in \eqref{Na def} imply that 
under the assumption that $\phi^{-1}(v) \in X(L)$ and $b\in J_j$ we have
\begin{align}\label{eq:NphiHv cap Sb}
|N_{\phi(H)}(v)\cap S_{b}| = \left\{ \begin{array}{ll}
1 & \text{ if }\phi^{-1}(v) \in N^a(S_{b}) \text{ and } b\in L,\\
0 & \text{ else.}
\end{array} \right.
\end{align} 
We now define events (C$1_{\ell'}$) and (C$2^{b}_{\ell'}$) for all $0\leq \ell'<a'$ and $b\in (I_1\cup \dots \cup I_{\ell'})\cap J_j$ as follows. 
\begin{itemize}
\item [(C$1_{\ell'}$)] $|N_{A^a_{\ell'}}(v)\cap X(L)| = d_0^{-1} p(\vec{d'},a,\ell')(1\pm \xi_{\ell'})x_L m$ for all $L\in \mathcal{L}$.
\item [(C$2^b_{\ell'}$)] $|N_{A^a_{\ell'}}(v)\cap N^a(S_b)\cap X(L)| = d_0^{-1}d^{-1}_{i,j}p(\vec{d'},a,\ell') (1\pm \xi_{\ell'})x_L|S_b|$ for all $L\in \mathcal{L}$.
\end{itemize}
Let (C$2_{\ell'}$) be the event that (C$2^b_{\ell'}$) holds for all $b\in (I_1\cup \dots \cup I_{\ell'})\cap J_j$.
Note that if $\ell'=0$, then $(\text{C}1_{\ell'})$ holds by \eqref{definition p 0}, and $(\text{C}2_{\ell'})$ is vacuous. 

Note that conditional on $(\text{A}1_{\ell'})$, $p(\vec{d'},a,\ell')$ is the density of $A^a_{\ell'}$, so $(\text{C}1_{\ell'})$ is exactly what one would expect (the extra $d_0^{-1}$ arises since we assumed $X(L) \subseteq N_{A_0^a}(v)$). 
Similarly as in (B1$_{\ell'}^b$), the extra $d^{-1}_{i,j}$ in $(\text{C}2^b_{\ell'})$ comes from the fact that $S_b\subseteq S\subseteq N_{G'}(v)$.

For all $0\leq \ell' < a'$ we next define the event $\mathcal{C}_{\ell'}$ by:
$$ \mathcal{C}_{\ell'} := 
\mathcal{A}_{\ell'} \wedge (\text{C}1_{\ell'}) \wedge (\text{C}2_{\ell'})$$
 and let $\mathcal{C}_{0}^{\ell'}:=\bigwedge_{j'=0}^{\ell'} \mathcal{C}_{j'}$. So in particular, $\mathcal{C}_0$ is the event which always holds. \vspace{0.2cm}

{\noindent \bf Subclaim 3.} For all $\ell' \in [a'-1]$, $\mathbb{P}[\mathcal{C}_{\ell'} \mid \mathcal{B}_S,\mathcal{C}_0^{\ell'-1} ] \geq 1 - (1-2c)^n$.

To prove Subclaim 3, note that $\mathcal{B}_S$ and $\mathcal{C}_0^{\ell'-1}$ are events only depending on the history of the Slender graph algorithm prior to Round $\ell'$. Thus Lemma \ref{slender lemmas}(ii) implies that for each $\ell'\in [a'-1]$,
\begin{align}\label{mathcal A mid C}
\mathbb{P}[ \mathcal{A}_{\ell'} \mid \mathcal{B}_S,\mathcal{C}_{0}^{\ell'-1}] \geq 1-(1-3c)^{Km} \geq 1- (1-3c)^n.
\end{align}
If $I_{\ell'}\cap N_{R_K}(a) =\emptyset$, then $A^{a}_{\ell'} = A^{a}_{\ell'-1}$ by \eqref{stay same}
and $p(\vec{d'},a,\ell')=p(\vec{d'},a,\ell'-1)$ by (\ref{P relation}).
Moreover, $I_{\ell'}\cap J_j=\emptyset$ as $J_j \subseteq N_{R_K}(a)$. So in this case (C$1_{\ell'}$) and (C$2_{\ell'}$) are immediate from $\mathcal{C}_{\ell'-1}$. 

So let us next assume that  $I_{\ell'}\cap N_{R_K}(a) =\{\ell\}$.
We now show that 
\begin{align}\label{C1ell}
\begin{split}
\mathbb{P}[(\text{C}1_{\ell'})\mid \mathcal{B}_S,\mathcal{C}_0^{\ell'-1}] 
\geq 1-(1-3c)^n.
\end{split}
\end{align}
Recall that $\sigma_\ell:Y'_\ell \rightarrow U'_\ell$ for $\ell \in I_{\ell'}$ denotes the bijection chosen in Round $\ell'$ of the Slender graph algorithm. 
By the definition of $A^{a}_{\ell'}$ in \eqref{Aji x def}, (similarly as in (\ref{equivalence}))
\begin{align}\label{Q equiv}
 N_{A^{a}_{\ell'}}(v)\cap X(L) = N_{A^{a}_{\ell'-1}}(v)\cap X(L) \cap \tau_{\ell,a}(\sigma_\ell^{-1}(N_{G'}(v)\cap U'_\ell)).
 \end{align}
Since we assume that (C$1_{\ell'-1}$) holds, for all $L\in \mathcal{L}$ we have 
\begin{align}\label{C1ell-1}
|N_{A^{a}_{\ell'-1}}(v) \cap X(L)| = d_0^{-1} p(\vec{d'},a,\ell'-1)(1\pm \xi_{\ell'-1})x_L m.
\end{align} 
For each $L \in \mathcal{L}$, we let
\begin{align}\label{QI define}
Q_L^\ell:= \tau_{a,\ell}(N_{A^{a}_{\ell'-1}}(v) \cap X(L)).
\end{align} 
So $Q^\ell_{L} \subseteq Y'_\ell$ and
\begin{align}\label{QI size a}
|Q_{L}^\ell|\stackrel{(\ref{C1ell-1})}{=}d_0^{-1} p(\vec{d'},a,\ell'-1)(1 \pm \xi_{\ell'-1})x_L m.
\end{align} 
Using that $\sigma_\ell\circ \tau_{a,\ell}: Y'_a\rightarrow U'_\ell$ is a bijection, we obtain 
\begin{eqnarray*}
\sigma_\ell(Q_L^\ell)\cap N_{G'}(v)
&\stackrel{(\ref{QI define})}{=}&\sigma_\ell\left(\tau_{a,\ell}\left( N_{A^{a}_{\ell'-1}}(v) \cap X(L)\right)\right)
\cap (N_{G'}(v)\cap U'_\ell) \nonumber \\ 
 &=& \sigma_\ell\circ \tau_{a,\ell} \left( N_{A^{a}_{\ell'-1}}(v)\cap X(L) \right) \cap \sigma_\ell\circ \tau_{a,\ell} \left (\tau_{\ell,a}(\sigma_\ell^{-1}(N_{G'}(v)\cap U'_\ell)) \right)\nonumber \\
&=& \sigma_\ell\circ \tau_{a,\ell} \left( N_{A^{a}_{\ell'-1}}(v)\cap X(L) \cap \tau_{\ell,a}(\sigma_\ell^{-1}(N_{G'}(v)\cap U'_\ell)) \right) \nonumber \\ 
&\stackrel{(\ref{Q equiv})}{=}& \sigma_\ell\circ \tau_{a,\ell} ( N_{A^{a}_{\ell'}}(v)\cap X(L)).
\end{eqnarray*}
Since $\sigma_\ell\circ \tau_{a,\ell}$ is a bijection, this implies that
\begin{align}\label{sigma Q N v size}
|\sigma_\ell(Q_L^\ell)\cap N_{G'}(v)| = |N_{A^{a}_{\ell'}}(v) \cap X(L)|.
\end{align}
We will now apply (M$'$4)$_{\ell'}$ with $Q^{\ell}_L$, $N_{G'}(v)\cap U'_{\ell}$ playing the roles of $S$, $T$.
Note that $|N_{G'}(v)\cap U'_\ell| = (1\pm 2\epsilon^{1/3}) d'_{a,\ell} m$ by ($\mathcal{G}'$) in Step 4 and also that 
\begin{align}\label{eq:QellLineq}
|Q_L^\ell|\geq d^wx_Lm \geq d^w(f(\epsilon))^3m \geq h'(4K\Delta_R \sqrt{\xi_{\ell'-1}})m
\end{align}
by \eqref{mathcal I define} and \eqref{QI size a}, so the conditions of (M$'$4)$_{\ell'}$ are satisfied. Moreover 
\begin{flalign*}
&(1\pm h'(4K\Delta_R\sqrt{\xi_{\ell'-1}}))|N_{G'}(v)\cap U'_{\ell}| |Q^{\ell}_L|/m &&
\end{flalign*}
\begin{eqnarray*} 
 &\stackrel{\eqref{QI size a}}{=}& (1\pm h'(4K\Delta_R\sqrt{\xi_{\ell'-1}})) (1\pm 2\epsilon^{1/3})
d'_{a,\ell} d_0^{-1}p(\vec{d'},a,\ell'-1)(1 \pm \xi_{\ell'-1}) x_L m \\
&\stackrel{\eqref{P relation},\eqref{eq:xitexplicit}}{=}&  d_0^{-1} p(\vec{d'},a,\ell') (1\pm \xi_{\ell'}) x_L m.
\end{eqnarray*}
So (M$'$4)$_{\ell'}$ implies 
\begin{eqnarray*}\label{C1ell : BSCell'}
\mathbb{P}[(\overline{\text{C}1_{\ell'}})\mid \mathcal{B}_S,\mathcal{C}_0^{\ell'-1}] 
&\stackrel{\eqref{sigma Q N v size}}{\leq}& 
\sum_{L\in \mathcal{L}} \mathbb{P}[|\sigma_\ell(Q_L^\ell)\cap N_{G'}(v)| \neq d_0^{-1} p(\vec{d'},a,\ell')(1\pm \xi_{\ell'})x_L m \mid \mathcal{B}_S, \mathcal{C}_0^{\ell'-1} ] \nonumber \\
&\leq & |\mathcal{L}|(1-4Kc)^{m} \leq (1-3c)^n,
\end{eqnarray*}
and so (\ref{C1ell}) holds. 

Our next aim is to show that for all $\ell' \in [a'-1]$ we have
\begin{align}\label{C2ell}
\mathbb{P}[(\text{C}2_{\ell'}) \mid \mathcal{B}_S,\mathcal{C}_0^{\ell'-1}] \geq 1-K(1-3c)^n.
\end{align}
Consider any $b\in (I_1\cup\dots\cup I_{\ell'})\cap J_j$. We first consider the case when $b\in I_{\ell'}$. Then $N_{R_K}(a)\cap I_{\ell'} =\{b\}$. 
Note 
\begin{eqnarray*}(1\pm h'(4K\Delta_R\sqrt{\xi_{\ell'-1}}))|Q_L^b||S_b|/m &\stackrel{\eqref{QI size a}}{=}& d_0^{-1} p(\vec{d'},a,\ell'-1)(1\pm \xi_{\ell'})x_L|S_b|\\ 
&=& d^{-1}_0 d^{-1}_{i,j} p(\vec{d'},a,\ell')(1\pm \xi_{\ell'})x_L |S_b|,
\end{eqnarray*} 
because $d_{i,j} = d'_{a,b}$. 
This shows that, for each $L\in \mathcal{L}$,  
we can apply (M$'$4)$_{\ell'}$ with $Q^{b}_L$, $S_b$ playing the roles of $S$, $T$ to see that
\begin{align}\label{sigma Q cap Sb}
\mathbb{P}[|\sigma_b(Q_L^b)\cap S_b| &= d_0^{-1}d^{-1}_{i,j} p(\vec{d'},a,\ell')(1\pm \xi_{\ell'})x_L |S_b| \mid \mathcal{B}_S , \mathcal{C}_0^{\ell'-1}]\nonumber \\ &\geq 1-(1-4Kc)^{m} \geq 1 - (1-3c)^n.
\end{align}
(Here we also use \eqref{eq:QellLineq} and that
$|S_b| = (1\pm \epsilon^{1/3})\frac{|S|}{K}$ by (BB1) to verify the conditions of (M$'$4)$_{\ell'}$.)
We now consider $\tau_{b,a}( \sigma_{b}^{-1}(\sigma_b(Q_L^b)\cap S_b)) \subseteq Y'_a$. Since $\tau_{b,a}, \sigma_b$ are both bijections, 
\begin{align}\label{psi sigma sigma Q Sb}
|\tau_{b,a}( \sigma_b^{-1}(\sigma_b(Q_L^b)\cap S_b))|= |\sigma_b(Q_L^b)\cap S_b|.
\end{align}
Moreover, using that $S_b\subseteq N_{G'}(v) \cap U'_b$,
\begin{eqnarray}\label{psi sigma sigma Q Sb 2}
\tau_{b,a}( \sigma_b^{-1}(\sigma_b(Q_L^b)\cap S_b))&=& \tau_{b,a}(Q_L^b)\cap \tau_{b,a}(\sigma_b^{-1}(S_b))\nonumber \\ 
&\stackrel{(\ref{QI define})}{=}& N_{A_{\ell'-1}^a}(v)\cap X(L) \cap \tau_{b,a}( \sigma_b^{-1}(S_b)) \nonumber \\
&=& N_{A_{\ell'-1}^a}(v)\cap X(L) \cap \tau_{b,a}(\sigma^{-1}_b(S_b)) \cap \tau_{b,a}( \sigma_b^{-1}(N_{G'}(v)\cap U'_b)) \nonumber \\
&\stackrel{\eqref{Na def},\eqref{Q equiv}}{=}& N_{A^a_{\ell'}}(v)\cap X(L) \cap N^a(S_b).
\end{eqnarray}
So (\ref{sigma Q cap Sb}), (\ref{psi sigma sigma Q Sb}) and (\ref{psi sigma sigma Q Sb 2}) together imply that for each $b\in I_{\ell'}\cap J_j$
\begin{align} \label{C2ell1}
\mathbb{P}[ (\text{C}2_{\ell'}^b)
\mid \mathcal{B}_S, \mathcal{C}_0^{\ell'-1}]   \geq 1 -(1-3c)^n.
\end{align}
So assume next that $b\in (I_1\cup \dots \cup I_{\ell'-1})\cap J_j$. (In particular, this means that $\ell'\geq 2$.) Then $\mathcal{C}_{\ell'-1}$ implies that for all $L\in \mathcal{L}$ we have
$$|N_{A_{\ell'-1}^a}(v) \cap N^a(S_b)\cap X(L)| = d_0^{-1}d^{-1}_{i,j}p(\vec{d'},a,\ell'-1)(1\pm \xi_{\ell'-1})x_L |S_b|.$$ 
Thus, this time we consider $\ell \in I_{\ell'}\cap N_{R_K}(a)$ and for each $L\in \mathcal{L}$ we let 
$$Q'^\ell_L:= \tau_{a,\ell}(N_{A^{a}_{\ell'-1}}(v)\cap N^a(S_b)\cap X(L)). $$ Note $Q'^\ell_L \subseteq Y'_\ell$. 
Similarly as in \eqref{sigma Q cap Sb} (but with $N_{G'}(v) \cap U'_\ell$ playing the role of $T$) one can use (M$'$4)$_{\ell'}$ to see that 
\begin{align}\label{C2ell-0}
\mathbb{P}[|\sigma_\ell(Q'^\ell_L)\cap (N_{G'}(v)\cap U'_\ell)|= d_0^{-1}d_{i,j}^{-1}p(\vec{d'},a,\ell') (1\pm \xi_{\ell'})x_L|S_b|\mid \mathcal{B}_S,\mathcal{C}_0^{\ell'-1}] \geq 1-(1-3c)^n.
\end{align}
\COMMENT{because $(1\pm h(4K\Delta_R\sqrt{\xi_{\ell'-1}})) d_0^{-1} d'_{\ell,a}(d^{-1}_{i,j} p(\vec{d'},a,\ell'-1) \pm \xi_{\ell'-1}) x_L|S_b| =d_0^{-1}d^{-1}_{i,j} p(\vec{d'},a,\ell') (1\pm \xi_{\ell'})x_L|S_b|$. }
Similarly as in \eqref{psi sigma sigma Q Sb 2},
\begin{eqnarray*}
\tau_{\ell,a}\left(\sigma_\ell^{-1}\left( \sigma_\ell(Q'^\ell_L)\cap (N_{G'}(v)\cap U'_\ell) \right)\right) &=&N_{A^a_{\ell'-1}}(v)\cap N^a(S_b)\cap X(L) \cap \tau_{\ell,a}\left( \sigma^{-1}_\ell\left( N_{G'}(v)\cap U'_\ell \right)\right)\\
& \stackrel{(\ref{Q equiv})}{=} &N_{A^{a}_{\ell'}}(v) \cap X(L) \cap N^a(S_b).
\end{eqnarray*}
Thus $|\sigma_\ell(Q'^\ell_L)\cap (N_{G'}(v)\cap U'_\ell)|= |N_{A^{a}_{\ell'}}(v) \cap X(L) \cap N^a(S_b)|$. So if $b\in (I_1\cup\dots\cup I_{\ell'-1})\cap J_j$, we have 
$$\Prob[(\text{C}2^b_{\ell'})\mid \mathcal{B}_S,\mathcal{C}_0^{\ell'-1}]\geq 1-(1-3c)^n.$$
Together with (\ref{C2ell1}) 
and a union bound taken over all $b\in(I_1\cup\dots \cup I_{\ell'}) \cap J_j$, this implies (\ref{C2ell}).
Thus, by (\ref{mathcal A mid C}), (\ref{C1ell}), (\ref{C2ell}) we obtain
$$\mathbb{P}[\mathcal{C}_{\ell'} \mid \mathcal{B}_S, \mathcal{C}_0^{\ell'-1}] \geq 1-(1-2c)^n.$$ 
This completes the proof of Subclaim 3.\newline
\vspace{0.2cm}

We now proceed with the proof of Claim \ref{claim:B1} in Case 2. Subclaim 3, (\ref{mathcal B prob}) and the fact that $\mathcal{C}_0$ always holds together imply that
\begin{align}\label{B C1a-1}
\mathbb{P}[\mathcal{B}_S,\mathcal{C}_0^{a'-1}] = \mathbb{P}[\mathcal{B}_S]\prod_{\ell'=1}^{a'-1} \mathbb{P}[\mathcal{C}_{\ell'} \mid \mathcal{C}_0^{\ell'-1},\mathcal{B}_S] \geq 1 - (1-c)^n.
\end{align}
Consider any $L\in \mathcal{L}$. We now compute the expectation of $|N_{\phi(H)}(v)\cap S|$ conditional on $\mathcal{B}_S,\mathcal{C}_0^{a'-1}$, and $\phi^{-1}(v) \in X(L)$. Since $\mathcal{C}_{a'-1}$ holds we have $\mathcal{A}_{a'-1}$, (C$1_{a'-1}$) and (C$2_{a'-1}$).
By Lemma \ref{slender lemmas}(i), $|N_{A(a)}(v)| = |N_{A^a_{a'-1}}(v)| \pm 4K\Delta_R \xi_{a'-1}m$. Together with  (C$1_{a'-1}$), (C$2_{a'-1}$) this implies that for any $b\in J_j$
\begin{align}\label{NA' X size1}
&|N_{A(a)}(v)\cap X(L)|=d_0^{-1} p(\vec{d'},a,a'-1) x_L m \pm 5K\Delta_R \xi_{a'-1} m, \\  
\label{NA' X size2}
&|N_{A(a)}(v)\cap X(L) \cap N^a(S_{b})|=d_0^{-1} d^{-1}_{i,j} p(\vec{d'},a,a'-1)x_L |S_{b}| \pm 5K\Delta_R \xi_{a'-1}m.
\end{align} (Here we also use that $J_j \subseteq I_1\cup \dots \cup I_{a'-1}$ by \eqref{aa'bb'} and since we are in Case 2.) 
 $(\text{M}'2)_{a'}$ and (\ref{eq:functions}) together imply that for every $L\in \mathcal{L}$ and every $b\in J_j$ we have
\begin{eqnarray}\label{B1 1} 
\mathbb{P}[ \phi^{-1}(v) \in X(L)\cap N^a(S_{b}) \mid \mathcal{B}_S,\mathcal{C}_0^{a'-1}] &=& (1\pm h(4K\Delta_R\sqrt{\xi_{a'-1}}))\frac{|N_{A(a)}(v)\cap X(L) \cap N^a(S_{b})|}{p(\vec{d'},a,a'-1)  m} \nonumber \\
&\stackrel{(\ref{NA' X size2})}{=}&  (1\pm q_*(\epsilon^{1/3})/2 ) \frac{x_L |S_{b}|}{d_0 d_{i,j}m} \stackrel{(\text{BB}1)}{=} (1\pm q_*(\epsilon^{1/3}))\frac{x_L|S|}{d_0 d_{i,j}n} \nonumber
\end{eqnarray}
and
\begin{eqnarray*}
\mathbb{P}[\phi^{-1}(v) \in X(L) \mid \mathcal{B}_S,\mathcal{C}_0^{a'-1}] &=& (1\pm h(4K\Delta_R\sqrt{\xi_{a'-1}})) \frac{|N_{A(a)}(v)\cap X(L)| }{ p(\vec{d'},a,a'-1) m} \\
&\stackrel{(\ref{NA' X size1})}{=}& (1\pm q_*(\epsilon^{1/3}))\frac{x_L}{d_0}.
\end{eqnarray*}
Thus
\begin{align}\label{B1 0}
\mathbb{P}[\phi^{-1}(v)\in N^a(S_{b}) \mid \mathcal{B}_S,\mathcal{C}_0^{a'-1}, \phi^{-1}(v) \in X(L) ] 
&= \frac{\mathbb{P}[ \phi^{-1}(v)\in X(L)\cap N^a(S_{b}) \mid \mathcal{B}_S,\mathcal{C}_0^{a'-1}]}{\mathbb{P}[\phi^{-1}(v)\in X(L) \mid \mathcal{B}_S,\mathcal{C}_0^{a'-1}]} \nonumber \\ &= (1\pm 3q_*(\epsilon^{1/3})) \frac{|S|}{d_{i,j}n}.
\end{align}
Thus for any $L\in \mathcal{L}$
\begin{flalign*}
&\mathbb{E}[|N_{\phi(H)}(v)\cap S | \mid \mathcal{B}_S,\mathcal{C}_0^{a'-1}, \phi^{-1}(v)\in X(L)] &&
\end{flalign*}
\begin{eqnarray} \label{B1 2}
 &\stackrel{\eqref{eq:NphiHv cap Sb}}{=}& \sum_{b\in L} \mathbb{P}[\phi^{-1}(v)\in N^a(S_{b}) \mid \mathcal{B}_S,\mathcal{C}_0^{a'-1}, \phi^{-1}(v) \in X(L)] \nonumber \\
&\stackrel{(\ref{B1 0})}{=}& \sum_{b\in L} (1\pm 3q_*(\epsilon^{1/3}))\frac{|S|}{d_{i,j}n} \stackrel{\eqref{mathcal I define}}{=} \left(1\pm \frac{f(\epsilon)}{3}\right) \frac{k_{i,j}|S|}{d_{i,j}n}.
\end{eqnarray}
Moreover, by $(\text{M}'2)_{a'}$ and (\ref{size of X'}), 
\begin{align}\label{X' prob}
\mathbb{P}[\phi^{-1}(v)\in X'\mid \mathcal{B}_S,\mathcal{C}_0^{a'-1}] &\leq
2\frac{|X'|}{p(\vec{d'},a,a'-1) m} \leq \frac{2f(\epsilon)^{11/4}}
{d_0 d^w} \leq f(\epsilon)^{8/3}.
\end{align}
Furthermore, note that $|\phi(N_H(\phi^{-1}(v)))\cap S| \leq \Delta(H)\leq (k+1)\Delta_R$ always holds. (Recall $k_{i,j}\leq k$ for all $ij \in E(R)$.) Thus,
\begin{flalign*}
&\mathbb{E}\left[|N_{\phi(H)}(v)\cap S| \mid \mathcal{B}_S,\mathcal{C}_0^{a'-1}\right] &&
\end{flalign*}
\begin{eqnarray}\label{B1 3}
&=& \sum_{L\in \mathcal{L}} \mathbb{E}\left[|N_{\phi(H)}(v)\cap S|\mid \mathcal{B}_S,\mathcal{C}_0^{a'-1},\phi^{-1}(v)\in X(L)\right]\cdot \mathbb{P}[\phi^{-1}(v)\in X(L) \mid \mathcal{B}_S,\mathcal{C}_0^{a'-1}]\nonumber \\ 
& &\pm (k+1)\Delta_R \mathbb{P}[\phi^{-1}(v)\in X' \mid \mathcal{B}_S,\mathcal{C}_0^{a'-1}] \nonumber\\
&\stackrel{(\ref{B1 2}, \ref{X' prob})}{=}&\sum_{L\in \mathcal{L}} \left(1\pm \frac{f(\epsilon)}{3}\right)\frac{k_{i,j}|S|}{d_{i,j}n} \mathbb{P}[\phi^{-1}(v)\in X(L) \mid \mathcal{B}_S,\mathcal{C}_0^{a'-1}] \pm (k+1)\Delta_R f(\epsilon)^{8/3}\nonumber\\
& =& \left(1\pm \frac{f(\epsilon)}{3}\right)\frac{k_{i,j}|S|}{d_{i,j}n}\mathbb{P}[\phi^{-1}(v)\notin X' \mid \mathcal{B}_S,\mathcal{C}_0^{a'-1}]\pm f(\epsilon)^{7/3}\nonumber \\
& \stackrel{(\ref{X' prob})}{=}& \left(1\pm \frac{f(\epsilon)}{2}\right) \frac{k_{i,j}|S|}{d_{i,j}n}.
\end{eqnarray}
Therefore, by Subclaim 3, (\ref{mathcal B prob}), (\ref{B C1a-1}) and (\ref{B1 3})
\begin{eqnarray*} 
\mathbb{E}\left[|N_{\phi(H)}(v)\cap S|\right] & = &
\mathbb{E}\left[|N_{\phi(H)}(v)\cap S| \mid \mathcal{B}_S,\mathcal{C}_0^{a'-1}\right]\mathbb{P}[\mathcal{B}_S,\mathcal{C}_0^{a'-1}] \\
& &\pm (k+1)\Delta_R\mathbb{P}[\overline{\mathcal{B}_S}]\pm (k+1)\Delta_R\sum_{\ell'=1}^{a'-1} \mathbb{P}[\overline{\mathcal{C}_{\ell'}}\mid \mathcal{C}_0^{\ell'-1},\mathcal{B}_S]\\
 &=& \left(1\pm \frac{f(\epsilon)}{2}\right) \frac{k_{i,j}|S|}{d_{i,j}n}(1\pm (1-c)^n) \pm (k+1) \Delta_R\left( 2(1-2c)^n + w (1-2c)^n\right) \\
& =& (1\pm f(\epsilon)) \frac{k_{i,j}|S|}{d_{i,j}n}.
\end{eqnarray*}
This completes the proof of Case 2 of Claim \ref{claim:B1}.
\end{proof}

\begin{claim}
(B2) holds.
\end{claim}
\begin{proof}
Note that if the Slender graph algorithm applied to $\mathcal{S}$ as defined in Step 4 
does not fail, the graphs $F_i$ satisfy property (i) in Lemma \ref{Slendered blow-up}. Since we defined $N_x = N_{H_*}(x)$ in Step 4, we have $x\in N_y$ if only if $y\in N_x$ for $x,y\in V(H)$.
Also, since $H_*[Y_i,Y_j]$ is a matching if $ij\in E(R_K)$ and $H_*[Y_i,Y_j]$ is empty otherwise, for every $x\in Y_i$ we have $|N_x \cap Y_j|\leq 1$ for all $j\in[Kr]$, and $|N_x\cap Y_j|=0$ if $j\notin N_{R_K}(i)$. Thus for all $x\in V(H),|N_x|\leq \Delta(R_K)\leq K\Delta_R$ and so (B2.1) holds. Properties (B2.2) and (B2.4) are immediate consequences of property (i) in Lemma~\ref{Slendered blow-up} (note that $H_*$, $A^*_0$ play the roles
of $H$, $A_0$ in Lemma~\ref{Slendered blow-up} and $A^*_0\subseteq A_0$) and the fact that $q_*(\epsilon^{1/3}) = \epsilon^{(1/3)(1/300)^{w+1}} \leq \epsilon^{(1/300)^{w+2}} =f(\epsilon)$ by (\ref{eq:functions}) with room to spare. Finally, (B2.3) holds since we assume that the Uniform embedding algorithm does not fail (and thus $(\mathcal{P})$ in Step 3 holds).
\end{proof}                     
            
\begin{claim}\label{distance two}
(B3) holds.
\end{claim}
\begin{proof}
Assume $u\in U_i,  v\in U_j$ and so $\phi^{-1}(u)\in Y'_i$ and $\phi^{-1}(v)\in Y'_j$. If $i,j\in I_{\ell'}$, then the fact that $(R_K)^2[I_{\ell'}]$ is an independent set implies that  
$$\mathbb{P}\left[N_H(\phi^{-1}(u))\cap N_H(\phi^{-1}(v))\neq \emptyset\right] =0.$$

Therefore, without loss of generality, we may assume that $i\in I_{\ell'}, j\in I_{\ell''}$ with $\ell'<\ell''$. Then during Round $\ell''$ of the Slender graph algorithm, when we are about to determine how to map $Y'_j$ to $U'_j$, $x:=\phi^{-1}(u)$ is already determined. Let $Y''_j \subseteq Y'_j$ be the set of all those vertices which are at distance two from $x$ in $H$. So $|Y''_j|\leq (\Delta(H))^2 \le ((k+1)\Delta_R)^2$.
Recall that $\mathcal{A}_0^{\ell''-1}$ was defined before Lemma \ref{slender lemmas}. By ($M'2)_{\ell''}$ we have 
$$\mathbb{P}[\phi^{-1}(v)\in Y''_j\mid \mathcal{A}_0^{\ell''-1}] \leq \frac{2((k+1)\Delta_R)^2}{p(\vec{d'},j,\ell''-1)m} \leq \frac{1}{2\sqrt{n}}.$$
\noindent
Moreover, Lemma \ref{slender lemmas}(ii) implies that 
$\mathbb{P}[\mathcal{A}_0^{\ell''-1}] \geq 1 - (1-c)^n.$ 
Thus
\begin{align*}
\mathbb{P}[N_H(\phi^{-1}(u))\cap N_H(\phi^{-1}(v))\neq \emptyset]
&\leq \mathbb{P}[N_H(\phi^{-1}(u))\cap N_H(\phi^{-1}(v))\neq \emptyset \mid \mathcal{A}_0^{\ell''-1}]\mathbb{P}[\mathcal{A}_0^{\ell''-1}] + \mathbb{P}[\overline{\mathcal{A}_0^{\ell''-1}}] \\&\leq \mathbb{P}[\phi^{-1}(v) \in Y''_j\mid \mathcal{A}_0^{\ell''-1}] + (1-c)^n < \frac{1}{\sqrt{n}}.
\end{align*}
This proves (B3.1).

To show (B3.2), again we may assume that $v\in U_j$, $\phi^{-1}(v) \in Y'_j$ and $j\in I_{\ell''}$. Then we determine how to map $Y'_j$ to $U'_j$ during Round $\ell''$ of the Slender graph algorithm.
Note that  $|Y'_j\cap Z| \leq |Z| \leq \gamma^3 n$.
 In Round $\ell''$ of the Slender graph algorithm, by ($M'2)_{\ell''}$ we have 
$$\mathbb{P}[\phi^{-1}(v)\in Y'_j\cap Z \mid \mathcal{A}_0^{\ell''-1}] \leq \frac{2 |Y'_j\cap Z| }{p(\vec{d'},j,\ell''-1)m} \leq \frac{2\gamma^3 n}{p(\vec{d'},j,\ell''-1)m} \leq \frac{\gamma^2}{2}.$$
\noindent
Moreover, Lemma \ref{slender lemmas}(ii) implies that 
$\mathbb{P}[\mathcal{A}_0^{\ell''-1}] \geq 1 - (1-c)^n.$ 
Thus
\begin{align*}
\mathbb{P}[\phi^{-1}(v)\in Z]
&\leq \mathbb{P}[\phi^{-1}(v)\in Z \mid \mathcal{A}_0^{\ell''-1}]\mathbb{P}[\mathcal{A}_0^{\ell''-1}] + \mathbb{P}[\overline{\mathcal{A}_0^{\ell''-1}}] \\&\leq \mathbb{P}[\phi^{-1}(v)\in Y'_j\cap Z \mid \mathcal{A}_0^{\ell''-1}] + (1-c)^n < \gamma^2.
\end{align*}
This proves (B3.2).

\end{proof}

\begin{claim}
(B4) holds.
\end{claim}            
\begin{proof}
Given $v\in V(G)$, as before, let $a$ denote the index such that $v\in U'_a$, and $a'$ be such that $a\in I_{a'}$. Let $f_{\ell'}$ be the partial embedding we obtain after Round $\ell'$ in the Slender graph algorithm. For all $\ell \in [Kr]$ let $W^{\ell}_{\ell'}$ be the set of vertices in $U'_{\ell}$ which are incident to an edge in $E(G'')\cap f_{\ell'}(E(H))$. (So $W_{\ell'}^{\ell} \subseteq U_{\ell}$.) Let $\tau_{i,j}: Y'_{i} \rightarrow Y'_{j}$ be as defined in \eqref{tau def}.

Recall the number of rounds $w$ was defined in (V1). For each $0\leq i'\leq w$ we define $I_1^{i'}:= I_1\cup \dots \cup I_{i'}$. For $i\in[Kr]$ and $i'\in [w]$ we let $R_K^{i,i'}:= |N_{R_K}(i)\cap I_{1}^{i'}|$. So $R_K^{i,i'} \leq K\Delta_R$. 
For each $0 \le \ell' \le w$ we define the following event:
\begin{itemize}
\item[(D$1_{\ell'}$)] $|W_{\ell'}^{\ell}|\leq R_K^{\ell,\ell'} \gamma^{4/5}n$ for all $\ell\in N_{R_K}(a)\cap I_{1}^{\ell'}$.
\end{itemize}
Note that $(\text{D}1_0)$ is vacuously true. For $a' \le \ell \le w$ we define the following event:
\begin{itemize}
\item[(D$2_{\ell'}$)] $\phi(\tau_{a,\ell}(\phi^{-1}(v))) \notin W_{\ell'}^{\ell}$ for all $\ell \in N_{R_K}(a)\cap I_{1}^{\ell'}$.
\end{itemize}
Recall that $\mathcal{A}_{\ell'}$ was defined before Lemma \ref{slender lemmas}. For each $0\leq \ell'\leq w$ we define $\mathcal{D}_{\ell'}$ and $\mathcal{D}^*_{\ell'}$ as follows:
\begin{align} \label{mathcal D define}
\mathcal{D}_{\ell'} := \left\{ \begin{array}{ll}
\mathcal{A}_{\ell'} \wedge (\text{D}1_{\ell'}) &\text{ if } \ell' <a',\\
\mathcal{A}_{\ell'} \wedge (\text{D}1_{\ell'}) \wedge (\text{D}2_{\ell'}) &\text{ if } \ell'\geq a',\\
\end{array} \right.
\end{align}
\begin{align} \label{mathcal D* define}
\mathcal{D}^*_{\ell'} := 
\mathcal{A}_{\ell'} \wedge (\text{D}1_{\ell'}) 
\end{align}
We denote $\mathcal{D}_{0}^{j'}:= \bigwedge_{\ell'=0}^{j'} \mathcal{D}_{\ell'}$ and $\mathcal{D}^{*,j'}_{0} = \bigwedge_{\ell'=0}^{j'} \mathcal{D}^*_{\ell'}$.
Note that (D$2_{w}$) implies that there is no $y\in V(H)$ with $\phi^{-1}(v)y\in E(H)$ such that $\phi(y)$ is incident to an edge in $\phi(E(H))\cap E(G'')$. Thus $v\notin \phi_2(H,G,G'')$. Accordingly, our aim is to show that both $\mathbb{P}[\mathcal{D}_{\ell'}\mid \mathcal{D}_0^{\ell'-1}]$ and $\mathbb{P}[\mathcal{D}^{*}_{\ell'}\mid \mathcal{D}_0^{*,\ell'-1}]$ are close to $1$ (see \eqref{prob ell>a} and \eqref{D* relation}). 

Consider any $\ell \in I_{\ell'}$ and let $A(\ell)\subseteq A^{\ell}_{\ell'-1}$ be as defined in \eqref{def A(i)}. 
Since $\mathcal{D}_{0}^{\ell'-1}$ and $\mathcal{D}_0^{*,\ell'-1}$ are events only depending on the history prior to Round $\ell'$ and are contained in $\mathcal{A}_{\ell'-1}$, Lemma~\ref{slender lemmas}(ii) together with the fact that $(1-3c)^{Km} \leq (1-3c)^n$ implies that 
\begin{align}\label{D0ell}
\mathbb{P}[\mathcal{A}_{\ell'} \mid \mathcal{D}_0^{*,\ell'-1}] \geq 1-(1-3c)^{n} \enspace \text{ and } \enspace \mathbb{P}[ \mathcal{A}_{\ell'} \mid \mathcal{D}_0^{\ell'-1}] \geq 1-(1-3c)^n.
\end{align}
We now show that $\mathbb{P}[(\text{D}1_{\ell'}) \mid \mathcal{D}_0^{\ell'-1}]$ is close to $1$ (see \eqref{D1ell}).
We say that $y\in Y'_{\ell}$ is {\em dangerous} for $u\in U'_{\ell}$ if there exists $q\in N_{R_K}(\ell)\cap I_1^{\ell'-1}$ such that
\begin{align} \label{1condi}
f_{\ell'-1}(\tau_{\ell,q}(y))u \in E(G'').
 \end{align} 
Note that if $y$ is dangerous for $u$ then $u\in U_\ell$,%
   \COMMENT{and $\tau_{\ell,q}(y)\in Y_q$, but we might have $y\in Y'_{\ell}\setminus Y_\ell$}
and choosing $\phi(y)=u$ would mean $u\in W_{\ell'}^\ell$, unless
$y\tau_{\ell,q}(y)\in E(H'_*)\setminus E(H)$. 
Let $A'_\ell$ be the subgraph of $A(\ell)$ such that $yu \in E(A'_\ell)$ if $y$ is dangerous for $u \in U'_{\ell}$. Note that $E(A'_\ell)=E(A'_\ell[Y'_\ell, U_\ell])$.  
 
We now bound $d_{A'_{\ell}}(u)$ for $u\in U'_\ell$.
Since $d_{G''}(u)\leq \gamma n$, there are at most $\gamma n$ vertices $y'\in V(H)$ such that $f_{\ell'-1}(y') \in N_{G''}(u)$. Thus for any fixed $u$, (\ref{1condi}) can occur for at most $\gamma n$ distinct vertices $y$ in $Y'_\ell$, i.e. $d_{A'_{\ell}}(u) \leq \gamma n$.

Next we bound $d_{A'_{\ell}}(y)$ for $y\in Y'_\ell$.
Note that $uy \in E(A'_\ell)$ implies that 
$$u \in \bigcup_{q\in N_{R_K}(\ell)\cap I_{1}^{\ell'-1}
} N_{G''}(f_{\ell'-1}(\tau_{\ell,q}(y)))\cap U'_{\ell},$$ 
so $d_{A'_\ell}(y) \leq \Delta(R_K)\Delta(G'')\leq K\Delta_R \gamma n$. Thus
\begin{align}\label{A'ell max degree}
\Delta(A'_\ell)\leq K\Delta_R \gamma n \leq K^2 \Delta_R \gamma m \leq \gamma^{9/10}m.
\end{align}

Note that once $f_{\ell'}$ is determined, if a vertex $u \in U'_\ell$ is in $W_{\ell'}^\ell$, then there must exist an index $q\in I_1^{\ell'-1}\cap N_{R_{K}}(\ell)$ such that $f_{\ell'-1}(\tau_{\ell,q}(f^{-1}_{\ell'}(u)))u \in E(G'')$. Thus $f^{-1}_{\ell'}(u)u \in E(A'_\ell[Y_\ell,U_\ell])$.%
\COMMENT{if $u\in W^\ell_{\ell'}$ then $u\in U_\ell$ and so $f^{-1}_{\ell'}(u)\in Y_\ell$}
Since we are conditioning on $\mathcal{D}_0^{\ell'-1}\subseteq \mathcal{A}_0^{\ell'-1}$, 
by $(\text{M}'3)_{\ell'}$ of Lemma~\ref{slender lemmas} (with $\gamma^{9/10}$ playing the role of $d'$ and $A'_\ell[Y_\ell,U_\ell]$ the role of $A'$)
\begin{align}\label{W ell prob}
\Prob\left[|W_{\ell'}^{\ell}|\geq \gamma^{4/5}n\mid \mathcal{D}_0^{\ell'-1}\right] \leq \mathbb{P}\left[|W_{\ell'}^{\ell}| \geq 8 \gamma^{9/10}m / p(\vec{d'},\ell,\ell'-1) \mid \mathcal{D}_0^{\ell'-1}\right] \leq (1-3c)^n.
\end{align}
Assume 
\begin{align}\label{eq:Wellineq}
|W_{\ell'}^\ell| \leq \gamma^{4/5} n \text{ for all } \ell \in I_{\ell'}.
\end{align}
We now consider $|W_{\ell'}^q|$ for $q\in I_1^{\ell'-1}$ under this assumption.
Note that if $u\in W_{\ell'}^{q} \setminus W_{\ell'-1}^{q}$ then $u$ must be incident to an edge in $E(G'')\cap f_{\ell'}(E(H)) \setminus ( E(G'')\cap f_{\ell'-1}(E(H)))$. Recall that for each $q\in I_{1}^{\ell'-1}$, we have $|N_{R_K}(q) \cap I_{\ell'}| \leq 1$. We let $q_{\ell'}$ be such that $\{q_{\ell'}\} = N_{R_K}(q)\cap I_{\ell'}$ if it exists. Then if $q_{\ell'}$ exists,
$$
|W_{\ell'}^{q}\setminus W_{\ell'-1}^{q}| \leq |W_{\ell'}^{q_{\ell'}}|
\stackrel{(\ref{eq:Wellineq})}{\leq} \gamma^{4/5} n.
$$
If $N_{R_K}(q)\cap I_{\ell'}=\emptyset$, then $W_{\ell'}^{q} = W_{\ell'-1}^{q}$. Suppose $\mathcal{D}_{\ell'-1}$ holds in addition to \eqref{eq:Wellineq}. Since $\mathcal{D}_{\ell'-1}$ implies (D$1_{\ell'-1}$), it follows that 
\begin{align*} |W_{\ell'}^{q}| &\leq |W_{\ell'-1}^{q}| + |W_{\ell'}^{q}\setminus W_{\ell'-1}^{q}| \stackrel{(\text{D}1_{\ell'-1})}{\leq}  R_{K}^{q,\ell'} \gamma^{4/5} n.
\end{align*}
In other words, (D$1_{\ell'}$) holds. Thus, for all $\ell' \in [w]$ 
\begin{eqnarray}\label{D1ell}
\mathbb{P}[(\text{D}1_{\ell'}) \mid \mathcal{D}_0^{\ell'-1}] &\geq& \mathbb{P}[|W_{\ell'}^{\ell}| \leq \gamma^{4/5}n \text{ for all }\ell\in I_{\ell'} \mid \mathcal{D}_0^{\ell'-1}] \nonumber \\
&\stackrel{\eqref{W ell prob}}{\geq}& 1- |I_{\ell'}|(1-3c)^n \geq 1 - Kr (1-3c)^n.
\end{eqnarray}
The same argument also shows that for any $\ell'\in [w]$,
\begin{align}\label{D*1ell}
\mathbb{P}[(\text{D}1_{\ell'}) \mid \mathcal{D}^{*,\ell'-1}_{0}] \geq 1- Kr(1-3c)^n.
\end{align}
In particular, \eqref{D1ell} together with (\ref{D0ell}) gives us that for any $\ell' < a'$, 
\begin{align}\label{prob ell<a}
\mathbb{P}[\mathcal{D}_{\ell'} \mid \mathcal{D}_0^{\ell'-1}]
\geq 1 - 2Kr(1-3c)^n.
\end{align}
Now we consider (D$2_{\ell'}$) for the case when $\ell'=a'$. 
Let $N^a(v)$ be the set of vertices $y\in Y'_a$ which 
\begin{itemize}
\item[(i)] are dangerous for $v$ (so there exists $q\in N_{R_K}(a)\cap I_1^{a'-1}$ such that
$f_{a'-1}(\tau_{a,q}(y))v \in E(G'')$) or
\item[(ii)] for which there exists $q \in N_{R_K}(a)\cap I_{1}^{a'-1}$ such that
$f_{a'-1}(\tau_{a,q}(y)) \in W_{a'-1}^{q}.$
\end{itemize}
Note that $f_{a'}^{-1}(v) \notin N^a(v)$ guarantees that (D$2_{a'}$) holds.
(Indeed, 
first observe that $N_{R_K}(a) \cap I_1^{a'}=N_{R_K}(a) \cap I_1^{a'-1}$ since $a \in I_{a'}$.
Write $x:=f_{a'}^{-1}(v)$.
So if $x \notin N^a(v)$, then (ii) means that $f_{a'}(\tau_{a,\ell}(x)) \notin W_{a'-1}^\ell$ for all $\ell \in N_{R_K}(a) \cap I_1^{a'}$.
But if $f_{a'}(\tau_{a,\ell}(x)) \in W_{a'}^\ell \setminus W_{a'-1}^\ell$, then $f_{a'}(\tau_{a,\ell}(x))$
is incident to an edge in $E(G'') \cap f_{a'}(E(H)) \setminus (E(G'') \cap f_{a'-1}(E(H)))$.%
\COMMENT{(by the definition of $W_{a'}^\ell$ and $W_{a'-1}^\ell$).
But the unique edge at $f_{a'}(\tau_{a,\ell}(x))$ in $f_{a'}(E(H)) \setminus f_{a'-1}(E(H))$ is $f_{a'}(\tau_{a,\ell}(x))v$
(since $\ell$ in $N_{R_K}(a)$ and so $N_{R_K}(\ell) \cap I_{a'}= \{a\}$ and $f_{a'}(x)=v$).
So $f_{a'}(\tau_{a,\ell}(x)) \in W_{a'}^\ell \setminus W_{a'-1}^\ell$ implies $f_{a'}(\tau_{a,\ell}(x))v \in E(G'')$.}
So $f_{a'}(\tau_{a,\ell}(x))v \in E(G'')$, i.e.~$x$ is dangerous for $v$, which means that (i) holds, a contradiction.
Thus $f_{a'}(\tau_{a,\ell}(x)) \notin W_{a'}^\ell$, as required.)

Now we estimate $|N^a(v)|$. By (\ref{A'ell max degree}) there are at most $K\Delta_R\gamma n$ vertices $y$ such that $y$ is dangerous for $v$.
Also $y$ satisfies (ii) if and only if $y\in \tau_{q,a}(f^{-1}_{a'-1}(W_{a'-1}^{q}))$ for some $q \in N_{R_K}(a)\cap I_{1}^{a'-1}$. 
Suppose $\mathcal{D}_{a'-1}$ and thus (D$1_{a'-1}$) holds. This shows that the number of vertices $y$ satisfying (ii) is at most $\sum_{q\in N_{R_K}(a)\cap I_1^{a'-1}} |W_{a'-1}^{q}| \leq K^2\Delta_R^2 \gamma^{4/5} n$. Thus 
\begin{align*}
|N^a(v)| \leq K^2 \Delta_R^2 \gamma^{4/5}n + K\Delta_R\gamma n\leq \gamma^{2/3}m.
\end{align*}
So by $(\text{M}'2)_{a'}$,
\begin{align}\label{D2ell}
\mathbb{P}[ f_{a'}^{-1}(v) \in N^a(v)\mid \mathcal{D}_0^{a'-1}] \leq (1+ h(4K\Delta_R\sqrt{\xi_{a'-1}}))\frac{ \gamma^{2/3} m}{p(\vec{d'},a,a'-1)m} \leq \gamma^{3/5}.
\end{align}
Thus by (\ref{D0ell}), (\ref{D1ell}) and (\ref{D2ell}),
\begin{align}\label{prob ell=a}
\begin{split}
\mathbb{P}[\mathcal{D}_{a'} \mid \mathcal{D}_0^{a'-1}] &= \mathbb{P}[\mathcal{A}_{a'}, (\text{D}1_{a'}), (\text{D}2_{a'}) \mid \mathcal{D}_0^{a'-1}]  \geq 1- (Kr+1)(1-3c)^n  -   \gamma^{3/5}  \geq 1 - \frac{\gamma^{2/3}}{w}.
\end{split}
\end{align}

Finally we consider $(\text{D}2_{\ell'})$ for the case when $\ell'>a'$. Let $x:= \phi^{-1}(v)$. 
For any $q\in N_{R_K}(a)$, let $y'_{q}:= \tau_{a,q}(x)$, and let $y'_{q,\ell}:= \tau_{q,\ell}(y'_{q})$ and $y'_{q,q}:=y'_{q}$
for $\ell \in N_{R_K}(q)$. Note that $y'_{q,\ell} \in U'_{\ell}$.
Let $N^*_{R_K}(a)$ be the set of all those $q\in N_{R_K}(a)$ for which either $|N_{R_K}(q)\cap I_{\ell'}|=1$ or $q \in I_{\ell'}$. For $q\in N^*_{R_K}(a)\setminus I_{\ell'}$, let $\ell_{q}$ be such that $N_{R_K}(q)\cap I_{\ell'} =\{\ell_q\}$, and for $q\in N^*_{R_K}(a)\cap I_{\ell'}$, let $\ell_q:=q$.
Then (\ref{A'ell max degree}) and $(\text{M}'1)_{\ell'}$ of Lemma~\ref{slender lemmas} together imply that  
\begin{align}\label{gamma4/5}
\sum_{q\in N^*_{R_K}(a)} \mathbb{P}[ f_{\ell'}(y'_{q,\ell_q}) \in N_{A'_{\ell_q}}(y'_{q,\ell_q}) \mid \mathcal{D}_0^{\ell'-1}] 
\leq  \sum_{q\in N^*_{R_K}(a)}2  \frac{\gamma^{9/10} }{p(\vec{d'},\ell_q,\ell'-1)}\leq \gamma^{4/5}.
\end{align}\COMMENT{We want to avoid artificial vertices to apply $(M'1)_{\ell'}$. To see that this is ok note that firstly if $y'_{q,\ell_q}$
is an artificial vertex then $ f_{\ell'}(y'_{q,\ell_q})\notin N_{A'_{\ell_q}}(y'_{q,\ell_q})$ since $N_{A'_{\ell_q}}(y'_{q,\ell_q})\subseteq V(G)$.
So in the sum we only need to consider $q\in N^*_{R_K}(a)$ for which $y'_{q,\ell_q}$ is not an artificial vertex. But then we can
apply $(M'1)_{\ell'}$ since $N_{A'_{\ell_q}}(y'_{q,\ell_q})\subseteq V(G)$.}
Note that if $\mathcal{D}_0^{\ell'-1}$ (and thus (D$2_{\ell'-1}$)) holds, then $f_{\ell'-1}(y'_{q})\notin W_{\ell'-1}^{q}$ for all $q \in N_{R_K}(a)\cap I_{1}^{\ell'-1}$. Thus if in addition we have $f_{\ell'}(y'_{q,\ell_q}) \notin N_{A'_{\ell_q}}(y'_{q,\ell_q})$ for all $q\in N^*_{R_K}(a)$, then $f_{\ell'}(y'_{\ell}) \notin W_{\ell'}^{\ell}$ for all $\ell \in N_{R_K}(a)\cap I_{1}^{\ell'}$, which implies (D$2_{\ell'}$).%
\COMMENT{As an intermediate step, the weaker additional assumption
$f_{\ell'}(y'_{q,\ell_q}) \notin N_{A'_{\ell_q}}(y'_{q,\ell_q})$ for all $q\in N^*_{R_K}(a) \setminus I_{\ell'}$ already gives
 $f_{\ell'}(y'_{\ell}) \notin W_{\ell'}^{\ell}$ for all $\ell \in N_{R_K}(a)\cap I_{1}^{\ell'-1}$.} 
Hence \eqref{gamma4/5} gives
\begin{align}\label{D3ell}
\mathbb{P}[(\text{D}2_{\ell'}) \mid \mathcal{D}_0^{\ell'-1}] \geq 1- \gamma^{4/5}.
\end{align}
Now (\ref{D0ell}), (\ref{D1ell}) and (\ref{D3ell}) together imply that for $\ell'>a'$
\begin{align}\label{prob ell>a}
\mathbb{P}[\mathcal{D}_{\ell'} \mid \mathcal{D}_0^{\ell'-1}] \geq 1- (Kr+1)(1-3c)^n -  \gamma^{4/5} \geq 1-\frac{\gamma^{2/3}}{w}.
\end{align}
Therefore, by (\ref{prob ell<a}), (\ref{prob ell=a}) and (\ref{prob ell>a}),
\begin{align*}
\mathbb{P}[\mathcal{D}_0^{w}] = \prod_{\ell'=1}^{w}\mathbb{P}[\mathcal{D}_{\ell'}\mid \mathcal{D}_0^{\ell'-1}]\geq \left(1-\frac{\gamma^{2/3}}{w}\right)^{w} \geq 1 - \gamma^{1/2}.
\end{align*}

Recall that $\mathcal{D}_0^{w}$ implies 
$v\notin \phi_2(H,G,G'').$ Hence, $\mathbb{P}[v\in \phi_2(H,G,G'')] \leq 1- \mathbb{P}[\mathcal{D}_0^{w}]  \leq \gamma^{1/2}, $ i.e. (B4.1) holds.

Now we show that (B4.2) holds.
Assume that $(\text{D}1_{w})$ holds. 
If $u \in U_{j}$ satisfies $u\in \phi_2(H,G,G'')$, then either $u\in W^j_{w}$ or there exists $v,\ell$ such that $v\in W^\ell_w$, $\ell \in N_{R_K}(j)$ and $uv \in \phi(E(H))$. 
Since $\phi(E(H))$ is a matching between $U_\ell$ and $U_j$, there can be at most $|W^\ell_{w}|$ vertices $u \in U_j$ such that $uv \in \phi(E(H))$ for some $v\in W^{\ell}_w$. Thus (D$1_w$) implies that for all $j\in [Kr]$ we have
\begin{align*}|\{u : u \in U_{j}, u\in \phi_2(H,G,G'') \}| &\leq |W^j_{w}| + \sum_{\ell \in N_{R_K}(j)} |W^\ell_{w}| \leq R^{j,w}_{K} \gamma^{4/5} n + \sum_{\ell \in N_{R_K}(j)} R^{\ell,w}_{K} \gamma^{4/5} n \\ 
&\leq (K\Delta_R+1)(K\Delta_R)\gamma^{4/5} n \leq \gamma^{3/5} n. \end{align*}
Also, by \eqref{D0ell} and \eqref{D*1ell},  
\begin{align}\label{D* relation}
\mathbb{P}[\mathcal{D}^*_{\ell'} \mid \mathcal{D}^{*,\ell'-1}_{0}] \geq 1 - 2Kr(1-3c)^n.
\end{align}
Thus we get
$$\mathbb{P}[(\text{D}1_{w})] \geq \mathbb{P}[\mathcal{D}^{*,w}_0]
\geq \prod_{\ell'=1}^{w} \mathbb{P}[\mathcal{D}^*_{\ell'} \mid \mathcal{D}^{*,\ell'-1}_{0} ] \geq (1 - 2Kr(1-3c)^n )^w \geq 1 - (1-2 c)^n.$$
Thus  
$$\mathbb{P}[|\{u : u \in U_{j}, u\in \phi_2(H,G,G'') \}|\leq \gamma^{3/5}n \text{ for all }j\in[Kr]] \geq 1 - (1-2c)^n,$$
i.e. (B4.2) holds.
\end{proof}  

We can now deduce (B5) from (B1) and (B3).

\begin{claim}\label{s edges}
(B5) holds.
\end{claim}

\begin{proof}
First, by part (B1) of Lemma \ref{modified blow-up}, for any $T\subseteq N_{G}(v_1,\dots, v_s)\cap V_j$ with $|T| \geq f(\epsilon) n$,
\begin{align}\label{s vertices}
&\mathbb{E}\left[ |\phi(E(H))\cap \{v_\ell v: v\in T, \ell\in [s]\}| \right]  = \sum_{\ell=1}^{s} \mathbb{E}\left[ |N_{\phi(H)}(v_\ell)\cap T| \right] = (1\pm f(\epsilon))\frac{k_{i,j}s}{d_{i,j}n}|T|.
\end{align}
Let $W:=N_G(v_1,\dots,v_s)\cap V_j$. Consider $$T_1:=\left\{ v \in W: \mathbb{P}[B_v=1] > (1+2f(\epsilon)) \frac{k_{i,j}s}{d_{i,j} n} \right\} \text{ and }T_2:=\left\{ v \in W: \mathbb{P}[B_v=1]<(1-2f(\epsilon)) \frac{k_{i,j}s}{d_{i,j} n}\right\}.$$ 
First, we show that $|T_1| \leq f(\epsilon) n$. Assume $|T_1| > f(\epsilon) n$.
Since $B_v=1$ implies that $|\phi(E(H))\cap \{v_\ell v : \ell\in [s]\}|\geq 1$, we get
$$\mathbb{E}\left[|\phi(E(H))\cap \{v_\ell v: v\in T_1, \ell\in [s]\}|\right] \geq \sum_{v\in T_1} \mathbb{P}[B_v=1] > (1+2f(\epsilon)) \frac{k_{i,j}s|T_1|}{d_{i,j}n},$$ which is a contradiction to (\ref{s vertices}). Thus $|T_1| \leq f(\epsilon) n$.

Now, we show that $|T_2| \leq f(\epsilon) n$.  Assume $|T_2| > f(\epsilon) n$. 
Let $A(v_1,\dots, v_s)$ be  the number of pairs $v_\ell,v_{\ell'}$ such that $\phi^{-1}(v_\ell)$ and $\phi^{-1}(v_\ell')$ share a common neighbour in $H$. Then by part (B3.1) of Lemma \ref{modified blow-up}, 
\begin{align}\label{A common nbr} \mathbb{E}[A(v_1,\dots, v_s)] \leq \frac{s^2}{\sqrt{n}}.
\end{align}
Let $a$ be defined by $$|\phi(E(H))\cap \{v_\ell v: v\in T_2, \ell \in [s]\}|  = a+\sum_{v\in T_2} B_v.$$
Thus $a>0$. 
By considering the bipartite subgraph of $\phi(H)$ spanned by $\{v_1,\dots, v_s\}$ and $T_2$, it is easy to see that there are at least $a/\Delta(H) \geq a/(k+1)$ pairs $v_\ell,v_{\ell'}$ for which there exists a vertex $v\in T_2$ such that $v_\ell v, vv_{\ell'} \in \phi(E(H))$. Thus in this case $\phi^{-1}(v_\ell)$ and $\phi^{-1}(v_{\ell'})$ share a common neighbour in $H$, and so $a \leq (k+1)A(v_1,\dots, v_s)$.
Therefore, 
$$ |\phi(E(H))\cap \{v_\ell v: v\in T_2, \ell \in [s]\}|  \leq \sum_{v\in T_2} B_v + (k+1) A(v_1,\dots, v_s).$$
Thus by (\ref{A common nbr}),
\begin{align*} 
\mathbb{E}\left[|\phi(E(H))\cap \{v_\ell v: v\in T_2, \ell \in [s]\}| \right] &\leq \sum_{v\in T_2} \mathbb{P}[B_v=1] + (k+1)\mathbb{E}[A(v_1,\dots,v_s)] \\
&\leq (1-2f(\epsilon)) \frac{k_{i,j}s|T_2|}{d_{i,j}n} + \frac{(k+1)s^2}{\sqrt{n}} \\ 
&\leq \left(1-\frac{3}{2}f(\epsilon)\right) \frac{k_{i,j}s|T_2|}{d_{i,j} n},
\end{align*} a contradiction to (\ref{s vertices}). Thus $|T_1|+|T_2| \leq 2f(\epsilon)n$, which proves (B5).
\end{proof}

\begin{claim}
(B6) holds.
\end{claim}
\begin{proof}
We are given sets $Q\subseteq X_i, W\subseteq V_i$ with $|Q|,|W| \geq f(\epsilon) n$. For all $\ell \in J_i$, let $Q_\ell:= Q\cap Y_\ell$, and define $q_\ell$ by $|Q_\ell| = q_\ell n$. Then 
\begin{align}\label{q sum} \sum_{\ell \in J_i} q_\ell  = \frac{|Q|}{n}.\end{align} 
Similarly for all $\ell\in J_i$, let $W_\ell:= W\cap U_\ell$.
We define the following event.
\begin{itemize}
\item[(\text{BW}1)] $|W_{\ell}| = (1\pm \epsilon^{1/3})\frac{|W|}{K} \text{ for all } \ell \in J_i.$
\end{itemize}
By a similar argument as in \eqref{mathcal B prob 1}, we have
\begin{align}\label{R1}
\mathbb{P}[(\text{BW}1) \text{ holds }] \geq 1- (1-3c)^n.
\end{align}

Similarly as in \eqref{aa'bb'}, let $\ell'_1 < \dots < \ell'_{K}$ be the indices such that $J_i \subseteq \bigcup_{j=1}^{K} I_{\ell'_{j}}$ and $J_i \cap I_{\ell'_{j}} = \{\ell_{j}\}$ for some $\ell_{j}$. So $J_i= \{\ell_1,\dots, \ell_K\}$. For $j\in [K]$ define the event $(\text{QW}_{\ell'_{j}})$ by
\begin{itemize}
 \item[$(\text{QW}_{\ell'_{j}})$] $|f_{\ell'_{j}}(Q_{\ell_{j}})\cap W_{\ell_{j}}| =  q_{\ell_{j}} |W| \pm f(\epsilon)^3 m.$
\end{itemize}
For $0 \le \ell'\le w$, where $w$ is defined as in (V1) of the definition of a valid input, we define 
$$\mathcal{Q}_{\ell'} := \left\{ \begin{array}{ll} \mathcal{A}_{\ell'}\wedge (\text{QW}_{\ell'}) & \text{ if }\ell'=\ell'_{j} \text{ for some }j\in [K], \\
\mathcal{A}_{\ell'} &\text{ otherwise,}
\end{array}\right.$$
and let $\mathcal{Q}_0^{\ell'} = \bigwedge_{\ell''=0}^{\ell'} \mathcal{Q}_{\ell''}$. In particular, $\mathcal{Q}_0$ is the event which always occurs.
Since $\mathcal{Q}_0^{\ell'-1}$ and $(\text{BW}1)$ are events which only depend on the history of algorithm prior to Round $\ell'$, Lemma~\ref{slender lemmas}(ii) implies that for all $\ell'\in [w]$
\begin{align}\label{QA}
\mathbb{P}[\mathcal{A}_{\ell'} \mid \mathcal{Q}_0^{\ell'-1},(\text{BW}1) ] \geq 1- (1-3c)^{Km} \geq 1-(1-3c)^n.
\end{align}
Now we show that when $\ell'=\ell'_{j}$ for some $j\in [K]$, we have
\begin{align}\label{QR}
\mathbb{P}[ (\text{QW}_{\ell'}) \mid \mathcal{Q}_0^{\ell'-1},(\text{BW}1) ] \geq 1-(1-3c)^n.
\end{align}
Let $\ell:= \ell_{j}$. 
If $q_{\ell} < f(\epsilon)^3/K$, then we have $|f_{\ell'}(Q_{\ell})\cap W_{\ell}| = q_{\ell}|W| \pm f(\epsilon)^3 m$, so we immediately get $(\text{QW}_{\ell'})$. So suppose $q_{\ell} \geq f(\epsilon)^3/K$ and note that $f(\epsilon)^3 \geq 2K h'(4K\Delta_R \sqrt{\xi_{\ell'-1}})$.
Also note that by (BW1) we have $|W_\ell|=(1\pm \epsilon^{1/3})|W|/K\ge h'(4K\Delta_R\sqrt{\xi_{\ell-1}})$
and 
$$
\left(1\pm h'(4K\Delta_R\sqrt{\xi_{\ell-1}})\right)\frac{|W_\ell||Q_\ell|}{m}\stackrel{\text{(BW1)}}{=}\left(1\pm \frac{f(\epsilon)^3}{2K}\right)(1\pm \epsilon^{1/3})\frac{|W|}{K}Kq_\ell=q_\ell |W|\pm f(\epsilon)^3m.
$$
Thus we can apply (M$'$4)$_{\ell'}$ with $Q_\ell$, $W_\ell$ playing the roles of $S$, $T$ to see that
\begin{align}
\mathbb{P}[ (\text{QW}_{\ell'}) \mid \mathcal{Q}_0^{\ell'-1},(\text{BW}1) ]  
&= \mathbb{P}[|f_{\ell'}(Q_{\ell})\cap W_{\ell}| =  q_{\ell}|W|\pm f(\epsilon)^3 m \mid \mathcal{Q}_0^{\ell'-1},(\text{BW}1) ] \nonumber \\
&\geq \mathbb{P}[|f_{\ell'}(Q_{\ell})\cap W_{\ell}| =  (1\pm h'(4K\Delta_R \sqrt{\xi_{\ell'-1}}))  |W_{\ell}||Q_\ell|/m \mid \mathcal{Q}_0^{\ell'-1},(\text{BW}1) ] \nonumber\\
&\geq 1-(1-4Kc)^m \geq 1-(1-3c)^n. \nonumber
\end{align}
Hence, \eqref{QR} holds.
Therefore, \eqref{QA} together with \eqref{QR} imply
\begin{align}
\mathbb{P}[\mathcal{Q}_{\ell'} \mid \mathcal{Q}_0^{\ell'-1},(\text{BW}1) ] \geq 1-2(1-3c)^n.
\end{align}
Thus by \eqref{R1}
$$\mathbb{P}[\mathcal{Q}_0^w,(\text{BW}1) ] \geq 1 - (2w+1)(1-3c)^n \geq 1-(1-2c)^n.$$
Note that if $\mathcal{Q}_0^w$ 
holds, then 
\begin{align*}
|\phi(Q)\cap W| =\sum_{\ell \in J_i} |\phi(Q_\ell) \cap W_{\ell}| 
= \sum_{j \in [K]} \left(q_{\ell_j}|W| \pm f(\epsilon)^3 m\right) \stackrel{\eqref{q sum}}{=} (1\pm f(\epsilon))\frac{|Q||W|}{n}.
\end{align*}
In the final equality we used that $|Q|,|W| \geq f(\epsilon)n$. This proves (B6).
\end{proof}

\subsection{A blow-up lemma for partially prescribed embeddings}
We now deduce from Lemma~\ref{modified blow-up} another extension of the blow-up lemma, which we shall also apply in our main algorithm in the next section. Suppose we are given an embedding $\phi$ of $H$ into $G$ and a set $Z$ of vertices whose embedding is unsuitable (this will be the case if $\phi(H)$ overlaps with the edges of previously embedded graphs in the main packing algorithm in Section~\ref{sec:main}). Then we can re-embed these vertices using edges of a `patching graph' $P$ provided $\phi$ was well behaved with respect to $P$. This notion of being well behaved is captured by the candidacy bigraphs $F$ being super-regular. Let $Z_i:=Z\cap X_i$ and $W_i:=\phi(Z_i)\subseteq V_i$. A candidacy bigraph $F$ will encode the possible new images of a vertex, i.e. we may only embed $z\in Z_i$ to $w\in W_i$ if $zw \in E(F)$. So the case of Lemma~\ref{patching} when each $F[Z_i,W_i]$ is a complete bipartite graph means that there are no constraints.

\begin{lemma}\label{patching}
Suppose $0<1/n \leq 1/m \ll \delta \ll \beta',\beta, 1/k, 1/(C+1) \leq 1$, and $1/m \ll 1/r$ and $m\leq n-C$. 
Let $R$ be a graph on $[r]$ with $\Delta(R) \leq k$. Let $\vec{\beta}$ be a symmetric $r\times r$ matrix such that $\min_{ij \in E(R)} \beta_{i,j} = \beta$. 
Suppose $H$ is a graph admitting vertex partition $(R, \mathcal{X})$ with $\mathcal{X}=(X_1,\dots,X_r)$ such that $\Delta(H)\leq k$ and $\Delta(H[X_i,X_j])\leq 1$ for all $ij\in E[R]$.  Suppose $A_0$ is a graph with bipartition $(V(H),V(P))$ such that $E(A_0) = \bigcup_{i\in [r]} E(A_0[X_i,V_i])$.
Suppose $P$ is a graph admitting vertex partition $(R,\mathcal{V})$ with $\mathcal{V}=(V_1,\dots,V_r)$, where
$\max_{i\in [r]}|V_i|=n$ and $n-C\leq |V_i|=|X_i| \leq n$. Suppose further that $\phi: V(H) \rightarrow V(P)$ is a bijection between $V(H)$ and $V(P)$ such that $\phi(X_i)=V_i$.
Suppose $N$ is an $(H,R,\mathcal{X})$-candidacy hypergraph and $F$ is an $(H,P,R,A_0,\phi,\mathcal{X},\mathcal{V},N)$-candidacy bigraph.
Suppose also that $Z_i \subseteq X_i$, and $W_i=\phi(Z_i) \subseteq V_i$ are sets such that $|Z_i|=|W_i|=m$, and let $Z:=\bigcup_{i=1}^{r} Z_i$, and $W:=\bigcup_{i=1}^{r} W_i$. 
Finally, suppose the following conditions hold:

\begin{itemize}\label{patching condition 1}
\item[(a)]
$F[Z_i, W_i]\textit{ is }(\delta,\beta')\textit{-super-regular for every }i\in [r]$.

\item[(b)]
$P[W]\textit{ is }(\delta,\vec{\beta})\textit{-super-regular with respect to }(R,W_1,\dots, W_r)$.
\end{itemize}

Then we can find a bijection $\phi': V(H) \rightarrow V(P)$ with $\phi'(X_i)=V_i$ for all $i\in [r]$, and such that

\begin{enumerate}
\item [(i)]$\phi'(x)=\phi(x)$ for every $x \notin Z$,

\item  [(ii)]$\phi'(x)\phi'(y) \in E(P)$ for every edge $xy\in E(H)$ with $\{x,y\} \cap Z \neq \emptyset$,

\item[(iii)] $\phi'(x) \in N_{F}(x) \subseteq N_{A_0}(x)$ for every $x\in Z$.
\end{enumerate}

\end{lemma}
\begin{proof}
Choose additional constants $c,\gamma$ satisfying $1/m \ll c \ll \delta \ll \gamma \ll \beta',\beta,1/k,1/(C+1)$. Let $\mathcal{Z}:=(Z_1,\dots, Z_r)$ and $\mathcal{W}:=(W_1,\dots, W_r)$. Let $Q$ be an $r$-partite graph admitting vertex partition $(R,\mathcal{W})$ such that $Q[W_i,W_j]$ is a complete bipartite graph for all $ij \in E(R)$. Let $\vec{\tau}$ be the $r\times r$ matrix such that $\tau_{i,j}=1$ for $ij\in E(R)$
and $\tau_{i,j}=0$ for $ij\notin E(R)$. Since $\Delta(H[Z_i,Z_j])\le 1$ for all $ij\in E(R)$, we can add edges to $H[Z]$ to obtain a graph
$H'\supseteq H[Z]$ such that $H'[Z_i,Z_j]$ is a perfect matching for each $ij\in E(R)$.
Apply Lemma~\ref{modified blow-up} with the following graphs and parameters.\newline

{\small
\begin{tabular}{c|c|c|c|c|c|c|c|c|c|c|c}
object/parameter & $P[W]$ & $Q$ &  $H'$ & $F[Z\cup W]$ &$R$ & $\mathcal{W}$& $\mathcal{Z}$& $m$ & $c$ & $\delta$ & $r$  \\ \hline
playing the role of & $G$& $P$&  $H$ & $A_0$ & $R$ & $\mathcal{V}$& $\mathcal{X}$& $n$ & $c$ & $\epsilon$ & $r$
\\ \hline \hline
object/parameter & $\gamma$ & $1$& $\beta$  & $ \beta'$& $1$& 
$k$ & $0$ & $\vec{\beta}$ & $\vec{\tau}$ & $ \vec{\tau}$ \\ \hline
playing the role of &$\gamma $ & $\beta$& $d$ & $d_0$ &  $k$&$\Delta_R$ & $C$ & $\vec{d}$ & $\vec{\beta}$ & $\vec{k}$ \\ 
\end{tabular}
}\newline \vspace{0.2cm}

\noindent Then by Lemma~\ref{modified blow-up}, with probability at least $1-(1-c)^m$, we get an embedding $\phi':Z\rightarrow W$
of $H'$ into $P[W]$ (and thus also of $H[Z]$ into $P[W]$) such that for all $x\in Z$, $\phi'(x) \in N_{F}(x)$. Let $\phi'(x):=\phi(x)$ if $x\notin Z$. Then (i) holds by definition.

Since $\phi'$ is an embedding of $H[Z]$ into $P[W]$, $\phi'(x)\phi'(y) \in E(P)$ for every edge $xy \in E(H[Z])$. If $xy\in E(H)$, $x\in Z_i$ and $y\notin Z$, then $\phi'(y)=\phi(y)$. Since $F$ is an $(H,P,R,A_0,\phi,\mathcal{X},\mathcal{V},N)$-candidacy bigraph,
$\phi'(x) \in N_{F}(x)$ and (CB2) imply that 
\begin{align*}
\phi'(x) &\in N_{F}(x) \subseteq N_{A_0}(x)\cap \bigcap_{y'\in N_x} N_{P}(\phi(y'))\cap V_i \subseteq N_{A_0}(x)\cap \bigcap_{y'\in N_H(x)} N_{P}(\phi(y'))\cap V_i \\ &\subseteq N_{A_0}(x)\cap N_{P}(\phi(y)) = N_{A_0}(x)\cap N_{P}(\phi'(y)).
\end{align*}
Thus (ii) and (iii) hold. (Note that we do not require properties (B1)--(B6) in this application of Lemma~\ref{modified blow-up}.)
\end{proof}



\section{Main packing algorithm}\label{sec:main}

In this section we combine the results derived in previous sections to establish Theorem~\ref{main lemma}, which we consider to be the main packing result of this paper. It guarantees an approximate decomposition of a super-regular graph $G$ into bounded degree graphs $H_1,\dots,H_s$, provided the $H_i$ reflect the large scale structure of $G$. More precisely, we assume that $G$ has a reduced graph $R$ of moderate degree, and for each edge $ij$ of $R$, the corresponding pair is $\epsilon$-regular in $G$, and this pair also corresponds to an almost regular bipartite graph in each $H_i$.

\begin{thm}\label{main lemma}
Suppose 
$0<1/n\ll c \ll \epsilon \ll \eta , \eta', \alpha, d, d_0, 1/k, 1/(C+1), 1/\Delta_R$ and $1/n\ll 1/r$. Let $s\in \mathbb{N}$ be an integer such that $s \leq \eta^{-1} n$.
Suppose the following assertions hold.

\begin{itemize}
\item[(S1)] $R$ is a graph on $[r]$ with $\Delta(R)\leq \Delta_R$.
\item[(S2)] $\vec{d}$ and $\vec{k}^{i}$ are symmetric $r\times r$ matrices for all $i\in[s]$ such that $\min_{jj'\in E(R)}d_{j,j'}=d$,
$k^{i}_{j,j'}\in\mathbb{N}$ and $k^{i}_{j,j'}\leq k$ for all $jj'\in E(R)$, and $d_{j,j'}=k^{i}_{j,j'}=0$ if $jj'\notin E(R)$.  \item[(S3)] For all $i\in[s]$, $H_i$ is an $(R,\vec{k}^{i},C)$-near-equiregular graph with respect to $(R,X_1,\dots, X_r)$ such that $\max_{j\in [r]} |X_j|=n$ and
$n-C\leq |X_j|\leq n$ for all $j\in [r]$. 
\item[(S4)] For all $jj'\in E(R)$, $\sum_{i=1}^{s} k^{i}_{j,j'} \leq (1-\alpha) d_{j,j'} n$.
\item[(S5)] $G$ is an $(\epsilon,\vec{d})$-super-regular graph with respect to $(R,V_1,\dots, V_r)$ such that $|X_j|=|V_j|$ for all $j\in [r]$.
\item[(S6)] For all $i\in [s]$, $A_i$ is a bipartite graph with bipartition $(V(H_i),V(G))$ such that $N_{A_i}(X_j)=V_j$ and
$A_i[X_j,V_j]$ is $(\epsilon,d_0)$-super-regular for all $j\in[r]$.
\item[(S7)] For all $j\in [r]$, there is a collection $\mathcal{Q}_{j}$ of subsets of $X_j$ and a collection $\mathcal{W}_{j}$
of subsets of $V_j$ such that $|Q|,|W|\ge \eta' n$ for all $Q\in \mathcal{Q}_{j}$, $W\in \mathcal{W}_{j}$ and such that
$|\mathcal{Q}_{j}| |\mathcal{W}_{j}| \leq (1+c)^n$.
\item[(S8)] $\Lambda$ is a graph with  
$V(\Lambda) \subseteq [s]\times \bigcup_{j=1}^{r} X_j$ and $\Delta(\Lambda) \leq (1-2\alpha)d_0 n$ such that for all $(i,x)\in V(\Lambda)$ and $i' \in [s]$ we have
$ |\{x' : (i',x') \in  N_{\Lambda}((i,x))\}| \leq k.$
Moreover, for all $i\in [s]$ and $j\in [r]$, we have $|\{ (i,x) \in V(\Lambda) : x\in X_j\} | \leq \epsilon |X_j|.$
\end{itemize}
Then for each $i\in [s]$ there exists an embedding $\phi'_i: V(H_i) \rightarrow V(G)$ of $H_i$ into $G$
such that the following assertions hold.
\begin{itemize}
\item[(T1)] $\phi'_i(x) \in N_{A_i}(x)$ for all $x\in V(H_i)$.
\item[(T2)] $\phi'_i(E(H_i)) \cap \phi'_{i'}(E(H_{i'})) =\emptyset$ for $i\neq i'$.
\item[(T3)] For all $j\in[r]$ and $i \in [s]$ and any sets $Q\in \mathcal{Q}_{j}$ and $W\in \mathcal{W}_{j}$, we have $|\phi'_i(Q)\cap W| = (1\pm \eta')\frac{|Q||W|}{n}$.\COMMENT{We need this property just because we want it for the new paper with Felix. This property is not used in this paper.}
\item[(T4)] For all $(i,x)(i',y) \in E(\Lambda)$, we have $\phi'_i(x) \neq \phi'_{i'}(y)$.
\end{itemize}
\end{thm} 

(T1) allows us to prescribe `target sets' for embedding some special vertices in each $H_i$
(see the remark below) and is used in \cite{JKKO}. 
(T4) allows us to to avoid undesired `collisions' between embeddings of special vertices
belonging to different $H_i$,
and will be applied in further work in progress on packing spanning trees.
We also believe (T3) to be of independent interest, with potential further applications. 

On the other hand, (T1), (T3) and (T4)  will not be required when we apply Theorem~\ref{main lemma}
in Section~\ref{sec:concl}.
Note that in such a situation, i.e.~when we apply Theorem~\ref{main lemma} but e.g.~conclusion (T3) is not required, we can ignore condition (S7)
along with the parameters $c$ and $\eta'$. Similarly, if we do not require (T1), then we can ignore (S6) along with the
parameter $d_0$ (by taking $A_i[X_j,V_j]$ to be a complete bipartite graph) and if we do not require (T4), then we can ignore (S8).

\begin{remark}
In \cite{JKKO} and in Section~\ref{sec:stacking}, it is convenient to use that Theorem~\ref{main lemma} holds even if we replace (S6) by the following.
\begin{itemize}
\item[(S6$'$)] 
For all $i\in [s]$, let $X^i_j\subseteq X_j$ be a subset with $|X^i_j|\leq \epsilon n$ and let $A_i$ be a bipartite graph with bipartition $(V(H_i),V(G))$ such that, for all $j\in [r]$,
\begin{itemize}
\item[$\bullet$] $N_{A_i}(X_j)=V_j$,
\item[$\bullet$] $d_{A_i}(x)\geq d_0 n$ for all $x\in X^i_j$, and $N_{A_i}(x)= V_j$ for all $x\in X_j\setminus X^i_j$.
\end{itemize}
\end{itemize}
\end{remark}
Indeed, if (S6$'$) is given, Lemma~\ref{lem: deletion of some edges still contain super-regular} gives a subgraph $A'_i$ of $A_i$ which satisfies (S6) (with $\epsilon^{1/3}$
playing the role of $\epsilon$).
Thus Theorem~\ref{main lemma} holds even after we replace (S6) by (S6$'$). In other words,
(S6$'$) and (T1) together imply that for each $i\in [s]$ we can specify a linearly sized target set for a small linear fraction of the
vertices of~$H_i$.

\medskip

Let us now briefly sketch the proof idea of Theorem~\ref{main lemma}. The desired packing will be constructed via a randomised algorithm, called the \emph{Main packing algorithm}, which will be shown to succeed with high probability. The algorithm runs in $T$ `rounds', indexed by a parameter $t$. In Round~$t$ we take a collection of $\gamma n$ graphs $H_i$  (where $\gamma$ is a small constant and $i \in I_t$ with $|I_t|=\gamma n$) and embed them into the current remainder $G^t$ of $G$. 
For this, we first apply Lemma~\ref{modified blow-up} to each $H_i$ with $i \in I_t$ independently, in order to define embeddings $\phi_i$ 
of $H_i$  into $G^t$ for all $i \in I_t$. We then let $G^{t+1}$ be the graph obtained from $G^t$
by deleting the edges in all the $\phi_i(H_i)$ with $i \in I_t$ from $G^t$, and proceed to the next round. 

Clearly, in each round the embeddings $\phi_i$ with $i \in I_t$ need not be pairwise edge-disjoint. 
However, the overlap can be shown to be small, which will allow us to apply Lemma \ref{patching} in order to define slightly modified new embeddings $\phi'_i$, which will be pairwise edge-disjoint, as desired. The `patching graph' $P\subseteq G$, used in Lemma \ref{patching} to carry out these `patching procedures' in all rounds, is set aside at the beginning of the algorithm and its edges are not used for any $\phi_i$.
These patching procedures are also designed to ensure that (T4) holds. Since they are repeated in every round, we have to ensure that certain graphs remain sufficiently super-regular throughout the algorithm.

Since the main packing algorithm is randomised, at various steps there will be a chance of failure, and our goal is to show that the total probability of all possible failures is small. 
\vspace*{0.2cm}

In order to describe the main packing algorithm, we need the following definitions. Choose additional constants $\beta, \gamma,\delta$ so that 
$$\frac{1}{n} \ll c\ll \epsilon \ll \gamma \ll \delta \ll \eta, \eta' ,\alpha, \beta,  d, d_0, \frac{1}{k}, \frac{1}{C+1}, \frac{1}{\Delta_R} \;\; \text{ and } \;\; \beta \ll \alpha, d, d_0.$$ 
Let $\vec{\beta}$ be the $r\times r$ matrix with entries $\beta_{j,j'}:= {\beta} d_{j,j'}/{d}$ for each $j,j'\in [r]$. So $\min_{jj'\in E(R)} \beta_{j,j'} =\beta$. Let $I_t:= \{(t-1)\gamma n+1, \dots, t\gamma n\}$. Let
\begin{align}\label{eq:Tdk}
T:=\left\lceil \frac{s}{\gamma n}\right\rceil \leq \frac{1}{\gamma \eta}, \enspace \vec{d}^1:=\vec{d}-\vec{\beta}, \text{ and } K:=(k+1)^2\Delta_R,
\end{align}
\begin{align}\label{eq:mainalgparams}
d^{t+1}_{j,j'}: = d^t_{j,j'}\prod_{i \in I_t}\left(1 - \frac{k^i_{j,j'}}{ d^t_{j,j'} n}\right) \text{ for all $t\in [T]$ and $jj' \in E(R)$}. 
\end{align}
$T$ will be the total number of rounds and the $d^t_{j,j'}$ track the densities of the unused leftover $G^t[V_j,V_{j'}]$ in the $t$-th round.
Let 
\begin{align}\label{eq:epsrecursion}
\epsilon_1:=\epsilon^{1/3} \;\; \text{ and } \;\; \epsilon_{t+1}: = q(\epsilon_t) \text{ for all } t\in [T],
\end{align}
where $q$ is the function defined in \eqref{eq:functions}. Note that by the choice of the functions in \eqref{eq:functions} (which depend only on $w:= (K\Delta_R)^2(\Delta_R+1)$ and the argument), we can assume that
\begin{align}\label{eq:epsilont}
\epsilon_T \ll \gamma.
\end{align}
  We are now ready to describe the main packing algorithm. 
\newline \vspace{0.2cm} 

\noindent {\bf Main packing algorithm} \vspace{0.2cm}

\noindent {\bf Round 0.} 
It will be convenient that in each of the $T$ rounds we embed exactly $\gamma n$ graphs. For this, let $H'$ be an arbitrary $(R,\vec{k'},C)$-near-equiregular graph with vertex partition $(R,X_1,\dots, X_r)$ where $k'_{j,j'}:=1$ for all $jj'\in E(R)$
and $k'_{j,j'}:=0$ for all $jj'\notin E(R)$. 
Also let $A$ be a union of $r$ complete bipartite graphs between $X_j$ and $V_j$ for $j\in [r]$. Now we let $(H_{p},A_{p}):= (H',A)$ for $s+1 \leq p \leq T\gamma n$.\COMMENT{Here, we are assuming that $\gamma n$ is an integer.} 
Let $\mathcal{H}:=\{ (H_i,A_i) : i \in [T\gamma n] \}$ and $\Phi: = \emptyset$. We apply Lemma \ref{preparing patching graph} to
find a graph $P\subseteq G$ such that $P^1:=P$ is $(\epsilon_1,\vec{\beta})$-super-regular 
and $G^1:=G-P^1$  is $(\epsilon_1,\vec{d}^1)$-super-regular with respect to $(R,\mathcal{V})$, where $\mathcal{V}:=(V_1,\dots, V_r)$. Let $\vec{d'}^t,\vec{\beta'}$ be the $Kr\times Kr$ matrices with entries $d'^t_{\ell,\ell'} := d^t_{j,j'}, \beta'_{\ell,\ell'
}:=\beta_{j,j'}$ where $\lceil \ell/K \rceil = j, \lceil \ell'/K \rceil=j'$ and $j,j'\in[r]$. 
Recall that $R_K$ denotes the $K$-fold blow-up of $R$. 
Let $\mathcal{X}:=(X_1,\dots, X_r)$. 
Let $t:=1$, and proceed to Round~$1$. \vspace{0.2cm}

\noindent {\bf Round ${\bf t}$.} \vspace{0.2cm}

\noindent {\bf Step 1.} Assume that in Round $t-1$ we have defined 
\begin{itemize}
\item[(A1)] a graph $G^t\subseteq G^1$, which is $(\epsilon_t,\vec{d}^t)$-super-regular with respect to $(R,V_1,\dots, V_r)$,
\item[(A2)] a graph $P^t\subseteq P^1$, which is $(\delta^{1/4},\vec{\beta})$-super-regular with respect to $(R,V_1,\dots, V_r)$, and  \COMMENT{The latter will be proved in Claim \ref{Pt}. It is not really necessary to state it here.} 
\item[(A3)] graph embeddings $\phi'_i : V(H_i) \rightarrow V(G)$ of $H_i$ into $G$ for each $i\in \bigcup_{t'=1}^{t-1} I_{t'}$.
\end{itemize}
$G^t\cup P^t$ is the set of currently available edges, the main part of the embedding in Round $t$ will be done in $G^t$, the patching graph $P^t$ will be used to partially re-embed copies of the~$H_i$ (for $i\in I_t$) in order to make these copies edge-disjoint from each other and in order to ensure (T4) `within' Round~$t$.

For all $i\in I_t$ and $j\in [r]$, let 
\begin{align}\label{eq: X'j def}
X'_{j}(i):= \{ x\in X_{j} : (i,x) \in V(\Lambda) \}.
\end{align}
Then (S8) implies that $|X'_j(i)|\leq \epsilon |X_j|$.
For each $x \in X'_j(i)$, we consider 
$$D'_i(x)  :=\left\{  \phi'_{i'} (x') : (i',x') \in  N_{\Lambda}( (i,x) ), i' \leq (t-1)\gamma n\right\}, \enspace \text{ and } \enspace
B'_i(x)  := N_{A_i}(x) \setminus D'_i(x).$$
Note that, in order to ensure (T4), when embedding $x$ it is necessary to avoid the vertices in $D'_i(x)$, i.e.~we need to embed $x$ into $B'_i(x)$.
To ensure this property, we will now update $A_i$ accordingly.
For all $i\in I_t$, $j\in [r]$ and $x \in X'_j(i)$, (S6) and (S8) imply that
$|B'_i(x)| \ge |N_{A_i}(x)| - |N_{\Lambda}((i,x))| \geq \alpha d_0 |X_j|$.
For all $i\in I_t$ and $j\in [r]$, $A_i[X_j, V_j]$ is $(\epsilon,d_0)$-super-regular, thus
we can apply Lemma~\ref{lem: deletion of some edges still contain super-regular} to obtain an
$(\epsilon_1,\alpha d_0)$-super-regular subgraph $A''_i[X_j, V_j]$ of $A_i[X_j, V_j]$ such that 
\begin{equation}\label{eq: disjointness}
\text{for each $x\in X'_j(i)$, we have $N_{A''_i[X_j,V_j]}(x) \cap D'_i(x) = \emptyset$.}
\end{equation}
For each $i\in I_t$, let
\begin{align}\label{eq: A'' def}
A''_i:= \bigcup_{j=1}^{r} A''_i[X_j,V_j] \subseteq \bigcup_{j=1}^{r} A_i[X_j,V_j].
\end{align}
We apply Lemma~\ref{modified blow-up} with the following graphs and parameters for each $i\in I_t$ one by one in increasing order to obtain an embedding $\phi_i$ of $H_i$ into $G^t$
(we may do so because of \eqref{eq:epsilont}, this is justified in detail in the proof of Claim \ref{type1}). Note that these applications of Lemma~\ref{modified blow-up} are carried out independently from each other.
\vspace*{0.2cm}

{\small

\noindent
\begin{tabular}{c|c|c|c|c|c|c|c|c|c|c|c|c|c|c|c|c|c|c|c}
object/parameter & $G^t$&$P$ &$H_i$ & $R$ & $A''_i$ & $\mathcal{V}$&$\mathcal{X}$ &$2c$ &$\epsilon_t$ &$4k\Delta_R\gamma $&$\vec{\beta}$ &$\vec{d}^t$ & $\alpha d_0$ &$\vec{k}^i$ &$k$ & $r$ & $\Delta_R$ & $C$ & $n$ \\ \hline
playing the role of &$G$ &$P$ &$H$ & $R$ & $A_0$ &$\mathcal{V}$ &$\mathcal{X}$ &$c$ &$\epsilon$ &$\gamma$&$\vec{\beta}$ &$\vec{d}$ &$d_0$ & $\vec{k}$&$k$ &$r$& $\Delta_R$&$C$ & $n$\\ 
\end{tabular}
}\newline
\vspace*{0.2cm}

\noindent Note that we apply it with $P$ rather than $P^t$; this is crucial to ensure that the regularity property of the candidacy bigraph in (F5) below is sufficiently strong. If the Uniform embedding algorithm fails for some $i\in I_t$, then end the algorithm with \textbf{failure of type 1}.

We also define the following event which will be relevant for (T3).
\begin{itemize}
\item[$(\text{QW}_i)$] For all  $j\in [r]$ and any sets $Q\in \mathcal{Q}_{j}$ and $W\in \mathcal{W}_{j}$, we have $|\phi_i(Q) \cap W| = (1\pm \gamma)\frac{|Q||W|}{n}$.
\end{itemize}
If $(\text{QW}_i)$ fails, then we end the algorithm with \textbf{failure of type 2}.

The above application of Lemma~\ref{modified blow-up} gives a tuple $(\phi_i,\mathcal{X}^i, \mathcal{V}^i,F_i ,N^i)$ satisfying (B1)--(B6) of Lemma \ref{modified blow-up}. 
In particular, for each $i\in I_t$,
\begin{itemize}
\item[(F0)] $G^t,P$ both admit the vertex partition $(R_K,\mathcal{V}^i)$, and $H_i$ admits the vertex partition $(R_K,\mathcal{X}^i)$,
\item[(F1)] $\phi_i: H_i \rightarrow G^t$ is an embedding such that $\phi_i(x) \in N_{A''_i}(x)$,
\item[(F2)] the partitions $\mathcal{X}^i=(X_1^i,\dots,X_{Kr}^i)$ refining $X_1,\dots, X_r$ and  $\mathcal{V}^i=(V_1^i,\dots,V_{Kr}^i)$ refining $V_1,\dots, V_r$ with $\max_{j\in [Kr]} |X_j^i|=m'$
and $m'-C\le |X_j^i|=|V_j^i|\le m'$, where $m'=\lceil n/K\rceil$, act as $\mathcal{Y}$ and $\mathcal{U}$ in Lemma \ref{modified blow-up}, respectively,
\item[(F3)] $N^i$ is an $(H_i,R_K,\mathcal{X}^i)$-candidacy hypergraph with $|N^i_x| \leq K\Delta_R$ for all $x\in V(H_i)$, 
\item[(F4)]
$P$ is $(\epsilon^{1/3}_t,\vec{\beta'})$-super-regular with respect to $(R_K,\mathcal{V}^i)$,
\item[(F5)] $F_i=\bigcup_{j\in [Kr]} F_i[X^i_j,V^i_j]$,
where $F_i[X_j^i,V_j^i]$ is $(f(\epsilon_t),\alpha d_0 p(R_K,\vec{\beta'},j))$-super-regular. Moreover, $F_i$ is an $(H_i,P,R_K,A''_i,\phi_i,\mathcal{X}^i,\mathcal{V}^i,N^i)$-candidacy bigraph. In particular, for all $x\in X_j^i$,
\begin{align}\label{eq:recallbigraph}
N_{F_i}(x)=N_{F_i[X^i_j,V^i_j]}(x) \subseteq N_{A''_i}(x)\cap\bigcap_{y\in N^i_x} N_{P}(\phi_i(y))\cap V_j^i.
\end{align}
\end{itemize}

For \eqref{eq:recallbigraph} recall that candidacy bigraphs are defined before Lemma~\ref{modified blow-up}. As remarked after the definition of a candidacy hypergraph, (F3) implies that $\Delta(H_i[X_j^i,X_{j'}^i])\leq 1$ for all $j\neq j'\in [Kr]$.
Define $F'_i\subseteq F_{i}$ by
\begin{align}\label{F' definition}
E(F'_i[X_j^i,V_j^i]):=\left\{xv\in E(F_{i}): x\in X^{i}_j, v\in V_j^{i}\cap \bigcap_{y\in N_x^{i}}N_{P^t}(\phi_{i}(y)) 
\right\}. 
\end{align}
The graph $F'_i$ can be viewed as an update of the candidacy bigraph $F_i$ which accounts for further restrictions imposed in the current round by the fact that the edges of $P-P^t$ are no longer available for the patching process.

Let 
\begin{align}\label{eq:updateGt}
G^{t+1}:= G^t - \bigcup_{i\in I_t} \phi_i(E(H_i)).
\end{align}

If $G^{t+1}$ is not $(\epsilon_{t+1} ,\vec{d}^{t+1})$-super-regular with respect to $(R,V_1,\dots, V_r)$, then end the algorithm with \textbf{failure of type 3}. Otherwise, proceed to Step 2.
\vspace{0.2cm}

\noindent {\bf Step 2.} Observe that in Step 1 we allowed the edge sets of different $\phi_i(H_i)$ to intersect and we allowed $\phi_{i}(x) = \phi_{i'}(y)$ for $(i,x)(i',y) \in E(\Lambda)$ with $i,i'\in I_t$. In Steps 2--4 we aim to resolve these issues by altering the embeddings $\phi_i$ in order to make them edge-disjoint and to satisfy (T4). 
For each $i\in I_t$, let 
\begin{align*}
E_i:&=\bigcup_{j\in I_t\setminus \{i\}}(\phi_i(E(H_i))\cap \phi_j(E(H_j))),\\
U_i:&=\{v\in e: e \in E_i\} \cup \bigcup_{j\in [r]} \phi_i(X'_j(i)),\\
NU_i:&=\bigcup_{j\in [r]} \phi_i(X'_j(i)) \cup \bigcup_{u\in U_i, e\in E(H_i)}  \{v\in V(G): \phi_i(e)=uv\}.
\end{align*}
Note that in particular we have $U_i\subseteq NU_i$.%
\COMMENT{12/7: added $\bigcup_{j\in [r]} \phi_i(X'_j(i)) \cup$ in the def of $NU_i$ to make sure that this sentence still holds - if a vertex in $(X'_j(i)$ is isolated in $H_i$ (which might happen if $R$ contains isolated vertices), then it would not lie in $NU_i$
with the old definition - obviously we don't care about this case, but we also don't want to make the statemet of Thm 6.1 more complicated. This lead to changes in Claim 6.11.}
We define the following two events.

\begin{itemize}
\item [(U1)] $|\{i\in I_t: v\in NU_i\}| < \gamma^{4/3} n \text{ holds for all } v\in V(G).$

\item [(U2)] 
$|NU_i \cap V^i_{j}| \leq \gamma^{2/5} m' \text{ for all } i\in I_t, \text{ and } j\in [Kr].$ 
\end{itemize}
If (U1) or (U2) does not hold, then end the algorithm with \textbf{failure of type 4}. Otherwise, proceed to Step 3.   \vspace{0.2cm}

\noindent {\bf Step 3.} Our aim is now to change the embedding $\phi_i$ for the vertices in $U_i$. It turns out that it is much easier to change the embedding on $U_i\cup Y_i$, where $Y_i$ is randomly chosen vertex set of appropriate size. For each $i\in I_t$ we choose a set $Y_i\subseteq V(G)\setminus U_i$ uniformly at random, subject to  $|(U_i\cup Y_i)\cap V^i_{j}| = \delta m'$ for all $j\in [Kr]$ (this is possible since, by (U2) above, we have $|U_i\cap V^i_{j}|\leq |NU_i\cap V^i_{j}|\leq \gamma^{2/5} m'\leq \delta m'$ for all $i\in I_t$). Let 
\begin{align}\label{eq:WiZidef}
W_i:=U_i\cup Y_i \; \text{ and } \; Z_i:=\phi_i^{-1}(W_i).
\end{align}
We define the following events.

\begin{itemize}
\item [(W1)] For all $i\in I_t$ the graph $P^t[W_i]$ is $( \delta^{1/25},\vec{\beta'})$-super-regular with respect to $(R_K,V^i_1 \cap W_i,\dots, V^{i}_{Kr} \cap W_i)$.
\item [(W2)] For all $v\in V(G)$ we have
\begin{itemize}
\item[(i)] $|\{i\in I_t: v\in W_i\}|\le 2\delta \gamma n$,
\item[(ii)] $\sum_{i\in I_t} |\{ e\in E(H_i): v\in \phi_i(e) , \phi_i(e)\cap W_i\neq \emptyset \}| \leq 3k\Delta_R\delta\gamma n$.
\end{itemize}
\item [(W3)] For all $i\in I_t$ and $j\in [Kr]$ the graph $F'_i[ Z_i\cap X^i_j,W_i\cap V^i_j]$ is $(\delta^{1/20},\alpha d_0 p(R_K,\vec{\beta'},j))$-super-regular. 

\end{itemize}
\noindent
So part~(ii) of (W2) says that the total number of necessary changes at a given vertex in the current round will be small, and (W3) implies that $\phi_i$ is compatible with the structure of $P^t$, so we can use $P^t$ to modify $\phi_i$.
If one of (W1)--(W3) fails, then end the algorithm with \textbf{failure of type~5}. 
Otherwise, proceed to Step~4.1. \vspace{0.2cm}

\vspace{0.2cm}

\noindent {\bf Step 4.$\ell$.} Consider $i=(t-1)\gamma n+\ell$.
Define
\begin{align}\label{eq:defPti}
P^t_{i} := P^t - \bigcup_{j\in I_t:j<i}\phi_j'(E(H_j)).
\end{align}
Let $F^*_i\subseteq F'_i$ be defined by
\begin{align}\label{eq:Ft*idef}
E(F^*_i[X_j^i,V_j^i]):=\left\{xv\in F'_i: x\in X^{i}_j, v\in V_j^{i}\cap \bigcap_{y\in N_x^{i}}N_{P_{i}^t}(\phi_{i}(y))\right\}. 
\end{align}
In other words, $F^*_i$ is the maximal subgraph of $F'_i$ that is an $(H_{i},P^t_{i},R_K,A''_i,\phi_{i},\mathcal{X}^i,\mathcal{V}^{i},N^{i})$-candidacy bigraph. It can be viewed as a further update of the candidacy bigraph $F'_i$ which takes into account restrictions arising
from the $\ell-1$ embeddings $\phi'_j$ with $j\in I_t$, $j<i$ already made in the current round. Note that
\begin{align}\label{eq:ptiftidef}
P^t_{(t-1)\gamma n+1} = P^t \;\;\; \text{ and } \;\;\; F^*_{(t-1)\gamma n+1} = F'_{(t-1)\gamma n+1}.
\end{align}
We now check whether the following conditions hold.

\begin{itemize}
\item[($a'$)] For all $j\in [Kr]$ the graph $F^*_i[Z_{i}\cap X^{i}_j,W_{i}\cap V^{i}_j]$ is $(\delta^{1/50},\alpha d_0 p(R_K,\vec{\beta'},j))$-super-regular.
\item[($b'$)] $P^t_{i}[W_{i}]$ is $(4\delta^{1/50},\vec{\beta'})$-super-regular with respect to $(R_K, V^{i}_1\cap W_{i}, \dots, V^{i}_{Kr}\cap W_{i})$.
\end{itemize}
 
If one of $(a')$ and $(b')$ does not hold, end the algorithm with \textbf{failure of type 6}.
Otherwise, note that by \eqref{p R def}, $p(R_{K},\vec{\beta}',j) = p(R,\vec{\beta},j)^K \geq \beta^{K\Delta_R}$ as $\beta =\min_{ij\in E(R)} \beta_{i,j}$. 

Our aim is to apply Lemma~\ref{patching} in order to obtain $\phi'_i$. Before this, we need to update $F^*_i$ further in order to ensure
that (T4) will be satisfied `within' the current round ((\ref{eq: disjointness}) will ensure that (T4) still holds `with respect to previous rounds'). Recall that we have defined $X'_{j'}(i)$ in \eqref{eq: X'j def} for each $j'\in [r]$.
For each $j\in [Kr]$, by (F2) there exists $j'\in [r]$ such that $X^i_j \subseteq X_{j'}$.
Note that $X^i_j\cap X'_{j'}(i)\subseteq  X^i_j\cap Z_i$.%
\COMMENT{12/7: added this sentence}
For each $x \in X^i_j\cap X'_{j'}(i)$,%
\COMMENT{12/7: deleted $\cap Z_i$}
we consider 
\begin{align*}
D''_i(x) & := \left\{  \phi'_{i'} (x') : (i',x') \in  N_{\Lambda}( (i,x) ), (t-1)\gamma n<  i' < i \right\},\\
B''_i(x) & := N_{F^*_i}(x) \setminus D''_i(x).
\end{align*}
Then $|D''_i(x)| \leq k\gamma n$ by (S8), thus for each $x \in  X^i_j\cap X'_{j'}(i)\subseteq  X^i_j \cap Z_i$,%
\COMMENT{12/7: changed $x \in  X^i_j\cap X'_{j'}(i)\cap Z_i$ to
$x \in  X^i_j\cap X'_{j'}(i)\subseteq  X^i_j \cap Z_i$ and deleted ``(S8) implies that''}
we have
$$|B''_i(x)| \ge |N_{F^*_i}(x)| - |D''_i(x)| \stackrel{(a')}{\geq} 
(\alpha d_0 p(R_K,\vec{\beta'},j)-\delta^{1/50})|Z_i\cap X^i_{j}| - k\gamma n \geq  \beta^{K\Delta_R+1} |Z_i\cap X^i_{j}|.
$$
Also $|X^i_j\cap X'_{j'}(i)| \leq |X'_{j'}(i)| \leq \epsilon |X_j| \leq \delta^{1/50}  |Z_i\cap X^i_{j}|$.%
\COMMENT{12/7: replaced $X^i_j\cap X'_{j'}(i)\cap Z_i$ by $X^i_j\cap X'_{j'}(i)$, also in the table below}
Together with $(a')$ this ensures that for each $j\in [Kr]$ we can apply Lemma~\ref{lem: deletion of some edges still contain super-regular} with the following graphs and parameters.\newline

{\small

\noindent
\begin{tabular}{c|c|c|c|c|c|c}
object/parameter & $F^*_i[Z_{i}\cap X^{i}_j,W_{i}\cap V^{i}_j]$ & $X^i_j\cap X'_{j'}(i)$ & $\delta^{1/50}$ & $\alpha d_0p(R_{K},\vec{\beta}',j)$ & $ \beta^{K\Delta_R+1}$ & $B''_i(x)$  \\ \hline
playing the role of &$G[A,B]$ & $A'$ & $\epsilon$ & $d$ & $d_0$ & $B'(v)$\\ 
\end{tabular}
}\newline
\vspace*{0.2cm}

\noindent
We obtain a $(\delta^{1/150}, \beta^{K\Delta_R+1})$-super-regular subgraph $F^{i}[Z_{i}\cap X^{i}_j,W_{i}\cap V^{i}_j]$ of $F^*_i[Z_{i}\cap X^{i}_j,W_{i}\cap V^{i}_j]$.
Let $F^i$ be the graph on $V(F^*_i)$ with edge set
$$E(F^i) := \bigcup_{j\in [Kr]}  E(F^{i}[Z_{i}\cap X^{i}_j,W_{i}\cap V^{i}_j]).$$ 
Then $F^i$ is a spanning subgraph of $F^*_i$ such that%
   \COMMENT{12/7: replaced $x\in X'_{j'}(i)\cap Z_i $ by $x\in X'_{j'}(i)$ below}
\begin{equation}\label{eq: disjointness 2}
\text{for all $j'\in [r]$ and $x\in X'_{j'}(i)$, we have $N_{F^{i}}(x) \cap D''_i(x) = \emptyset$.}
\end{equation}
Moreover,
\begin{itemize}
\item[($a''$)] $F^i[Z_{i}\cap X^{i}_j,W_{i}\cap V^{i}_j]$ is $(\delta^{1/150},  \beta^{K\Delta_R+1})$-super-regular for each $j\in [Kr]$.  
\end{itemize}
Let $A'_i:=\bigcup_{j\in [Kr]} A''_i[X^i_j,V^i_j]$.%
  \COMMENT{Need to get rid of cross edges in order to be able to apply Lemma~\ref{patching}. This is fine for candidacy bigraph
since there we look at $N_{A_i}(x)\cap V^i_j=N_{A'_i}(x)$ for $x\in X^i_j$ anyway.}
Note that since $F^i$ is a spanning subgraph of $F^*_i$, $F^i$ is an $(H_i,P^t_i,R_K,A'_i,\phi_i,\mathcal{X}^i,\mathcal{V}^i,N^i)$-candidacy bigraph. 
We apply Lemma \ref{patching} with the following graphs and parameters to find a new embedding $\phi_i'$ of $H_i$ into $G^t\cup P_i^t \subseteq G^t\cup P^t$ in order to make sure that the images of different $\phi_i'$ are edge-disjoint and (T4) holds
(see Claim~\ref{disjoint} for the proof).
\vspace*{0.2cm}

\noindent
{\small
\begin{tabular}{c|c|c|c|c|c|c|c|c|c|c}
object/parameter & $H_i$& $P^t_{i}$&  $F^i$&$A'_i$ &$N^i$ & $\mathcal{V}^i$& $\mathcal{X}^i$& $R_K$ & $W_i\cap V^i_j$ & $Z_i\cap X^i_j$ 
\\ \hline
playing the role of & $H$ & $P$&  $F$&$A_0$ & $N$&$\mathcal{V}$& $\mathcal{X}$& $R$ & $W_i$ & $Z_i$
\\ \hline \hline
object/parameter & $m'$& $\delta m'$  & $ \delta^{1/150}$& $\beta$& $\vec{\beta'}$ & 
$\beta^{K\Delta_R+1}$ & $Kr$ & $K\Delta_R$ & $C$  & $\phi_i$ \\ \hline
playing the role of & $n$& $m$ & $\delta$ &  $\beta$& $\vec{\beta}$ & $\beta'$ & $r$ & $k$ & $C$ & $\phi$\\ 
\end{tabular}
}\newline \vspace{0.2cm}

\noindent (Lemma~\ref{patching} can be applied by $(a'')$ and $(b')$.)
Then Lemma~\ref{patching}(i) and (iii) imply that
\begin{align}\label{eq:new}
\phi'_i(x)=\phi_i(x) \ \text{ for all $x\notin Z_i$ and }
\phi'_i(x)\in N_{F^i}(x)\subseteq N_{A'_i}(x) \ \text{ for all $x\in Z_i$.}
\end{align}
Together with (F1) and the fact that $N_{A'_i}(x)\subseteq N_{A''_i}(x)$ this implies that
\begin{align}\label{eq:hi}
\phi'_i(x)\in N_{A''_i}(x) \ \ \text{for all } x\in V(H_i).
\end{align}

If $\ell<\gamma n$, proceed to Step 4.$(\ell+1)$. If $\ell=\gamma n$ and $t\leq T-1$, then define the graph
\begin{align}\label{eq:Pt}
P^{t+1} :=  P^t - \bigcup_{i\in I_t} \phi'_i(E(H_i)),
\end{align}
and proceed to Round $t+1$.
If $\ell=\gamma n$ and $t=T$, then we end the algorithm with success.
\vspace{0.2cm}

\noindent
Let us now prove some properties of the above algorithm. The first one concerns the density $d^t_{j,j'}$ of $G^t[V_j,V_{j'}]$.

\begin{claim}\label{number of rounds}
For all $t\in [T]$ and all $jj'\in E(R)$ we have $d^t_{j,j'} \geq \alpha d/2$.
\end{claim}
\begin{proof}
Clearly, we may assume that $t \ge 2$. Then for $jj'\in E(R)$,
\begin{align}\label{di}
d^{t}_{j,j'} \stackrel{\eqref{eq:mainalgparams}}{\geq} 
d^{t-1}_{j,j'}\left(1 - \frac{\sum_{i \in I_{t-1}}k^{i}_{j,j'}}{d^{t-1}_{j,j'} n}\right)
= d^{t-1}_{j,j'}-\frac{\sum_{i \in I_{t-1}}k^{i}_{j,j'}}{n}.
\end{align}
By (\ref{di}), we get 
\begin{align*}
d^{t}_{j,j'}&\geq d^{T}_{j,j'} \geq d^1_{j,j'}- \sum_{i=1}^{(T-1)\gamma n} \frac{k^i_{j,j'}}{n} \stackrel{(\text{S}4)}{\geq} d^1_{j,j'} - (1-\alpha)d_{j,j'} \stackrel{\eqref{eq:Tdk}}{\geq} (1-\beta/d) d_{j,j'} - (1-\alpha) d_{j,j'}\\
&\geq \alpha d_{j,j'}/2\geq \alpha d/2.
\end{align*}
\end{proof}
\noindent
Part~(ii) of the next claim ensures that the super-regularity assumption (A2) in Step 1 is justified.
\begin{claim} \label{Pt}
Let $t\in [T]$. Assuming no failure prior to Round~$t$, the following hold.
\begin{itemize}
\item[{\rm (i)}] For any vertex $v\in V(P)$, $d_{P - P^{t}}(v) \leq \delta^{3/5}n$.
\item[{\rm (ii)}] $P^{t}$ is $(\delta^{1/4} , \vec{\beta})$-super-regular with respect to $(R,\mathcal{V})$
and $(\delta^{1/4},\vec{\beta'})$-super-regular with respect to $(R_K,\mathcal{V}^i)$. 
\item[{\rm (iii)}] For all $i\in I_t$ and all $jj'\in E(R_K)$, let
$$SP^t_{j,j'}(i):=\left\{\{v,v'\} \in \binom{V_j^i}{2} : |N_{P^t}(\{v,v'\})\cap V_{j'}^i| \neq ( \beta_{j,j'}'^2 \pm 3\delta^{3/5})m'\right\}.$$
Then $|SP^t_{j,j'}(i)|\leq \gamma m'^2$.
\end{itemize}
\end{claim}
\begin{proof}
Note that for each $v\in V(G)$ and all $i\in I_1\cup\dots\cup I_t$
$$|\{e\in E(H_i): v\in \phi'_i(e),\phi_i(e)\cap W_i\neq\emptyset\}|=|\{e\in E(H_i): v\in \phi_i(e),\phi_i(e)\cap W_i\neq\emptyset\}|\pm \Delta_R,
$$
where we need the $\pm \Delta_R$ only if $v\in W_i$.\COMMENT{Since then $v$ has degree between $\sum_{j'\in N_{R}(j) } k^i_{j,j'}$ and $\sum_{j'\in N_{R}(j)} (k^i_{j,j'}+1)$ and $|N_{R}(j)|\leq \Delta_R$.}
\noindent
Moreover, since the algorithm has not ended with failure before, (W2) in Step~3 was satisfied in all Rounds~$t'$ with $t'<t$.
Hence, for each $v\in V(G)$ and each $t'<t$
\begin{eqnarray}\label{eq:ApplW2}
d_{P^{t'}-P^{t'+1}}(v)&\stackrel{\eqref{eq:Pt}}{=}&\sum_{i\in I_{t'}}|\{ e\in E(H_i): v\in \phi'_i(e) , \phi_i(e)\cap W_i\neq \emptyset \}| \nonumber \\
&\leq& \sum_{i\in I_{t'}}|\{ e\in E(H_i): v\in \phi_i(e) , \phi_i(e)\cap W_i\neq \emptyset \}|+\Delta_R|\{i\in I_{t'}:v\in W_i\}|\nonumber \\
& \stackrel{\text{(W2)}}{\leq} & 5k\Delta_R\delta\gamma n.
\end{eqnarray} 
Thus 
\begin{align}\label{max deg P Pt}
d_{P - P^{t}}(v) \leq \sum_{t'=1}^{t-1} 5k\Delta_R\delta\gamma n \leq T\cdot 5k\Delta_R\delta\gamma n \stackrel{\eqref{eq:Tdk}}{\leq} \delta^{3/5}n.
\end{align}
The graph $P^1=P$, defined in Round $0$, is $(\epsilon^{1/3}, \vec{\beta})$-super-regular with respect to $(R,\mathcal{V})$ by construction,
and so also $(\delta^{3/5},\vec{\beta})$-super-regular since $\epsilon \ll \delta$. Thus
Proposition \ref{regularity after edge deletion} implies that $P^t$ is $(\delta^{1/4},\vec{\beta})$-super-regular
with respect to $(R,\mathcal{V})$. Similarly, (F4) and Proposition~\ref{regularity after edge deletion} together imply that $P^t$ is $(\delta^{1/4},\vec{\beta'})$-super-regular with respect to $(R_K,\mathcal{V}^i)$.

For all $i\in I_t$ and $jj'\in E(R_K)$, let 
\begin{align*}
SP_{j,j'}(i):=\left\{\{v,v'\} \in \binom{V_j^i}{2}: |N_{P}(\{v,v'\})\cap V_{j'}^i| \neq ( \beta_{j,j'}'^2\pm 3\epsilon_t^{1/3})m'\right\}.
 \end{align*}
Then \eqref{max deg P Pt} implies that $SP^t_{j,j'}(i)\subseteq SP_{j,j'}(i)$. Also (F4) and Proposition~\ref{regularity implies codegree}
together imply that $|SP_{j,j'}(i)|\leq \epsilon_t^{1/3} m'^2$. 
Thus $|SP^t_{j,j'}(i)|\leq \epsilon_t^{1/3} m'^2 \leq \gamma m'^2$. This completes the proof of Claim~\ref{Pt}.
%
\end{proof}

The following claim implies that the embedded copies of $H_1,\dots, H_{s}$ are indeed edge-disjoint and satisfy (T1) and (T4).

\begin{claim}\label{disjoint}
Assume that in Round $t$ we have constructed the embeddings $\phi'_i$ for all $i\in I_t$. 
Then the following statements hold.
\begin{itemize}
\item[(i)] $\bigcup_{i\in I_t} \phi'_i(E(H_i)) \subseteq E(G^t) \cup E(P^t)$, 
\item[(ii)]  $\phi'_i(E(H_i))\cap \phi'_{i'}(E(H_{i'}))=\emptyset$ \text{ for all distinct } $i,i'\in I_t$,
\item[(iii)] $\bigcup_{i\in I_t} \phi'_i(E(H_i)) \cap ( E(G^{t+1})\cup E(P^{t+1}) ) =\emptyset$,
\item[(iv)] $\phi'_i(x) \in N_{A_i}(x)$,
\item[(v)]  if $(i,x)(i',y) \in E(\Lambda)$ and $i,i' \in \bigcup_{t'=1}^{t} I_{t'}$,  then $\phi'_i(x) \neq \phi'_{i'}(y)$.
\end{itemize}
\end{claim}
\begin{proof}
By Lemma \ref{patching}, if $e=uv \in \phi'_i(E(H_i))\setminus E(G^t)$, then $uv \in  E(P^t_i) \subseteq E(P^t)$.  Thus (i) holds. Statements (ii) and (iii) are immediate consequences of \eqref{eq:updateGt}, \eqref{eq:defPti} and \eqref{eq:Pt}. (iv) follows from \eqref{eq: A'' def} and \eqref{eq:hi}. 
To show (v), we assume that $i'< i$. (Note that (v) is trivial if $i=i'$.)   Since $(i,x) \in V(\Lambda)$,
\eqref{eq: X'j def} and \eqref{eq:WiZidef} together with the definition of $U_i$ imply that $x\in Z_i$. Consider the following cases.

\noindent \emph{Case 1.} $i'\in \bigcup_{t'=1}^{t-1} I_{t'}$.
Then $\phi'_{i'}(y) \in D'_i(x)$. Thus \eqref{eq: disjointness} and 
\eqref{eq:hi} imply that $\phi'_{i'}(y)\neq  \phi'_i(x)$.

\noindent \emph{Case 2.} $i'\in I_t$. Then $\phi'_i(y) \in D''_i(x)$. So \eqref{eq: disjointness 2}  and  \eqref{eq:new} imply that $\phi'_{i'}(y)\neq  \phi'_i(x)$.
\end{proof}

\begin{claim}\label{cl:F't}
Let $t\in [T]$. Assume that the main packing algorithm did not fail prior to Round~$t$ and that within Round~$t$ failure of type~1 did
not occur. Then for all $i\in I_t$, $j\in [Kr]$ and any vertex $z\in V^i_j\cup X^i_j$,
$$d_{F'_i}(z) = (\alpha d_0 p(R_K,\vec{\beta}',j) \pm \delta^{1/2}) m'.$$ Moreover, let
$$S_j(i):=\left\{\{x,x'\} \in \binom{X^i_j}{2} : |N_{F'_i}(\{x,x'\})| \neq ((\alpha d_0 p(R_K,\vec{\beta}',j) )^2 \pm \delta^{1/2})m'\right\}.$$
Then $|S_j(i)|\leq \gamma m'^2$.
\end{claim}
\begin{proof}
By (\ref{F' definition}), for all $x\in X^i_j$ and $v\in V^{i}_j$, if $xv \in E(F_i - F'_i)$ then there exists $y\in N_x^{i}$ such that $\phi_i(y)v \in E(P - P^t)$. Since $d_{P-P^t}(\phi_i(y)) \leq \delta^{3/5} n$ by Claim~\ref{Pt}(i), \begin{align}\label{F' x}
d_{F_i - F'_i}(x) \leq |N^i_x| \delta^{3/5}n \stackrel{\text{(F3)}}{\leq} K\Delta_R \delta^{3/5} n.
\end{align}
Similarly, for all $x \in X^i_j$ and $v \in V^{i}_j$, if $xv \in E(F_i - F'_i)$ then 
\begin{align}\label{eq:xvunions}
v\in \bigcup_{y\in N_x^i}N_{P-P^t}(\phi_i(y))\Leftrightarrow v\in \bigcup_{y\colon x\in N_y^i}N_{P-P^t}(\phi_i(y)) \Leftrightarrow x \in \bigcup_{u \in N_{P-P^t}(v)} N^i_{\phi_i^{-1}(u)}.
\end{align} 
Since $d_{P-P^t}(v) \leq \delta^{3/5} n$, we have $d_{F_i - F'_i}(v) \leq K\Delta_R \delta^{3/5}n$. Together with (\ref{F' x}), this implies that
\begin{align}\label{eq:max deg F'F}
\Delta(F_i-F'_i)\leq K\Delta_R \delta^{3/5} n.
\end{align}
By \eqref{eq:epsilont}, $f(\epsilon_t) < \gamma < \delta^{1/2}/2$. Thus for each vertex $z\in V^i_j\cup X^i_j$, 
$$d_{F'_i}(z) = d_{F_i}(z) \pm K\Delta_R \delta^{3/5} n \stackrel{\text{(F5)}}{=} (\alpha d_0 p(R_K,\vec{\beta'},j) \pm f(\epsilon_t) )m' \pm K\Delta_R\delta^{3/5} n = (\alpha d_0 p(R_K,\vec{\beta'},j) \pm \delta^{1/2}) m'.$$
Let 
 \begin{align*}
S_{i,j}:=\left\{\{x,x'\} \in \binom{X^i_j}{2}: |N_{F_i}(\{x,x'\})| \neq ( (\alpha d_0p(R_K,\vec{\beta'},j))^2\pm 3 f(\epsilon_t))m'\right\}.
 \end{align*}
Then \eqref{eq:max deg F'F} implies that $S_j(i)\subseteq S_{i,j}$. Also (F5) together with Proposition \ref{regularity implies codegree} implies that $|S_{i,j}|\leq f(\epsilon_t)m'^2$. 
Thus  $|S_j(i)|\leq f(\epsilon_t) m'^2 \leq \gamma m'^2$.
\end{proof}

We will now show that each type of failure occurs with very small probability.  

\subsection{Failure of type 1}

\begin{claim} \label{type1}
Failure of type 1 occurs with probability at most $(1-c)^n$.
\end{claim}
\begin{proof}
Suppose $t\in [T]$ and we are in Round~$t$. Then we can assume that there was no failure in a previous round. So by the definition
of failure of type 3 (if $t>1$) or as observed in Round~0 (if $t=1$), 
the graph $G^t$ is $(\epsilon_{t},\vec{d}^t)$-super-regular with respect to $(R,V_1,\dots,V_r)$. Moreover, as stated in Round~$0$, $P$ is $(\epsilon_1,\vec{\beta})$-super-regular with respect to $(R,V_1,\dots,V_r)$. Since by Claim \ref{number of rounds} and \eqref{eq:epsilont} we have $\epsilon_{t}\leq \epsilon_T \ll \gamma \ll \alpha d/2\leq d^t_{j,j'}$ for all $jj'\in E(R)$, Lemma \ref{modified blow-up} can be applied (with $2c$ playing the role of $c$), and for each $i\in I_t$ the Uniform embedding algorithm fails with probability at most $(1-2{c})^n$.

Since the number of times we apply the Uniform embedding algorithm over all rounds is at most $s \leq \eta^{-1} n $, the probability that failure of type 1 ever occurs is at most $\eta^{-1} n (1-2{c})^n \leq (1-{c})^n$. 
\end{proof}

\subsection{Failure of type 2}

We define the following event.
\begin{itemize}
\item[$(\text{QW}'_i)$] For all $j\in [r]$ and any sets $Q\in \mathcal{Q}_{j}$ and $W\in \mathcal{W}_{j}$, we have $|\phi'_i(Q) \cap W| = (1\pm \eta')\frac{|Q||W|}{n}$.
\end{itemize}

\begin{claim}\label{type2}
Failure of type 2 occurs with probability at most $(1-c)^n$. Moreover, $(\textup{QW}_i)$ implies $(\textup{QW}'_i)$.
\end{claim}
\begin{proof}
For given $Q\in \mathcal{Q}_j, W\in \mathcal{W}_j$, $i\in [s]$, let $\mathcal{E}_i(Q,W)$ be the event that $|\phi_i(Q)\cap  W| = (1\pm \gamma) \frac{|Q||W|}{n}$. Then we have
$$(\text{QW}_{i}) =\bigwedge_{j\in [r], Q\in \mathcal{Q}_j,W\in \mathcal{W}_j} \mathcal{E}_i(Q,W).$$

Note that (B6) of Lemma \ref{modified blow-up} together with the fact that $f(\epsilon_t) \ll \gamma$ imply that for given $i\in I_t$, $Q\in \mathcal{Q}_{j}, W\in \mathcal{W}_{j}$ we have
$$\mathbb{P}[ \mathcal{E}_i(Q,W)] \geq 1-(1-2c)^n.$$
A union bound and (S7) imply that
\begin{align}\label{QW11}
\mathbb{P}[(\text{QW}_i)]  \geq 1 - \sum_{j\in[r]} |\mathcal{Q}_j||\mathcal{W}_j|(1-2c)^n \geq 1- r (1+c)^n(1-2c)^n.
\end{align}
Since $1\leq i\leq s$, there are $s\leq \eta^{-1} n$ possible values for which the failure of type 2 can occur. Therefore, failure of type 2 ever occurs with probability at most $ rs(1+c)^n(1-2c)^n < (1-c)^n$.

Now we show the `moreover' part of Claim~\ref{type2}. Let us assume that $(\text{QW}_i)$ holds. 
Note that for $j\in [r]$, 
\begin{align}\label{delta different}
|\{ x\in X_j : \phi'_i(x) \neq \phi_i(x)\}| \leq |W_i|\stackrel{\eqref{eq:WiZidef}}{\leq} \delta n.
\end{align}
So for any sets $Q\in \mathcal{Q}_j, W\in \mathcal{W}_j$ we have
$$
|\phi'_i(Q)\cap W| \stackrel{\eqref{delta different}}{=} |\phi_i(Q)\cap W| \pm \delta n \stackrel{(\text{QW}_i)}{=} (1\pm \gamma )\frac{|Q||W|}{n} \pm \delta n 
=(1\pm \eta' )\frac{|Q||W|}{n}.
$$
We obtain the final identity since $|Q|,|W| \geq \eta' n$ and $\gamma\ll\delta \ll \eta'$. Thus $(\text{QW}_i)$ implies $(\text{QW}'_i)$.
\end{proof}

\subsection{Failure of type 3}

In this subsection we show that failure of type 3 occurs with small probability. We will achieve this by checking that the degrees and codegrees of $G^t$ decrease by the expected amount when constructing $G^{t+1}$. Then we invoke Theorem~\ref{codegree implies regularity}. 

\begin{claim} \label{succession}
Assume that we are in Step 1 of Round $t \in [T]$ and condition on there being no failure of type 1 in this step
(in particular,  $G^t$ is $(\epsilon_t,\vec{d}^t)$-super-regular with respect to $(R,V_1,\dots,V_r)$). 
Then, with probability at least $1-(1-{2c})^n$, $G^{t+1}$ is $(\epsilon_{t+1}, \vec{d}^{t+1})$-super-regular with respect to $(R,V_1,\dots,V_r)$. 
\end{claim}
\begin{proof}
By \eqref{eq:epsrecursion}, this means that we need to show that with high probability $G^{t+1}[V_j,V_{j'}]$ is $(q(\epsilon_t), d^{t+1}_{j,j'})$-super-regular for all $jj' \in E(R)$.

Fix an outcome of the algorithm running for Rounds $1,\dots,t-1$. (For the proof of this claim, all probabilities and expectations are conditioned on this outcome.) Recall that we also condition on there being no failure of type 1 in Step 1.
Let $S$ be a set of $s'$ vertices in $V_{j'}$ with $s'\leq 2$, and let $N_j(S):= N_{G^t}(S)\cap V_j$. Also we assume that $S$ satisfies
\begin{align}\label{eq:NiS}
|N_j(S)|=((d^t_{j,j'})^{s'} \pm 3\epsilon_t)n.
\end{align}
We are applying the Uniform embedding algorithm to $G^t$ and $H_i$ for $i\in I_t$ to obtain embeddings $\phi_i$ for all $i\in I_t$.
Then $$N_{G^{t+1}}(S)\cap V_j = \{ u\in N_j(S) : v u \notin \phi_i(E(H_i)) \text{ for all } i \in I_t \text{ and all } v\in S \}.$$
For $i\in I_t\cup \{(t-1)\gamma n\}$, let $Q_i$ be the random variable defined by 
$$Q_i := \mathbb{E}\left[|N_{G^{t+1}}(S)\cap V_j| \mid \phi_{(t-1)\gamma n+1},\dots, \phi_{i}\right].$$
Note that this is an exposure martingale. Also, if we change one $\phi_{i'}$, then the value of 
$|N_{G^{t+1}}(S)\cap V_j|$ changes by at most $s'(k+1)$, thus 
\begin{align}\label{sk lipsitz}
|Q_{i+1} - Q_i| \leq s'(k+1).
\end{align} 
In other words, the martingale $Q_i$ is $s'(k+1)$-Lipschitz. Our aim now is to compute $\mathbb{E}[Q_{t\gamma n}]=\E[|N_{G^{t+1}}(S)\cap V_j|]$ and then to apply Azuma's inequality. For each $u\in N_j(S)$ and $i\in I_t$, let $B^i_u$ be the random variable such that 
$$
B^i_u:= \left\{\begin{array}{ll}
1 & \text{ if } v u \in \phi_i(E(H_i)) \text{ for some } v\in S, \\
0 & \text{ otherwise}. 
\end{array}\right.
$$
It is easy to see that $$Q_{t\gamma n} = |N_{G^{t+1}}(S)\cap V_j| = \sum_{u \in N_j(S)} \prod_{i\in I_t} (1-B^i_u).$$
Let $p_{i,u}:=\mathbb{P}[B^i_u=1] $. Then the fact that $\phi_i, \phi_{i'}$ are independent for $i\neq i'\in I_t$ implies that 
\begin{align}\label{eq:Eqgamman}
Q_{(t-1)\gamma n}=\E[Q_{t \gamma n}]=\sum_{u\in N_j(S)} \prod_{i\in I_t}(1-p_{i,u}).
\end{align}
We will show that $p_{i,u}$ is usually very close to $s'k^i_{j,j'}/(d^t_{j,j'}n)$. For this, let $A,B\subseteq N_j(S)\times I_t$ be defined as follows.
\begin{align}\label{AB def} 
A:= \left\{ (u,i): p_{i,u} > \frac{s'k^i_{j,j'}}{d^t_{j,j'} n}(1+ 2f(\epsilon_t))\right\} \text{ and }
B:= \left\{ (u,i): p_{i,u} < \frac{s'k^i_{j,j'}}{d^t_{j,j'} n}(1- 2f(\epsilon_t))\right\}.
\end{align}
Let $A_i^\diamond$ be the set of vertices $u\in N_j(S)$ such that $(u,i) \in A$, and let $B_i^\diamond$ be defined analogously. Let $A^*_u$ be the set of indices $i$ such that $(u,i) \in A$ and define $B^*_u$ analogously.
Then $|A_i^\diamond|,|B_i^\diamond| < 2f(\epsilon_t) n$\COMMENT{proof of (B5) shows $f(\epsilon_t)n$ but statement of (B5) is $|A_i^\diamond|+|B_i^\diamond|\leq 2f(\epsilon_t)n$} by (B5) of Lemma~\ref{modified blow-up}.
Therefore 
\begin{align}\label{A size}
|A| =\sum_{i\in I_t} |A_i^\diamond| < 2f(\epsilon_t)\gamma n^2 < f(\epsilon_t)n^2, 
\end{align}
and similarly 
\begin{align}\label{B size}
|B|< f(\epsilon_t)n^2.
\end{align}
Moreover, we have
\begin{align}\label{A ell sum}
\sum_{u\in A_i^\diamond} p_{i,u} \leq \frac{4 s' f(\epsilon_t)k^{i}_{j,j'}}{d^t_{j,j'}}.
\end{align} 
Indeed, if not, consider a set $A'\subseteq N_j(S)$ of size $2f(\epsilon_t) n$ with $A_i^\diamond\subseteq A'$ and obtain a contradiction via
\begin{align*} 
\frac{4 s' f(\epsilon_t)k^{i}_{j,j'}}{d^t_{j,j'}} &<\sum_{u \in A_i^\diamond} p_{i,u} \leq \sum_{u \in A'} p_{i,u} 
=\mathbb{E}\left[ |\{u \in A': vu\in \phi_i(E(H_i)) \text{ for some }v\in S\}|\right] \\
&\leq (1+ f(\epsilon_t))\frac{k^{i}_{j,j'}|A'|}{d^t_{j,j'}n}s' < \frac{3s'f(\epsilon_t)k^i_{j,j'}}{d^t_{j,j'}}.
\end{align*}
Note that the second line follows from (B1) of Lemma \ref{modified blow-up}.
We also obtain 
\begin{align}\label{B ell sum}
\sum_{u\in B_i^\diamond} p_{i,u} \leq \frac{s' k^i_{j,j'}}{d^t_{j,j'} n}|B_i^\diamond| \leq\frac{2 s'f(\epsilon_t)k^i_{j,j'}}{d^t_{j,j'} },
\end{align}
because $p_{i,u} < \frac{s' k^{i}_{j,j'}}{d^t_{j,j'} n}$ for $u\in B_i^\diamond$ and $|B_i^\diamond|<2f(\epsilon_t)n$. Thus 
\begin{align}\label{sum sum}
\begin{split}
&\sum_{u\in N_j(S)} \sum_{i \in A^*_u } p_{i,u}  = \sum_{i\in I_t} \sum_{u \in A_i^\diamond} p_{i,u}  \stackrel{\eqref{A ell sum}}{\leq} 4s' f(\epsilon_t)k^i_{j,j'}\gamma n/d^t_{j,j'}  < f(\epsilon_t) n, \\
& \sum_{u\in N_j(S)} \sum_{i \in B^*_u} p_{i,u}  = \sum_{i\in I_t} \sum_{u \in B_i^\diamond} p_{i,u} \stackrel{\eqref{B ell sum}}{\leq} 2s'f(\epsilon_t)k^i_{j,j'}\gamma n/d^t_{j,j'} < f(\epsilon_t) n.
\end{split}\end{align}
 
Let $D$ be the set of vertices $u\in N_j(S)$ satisfying at least one of $\sum_{i\in A^*_u} p_{i,u} > f(\epsilon_t)^{1/2}$ and $\sum_{i\in B^*_u} p_{i,u} > f(\epsilon_t)^{1/2}$. By (\ref{sum sum}), we have
\begin{align}\label{eq:Cineq}
|D| \leq 2f(\epsilon_t)^{1/2} n.
\end{align}
Then $u \in N_j(S)\setminus D$ implies that 
\begin{align}\label{Aw Bw sum}
\sum_{i \in A^*_u } p_{i,u} \leq f(\epsilon_t)^{1/2} \;\; \text{ and } \;\;
\sum_{i \in B^*_u } p_{i,u} \leq  f(\epsilon_t)^{1/2} .
\end{align}
Similarly, let $D'$ be the set of vertices $u\in N_j(S)$ such that $|A^*_u| > f(\epsilon_t)^{1/2} n$ or $|B^*_u| > f(\epsilon_t)^{1/2} n$. Then, by (\ref{A size}) and (\ref{B size}), 
\begin{align}\label{eq:Dineq}
|D'|\leq 2 f(\epsilon_t)^{1/2} n.
\end{align}
Now, consider any vertex $u\in N_j(S)\setminus (D \cup D')$. 
Since $u \notin D'$, 
\begin{align}\label{eq:A*uB*u}
|A^*_u\cup B^*_u| \leq 2f(\epsilon_t)^{1/2} n.
\end{align}
Moreover, for every $u\in N_j(S)\setminus(D\cup D')$, we have (note that $A^*_u,B^*_u$ are disjoint)
\begin{align}\label{main calculation} 
\begin{split}
\prod_{i\in I_t} (1- p_{i,u})
&=  \prod_{i \in I_t\setminus(A^*_u\cup B^*_u) } (1- p_{i,u})\prod_{i \in A^*_u} (1- p_{i,u})\prod_{i \in B^*_u\setminus A^*_u} (1- p_{i,u}) \\
& = \prod_{i\in I_t \setminus (A^*_u \cup B^*_u)}\left(1 - (1\pm 2f(\epsilon_t))\frac{s' k^{i}_{j,j'}}{d^t_{j,j'}n}\right)\prod_{i \in A^*_u} (1- p_{i,u})\prod_{i \in B^*_u} (1- p_{i,u}).
\end{split}
\end{align}
Estimating the above products by sums yields
\begin{align}\label{calculation 1} 
\begin{split}\prod_{i \in A^*_u} \left(1- p_{i,u}\right)\prod_{i \in B^*_u} (1- p_{i,u}) 
&= \left(1 \pm \sum_{i \in A^*_u} p_{i,u}\right)\left(1 \pm \sum_{i\in B^*_u} p_{i,u}\right)
\stackrel{(\ref{Aw Bw sum})}{=} (1\pm  f(\epsilon_t)^{1/2} )^2\\ &= 1\pm  f(\epsilon_t)^{1/3}.
\end{split} 
\end{align}
Recall that $d^t_{j,j'}\geq \alpha d/2$ by Claim~\ref{number of rounds}. Combining \eqref{main calculation} and \eqref{calculation 1}, we obtain
\begin{align*}
\prod_{i\in I_t} (1- p_{i,u})&=  (1\pm  f(\epsilon_t)^{1/3})\prod_{i\in I_t \setminus (A^*_u \cup B^*_u)}\left(1 - (1\pm 2f(\epsilon_t))\frac{s' k^{i}_{j,j'}}{d^t_{j,j'}n}\right) \\
&= (1\pm f(\epsilon_t)^{1/4}) \left(1 \pm 3f(\epsilon_t) \frac{2s'k}{\alpha dn}\right)^{\gamma n}
\left(\prod_{i\in A^*_u\cup B^*_u} \left(1-\frac{s'k^{i}_{j,j'}}{d^t_{j,j'}n}\right) \right)^{-1}
\prod_{i\in I_t}\left(1 - \frac{s' k^{i}_{j,j'}}{d^t_{j,j'}n}\right)\\
&=(1\pm f(\epsilon_t)^{1/4}) \left(1 \pm 3f(\epsilon_t) \frac{2s'k}{\alpha dn}\right)^{\gamma n}
\left(1\pm \frac{3s'k}{\alpha dn}\right)^{|A^*_u\cup B^*_u|}
\prod_{i\in I_t}\left(1 - \frac{s' k^{i}_{j,j'}}{d^t_{j,j'}n}\right) \\
&\stackrel{\eqref{eq:A*uB*u}}{=}(1\pm 2 f(\epsilon_t)^{1/4}) \left(\prod_{i\in I_t}\left(1 - \frac{ k^{i}_{j,j'}}{d^t_{j,j'}n}\right)\right)^{s'}.
\end{align*}\COMMENT{
In the 2nd and 3rd equality, we used $d^t_{j,j'} \geq \alpha d/2$ and
$ (1 \pm 3f(\epsilon_t) \frac{2s'k}{\alpha dn})^{\gamma n} = 1\pm 8 f(\epsilon_t) \gamma s'k/(\alpha d) = (1\pm f(\epsilon_t)^{1/2})$
and $(1\pm \frac{3s'k}{\alpha dn})^{-|A^*_u\cup B^*_u|} = (1\pm \frac{3s'k|A^*_u\cup B^*_u|}{\alpha dn}) = (1\pm f(\epsilon_t)^{1/3})$.
}
\COMMENT{ 
\begin{align*}
\prod_{i\in I_t}\left(1 - \frac{ s'k^{i}_{j,j'}}{d^t_{j,j'}n}\right) =
\prod_{i\in I_t} \left(1 \pm  \frac{2 s'^2(k^{i}_{j,j'} )^2}{(d^t_{j,j'}n)^2} \right)
\left(\prod_{i\in I_t}\left(1 - \frac{ k^{i}_{j,j'}}{d^t_{j,j'}n}\right)\right)^{s'} = \left(1 \pm  \sum_{i\in I_t} \frac{3 {s'}^2(k^{i}_{j,j'} )^2}{(d^t_{j,j'}n)^2 } \right) \left(\prod_{i\in I_t}\left(1 - \frac{ k^{i}_{j,j'}}{d^t_{j,j'}n}\right)\right)^{s'} \\
= (1\pm \epsilon_t)\left(\prod_{i\in I_t}\left(1 - \frac{ k^{i}_{j,j'}}{d^t_{j,j'}n}\right)\right)^{s'}
\end{align*}}
Taking the sets $D$ and $D'$ into account, we get 
\begin{eqnarray*}
\E[Q_{t\gamma n}]
&\stackrel{\eqref{eq:Eqgamman}}{=}& (|N_j(S)|-|D|-|D'|)(1\pm 2f(\epsilon_t)^{1/4})\left(\prod_{i\in I_t}\left(1 - \frac{ k^{i}_{j,j'}}{d^t_{j,j'}n}\right)\right)^{s'} \pm (|D|+|D'|)
\\&\stackrel{\eqref{eq:Cineq},\eqref{eq:Dineq}}{=}& (1\pm 2f(\epsilon_t)^{1/4})\left(\prod_{i\in I_t}\left(1 - \frac{ k^{i}_{j,j'}}{d^t_{j,j'}n}\right)\right)^{s'}|N_j(S)| \pm 10f(\epsilon_t)^{1/2} n\\
&\stackrel{\eqref{eq:NiS}}{=}& (1\pm f(\epsilon_t)^{1/5})\left(\prod_{i\in I_t}\left(1 - \frac{ k^{i}_{j,j'}}{d^t_{j,j'}n}\right)\right)^{s'}|N_j(S)|.
\end{eqnarray*}
So, applying Azuma's inequality (Theorem \ref{Azuma}) to $Q_i$ with $i\in I_t\cup \{(t-1)\gamma n\}$, which is $s'(k+1)$-Lipschitz by (\ref{sk lipsitz}), we get
$$\mathbb{P}\left[Q_{t\gamma n} \neq \left(1\pm 2f(\epsilon_t)^{1/5}\right)\left(\prod_{i\in I_t}\left(1 - \frac{ k^{i}_{j,j'}}{d^t_{j,j'}n}\right)\right)^{s'}|N_j(S)|\right] < (1-3{c})^n$$ since ${c} \ll \epsilon$.
This holds for all sets $S\subseteq V_{j'}$ of $s'\leq 2$ vertices with $|N_j(S)|= ((d^t_{j,j'})^{s'}\pm 3\epsilon_t) n$. 

Let $D_{j,j'}:= \{\{u,v\} \in \binom{V_{j'}}{2}: |N_{G^t}(\{u,v\})\cap V_j| = ((d^t_{j,j'})^2 \pm 3\epsilon_t)n\}$. Since $G^t[V_j,V_{j'}]$ is $(\epsilon_t,d^t_{j,j'})$-super-regular, Proposition~\ref{regularity implies codegree} implies that%
  \COMMENT{here and when applying Theorem~\ref{codegree implies regularity} below we are pretending that $|V_j|=|V_{j'}|=n$, but there is always room
in the calculations. So perhaps best to gloss over this...}
$$|D_{j,j'}| \geq \binom{n}{2} - \epsilon_t n^2 \geq \frac{1}{2}(1- 2\epsilon_t) n^2\stackrel{\eqref{eq:functions}}{\geq} \left(\frac{1}{2} - 5q(\epsilon_t)^6\right) n^2.$$
Then with probability at least $1 - r^2( n + |D_{j,j'}|)(1-3{c})^n \geq 1- (1-2{c})^n$, for all $j,j' \in [r]$ and any $S\in D_{j,j'}$ we have
\begin{eqnarray*}
|N_{G^{t+1}}(S)\cap V_j| &=& (1\pm 2f(\epsilon_t)^{1/5}) \left(\prod_{i\in I_t}\left(1 - \frac{ k^{i}_{j,j'}}{d^t_{j,j'}n}\right)\right)^2((d^t_{j,j'})^2 \pm 3\epsilon_t) n \\
&\stackrel{\eqref{eq:functions}}{=}& (1\pm q(\epsilon_t)^{6})\left(d^t_{j,j'}\prod_{i\in I_t}\left(1 - \frac{ k^{i}_{j,j'}}{d^t_{j,j'}n}\right)\right)^2 n\\
&\stackrel{\eqref{eq:mainalgparams}}{=}&
(1\pm \epsilon_{t+1}^6)(d^{t+1}_{j,j'})^2n, 
\end{eqnarray*}
and for any $v\in V_{j'}$
$$|N_{G^{t+1}}(v) \cap V_j| = (1\pm q(\epsilon_t)^6)\left(d^t_{j,j'} \prod_{i\in I_t}\left(1 - \frac{ k^{i}_{j,j'}}{d^t_{j,j'}n}\right)\right) n\stackrel{\eqref{eq:mainalgparams}}{=}(1\pm \epsilon_{t+1}^6)d^{t+1}_{j,j'} n.$$ 
This together with Theorem~\ref{codegree implies regularity} implies that
$G^{t+1}$ is $(\epsilon_{t+1}, \vec{d}^{t+1})$-super-regular 
with probability at least $1- (1-2{c})^n$.
\end{proof}

\begin{claim}\label{cl:mainalgfailure2}
Failure of type 3 occurs with probability at most $(1-{c})^n$.
\end{claim}
\begin{proof}
By Claim \ref{succession}, for every $t\in [T]$, conditioned on no previous failure, $G^t$ is $(\epsilon_t,\vec{d^t})$-super-regular with probability at least $1- (1-2{c})^n$. 
Thus failure of type 3 ever occurs with probability at most $T(1-2{c})^n< (1-{c})^n$. 
\end{proof}

\subsection{Failure of type 4}

\begin{claim}\label{failure of type 4}
Failure of type $4$ occurs with probability at most $(1-{c})^n$. 
\end{claim}
\begin{proof}
Suppose we are in Round $t$, Step 2 and that $\phi_i$ for $i\in I_t$ are the embeddings we define in Round $t$, Step 1. We prove the following two subclaims. 

\vspace*{0.2cm}
\noindent {\bf Subclaim 1.} \textit{In Round $t$ (U1) fails  with probability at most } $3 r n (1-2{c})^n.$

To prove Subclaim 1, fix $v\in V(G)$, $i\in I_t$ and $j\in [r]$ such that $v\in V_j$. 
Let $G^{t,<}(i)$ be the spanning subgraph of $G^t$ with edge set  
$$E(G^{t,<}(i)) := \bigcup_{i'\in I_t, i'<i} \phi_{i'}(E(H_{i'})),$$ and, similarly, let $G^{t,>}(i)$ be the spanning subgraph of $G^t$ be defined by
$$E(G^{t,>}(i)) := \bigcup_{i'\in I_t, i'>i} \phi_{i'}(E(H_{i'})).$$
Let
$$
I_{v,i}:= \left\{\begin{array}{ll}
1 & \text{ if } v\in (\phi_{i})_2(H_i,G^t,G^{t,<}(i)),\\
0 & \text{ otherwise}, 
\end{array}\right.
$$
where $(\phi_{i})_2(H_i,G^t,G^{t,<}(i))$ is as defined at the beginning of Section \ref{sec:RR}.
Similarly, let
$$
J_{v,i}:= \left\{\begin{array}{ll}
1 & \text{ if } v\in (\phi_{i})_2(H_i,G^t,G^{t,>}(i)),\\
0 & \text{ otherwise}. 
\end{array}\right.
$$
Let%
\COMMENT{12/7: changed def of $L_{v,i}$, which leads to changes in (\ref{eq: size NHX'})}
$$L_{v,i}:=\left\{\begin{array}{ll}
1 & \text{ if } v\in \phi_{i}\left(X'_j(i)\cup \bigcup_{j'\in N_R(j)} (N_{H_i}(X'_{j'}(i))\cap V_j)\right),\\
0 & \text{ otherwise}. 
\end{array}\right.$$
Note that $L_{v,i}$ and $L_{v,i'}$ are independent for $i\neq i' \in I_t$.
Since $\Delta(H_{i'}) \leq (k+1)\Delta_R$ for all $i\in I_t$, we have
\begin{align}\label{eq:G<>}
\Delta(G^{t,<}(i)),\Delta(G^{t,>}(i)) \leq 2k\Delta_R\gamma n
\end{align}
and
\begin{align}\label{eq: size NHX'}
\left|X'_j(i)\cup \bigcup_{j'\in N_R(j)} (N_{H_i}(X'_{j'}(i))\cap V_j)\right| \leq
|X'_j(i)|+\sum_{j'\in N_R(j)} (k+1)|X'_{j'}(i)|\le \epsilon^{1/2} n.
\end{align}
\noindent
Recall that in Round $t$, Step 1, we apply Lemma \ref{modified blow-up} with $4k\Delta_R\gamma$ playing the role of $\gamma$ for each $i\in I_t$ independently to obtain embeddings $\phi_i$. Thus, by (B4.1) of Lemma \ref{modified blow-up}, for any $x_{(t-1)\gamma n+1}, \dots, x_{i-1} \in \{0,1\}$ we have 
$$\mathbb{P}[I_{v,i}=1 \mid I_{v,(t-1)\gamma n+1}=x_{(t-1)\gamma n+1}, \dots, I_{v,i-1}=x_{i-1} ] \leq (4k\Delta_R \gamma )^{1/2}\leq \gamma^{1/3}/4.$$\COMMENT{Jaehoon : If we merge $I$ and $J$ together to indicate $L_{v,i} = 1$ if $v\in (\phi_{i})_2(H_i,G^t,G^{t,\geq}(i)) \cup (\phi_{i})_2(H_i,G^t,G^{t,\leq}(i))$, then $L_{v,1}=x_1,\dots, L_{v,i-1}=x_{i-1}$ is conditioning on both the past and the future.(Since $\phi_{i},\phi_{i+1},\dots $ also affects the value of $L_{v,1},L_{v,2},\dots, L_{v,i-1}$. 
 Thus we can't directly say it's bounded by $(2kr\gamma)^{1/3}$ if we use $L$. That's why we divide $I$ and $J$ here.  }
Moreover, by (B3.2) of Lemma \ref{modified blow-up}, (\ref{eq: size NHX'}) and the fact that $\varepsilon\ll \gamma$ we have%
\COMMENT{12/7: had $(4k\Delta_R\gamma)^3 \leq \gamma^2$ before}
$$\mathbb{P}[L_{v,i} = 1 ] \leq (4k\Delta_R\gamma)^2 \leq \gamma.$$
Let $X, Y$ have binomial distribution with parameters $(\gamma n, \gamma^{1/3}/4)$ and $(\gamma n, \gamma)$, respectively, so $\E[X]=\gamma^{4/3}n/4$ and $\E[Y] = \gamma^2 n$.
Then, by Proposition~\ref{generalised-chernoff} and Lemma~\ref{Chernoff Bounds}, 
\begin{align}\label{II}
\mathbb{P}\left[ \sum_{i\in I_t} I_{v,i} \geq  \gamma^{4/3} n/3 \right] \leq \mathbb{P}\left[ X \geq  \frac{4}{3}\E[X] \right] \leq (1-2{c})^n.
\end{align}
A symmetric argument for $J_{v,i}$ (considering the reverse-ordering of $[\gamma n]$) yields 
\begin{align}\label{JJ}
\mathbb{P}\left[ \sum_{i\in I_t} J_{v,i} \geq \gamma^{4/3}n/3 \right] \leq \mathbb{P}\left[ X \geq  \frac{4}{3}\E[X] \right] \leq (1-2{c})^n.
\end{align}
Also, since $L_{v,i}$ and $L_{v,i'}$ are independent, by Lemma~\ref{Chernoff Bounds},%
\COMMENT{12/7: had ``by Proposition~\ref{generalised-chernoff} and Lemma~\ref{Chernoff Bounds}'' before}
\begin{align}\label{LL}
\mathbb{P}\left[ \sum_{i\in I_t} L_{v,i} \geq  \gamma^{4/3}n/3 \right] \leq \mathbb{P}\left[ Y \geq  2\E[Y] \right] \leq (1-2{c})^n.
\end{align}
Since $v\in NU_i$ implies that $I_{v,i}=1$ or $J_{v,i}=1$ or $L_{v,i}=1$, we obtain 
$$|\{i\in I_t :v\in NU_i\}| \leq  \sum_{i\in I_t} (I_{v,i}+ J_{v,i}+L_{v,i}).$$ Thus, by \eqref{II}, \eqref{JJ} and \eqref{LL}, for any given $v\in V(G)$,
\begin{align*}
\mathbb{P}\left[|\{i\in I_t: v\in NU_i\}|> \gamma^{4/3} n\right]\leq \mathbb{P}\left[ \sum_{i\in I_t} I_{v,i}+\sum_{i\in I_t}J_{v,i} + \sum_{i\in I_t} L_{v,i}\geq \gamma^{4/3} n\right] \leq 3(1-2{c})^n.
\end{align*}
Thus, with probability at least $1-3rn(1-2{c})^n$, (U1) holds in Round $t$ for all vertices. This completes the proof of Subclaim 1.
\vspace{0.2cm}

\noindent {\bf Subclaim 2.} \textit{ (U2) fails with probability at most $\gamma n(1-2c)^n$.}

To prove Subclaim 2, for $i\in I_t$ let $G^{t,<}(i)$ and $G^{t,>}(i)$ be the graphs defined in Subclaim 1. Let $\mathcal{E}_i$ be the event that 
$$|\{u \in V^i_j :  u\in (\phi_i)_2(H_i,G^t,G^{t,<}(i)\cup G^{t,>}(i)\}| \leq \gamma^{2/5}m'/2 \text{ for all } j\in[Kr].$$ 
Since $\Delta(G^{t,<}(i) \cup G^{t,>}(i)) \leq 4k\Delta_R \gamma n$ by \eqref{eq:G<>} and the random variable $\phi_i$ is independent from the random variable $G^{t,<}(i)\cup G^{t,>}(i)$, (B4.2) of Lemma \ref{modified blow-up} together with a union bound imply that 
$$\mathbb{P}\left[\bigwedge_{i\in I_t} \mathcal{E}_i\right] \geq 1-\gamma n(1-2c)^n.$$
(Here we use that $(4k\Delta_R\gamma )^{3/5}n \leq \gamma^{2/5} m'/2$.)
Let%
\COMMENT{12/7: new sentence}
$j'\in [r]$ be such that $V_j^i\subseteq V_{j'}$.
If $\mathcal{E}_i$ occurs for all $i\in I_t$, then%
\COMMENT{12/7: equation below changed}
$$|NU_i\cap V^i_j| \leq \frac{\gamma^{2/5} m'}{2} + \left|X'_{j'}(i)\cup \bigcup_{j''\in N_R(j')} (N_{H_i}(X'_{j''}(i))\cap V_{j'})\right| \stackrel{(\ref{eq: size NHX'})}{\leq} \frac{\gamma^{2/5}m'}{2} + \epsilon^{1/2} n \leq \gamma^{2/5}m'$$ holds for all $i\in I_t$ and $j\in [Kr]$. Thus (U2) fails with probability at most $\gamma n(1-2c)^n$.
 This completes the proof of Subclaim 2.
\vspace*{0.2cm}

To summarise, in Round $t$, (U1) fails with probability at most $3rn(1-2{c})^n$, and (U2) fails with probability at most $\gamma n(1-2c)^n$. Thus the probability that either one of them ever fails is at most $T(3rn (1-2{c})^n +  \gamma n(1-2c)^n) <(1-{c})^n$. 
\end{proof}

\subsection{Failure of type 5}

\begin{claim}
Failure of type 5 occurs with probability at most $(1-{c})^n$. 
\end{claim}
\begin{proof}
Suppose that the algorithm has reached Round $t$, Step 3, 
and no previous failure has occurred. In particular,
(U1),(U2) of Step 2 hold.

First, we show that (W1) holds with high probability. Let $i\in I_t$ and $j\neq j' \in [Kr]$ with $jj'\in E(R_K)$ be fixed. Since by (U2)
we have 
\begin{align}\label{eq:UiNUi}
|U_i\cap V_j^i| \leq |NU_i \cap V^i_j| \leq \gamma^{2/5} m',
\end{align}
it follows that 
\begin{align}\label{eq:Yi}
|Y_i\cap V^i_j| = \delta m' - | U_i\cap V^i_j| \stackrel{\eqref{eq:UiNUi}}{\geq} \delta m' - \gamma^{2/5} m'.
\end{align}
Thus, \eqref{eq:UiNUi} and \eqref{eq:Yi} together imply that for any given vertex $v\in V^i_j\setminus U_i$ we have
\begin{align}\label{Yi prob}
\mathbb{P}[ v\in W_i ] = \Prob[v\in Y_i] = \frac{|Y_i\cap V^i_j|}{|V^i_j\setminus U_i|} = (1\pm \gamma^{1/3})\delta.
\end{align}
\noindent
Let
$$\overline{SP}^t_{j,j'}(i):=\left\{\{v,v'\} \in \binom{V_j^i}{2} : |N_{P^t}(\{v,v'\})\cap V_{j'}^i|=( \beta_{j,j'}'^2 \pm 3\delta^{3/5})m'\right\}.$$
Let $S\subseteq V_j^i$ be such that either $|S|=1$ or $S$ consists of a pair of vertices in $\overline{SP}^t_{j,j'}(i)$.
Then the definition of $\overline{SP}^t_{j,j'}(i)$ together with the fact that by Claim~\ref{Pt}(ii)
$P^t$ is $(\delta^{1/4},\vec{\beta'})$-super-regular with respect to $(R_K,\mathcal{V}^i)$ implies that
$|N_{P^t}(S)\cap V^i_{j'}| =  (\beta_{j,j'}'^{|S|} \pm 2\delta^{1/4}) m'$. 
It follows that 
$$|(N_{P^t}(S)\cap V^i_{j'})\setminus U_i | \stackrel{\eqref{eq:UiNUi}}{=} (\beta_{j,j'}'^{|S|} \pm 2\delta^{1/4}  \pm \gamma^{2/5}) m' .$$ 
Thus, 
\begin{align*} 
\mathbb{E}\left[\left| N_{P^t}(S)\cap V^i_{j'} \cap Y_i\right|\right]
&=  (1\pm \gamma^{1/3})\delta\left|(N_{P^t}(S)\cap V^i_{j'})\setminus U_i \right| \\
& =  (1\pm \gamma^{1/3})\delta(\beta_{j,j'}'^{|S|} \pm 2\delta^{1/4} \pm \gamma^{2/5})m'= 
\left(\beta_{j,j'}'^{|S|} \pm 3\delta^{1/4}\right)\delta m'.
\end{align*}
By the choice of $Y_i$, the above random variable is hypergeometrically distributed. Thus by Lemma~\ref{Chernoff Bounds},
\begin{eqnarray}\label{eq:W1lowprob}
\Prob\left[| N_{P^t}(S)\cap V^i_{j'} \cap W_i| = (\beta_{j,j'}'^{|S|} \pm 5\delta^{1/4}) \delta m'\right] &\stackrel{\eqref{eq:UiNUi}}{\geq}& \Prob\left[| N_{P^t}(S)\cap V^i_{j'} \cap Y_i| = \left(\beta_{j,j'}'^{|S|} \pm 4\delta^{1/4}\right) \delta m'\right]\nonumber \\
&\geq& 1-(1-2{c})^n.
\end{eqnarray}
Moreover, 
\begin{align}\label{eq:WiVij}
|W_i\cap V^i_j|=|W_i\cap V^i_{j'}|=\delta m',
\end{align}
and $$\left| \binom{W_i\cap V^i_j}{2} \cap \overline{SP}^t_{j,j'}(i)\right|\geq \binom{\delta m'}{2} - \gamma m'^2 \geq
(\frac{1}{2}-25\delta^{1/4})(\delta m')^2$$ by Claim~\ref{Pt}(iii). So Theorem~\ref{codegree implies regularity} together with~\eqref{eq:W1lowprob} 
and a union bound over all $S\in V^i_j\cup \overline{SP}^t_{j,j'}(i)$
implies that $P^t[W_i\cap V^i_j, W_i\cap V^i_{j'}]$ is $(\delta^{1/25},\beta'_{j,j})$-regular with probability at least $1-m'^2(1-2{c})^n$.

Recall that (W1) holds if and only if $P^t[W_i\cap V^i_j,W_i\cap V^i_{j'} ]$ is $(\delta^{1/25} , \beta'_{j,j'})$-super-regular
for all $i\in I_t$ and $jj'\in E(R_K)$. So a union bound over all $i,j,j'$ gives  
\begin{align}\label{eq:W1fails}
\Prob[\text{(W1) fails in Round } t] \leq \gamma n (Kr)^2m'^2 (1-2 {c})^n.
\end{align} 

Secondly, we show that (W2)(i) holds with high probability. Fix $v\in V(G)$. Note that~(\ref{Yi prob}) together with
Proposition~\ref{generalised-chernoff} and Lemma~\ref{Chernoff Bounds} implies that
with probability at least $1-(1-2c)^n$ we have $|\{i\in I_t:v\in Y_i\}|\le 3\delta|I_t|/2=3\delta\gamma n/2$.
Moreover, by (U1) we have
\begin{align}\label{eq:U1f}
| \{i\in I_t: v\in U_i\}|\le | \{i\in I_t: v\in NU_i\}| \leq \gamma^{4/3} n.
\end{align}
Since $| \{i\in I_t: v\in W_i\}|=| \{i\in I_t: v\in U_i\}|+| \{i\in I_t: v\in Y_i\}|$, together with a union bound over all $v \in V(G)$ this implies that
\begin{align}\label{eq:W2ifails}
\Prob[\text{(W2)(i) fails in Round } t] \leq rn(1-2 {c})^n.
\end{align}

Next we show that (W2)(ii) holds with high probability. Fix $v\in V(G)$ and
consider 
$$C(v):= \sum_{i\in I_t} |\{ e\in E(H_i): v\in \phi_i(e) , \phi_i(e)\cap W_i\neq \emptyset\}|.$$
\noindent
For all integers $ p \in I_t\cup \{(t-1)\gamma n\}$ define 
$$ X_p(v) := \mathbb{E}\left[ C(v) \mid Y_{(t-1)\gamma n+1}, \dots, Y_p \right],$$
and note that $X_p(v)$ is an exposure martingale.
Moreover, changing $Y_p$ changes $C(v)$ by at most $(k+1)\Delta_R$, so the martingale is $(k+1)\Delta_R$-Lipschitz.
Furthermore, we have
\begin{eqnarray*}  \mathbb{E}\left[C(v)\right] &\leq& \sum_{i\in I_t: v \notin NU_i} \left((k+1)\Delta_R \mathbb{P}\left[v \in W_i\right] +  \sum_{uv \in \phi_i(E(H_i))}\mathbb{P}\left[u \in W_i\right]\right) \\
&+& (k+1)\Delta_R| \{i\in I_t: v\in NU_i\}|\\
&\stackrel{(\ref{Yi prob}, \ref{eq:U1f})}{\leq}& \sum_{i\in I_t\colon v \notin NU_i} \left((k+1)\Delta_R (1+\gamma^{1/3})\delta+  \sum_{uv \in \phi_i(E(H_i))}(1+  \gamma^{1/3})\delta\right) + (k+1)\Delta_R\gamma^{4/3} n\\
&\leq& 2(k+1)\Delta_R(1+\gamma^{1/3})\delta\left(\gamma n - | \{i\in I_t: v\in NU_i\}| \right) + (k+1)\Delta_R\gamma^{4/3} n\\
&\leq& \frac{5}{2}k\Delta_R\delta\gamma n. 
\end{eqnarray*}
(To see that we can indeed apply (\ref{Yi prob}), note that whenever $v\notin NU_i$ and $uv\in \phi_i(E(H_i))$ we have $u\notin U_i$.)
So, by Azuma's inequality (Theorem~\ref{Azuma}), we obtain $\Prob[ C(v) \geq 3k\Delta_R\delta\gamma n] 
\leq (1-2{c})^n.$ 
Hence, a union bound over all $v\in V(G)$ gives
\begin{align}\label{eq:W2fails}
\Prob[\text{(W2)(ii) fails in Round } t]\leq rn(1-2{c})^n.
\end{align}

Finally, we show that (W3) holds with high probability. Recall that $F'_i$ was defined in~(\ref{F' definition}).
For any $i\in I_t$ and $j\in [Kr]$, let
$$\overline{S}_{j}(i):=\left\{\{x,x'\} \in \binom{X^i_j}{2}: \left|N_{F'_i}(\{x,x'\})\right| = ( (\alpha d_0 p(R_K,\vec{\beta}',j))^2 \pm \delta^{1/2})m' \right\}.$$ 
Then Claim~\ref{cl:F't} implies $|\overline{S}_j(i)| \geq \binom{|X^i_j|}{2} - \gamma m'^2$.

Let $S\subseteq X_j^i$ be such that either $|S|=1$ or $S$ consists of a pair of vertices in $\overline{S}_{j}(i)$. 
Then Claim~\ref{cl:F't} and the definition of $\overline{S}_j(i)$ imply that
\begin{align}\label{common nbrs s}
 |N_{F'_i}(S)| = ( (\alpha d_0 p(R_K,\vec{\beta}',j))^{|S|} \pm \delta^{1/2} )m'.
\end{align}
By the same argument as in the proof of (W1) we have 
\begin{eqnarray*} 
\mathbb{E}\left[\left|N_{F'_i}(S)\cap Y_i\right|\right] &\stackrel{\eqref{Yi prob}}{=}& (1\pm  \gamma^{1/3})\delta\left|N_{F'_i}(S)\setminus U_i\right|  \\
&\stackrel{(\ref{eq:UiNUi}, \ref{common nbrs s}) }{=}& (1\pm \gamma^{1/3}) \left((\alpha d_0 p(R_K,\vec{\beta}',j))^{|S|}\pm \delta^{1/2} \pm \gamma^{2/5}\right) \delta m'\\
&=& \left((\alpha d_0 p(R_K,\vec{\beta}',j))^{|S|} \pm 2\delta^{1/2} \right) \delta m'.
\end{eqnarray*}
As in the proof of (W1), $|N_{F'_i}(S)\cap Y_i|$ has the hypergeometric distribution, so
$$\mathbb{P}\left[ |N_{F'_i}(S)\cap Y_i| = ((\alpha d_0 p(R_K,\vec{\beta}',j))^{|S|}\pm 3\delta^{1/2}) \delta m'\right] \geq 1-(1-2{c})^n.$$ 
Together with \eqref{eq:UiNUi} this implies
\begin{align*}
\mathbb{P}\left[ |N_{F'_i}(S)\cap W_i| = ((\alpha d_0 p(R_K,\vec{\beta}',j))^{|S|}\pm 4\delta^{1/2}) \delta m'\right] \geq 1-(1-2{c})^n.
\end{align*}

\COMMENT{Here we use that $\phi_i^{-1}(Y_i)\cap X_j^i$ is a random subset of $X_j^i\setminus \phi_i^{-1}(U_i)$ since $\phi_i$ is fixed in Step 3.}Together with a similar argument for vertices $v \in V_j^i$ this shows that with probability at least $1- 2Kr (rn)^2(1-2{c})^n$ for all $j\in [Kr]$, $S\subseteq X_j^i$ with $|S|=1$ or $S \in \overline{S}_j(i)$ and all $v\in V_j^i$ we have
\begin{align}\label{eq:www}
&|N_{F'_i}(S)\cap W_i| = ((\alpha d_0 p(R_K,\vec{\beta}',j))^{|S|}\pm 4\delta^{1/2}) \delta m', \nonumber\\
&|N_{F'_i}(v)\cap Z_i| = (\alpha d_0 p(R_K,\vec{\beta}',j)\pm 4\delta^{1/2}) \delta m'.
\end{align}
Moreover, 
\begin{align*}
\left|\binom{Z_i\cap X^i_j}{2} \cap \overline{S}_j(i)\right|
 \stackrel{\eqref{eq:WiVij}}{\geq} \binom{\delta m'}{2} - \gamma m'^2 \geq \frac{1}{2}(1- \delta^{6/20}) (\delta m')^2.
\end{align*}
By Theorem~\ref{codegree implies regularity}, this together with \eqref{eq:www} implies that $F'_i[Z_i\cap X^i_j,W_i\cap V^i_j]$ is $(\delta^{1/20},\alpha d_0 p(R_K,\vec{\beta}',j))$-super-regular. Hence, 
\begin{align}\label{eq:W3fails}
\Prob[\text{(W3) fails in Round } t] \leq 2Kr(rn)^2(1-2{c})^n.
\end{align}

\noindent
By \eqref{eq:W1fails}, \eqref{eq:W2ifails}, \eqref{eq:W2fails} and \eqref{eq:W3fails}, failure of type 5 occurs in some Round $t$ with probability at most $$T\cdot \left((Kr)^2 \gamma n m'+rn + rn+ 2Kr(rn)^2\right)(1-2{c})^n \leq (1-{c})^n.$$ 
\end{proof}

\subsection{Failure of type 6}

The following claim shows that failure of type 6 in fact never occurs. 

\begin{claim}\label{type6}
Assume that we are in Round $t$, Step $4.\ell$, and no failure has occurred so far, in particular, (W1)--(W3) of Step 3 hold. Then failure of type 6 does not occur.
\end{claim}
\begin{proof}

We need to check that $F^*_i[Z_i\cap X^i_j,W_i\cap V^i_j]$ and $P^t_i[W_i]$ satisfy ($a'$) and ($b'$). Note that for every vertex $v\in V(G)$ we have 
\begin{align}\label{eq:Pti degree}
 d_{P^t - P^t_i}(v) \leq \sum_{j\in I_t : j<i} \Delta(H_i)\le (k+1)\Delta_R \gamma n \leq \delta^3 n.
\end{align}

First, we prove ($a'$). If $\ell=1$, this is immediate from \eqref{eq:ptiftidef} and (W3). So assume that $\ell>1$. 
The definition of $F^*_i$ in \eqref{eq:Ft*idef} together with the definition of $F'_i$ in \eqref{F' definition} imply that whenever $xv \in E(F'_i)\setminus E(F^*_i)$ with $x\in X^i_{j}, v\in V^i_{j}$, we can find a vertex $y \in N^i_x$ such that $v \in N_{P^t-P^t_i}(\phi_i(y))$. Hence,
$$v \in \bigcup_{y\in N^i_x} N_{P^t- P^t_i}(\phi_i(y)).$$
By \eqref{eq:Pti degree} and the fact that $|N_x^i|\leq K\Delta_R$ (by (F3)) 
we have
$$\left|\bigcup_{y\in N^i_x} N_{P^t- P^t_i}(\phi_i(y))\right| \leq K\Delta_R \cdot \delta^3 n \leq \delta^2 m',$$ thus $x$ is incident to at most $\delta^2 m'$ edges in $E(F'_i)\setminus E(F^*_i)$.

Similarly, for any $xv \in E(F'_i)\setminus E(F^*_i)$ with $x\in X^i_j, v\in V^i_j$, an argument similar to that in \eqref{eq:xvunions} implies that 
 $$x\in \bigcup_{u \in N_{P^t- P^t_i}(v)} N^i_{\phi_i^{-1}(u)}.$$
As before, using \eqref{eq:Pti degree}, one can show that 
$v$ is incident to at most  $K\Delta_R \delta^3 n \leq \delta^2 m'$ edges in $E(F'_i) \setminus E(F^*_i)$. So 
$\Delta(F'_i - F^*_i)\leq \delta^2 m' \le 4\delta^{1/20}\cdot \delta m'$.

Moreover, by (W3), $F'_i[Z_i\cap X^i_j, W_i\cap V^i_j]$ is $(\delta^{1/20}, \alpha d_0p(R_K,\vec{\beta}',j))$-super-regular. 
Since $|Z_j\cap X_j^i|=|W_j\cap V_j^i|=\delta m'$, we can apply 
Proposition \ref{regularity after edge deletion} 
to deduce that $F^*_i[Z_i\cap X^i_j, W_i\cap V^i_j]$ is $(\delta^{1/50},\alpha d_0 p(R_K,\vec{\beta}',j))$-super-regular, and ($a'$) holds. 

Next, we prove ($b'$). By (W1), $P^t[W_i]$ is $(\delta^{1/25} , \vec{\beta'})$-super-regular with respect to $(R_K,V_1^i\cap W_i,\dots,V_{Kr}^i\cap W_i)$.
For $v \in V(P^{t}_i)$, \eqref{eq:Pti degree} implies that 
$$d_{P^t[W_i]}(v)- d_{P^t_i[W_i]}(v) \leq \delta^3 n \leq  4\delta^{1/25}\cdot \delta m'.$$
Thus Proposition \ref{regularity after edge deletion} implies that $P^t_i[W_i]$ is $(4\delta^{1/50} , \beta)$-super-regular, and ($b'$) holds.
\end{proof}

\begin{proof}[Proof of Lemma \ref{main lemma}]
Claims~\ref{type1}--\ref{type2} and \ref{cl:mainalgfailure2}--\ref{type6} imply that the main packing algorithm has failure probability of at most $5(1-{c})^n$. 
Thus, the main packing algorithm succeeds with high probability, and Claim \ref{disjoint} ensures that the embedded copies
$\phi'_1(H_1),\dots, \phi'_s(H_{s})$ of $H_1,\dots, H_{s}$ are pairwise edge-disjoint subgraphs of $G$ such that $\phi'_i(x)\in N_{A_i}(x)$ for all $x\in V(H_i)$. Thus (T1) and (T2) are satisfied for all $i\in [s]$. 
Claim~\ref{type2} also implies that, provided failure of type $2$ does not occur, $(\textup{QW}'_i)$ holds for all $i\in [s]$. This implies~(T3).
Finally, Claim~\ref{disjoint}(v) for all $i\in [s]$ implies (T4).
\end{proof}

The following lemma is a modification of Csaba's extension of the blow-up lemma~\cite{Csaba}, and it is easily implied from Theorem~\ref{main lemma} with the condition (S6) replaced with (S6$'$) and setting $s=1$.%
\COMMENT{Note Theorem~\ref{main lemma}only applies to near regular graphs. Can achieve this by using the maxflowmincut theorem for each bipartite subgraph between clusters to find a regular bipartite graph which contains it. Let $H'$ be the union of these bipartite graphs.}
It will be convenient (though not essential) to use it in Section~\ref{sec:stacking}.

\begin{lemma} \cite{Csaba} \label{thm:oldblowup}
Suppose $0< 1/n\ll \epsilon\ll 1/k, d, 1/\Delta_R, 1/C$ and that $R$ is a graph on $[r]$ with $\Delta(R)\leq\Delta_R$. Suppose $\vec{d}$ is a symmetric $r\times r$ matrix with entries in $[0,1]$ such that $\min_{ij\in E(R)}d_{i,j} \geq d$.  Suppose $H$ is an $r$-partite graph with partition $(R,X_1,\dots,X_r)$ and $G$ is an $r$-partite graph with partition $(R,V_1,\dots V_r)$ such that for all $i \in [r]$ we have $|X_i|=|V_i| = n \pm C$. Suppose that for all $i\in [r]$, we are given $W_i\subseteq X_i$ with $|W_i| \leq \epsilon |X_i|$ and $U_i \subseteq V_i$ with $|U_i| \leq \epsilon |V_i|$.
Suppose finally that $\Delta(H) \leq k$ and $G$ is $(\epsilon, \vec{d})$-super-regular with respect to $(R,V_1,\dots,V_r)$. Then $G$ contains a copy $\phi(H)$ of $H$, 
such that for each $i \in [r]$  we have $\phi(X_i)=V_i$ and $\phi(W_i) \cap U_i= \emptyset$.
\end{lemma}


\section{Packing graphs into near-equiregular graphs}\label{sec:stacking}

Our aim is to pack graphs $H_1,\dots,H_\ell$ of small maximum degree into a graph $G$ which is
$(\epsilon,\vec{d})$-regular with respect to a partition $(R,V_1,\dots,V_r)$. However, one of the conditions of the blow-up lemma for approximate decompositions (Theorem~\ref{main lemma}) is that all the $H_s$ are `degree balanced' with respect to $R$ and the corresponding densities $d_{i,j}$.
If $R$ is complete and the $d_{i,j}$ are all equal, this (roughly speaking)
translates to the requirement that all the bipartite graphs $H[V_i,V_j]$ are
$k$-regular for some constant $k$.
More generally, the requirement is (again roughly speaking) that for each edge $ij$
of $R$, the graph 
$H[V_i,V_j]$ is $k_{i,j}$-regular. 

Suppose that we are given a family of graphs $\mathcal{H}=\{H_1,\dots,H_\ell\}$ of small maximum degree, which do not necessarily satisfy the degree balance condition. 
In this section we show that it is often possible to pack the graphs in $\mathcal{H}$
together in a suitable way to obtain a new family of graphs which do satisfy the
degree balance condition required in Theorem~\ref{main lemma}.

The main result (Lemma~\ref{lem:section3newlemma}) of Subsection~\ref{subsec:31} shows that this is possible if the graphs in
$\mathcal{H}$ are `edge-balanced'
i.e. $e(H[V_i,V_j])$ is proportional to $d_{i,j}$.
In Subsection~\ref{subsec:32} we will apply Lemma~\ref{lem:section3newlemma} in order to show that in the setting of Theorems~\ref{main 1b} and~\ref{main 2}, that is when $R$ is a complete graph, we can achieve a packing of the $H_s$ which satisfies the degree balance condition for any family $\mathcal{H}$ of bounded degree graphs. In this way we will be able to satisfy the conditions of Theorem~\ref{main lemma}.

\subsection{Packing graphs into almost regular graphs}\label{subsec:31}

The aim of this subsection is to prove the following lemma, which shows that a sufficiently large collection of `edge-balanced' graphs of bounded maximum degree can be packed into a suitable near-equiregular graph $H$. Note that property (ii) will not needed in this paper, but is intended for further applications elsewhere.

\begin{lemma}\label{lem:section3newlemma}
Suppose $0<1/n\ll \epsilon \ll 1/s \ll  1/\Delta,1/\Delta_R$ and $\epsilon \ll 1/k, 1/(C+1)$. Let $R$ be a graph on $[r]$ with $\Delta(R)\leq \Delta_R$ and let $n\leq n_1\leq \dots \leq n_r\leq n+C$. Let $\vec{M}$ and $\vec{k}$ be symmetric $r\times r$ matrices such that $k_{i,j}\in \mathbb{N}$ and $M_{i,j} + s^{2/3} +1 \leq k_{i,j} \leq k$ for all $ij\in E(R)$. \COMMENT{Note that this holds even if $M_{i,j} < s^{2/3}$.}

Suppose that $L_1,\dots,L_s$ are graphs such that each $L_\ell$ admits a vertex partition $(R,X_1^\ell,\dots,X_r^\ell)$ with $|X_i^\ell|=n_i$, $\Delta(L_\ell)\leq \Delta$ and 
\begin{align}\label{eq:edgebalance}
\sum_{\ell=1}^{s} e(L_\ell[X_i^\ell,X_j^\ell])=M_{i,j}n
\end{align}
for all $ij\in E(R)$. 
Suppose that for each $\ell\in [s]$ and $i\in [r]$, we have $W^{\ell}_{i} \subseteq X^{\ell}_i$ with $|W^{\ell}_{i}| \leq \epsilon n$.
Then there exists an $(R,\vec{k},C)$-near-equiregular graph $H$ with vertex partition $(R,V_1,\dots,V_r)$ such that $|V_i|=n_i$ for all $i\in [r]$ and such that $L_1,\dots,L_s$ pack into $H$. 

Moreover, writing $\phi(L_\ell)$ for the copy of $L_\ell$ in this packing of $L_1,\dots,L_s$ in $H$, and for each $ij\in E(R)$ writing $J_{i,j}:=H[V_i,V_j]-(\phi(L_1)\cup\dots\cup \phi(L_s))$, we have
\begin{itemize}
\item[(i)] $\phi(X^{\ell}_i)= V_{i}$, and $\Delta(J_{i,j})\leq k_{i,j}-M_{i,j}+2s^{2/3}$ and for each $i\in [r]$,  
\item[(ii)] for all $i\in [r]$ and all distinct $\ell , \ell' \in [s]$, we have $\phi(W^{\ell}_i) \cap \phi(W^{\ell'}_i) =\emptyset$.
 \end{itemize}  
\end{lemma}

In order to prove Lemma~\ref{lem:section3newlemma}, we first show that the graphs $L_1\dots,L_s$ pack `perfectly' into a graph $H$ for which $H[V_i,V_j]$ is close to regular whenever $ij\in E(R)$ and $H[V_i,V_j]$ is empty if $ij\notin E(R)$. Note that $\Delta(L_\ell)\leq \Delta$ implies that 
\begin{align}\label{s1}
M_{i,j} \leq 2\Delta s.
\end{align} 

\begin{lemma}\label{packing to almost regular graph}
Assume the setup of Lemma~\ref{lem:section3newlemma}. Then there exists an $r$-partite graph $H$ admitting a vertex partition $(R,V_1,\dots, V_r)$ with $|V_i|=n_i$ such that $H$ has a decomposition into $L_1,\dots,L_s$ and $d_{H[V_i,V_j]}(v)= M_{i,j} \pm s^{2/3}$ for all $v\in V_i\cup V_j$ and
$ij \in E(R)$. 
Moreover, writing $\phi(L_\ell)$ for the copy of $L_\ell$ in this decomposition of $H$, for each $i\in [r]$ we have
\begin{itemize}
\item[(i)] $\phi(X^\ell_i)=V_i$,
\item[(ii)] for all $i\in [r]$ and all distinct $\ell , \ell' \in [s]$, we have $\phi(W^{\ell}_i) \cap \phi(W^{\ell'}_i) =\emptyset$.
 \end{itemize}  
\end{lemma}
\begin{proof}
For each $\ell \in [s]$, we will now refine the partition $X^\ell_1,\dots,X_r^\ell$ of $L_\ell$. The aim is that, firstly, within each refined partition class the vertex degrees are similar with respect to the other vertex classes (see Claim~\ref{B'qij}). 
Secondly, an analogue of the condition \eqref{eq:edgebalance} is still preserved (with appropriate scaling and relabelling) by the subpartitions (see Claim~\ref{Bqij}).
To achieve this, we first define a suitable `exceptional' class $X^{\ell,0}_i\subseteq X_i^\ell$ (for each $i\in [r]$) whose removal results in partition classes whose size is a large integer multiple.

For each $i\in [r]$ let $r_i(j)$ denote the $j$-th smallest number in $N_R(i)$. Let $r_i(j):= 0$ if $j> |N_R(i)|$ and let $M_{i,0}:=0$. Since $\Delta(R) \leq \Delta_R$, we have $r_i(\Delta_R+1)=0$ for each $i\in [r]$.
For every $x\in X^\ell_i$ and $j\leq |N_R(i)|$ we define 
$$d^\ell_j(x):= |N_{L_\ell}(x)\cap X_{r_i(j)}^\ell|, \text{ which will be called the \emph{$(j,\ell)$-degree} of } x.$$ We let $d^\ell_j(x):=0$ if $j> |N_R(i)|$. Note that for every $x\in X_i^\ell$ and every $j\in [\Delta_R]$ we have
\begin{align}\label{eq:jelldegrees}
0\leq d_j^\ell(x)\leq \Delta.
\end{align}
For a set $A\subseteq X^\ell_i$ we say that $A$ is \emph{$(j,\ell)$-regular} if every vertex $v \in A$
has the same $(j,\ell)$-degree, otherwise we say that $A$ is \emph{$(j,\ell)$-irregular}. 
We let $n'$ be the integer satisfying $n = (s^{\Delta_R}+1) n' + m'$ with $0\leq m' \leq s^{\Delta_R}$.
For each $i \in [r]$ define $m'_i$ to satisfy $n_i = (s^{\Delta_R}+1) n' + m_i'$. So $0 \leq m_i' \leq s^{\Delta_R} +C$. 
Let 
\begin{align}\label{eq:defhatn}
\hat{n}_{i,j}:=\begin{cases}
n_i  &\text{ if } j=0,\\
(s^{\Delta_R-j}+1)n'+m_i' &\text{ if } j\in[\Delta_R-1],\\
n'+m'_i &\text{ if } j=\Delta_R.
\end{cases}
\end{align}

\begin{claim}\label{X0 choice}
For all $\ell\in [s]$ and $i\in [r]$, there is a set $X_{i}^{\ell,0} \subseteq X^\ell_i$ with $|X_{i}^{\ell,0}|=\hat{n}_{i,\Delta_R}$ such that
for each $j\in [|N_R(i)|]$ we have 
\begin{align}\label{average X0} 
\sum_{\ell=1}^{s} D^{\ell,j}_i =M_{i,r_i(j)}\pm s^{3/5},
\end{align}
where $D^{\ell,j}_i$ is the average $(j,\ell)$-degree of the vertices in $X^{\ell,0}_i$, and
\begin{align}\label{irregular X0}
|\{\ell\in [s] : X^{\ell,0}_i \text{ is }(j,\ell)\text{-irregular}\}| \leq 3\Delta_R \Delta .\end{align}
\end{claim}
\begin{proof}
Let $\ell\in [s]$ and $i\in[r]$ be arbitrary but fixed. 
For each $0\leq j \leq \Delta_R$ we define a vertex set $Y_i^{\ell,j}\subseteq X_i^\ell$ inductively as follows. Let $Y_i^{\ell,0}:=X^\ell_i$.
Assume that for some $0\leq j < \Delta_R$ we have defined  $Y_i^{\ell,j}$ of size $\hat{n}_{i,j}$, and order the vertices in $Y_i^{\ell,j}$ as $y^{\ell,j}_1,\dots, y^{\ell,j}_{\hat{n}_{i,j}}$ so that their $(j+1,\ell)$-degrees are non-decreasing. Now we pick a number $a$ in $[\hat{n}_{i,j}]$ uniformly at random and let 
$$Y_i^{\ell,j+1}:= \{ y^{\ell,j}_q : a\leq q< a+ \hat{n}_{i,j+1} \}, \text{ where the index } q \text{ is considered modulo }\hat{n}_{i,j}. $$
We repeat this process until we obtain $Y_i^{\ell,\Delta_R}$ with $|Y^{\ell,\Delta_R}_i|=\hat{n}_{i,\Delta_R} = n'+m'_i$. 
Let $X^{\ell,0}_i := Y^{\ell,\Delta_R}_i$.
Consider any vertex $x \in X^\ell_i$. We have
$$\Prob[x\in X^{\ell,0}_i]=\prod_{j=1}^{\Delta_R} \mathbb{P}[x\in Y^{\ell,j}_i\mid x\in Y^{\ell,j-1}_i] = \prod_{j=1}^{\Delta_R} \frac{|Y^{\ell,j}_i|}{|Y^{\ell,j-1}_i|} = \frac{\hat{n}_{i,\Delta_R}}{n_i}.$$ 
Thus 
\begin{align*}
\mathbb{E}\left[ \sum_{\ell=1}^{s} D^{\ell,j}_i \right] &=\sum_{\ell=1}^{s} \sum_{x\in X^\ell_i}\mathbb{P}[x\in X^{\ell,0}_i] \frac{d^\ell_j(x)}{\hat{n}_{i,\Delta_R}} = \frac{1}{n_i} \sum_{\ell=1}^{s}\sum_{x\in X^\ell_i} d^\ell_j(x)
\stackrel{\eqref{eq:edgebalance}}{=} \frac{M_{i,r_i(j)}n}{n_i}  \\ & \stackrel{\eqref{s1}}{=} M_{i,r_i(j)}\pm \frac{3C\Delta s}{n}.
\end{align*}
Consider the exposure martingale $Z_i^\ell: = \mathbb{E}[\sum_{\ell'=1}^{s} D^{\ell',j}_i \mid X^{1,0}_i, \dots, X^{\ell,0}_i]$ for $0\leq \ell \leq s$.
Note that $D^{\ell,j}_i\leq \Delta$ and $D^{\ell,j}_i$ is determined by $X^{\ell,0}_i$. Thus $Z^\ell_i$ is $\Delta$-Lipschitz. So by Azuma's inequality (Theorem~\ref{Azuma}) we know that 
\begin{equation}\label{eq:azum}
\Prob\left[\sum_{\ell=1}^{s} D^{\ell,j}_i \neq  M_{i,r_i(j)}\pm s^{3/5}\right]\leq 2e^{-\frac{s^{6/5}}{8\Delta^2 s}}
< e^{-s^{1/6}}.
\end{equation} 
Let $Z_{i,j}^\ell$ be the indicator variable defined by 
  \begin{displaymath}
Z_{i,j}^\ell:= \left\{ \begin{array}{ll}
1 & \text{ if }X_i^{\ell,0} \text{ is }(j,\ell)\text{-irregular},\\
0 & \text{ if }X_i^{\ell,0} \text{ is }(j,\ell)\text{-regular}.\\
\end{array} \right.
\end{displaymath}
Also we let $Z_{i,j} := \sum_{\ell=1}^{s} Z_{i,j}^\ell$.
Note that since $X^{\ell,0}_i = Y^{\ell,\Delta_R}_i \subseteq Y^{\ell,j}_i$ for all $j \in [\Delta_R-1]$ we get that if $Y^{\ell,j}_i$ is $(j,\ell)$-regular, then $X^{\ell,0}_i$ is also $(j,\ell)$-regular. \COMMENT{(Moreover, if $j> |N_R(i)|$, then $X_i^{\ell,0}$ is always $(j,\ell)$-regular as the $(j,\ell)$-degree of any vertex in $X_i^{\ell,0}$ is $0$.)}

Note that there are at most $\Delta+1$ indices $q$ such that $y^{\ell,j-1}_{q}$ and $y^{\ell,j-1}_{q+1}$ have distinct $(j,\ell)$-degrees
(by~\eqref{eq:jelldegrees} and since the ordering $y^{\ell,j-1}_{1},\dots, y^{\ell,j-1}_{\hat{n}_{j,i}}$ of $Y_i^{\ell,j-1}$ is such that the $(j,\ell)$-degrees are non-decreasing). So for a given $j\in [\Delta_R]$, 
$$\Prob[ Z_{i,j}^\ell = 1 ] \leq \Prob[Y^{\ell,j}_i \text{ is } (j,\ell)\text{-irregular}]\leq
\frac{(\Delta+1) \hat{n}_{i,j}}{\hat{n}_{i,j-1}} \stackrel{\eqref{eq:defhatn}}{\leq} \frac{3 \Delta}{2s}.
$$
So $\mathbb{E}[Z_{i,j}] \leq \frac{3\Delta s}{2s} =\frac{3\Delta}{2}$, and hence by Markov's inequality, 
$$\mathbb{P}\left[ Z_{i,j} \geq 3\Delta_R \Delta \right] \leq \frac{1}{2\Delta_R}.$$
Together with~\eqref{eq:azum} and a union bound over all choices of $j\in [\Delta_R]$ this implies that
with probability at least $1- \Delta_R(e^{-s^{1/6}}+1/(2\Delta_R))>0$ both \eqref{average X0} and \eqref{irregular X0} hold for fixed $i\in [r]$, meaning that we can find sets $X_{i}^{\ell,0}$ as required for fixed $i$. Thus we may take such a choice for each $i\in [r]$, which completes the proof of the claim.
\end{proof}

We now proceed in a similar way for the remaining vertices. For each $\ell \in [s]$ and $i\in [r]$ let $X^{\ell,0}_i$ be as guaranteed by Claim \ref{X0 choice} and let
$$X^{\ell,*}_i := X^\ell_i\setminus X^{\ell,0}_i.$$ Then $|X^{\ell,*}_i| = s^{\Delta_R} n'$. We now order the vertices of $X^{\ell,*}_i$ in the order of non-decreasing $(1,\ell)$-degree and partition them into $s$ blocks, each of size $s^{\Delta_R -1}n'$. For each $p_1\in [s]$ we let  $A^\ell_i(p_1)$ denote the $p_1$-th block. We next order the vertices of each $A^\ell_i(p_1)$ in the order of  non-decreasing $(2,\ell)$-degree and partition them into $s$ blocks, each of size $s^{\Delta_R-2}n'$. For all $p_1,p_2\in [s]$ we let $A^\ell_i(p_1,p_2)$ denote the $p_2$-th block of $A^\ell_i(p_1)$. Continuing in this way, we obtain $s^{\Delta_R}$ blocks $A^\ell_i(p_1,p_2,\dots, p_{\Delta_R})$ of size $n'$ each.

We now show that for fixed $\ell\in [s]$, $i\in[r]$ and $j\in [\Delta_R]$ we have 
\begin{align}\label{A block irregular}
|\{(p_1,\dots, p_{\Delta_R} ) : A^\ell_i(p_1,\dots, p_{\Delta_R})\text{ is }(j,\ell)\text{-irregular} \}| \leq (\Delta+1) s^{\Delta_R-1}.
\end{align}
Note that, for each $j\in [\Delta_R]$, if $A^\ell_i(p_1, \dots, p_{j})$ is $(j,\ell)$-regular, then  $A^\ell_i(p_1,\dots, p_{\Delta_R})$ is also $(j,\ell)$-regular. But by~\eqref{eq:jelldegrees}, for any fixed $(p_1,\dots, p_{j-1})$ all but at most $\Delta+1$ of the blocks $A^\ell_i(p_1, \dots, p_{j})$ are $(j,\ell)$-regular.
This implies~(\ref{A block irregular}).

For each $\ell\in [s]$ and $i\in [r]$ we now relabel the $s^{\Delta_R}$ blocks $A^\ell_i(p_1,\dots, p_{\Delta_R})$ by $X^{\ell,j'}_i$ with $j' \in [s^{\Delta_R}]$.
Thus $X_i^{\ell,1},X_i^{\ell,2},\dots, X^{\ell,s^{\Delta_R}}_i$ forms a partition of $X_i^{\ell,*}$ into blocks of size $n'$.
Let $d^{\ell}_j(i,j')$ be the average $(j,\ell)$-degree of the vertices in $X^{\ell,j'}_i$ where the index $j'$ is considered modulo $s^{\Delta_R}$. Now we choose a random integer $ j'_{i,\ell} \in [s^{\Delta_R}]$ for each $\ell \in [s]$ and $i\in [r]$.

\begin{claim}\label{Bqij}
For any fixed $i\in [r],j\in [\Delta_R],q\in [s^{\Delta_R}]$, let $B_q(i,j)$ be the event that 
\begin{align}\label{ref Bqij}
\sum_{\ell=1 }^{s} d^\ell_j(i,j'_{i,\ell}+q) \neq  M_{i,r_i(j)}\pm 2s^{3/5}.
\end{align}
Then $\Prob[B_q(i,j)] \leq e^{-s^{1/6}}$.
\end{claim}
\begin{proof}
For each vertex $x \in X^{\ell,*}_i$, we have $\mathbb{P}[x \in X^{\ell,j'_{i,\ell}+q}_i]=1/{s^{\Delta_R}}$ since there exists one index $j'$ such that $x \in X^{\ell,j'}_i$, and the probability that $j'_{i,\ell}+q = j'$ is exactly ${1}/{s^{\Delta_R}}$.
So, 
\begin{align}\label{Ediljq}
\mathbb{E}\left[\sum_{\ell=1 }^{s} d^\ell_j(i,j'_{i,\ell}+q) \right]= \sum_{\ell=1}^{s} \sum_{ x\in X^{\ell,*}_i} \frac{d^\ell_j(x)}{n'} \cdot \frac{1}{s^{\Delta_R} }.
\end{align}
By the assumptions of Lemma~\ref{packing to almost regular graph} and~\eqref{average X0}, $$\sum_{\ell=1}^{s} \sum_{ x\in X^\ell_i} d^\ell_j(x) = M_{i,r_i(j)} n  \;\;\; \text{ and } \;\;\; \sum_{\ell=1}^{s} \sum_{x\in X^{\ell,0}_i} d^\ell_j(x) = (M_{i,r_i(j)} \pm s^{3/5})\hat{n}_{i,\Delta_R}.$$
Thus 
\begin{eqnarray*}\sum_{\ell=1}^{s} \sum_{ x\in X^{\ell,*}_i} d^\ell_j(x) &=& \sum_{\ell=1}^{s} \sum_{ x\in X^\ell_i \setminus X^{\ell,0}_i} d^\ell_j(x) = M_{i,r_i(j)} (n_i\pm C)-(M_{i,r_i(j)}\pm s^{3/5})\hat{n}_{i,\Delta_R}\\ 
&\stackrel{\eqref{eq:defhatn}}{=}& (M_{i,r_i(j)}\pm s^{3/5}) s^{\Delta_R}n'.
\end{eqnarray*}
Together with (\ref{Ediljq}) this implies that $$\mathbb{E}\left[\sum_{\ell=1 }^{s} d^\ell_j(i,j'_{i,\ell}+q) \right] = M_{i,r_i(j)}\pm s^{3/5}.$$
For fixed $i,j$ and all $0\leq \ell \leq s$, let $$Z^{\ell}_{i,j} := \mathbb{E}\left[\sum_{\ell'=1}^{s} d^{\ell'}_j(i,j'_{i,\ell'}+q) \mid j'_{i,1},\dots, j'_{i,\ell}\right].$$
Then $Z^{\ell}_{i,j}$ is an exposure martingale which is $\Delta$-Lipschitz by~\eqref{eq:jelldegrees}.
Therefore, by Azuma's inequality (Theorem \ref{Azuma}), 
$$\Prob[B_q(i,j)]\leq 2e^{-\frac{s^{6/5}}{2\Delta^2 s}} < e^{-s^{1/6}}.$$
\end{proof}

\begin{claim}\label{B'qij}
For any fixed $i\in [r],j\in [\Delta_R],q\in [s^{\Delta_R}]$, 
let $B'_q(i,j)$ be the event that 
$$|\{\ell\in [s]: X^{\ell,j'_{i,\ell}+q}_i \text{ is }(j,\ell)\text{-irregular}\}| \geq s^{3/5}.$$
Then $\Prob[B'_q(i,j)]\leq e^{-s^{1/6}}.$
\end{claim}
\begin{proof}
Let 
\begin{displaymath}
Z(i,j,q,\ell):= \left\{ \begin{array}{ll}
1 & \text{ if }X^{\ell, j'_{i,\ell}+q}_i \text{ is }(j,\ell)\text{-irregular},\\
0 & \text{ otherwise. }\\
\end{array} \right.
\end{displaymath}
\noindent
Let $Z(i,j,q) := \sum_{\ell=1}^{s} Z(i,j,q,\ell)$. By~\eqref{A block irregular}, 
$$\Prob\left[X^{\ell,j'_{i,\ell}+q}_i \text{ is } (j,\ell) \text{-irregular}\right] \leq \frac{(\Delta+1) s^{\Delta_R -1}}{s^{\Delta_R}} \le  \frac{2\Delta}{s}.$$
This implies that for all $\ell\in[s]$ and all $z_1,\dots,z_{\ell-1}\in\{0,1\}$, 
$$\Prob\left[Z(i,j,q,\ell)= 1 \mid Z(i,j,q,1)=z_1, \dots, Z(i,j,q,\ell-1)=z_{\ell-1} \right] \leq  \frac{2\Delta}{s}.
$$
Thus, by Lemma~\ref{Chernoff Bounds} and Proposition~\ref{generalised-chernoff}\COMMENT{Note that $\mathbb{E}[ Bin(s,2\Delta/s) ] = 2\Delta \leq s^{3/5}/2$.},
$\Prob[Z(i,j,q) \geq s^{3/5}] \leq 2e^{-s^{6/5}/2s} \leq e^{-s^{1/6}}$. 
\end{proof}

Claims \ref{Bqij} and \ref{B'qij} imply that for fixed $i\in [r]$, with probability at least $1 - \Delta_R s^{\Delta_R}(e^{-s^{1/6}} +e^{-s^{1/6}}) >1/2$ for all $j \in [\Delta_R] $ and $q\in [s^{\Delta_R}]$ neither $B_q(i,j)$ nor $B'_q(i,j)$ occurs.
Thus for fixed $i\in [r]$, there exists a choice of $j'_{i,\ell}$ for each $\ell\in [s]$ such that this is indeed the case. Hence we can choose such integers $j'_{i,\ell}$ for all $i\in [r]$.\COMMENT{Note that choice for $i\in [r]$ does not affect the choice for another index $i'\in [r]$.} We now re-label, letting $X^\ell_{i,q}:= X^{\ell,j'_{i,\ell}+q}_i$ for all $q\in [s^{\Delta_R}]$, and $X^{\ell}_{i,0}: = X^{\ell,0}_i$.

In order to construct the graph $H$ required in Lemma~\ref{packing to almost regular graph}, we now apply the blow-up lemma (Lemma~\ref{thm:oldblowup}). For this, we divide each $V_i$ into $V_{i,0}$ of size $n'+m'_i$, and $V_{i,q}$ of size $n'$ for all $q\in [s^{\Delta_R}]$. Let $G_1$ be the complete $r$-partite graph with vertex partition $V_1,\dots, V_r$, and consider the refined vertex partition $$\mathcal{P}:=(V_{1,0},\dots, V_{1,s^{\Delta_R}},V_{2,0},\dots, V_{2,s^{\Delta_R}},\dots,V_{r,s^{\Delta_R}})$$
of $G_1$ and the $(s^{\Delta_R}+1)$-fold blow-up $R_{s^{\Delta_R}+1}$ of $R$. Note that $G_1$ admits vertex partition $(R_{s^{\Delta_r}+1}, \mathcal{P})$.

For each $\ell\in [s]$ in turn, we will choose an embedding $\phi_\ell$ of $H_\ell$ into some subgraph $G_\ell$ of $G_1$. Suppose that for some $\ell\in [s]$
we have already defined $G_\ell$ and $\phi_{\ell'}$ for all $\ell'<\ell$.
For each $(i,q) \in [r] \times \{0, 1, \dots,  s^{\Delta_R} \}$, we let $U^{\ell}_{i,q} := V_{i,q}\cap \bigcup_{\ell' < \ell} \phi_{\ell'}(W^{\ell'}_i)$ and 
we choose an embedding $\phi_\ell: L_\ell \rightarrow G_\ell$ of $L_\ell$ into $G_\ell$ in such a way that for all $(i,q) \in [r] \times \{0,1,\dots, s^{\Delta_R}\}$ we have 
$\phi_{\ell}(X^{\ell}_{i,q})=V_{i,q}$  and
\begin{align}\label{eq: W U disjoint}
\phi_{\ell}(W^{\ell}_i \cap X^{\ell}_{i,q} ) \cap U^{\ell}_{i,q} =\emptyset.
\end{align}
 We then let $G_{\ell+1}:= G_\ell - \phi_\ell(E(L_\ell))$. We repeat this process until we obtain $G_{s+1}$. 

Note that this process succeeds because for all $x \in V(G_\ell)\setminus V_i,$ and all $q\in\{0,1,\dots,s^{\Delta_R}\}$ we have $|N_{G_\ell}(x)\cap V_{i,q}| \geq |V_{i,q}| - \Delta s \geq (1- \frac{\Delta s}{n'})|V_{i,q}|.$ So $G_\ell$ is $((\frac{\Delta s}{n'})^{1/2},1)$-super-regular with respect to $(R_{s^{\Delta_r}+1},\mathcal{P})$ for all $\ell\in [s]$. 
Also, for each $(i,q) \in [r] \times \{0,1,\dots, s^{\Delta_R}\}$ we have $|W^{\ell}_i|, |U^{\ell}_{i,q}| \leq s \epsilon  n \leq \epsilon^{1/2} n'$.
Thus we can keep applying the Blow-up lemma (Lemma~\ref{thm:oldblowup}) until we obtain $G_{s+1}$.  Let $H:= \bigcup_{\ell=1}^{s} \phi_\ell(L_\ell).$

We claim that $H$ is as desired. 
Given a vertex $x\in V_{i,q}$, let 
$d_{H,j}(x):= |N_H(x)\cap V_{r_i(j)}|.$ Note that $d^\ell_j(\phi^{-1}_\ell(x)) =d^{\ell}_j(i,j'_{i,\ell}+q)$ unless $X^\ell_{i,q}$ is $(j,\ell)$-irregular. But $X^\ell_{i,q}$ is $(j,\ell)$-irregular for at most 
$s^{3/5}$ indices $\ell$ because $B'_q(i,j)$ does not occur. Also $|d^{\ell}_j(\phi^{-1}_\ell(x))- d^{\ell}_j(i,j'_{i,\ell}+q)| \leq \Delta$ even when $X^\ell_{i,q}$ is $(j,\ell)$-irregular.
Thus 
\begin{eqnarray} 
d_{H,j}(x) &=& \sum_{\ell=1}^{s} d^\ell_j(\phi^{-1}_\ell(x)) =\sum_{\ell=1}^{s} d^{\ell}_{j}(i,j'_{i,\ell}+q) \pm \Delta s^{3/5}
\stackrel{(\ref{ref Bqij})}{=} M_{i,r_i(j)}\pm 2s^{3/5} 
\pm \Delta s^{3/5} 
\nonumber \\ &=& 
M_{i,r_i(j)}\pm s^{2/3}.\nonumber
\end{eqnarray}
Lemma~\ref{packing to almost regular graph}(i) follows from the observation that for all $\ell\in [s]$ and $i\in [r]$, we have $\phi(X_i^{\ell}) = \phi_{\ell}(\bigcup_{q=0}^{s^{\Delta_R}}X^{\ell}_{i,q}) = \bigcup_{q=0}^{s^{\Delta_R}}V_{i,q} = V_i$.
Lemma~\ref{packing to almost regular graph}(ii) follows from \eqref{eq: W U disjoint}. This completes the proof of Lemma~\ref{packing to almost regular graph}. 
\end{proof}

Now we show that we can transform an almost regular graph $H$ (as in the conclusion of Lemma~\ref{packing to almost regular graph}) into an $(R,\vec{k},C)$-near-equiregular graph by adding a small number of edges. First we show that we can achieve this if the vertex classes have equal size. We write $\overline{H}$ for the complement of $H$. 

\begin{lemma} \label{adding}
Suppose $r\in \mathbb{N}$ and $0<1/n\ll 1/s \ll 1/\Delta$. Let $R$ be a graph on $[r]$. 
Let $\vec{M}$ and $\vec{k}$ be symmetric $r\times r$ matrices such that $k_{i,j}\in \mathbb{N}$, $M_{i,j} \leq 2\Delta s$ and $ M_{i,j}+ s^{2/3}+1\leq k_{i,j} < n/2$ for all $ij\in E(R)$. Suppose $H$ is an $r$-partite graph with vertex partition $(R,V_1,\dots, V_r)$ such that $|V_i|=n$ for all $i\in [r]$ and that for all $ij\in E(R)$ and every vertex $x\in V_i$ we have $M_{i,j}- s^{2/3}-1\leq d_{H[V_i,V_j]}(x)\leq M_{i,j} + s^{2/3}$.\COMMENT{need to allow $M_{i,j}- s^{2/3}-1$ instead of $M_{i,j}- s^{2/3}$ to be able to apply this in the proof of Lemma \ref{adding2}} Then there exists a graph $H'$ with vertex partition $(R,V_1,\dots,V_r)$ such that $H\subseteq H'$ and $H'[V_i,V_j]$ is $k_{i,j}$-regular for all $ij \in E(R)$.
\end{lemma}
\begin{proof}
It suffices to prove the lemma for the case when $r=2$ and $R=K_2$, and $k_{1,2} \leq M_{1,2}+ 2 s^{2/3}$. (Indeed, if $r\geq 3$, we proceed in the same way as when $r=2$ for each pair $V_i,V_j$ with $ij \in E(R)$. Also if $M_{1,2} + 2s^{2/3} < k_{1,2} < n/2$, then we can set $k'_{1,2} := M_{1,2} + 2s^{2/3}$ and find $H''$ with $H\subseteq H''$ such that $H''[V_1,V_2]$ is $k'_{1,2}$-regular. Since $k'_{1,2}<n/2$, we have $\delta(\overline{H''}[V_1,V_2]) \geq n/2$, and thus $\overline{H''}[V_1,V_2]$ has a perfect matching which we can add to $H''$. We repeat this until we obtain a $k_{1,2}$-regular graph $H'[V_1,V_2]$. This is possible since $k_{1,2} < n/2$.) 

So suppose that $r=2$ and $R=K_2$ and $k_{1,2} \leq M_{1,2}+ 2 s^{2/3}$. We will apply the Max-flow min-cut theorem. For this sake let us define a digraph $G'$ and an edge capacity function $c$ as follows:
\begin{itemize}
\item
$V(G') := \{s,t\} \cup V_1 \cup V_2$,
\item
$E(G') := \{\overrightarrow{vv'} : vv' \in E(\overline{H}[V_1,V_2]), v\in V_1, v'\in V_2\}
 \cup \{ \overrightarrow{sv} : v\in V_1\} \cup \{ \overrightarrow{v't}: v'\in V_2\}$,
\item
$c(\overrightarrow{vv'}) := 1, c(\overrightarrow{sv}) := k_{1,2}-d_{H[V_1,V_2]}(v),\text{ and }  c(\overrightarrow{v't}) := k_{1,2}-d_{H[V_1,V_2]}(v') \text{ for all } v \in V_1, v'\in V_2$.
\end{itemize}
Let $$C^*:=\sum_{v\in V_1}c(\overrightarrow{sv})=\sum_{v\in V_1} (k_{1,2}-d_{H[V_1,V_2]}(v))=\sum_{v'\in V_2}c(\overrightarrow{v't})$$
be the capacity of the cut $\{\{s\}, V_1 \cup V_2  \cup \{t\}\}$. If the above cut 
is a minimum $(s,t)$-cut, then the Max-flow min-cut theorem ensures that there is an integer flow $f$ from $s$ to $t$ with value $C^*$. Let $E'$ be the set of all edges from $V_1$ to $V_2$ with value $1$ in $f$. Then $E'\subseteq E(\overline{H})$, so letting $E(H'[V_1,V_2])):= E(H[V_1,V_2])\cup E'$ will ensure that $H'[V_1,V_2]$ is $k_{1,2}$-regular. Thus it suffices to prove that any $(s,t)$-cut in $G'$ has capacity at least $C^*$.

Note that for any $v\in V_1, v'\in V_2$, 
\begin{align} \label{flow bounds}
1\leq c(\overrightarrow{sv}), c(\overrightarrow{v't}) \leq k_{1,2} - (M_{1,2} - s^{2/3}-1) \leq 4 s^{2/3}.\end{align}
Assume we have an $(s,t)$-cut $\{U,W\}$ with $s\in U, t\in W$ which has capacity $C'$. Let $U_1 := U\cap V_1, U_2:= U\cap V_2, W_1:= W\cap V_1, W_2 := W\cap V_2$. Then
$$C'=\sum_{v\in W_1} c(\overrightarrow{sv}) + e_{\overline{H}}(U_1,W_2) + \sum_{v'\in U_2} c(\overrightarrow{v't}).$$

First, assume that $|W_2| \geq M_{1,2}+5s^{2/3}$. Then for each $v\in U_1$, $e_{\overline{H}}(\{v\},W_2) \geq |W_2| - d_{H[V_1,V_2]}(v) \geq M_{1,2}+5s^{2/3} -(M_{1,2}+ s^{2/3}) \geq c(\overrightarrow{sv})$ by \eqref{flow bounds}, so we get
\begin{align*}
C' \geq \sum_{v\in W_1} c(\overrightarrow{sv}) + e_{\overline{H}}(U_1,W_2)  
\geq  \sum_{v\in W_1} c(\overrightarrow{sv})+  \sum_{v\in U_1} c(\overrightarrow{sv}) = C^*.
\end{align*}
So we may assume that $|W_2| < M_{1,2}+5s^{2/3}$. A similar argument shows that we may assume $|U_1| < M_{1,2}+5s^{2/3}$. \COMMENT{Assume, for a contradiction, that $|U_1| \geq M_{1,2}+5s^{2/3}$. Then for $v'\in W_2$, $e_{\overline{H}}(U_1,\{v'\}) \geq |U_1| - d_{H[V_1,V_2]}(v') \geq M_{1,2}+5s^{2/3}-(M_{1,2}+ s^{2/3}) > c(\overrightarrow{v't})$, so we get

\begin{align*}
C' \geq e(U_1,W_2) + \sum_{v'\in U_2} c(\overrightarrow{v't})  &= \sum_{v'\in U_2} c(\overrightarrow{v't}) + \sum_{v'\in W_2} e(U_1,\{v'\}) \\
& \geq \sum_{v'\in W_2} c(\overrightarrow{v't})+  \sum_{v'\in U_2} c(\overrightarrow{v't}) = C^*,
\end{align*}
which, again, contradicts the assumption that $C'<C^*$.} Hence $|U_2|,|W_1| \geq n- M_{1,2}-5s^{2/3}\geq n/2$ since $1/n \ll 1/(2\Delta s) \leq 1/M_{1,2}$. This implies that
\begin{align*}
C'&\geq \sum_{v\in W_1} c(\overrightarrow{sv}) + \sum_{v'\in U_2} c(\overrightarrow{v't})  \stackrel{(\ref{flow bounds})}{\geq} \sum_{v\in W_1} c(\overrightarrow{sv}) +  \sum_{v'\in U_2} 1  \geq   \sum_{v\in W_1} c(\overrightarrow{sv}) + n/2 \\
& \geq   \sum_{v\in W_1} c(\overrightarrow{sv}) +  |U_1|\cdot 4  s^{2/3} \stackrel{(\ref{flow bounds})}{\geq}  \sum_{v\in W_1} c(\overrightarrow{sv}) +  \sum_{v\in U_1} c(\overrightarrow{sv}) = C^*,
\end{align*}
thus every $(s,t)$-cut has capacity at least $C^*$, and we are done.
\end{proof}

We now extend Lemma \ref{adding} to the setting where the vertex class sizes might not be exactly equal.


\begin{lemma} \label{adding2}
Suppose $0<1/n\ll 1/s \ll  1/\Delta$ and $1/n \ll 1/k, 1/(C+1)$. Let $R$ be a graph on $[r]$ and let $n\leq n_1\leq \dots \leq n_r\leq n+C$. 
Let $\vec{M}$ and $\vec{k}$ be symmetric $r\times r$ matrices such that $M_{i,j}\leq 2\Delta s$, $k_{i,j}\in \mathbb{N}$, $M_{i,j}+s^{2/3}+1\leq k_{i,j} \leq k$ for all $ij\in E(R)$. Suppose $H$ is an $r$-partite graph with vertex partition $(R,V_1,\dots, V_r)$ such that $|V_i|=n_i$ for all $i\in [r]$ and that for all $ij\in E(R)$ and every vertex $x\in V_i$ we have 
$d_{H[V_i,V_j]}(x)= M_{i,j} \pm s^{2/3}$. Then there exists an $(R,\vec{k},C)$-near-equiregular graph $H''$ with vertex partition $(R,V_1,\dots,V_r)$ such that $H\subseteq H''$.
\end{lemma}
\begin{proof}
Consider $ij\in E(R)$ with $|V_i|\geq |V_j|$ and let $a:=|V_i|-|V_j|$. Greedily choose a set $A_{ij}$ of $a$ vertices in $V_i$ such that their neighbourhoods in $H[V_i,V_j]$ are disjoint (note that we can do this as $n \geq Ck^2$)\COMMENT{We take one vertex from $V_i$ and delete all vertices whose distance from it is two. (at most $k(k-1)$ vertices ) Now we again choose a vertex from the remaining of $V_i$ and delete all vertices within distance two from it. We can keep do this $C$ times as long as $Ck^2 \leq n$.}. 
Note that $H_{ij}:=H[V_i,V_j]-A_{ij}$ satisfies that for all $x\in V_i\cup V_j$, $M_{i,j} - s^{2/3} -1 \leq d_{H_{ij}}(x) \leq M_{i,j}+s^{2/3}$. Now apply Lemma~\ref{adding} to $H_{ij}$ to obtain a bipartite graph $H'_{ij}$ which is $k_{i,j}$-regular and is obtained from $H_{ij}$ by adding edges. Obtain $H''_{ij}$ from $H'_{ij}$ by adding $A_{ij}$ and connecting each $v\in A_{ij}$ to $k_{i,j}$ vertices in $V_j$ such that $N_{H[V_i,V_j]}(v)\subseteq N_{H''_{ij}}(v)$ and such that the $N_{H''_{ij}}(v)$ are disjoint for different $v\in A_{ij}$. Repeat this for all $ij\in E(R)$ and let $H''$ be the union of the $H''_{ij}$. From the construction of $H''$ it is immediate that $H''$ is $(R,\vec{k},C)$-near-equiregular.
\end{proof}

We can now combine the above results to obtain Lemma~\ref{lem:section3newlemma}. 

\begin{proof}[Proof of Lemma~\ref{lem:section3newlemma}.]
First we apply Lemma~\ref{packing to almost regular graph} to obtain $H$ as described there. So in particular, $H$ can be decomposed into copies $\phi(L_1),\dots \phi(L_s)$ of $L_1,\dots,L_s$ such that (i) and (ii) of Lemma~\ref{packing to almost regular graph} hold. Then we apply Lemma~\ref{adding2} to obtain $H'$ which satisfies the requirements of Lemma~\ref{lem:section3newlemma}. 

To check Lemma~\ref{lem:section3newlemma}(i), for each $ij\in E(R)$ recall $J_{i,j}=H'[V_i,V_j]-(\phi(L_1)\cup\dots\cup\phi(L_s))=(H-H')[V_i,V_j]$. Lemmas~\ref{packing to almost regular graph} and~\ref{adding2} together imply that each $x\in V_i$ satisfies $d_{H[V_i,V_j]}(x) = M_{i,j}\pm s^{2/3}$, and $d_{H'[V_i,V_j]}(x) \in   \{ k_{i,j}, k_{i,j}+1\}$.
Thus $\Delta(J_{i,j})\leq k_{i,j}- M_{i,j} + 2s^{2/3}$. 
Lemma~\ref{lem:section3newlemma}(ii) follows directly from Lemma~\ref{packing to almost regular graph}(ii).
\end{proof}


\subsection{Obtaining $(r,k,2)$-near-equiregular graphs}\label{subsec:32}
In this subsection we deal with the case when the `reduced graph' $R$ of the partitions satisfies $R=K_r$. As described at the beginning of Section~\ref{sec:stacking}, we will show that we can pack any suitable family $\mathcal{L}$ of bounded degree graphs into
`edge-balanced' graphs $H$ (see Lemma~\ref{packing to regular graph}). We then can apply Lemma~\ref{lem:section3newlemma} to pack these graphs $H$ together to form degree balanced graphs. We begin with a simple arithmetic observation.

\begin{prop} \label{arithm}
Suppose $r, \Delta, n\in \mathbb{N}$ with $r\geq 3\Delta+2$\COMMENT{Jaehoon: I think it is possible to replace it to $2\Delta+1$,
but it requires some tedious argument so I am not sure if it's worth doing.}, and let $\overline{n}:=rn+c$ for some integer $c$ with $0\leq c\leq r-1$. Then there are non-negative integers $a_i$ and $n_i$ for all $i\in [3]$ such that 
\begin{itemize}
\item[(i)] $a_2,a_3 \in \{0\}\cup\{\Delta+1,\Delta+2,\dots \}$, 
\item[(ii)] $a_1+a_2+a_3=r$,
\item[(iii)] $n_1\in\{n-1,n\}$ and $n_3=n_2+1=n_1+2$, 
\item[(iv)] $a_1n_1 + a_2n_2 + a_3n_3 = \overline{n}$.
\end{itemize}
\end{prop}
\begin{proof}
If $c=0$, then we let $n_1:=n$, $a_1 := r$, $a_2:=a_3:=0$.
If $0< c\leq \Delta$, then we let $n_1:=n-1$, $a_1:=\Delta+1-c$ and $a_3:=\Delta+1$. If $\Delta+1 \leq c \leq r-1$, then we let $n_1:=n, a_1:= r-c$ and $a_3:=0$. In all cases we let $a_2:=r-a_1-a_3$, $n_2:=n_1+1$ and $n_3:=n_1+2$. It is easy to check that conditions (i)--(iv) are satisfied.
\end{proof}

Now we can state our main goal in this subsection. 

\begin{lemma}\label{packing to regular graph}
Suppose $0<1/n \ll 1/M \ll \xi \ll 1/\Delta,1/r$ and that $\overline{n}=rn+c$ for some integer $c$ with $0\leq c \leq r-1$. Suppose $s,k\in \mathbb{N}$ and that  $(1+\xi)M \leq k \leq (1+2\xi)M$.
Suppose $L_1,\dots, L_{s}$ are graphs of order $\overline{n}$ with $\Delta(L_i)\leq \Delta$ and $e(L_i)\geq {n}/4$ for all $i\in [s]$, and such that $\sum_{i=1}^{s} e(L_i) =  M\binom{r}{2} n$. 
\begin{itemize}
\item[(i)] 
Suppose further that $c=0$ and each $L_i$ has an equitable $r$-colouring. Then there exists a graph $H$ with vertex classes $V_1,\dots, V_r$ such that $|V_i|=n$ for all $i\in [r]$, $H[V_i,V_j]$ is $k$-regular for all $1\leq i<j\leq r$, and such that $L_1,\dots , L_s$ pack into $H$.
\item[(ii)] 
Suppose alternatively that $r\geq 3\Delta+2$ and $a_j,n_j$ are as given by Proposition \ref{arithm} (for all $j\in [3]$). Then there exists
an $(r,k,2)$-near-equiregular graph $H$ with vertex classes $V_1,\dots, V_r$ and a partition of $[r]$ into $I_1,I_2,I_3$ where $|I_j|=a_j$ and $|V_i|=n_j$ for all $i\in I_j$, and such that $L_1,\dots , L_s$ pack into $H$.
\end{itemize}
Moreover, writing $\phi(L_\ell)$ for the copy of $L_\ell$ in this packing of $L_1,\dots,L_s$ in $H$, and for each $ij\in E(R)$ writing $J_{i,j}:=H[V_i,V_j]-(\phi(L_1)\cup\dots\cup \phi(L_s))$, we have $\Delta(J_{i,j})\leq 3\xi k$. 
\end{lemma} 

The proof of Lemma \ref{packing to regular graph} will be conducted via two claims. First we show how in case (i) of Lemma \ref{packing to regular graph} the partition classes of the $L_i$ can be individually re-arranged so that the union of the $L_i$ is edge-balanced, i.e. the total numbers of edges in the bipartite graphs induced by each pair of partition classes is almost equal for different pairs. 


\begin{claim} \label{permute}
Suppose $0<1/n \ll 1/M \ll 1/\Delta,1/r$. Suppose $L_1,\dots, L_s$ are graphs of order $rn$ such that each $L_\ell$ satisfies $\Delta(L_\ell)\leq \Delta$ and has an equitable $r$-colouring with colour classes $X^\ell_1, \dots, X^\ell_r$. If $e(L_\ell)\geq n/5$ for each $\ell\in [s]$ and $\sum_{\ell=1}^{s} e(L_\ell) = M\binom{r}{2} n\pm M^{3/5}n/2$, then for each $\ell\in [s]$ one can permute the subscript indices
of the colour classes $X^\ell_1,\dots,X^\ell_r$ to achieve that
\begin{align}\label{eq:permute}
\sum_{\ell=1}^{s} \frac{e(L_\ell[X^\ell_{j},X^\ell_{j'}])}{n} = M\pm M^{3/5}
\end{align}
 for all $j,j'\in [r]$ with $j\neq j'$.%
    \COMMENT{Claim is now slightly stronger than before (due to the new first part in the proof of Claim~\ref{permute2}). More precisely,
have $e(L_\ell)\geq n/5$ instead of $e(L_\ell)\geq n/4$, $\sum_{\ell=1}^{s} e(L_\ell) = M\binom{r}{2} n\pm M^{3/5}n/2$ instead of
$\sum_{\ell=1}^{s} e(L_\ell) = M\binom{r}{2} n$ and an error of $M^{3/5}$ instead of $M^{2/3}$.}
\end{claim}

\noindent
Note that the above constraints of $e(L_i)\geq n/5$ and $\Delta(L_i)\leq \Delta$ for all $i\in[s]$ imply 

$$\frac{sn}{5}\leq \sum_{i=1}^{s} e(L_i)\leq \frac{s\Delta \overline{n}}{2}. $$

\noindent
This, combined with $\sum_{i=1}^{s} e(L_i) =  M\binom{r}{2} n\pm M^{3/5}n/2$, yields

\begin{align} \label{s}
\frac{(r-1)M}{4\Delta} \leq s \leq 10M\binom{r}{2}.
\end{align}

\noindent
In other words, $s$ is of the same order of magnitude as $M$ in our hierarchy.

\begin{proof}[Proof of Claim~\ref{permute}]

For each $\ell\in [s]$ we choose a permutation $\sigma_\ell$ on $[r]$ uniformly at random. For all $j,j'\in [r]$ with $j\neq j'$ let $D_{j,j'}$ be the random variable defined by $$D_{j,j'} := \sum_{\ell=1}^{s} \frac{e(L_\ell[X^\ell_{\sigma_\ell(j)},X^\ell_{\sigma_\ell(j')}])}{n}.$$ 
Note that 
\begin{align*}
\E[e(L_\ell[X^\ell_{\sigma_\ell(j)},X^\ell_{\sigma_\ell(j')}])] &=
\sum_{i\neq i'} e(L_\ell[X^\ell_{i},X^\ell_{i'}]) \cdot \mathbb{P}[\sigma_\ell(j)=i, \sigma_\ell(j')=i'] = \frac{e(L_\ell)}{\binom{r}{2}}.
\end{align*}

\noindent
Thus $$\E[D_{j,j'}] = \frac{\sum_{\ell=1}^{s} e(L_\ell)}{\binom{r}{2}n}=M\pm M^{3/5}/2.$$

Consider the exposure martingale  
$J_{j,j'}^\ell:=\E[D_{j,j'} \mid \sigma_1,\dots,\sigma_\ell]$ for $0\leq\ell\leq s$. 
Note that this martingale is $\Delta$-Lipschitz since changing only one permutation $\sigma_i$ changes the value of $D_{j,j'}$ by at most $\Delta$ (as $0\leq e(L_\ell[X^\ell_{\sigma_\ell(j)},X^\ell_{\sigma_\ell(j')}])\leq \Delta n$ for any fixed $\ell,j,j'$). Thus, by Azuma's inequality (Theorem \ref{Azuma}), by (\ref{s}), and by the fact that $1/M\ll 1/\Delta,1/r$, we have
$$\mathbb{P}[ D_{j,j'} \neq M\pm M^{3/5} ] \leq 2 e^{\frac{-M^{6/5}}{8\Delta^2 s}} \leq e^{-M^{1/6}}.$$
Hence with probability at least $1- \binom{r}{2}  e^{-M^{1/6}} \geq 1/2$ we can permute the subscript indices so that~\eqref{eq:permute} holds for all $1\leq j\neq j'\leq r$.
\end{proof}

Next we prove a similar statement for case (ii) of Lemma \ref{packing to regular graph}.

\begin{claim} \label{permute2}
Suppose $0<1/n \ll 1/M\ll 1/\Delta,1/r$ and that $r\geq 3\Delta +2$. Let $c$ be an integer such that $0\leq c\leq r-1$.
Suppose that $L_1,\dots, L_s$ are graphs with $\Delta(L_\ell)\leq \Delta$, $|L_\ell|=\overline{n}=rn+c$ and
$e(L_\ell)\geq n/4$ for all $\ell \in [s]$. Suppose that $\sum_{\ell=1}^{s} e(L_\ell) = M\binom{r}{2} n$.
For all $j\in [3]$ let $a_j, n_j$ be as given by Proposition \ref{arithm}. 

Then there is a partition $I_1,I_2,I_3$ of $[r]$ with $|I_j|=a_j$ and, for each $\ell\in [s]$, there is
an $r$-colouring of $L_\ell$ with colour classes $X^\ell_1,\dots,X_r^\ell$ such that $|X^\ell_i|=n_j$ for all $i\in I_j$
and
$$\sum_{\ell=1}^{s} \frac{e(L_\ell[X^\ell_{i},X^\ell_{i'}])}{n} = M\pm M^{2/3}$$
for all $1\leq i\neq i'\leq r$.
\end{claim}
\begin{proof}

Recall from Proposition \ref{arithm} that $n_1\in \{n-1,n\}$. For each $L_\ell$ we apply the Hajnal-Szemer\'{e}di theorem (Theorem \ref{thm:HS}) to obtain an equitable $r$-colouring of $L_\ell$. Take this $r$-colouring, and take out none, one or two vertices from each colour class to ensure that the resulting vertex classes have size $n_1$. Let $P_\ell$ be the set of those `excess' vertices and let $Q^\ell_1, \dots, Q^\ell_r$ be the remaining colour classes. We have thus partitioned $V(L_\ell)$ into sets $Q^\ell_1, \dots, Q^\ell_r, P^\ell$ such that each $L_\ell[Q^\ell_i]$ is empty, $|Q^\ell_i|=n_1$ and $|P^\ell|= a_2+2a_3$.

Let $n':=rn_1$. For each $\ell\in [s]$ let $L'_\ell:=L_\ell\setminus P^\ell$. So $n'=|L'_\ell|$.
Since $|P^\ell|\le 2r$ by Proposition~\ref{arithm}(ii), at most $2r\Delta$ edges of $L_\ell$ are incident to~$P^\ell$. Thus $e(L'_\ell)\ge n'/5$
and $\sum_{\ell=1}^s e(L_\ell) = M\binom{r}{2} n'\pm n'$. Hence by Claim~\ref{permute} for each $\ell\in [s]$ we can
permute the subscript indices of the partition classes $Q^\ell_1,\dots,Q^\ell_r$ of $L'_\ell$ to achieve that for all $i,i'\in [r]$ with $i\neq i'$,
\begin{align}\label{eq:sum}
\sum_{\ell=1}^{s} \frac{e(L'_\ell[Q^\ell_{i},Q^\ell_{i'}])}{n'} = M\pm M^{3/5}.
\end{align}
Let $I_1:=[a_1]$, $I_2:=\{a_1+1,\dots,a_1+a_2\}$ and $I_3:=\{a_1+a_2+1,\dots,r\}$. Partition each $P^\ell$ into two sets
$P^\ell_2, P^\ell_3$ of size $a_2$ and $2a_3$ respectively.
Now define 
$$Y^\ell_1 := \bigcup_{i\in I_1} Q^\ell_i, \ \ \  Y^\ell_2 := \bigcup_{i\in I_2} Q^\ell_i \cup P^\ell_2, \ \ \  Y^\ell_3 := \bigcup_{i\in I_3} Q^\ell_i  \cup P^\ell_3.$$
Then $|Y^\ell_j|= a_jn_j$ for all $j\in [3]$ by Proposition \ref{arithm}(iii).
Moreover, by~\eqref{eq:sum} for all $j,j'\in [3]$ with $j\neq j'$ we have
\begin{align}\label{eq:W}
\sum_{\ell=1}^{s} e(L_\ell[Y_j^\ell]) & =(M\pm M^{3/5})\binom{a_j}{2}n'\pm \begin{cases}
0 &\text{ if } j=1,\\
2sa_j\Delta &\text{ if } j\in \{2,3\} 
\end{cases}
\\&=(M\pm 2M^{3/5}){\binom{a_j}{2}}n,\nonumber\\
\sum_{\ell=1}^{s} e(L_\ell[Y^\ell_j,Y^\ell_{j'}]) & = (M\pm  2M^{3/5})a_ja_{j'}n.
\end{align}

Recall that $|Y_j^\ell|=a_jn_j$, and that for each $j\in\{2,3\}$, either $a_j=0$ or $a_j\ge \Delta+1$.
Thus we can apply the Hajnal-Szemer\'{e}di theorem (Theorem \ref{thm:HS}) to obtain an equitable $a_j$-colouring of $L_\ell[Y_j^\ell]$
for each $j\in\{2,3\}$.%
   \COMMENT{Instead of choosing a new partition via the Hajnal-Szemer\'{e}di theorem, we could also use Hall's theorem to show that we
can extend the $Q^\ell_2$ and $Q^\ell_3$ into the required new partition by adding $1$ or $2$ vertices from $P^\ell_2$, $P^\ell_3$ to
each $Q^\ell_2$, $Q^\ell_3$ respectively. This would work if can ensure that $a_2, a_3\ge 2\Delta$ (or $0$), which we can probably achieve
if we make $r$ even larger compared to $\Delta$ than now. But the current way is not so complicated either, so probably best to leave things as they are.} 
Let $X^\ell_{a_1+1},\dots, X^{\ell}_{a_1+a_2}$ be the corresponding equipartition of $Y_2^\ell$ into sets of size $n_2$ and
let $X^{\ell}_{a_1+a_2+1},\dots ,X^{\ell}_r$ be the corresponding equipartition of $Y_3^\ell$ into sets of size $n_3$.
Let $\{X^\ell_1, \dots, X^\ell_{a_1}\}:= \{Q_i^\ell : i \in I_1\}$.

For each $\ell\in [s]$ we now choose a permutation $\sigma_{\ell}$ on $[r]$ such that $\sigma_{\ell}$ is a random permutation when restricted
to each of $I_1$, $I_2$ and $I_3$ (choose three independent permutations
on $I_1$, $I_2$ and $I_3$ uniformly at random, and combine them).
For all $1\leq i \neq i'\leq r$ define $$ D^{\ell}_{i,i'}:= e(L_\ell[X^\ell_{\sigma_{\ell}(i)}, X^{\ell}_{\sigma_\ell(i')}])
\ \ \ \text{ and } \ \ \ D_{i,i'} := \frac{1}{n}\sum_{\ell=1}^{s} D^{\ell}_{i,i'}.$$ 
Given $1\leq i\neq i'\leq r$ and $\ell \in [s]$, let $j,j' \in [3]$ be so that $i\in I_{j}, i'\in I_{j'}$. Then
\begin{align*}
\mathbb{E}[D^{\ell}_{i,i'}] & = \sum_{b\neq b', b\in I_j, b'\in I_{j'}} e(L_\ell[X^\ell_{b},X^\ell_{b'}]) \mathbb{P}\left[\sigma_\ell(i)=b, \sigma_\ell(i')=b'\right]\\
&= \begin{cases}
\sum_{b\neq b' \in I_j} \frac{e(L_\ell[X^\ell_{b},X^\ell_{b'}])}{a_j(a_j-1)}=\frac{2e(L_\ell[Y_j^\ell])}{a_j(a_j-1)} &\mbox{if } j=j' \text{ and } a_j\geq 2,\\
\sum_{b \in I_j, b'\in I_{j'}} \frac{e(L_\ell[X^\ell_{b},X^\ell_{b'}])}{a_j a_{j'}}=\frac{e(L_\ell[Y_j^\ell,Y_{j'}^\ell])}{a_j a_{j'}} & \mbox{if } j\neq j' \text{ and } a_j,a_{j'}\geq 1. \end{cases}
\end{align*}
Together with \eqref{eq:W} this implies that for every $1\le i\neq i'\le r$ 
$$\mathbb{E}[D_{i,i'}] = M\pm  2M^{3/5}.$$
Consider the exposure martingale $J_{i,i'}^\ell:= \mathbb{E}[D_{i,i'}\mid \sigma_1,\dots, \sigma_\ell]$ with $0\leq \ell \leq s$. Note that $J_{i,i'}^\ell$ is $\Delta$-Lipschitz. Thus by Azuma's inequality (Theorem~\ref{Azuma}) and the fact that $1/M\ll 1/\Delta,1/r$ we obtain
$$\mathbb{P}[ D_{i,i'} \neq M\pm M^{2/3} ] \leq 2 e^{\frac{-M^{4/3}/4}{2\Delta^2 s}} \stackrel{\eqref{s}}{\leq} e^{-M^{1/4}}.$$
Therefore with probability at least $1- {\binom{r}{2}}e^{-M^{1/4}} \geq 1/2$ we have
\begin{align}
\sum_{\ell=1}^{s} \frac{e(L_\ell[X^\ell_{\sigma_\ell(i)},X^\ell_{\sigma_\ell(i')}])}{n} = M\pm M^{2/3}
\end{align} 
for all $1\leq i\neq i'\leq r$. The claim now follows by relabelling the subscript indices of the $X_i^\ell$ according to $\sigma_\ell$.
\end{proof}

We now conclude the proof of Lemma~\ref{packing to regular graph}.

\begin{proof}[Proof of Lemma~\ref{packing to regular graph}]
To prove Lemma~\ref{packing to regular graph}(i), we first apply Claim~\ref{permute}. We then apply Lemma~\ref{lem:section3newlemma} with $C=0$, with $K_r$
playing the role of $R$ and with $\vec{k}$, $\vec{M}$ being the matrices with entries $k_{i,j}:=k$ and $M_{i,j}:=\sum_{\ell=1}^{s}\frac{e(L_\ell[X_i^\ell,X_j^\ell])}{n}= M\pm M^{3/5}$ for all $i\neq j$ (note that $s^{2/3}+ M^{3/5}+1 \leq \xi M$ by~\eqref{s}). To prove Lemma~\ref{packing to regular graph}(ii),\COMMENT{This time we have $M_{i,j}= M\pm M^{2/3}$, but this is still ok since $s^{2/3}+ M^{2/3} \leq \xi M$.} we proceed similarly, except that we first apply Claim~\ref{permute2}, and that we apply Lemma~\ref{lem:section3newlemma} with $C=2$ instead of $C=0$.

In either case, Lemma~\ref{lem:section3newlemma}(i) implies that $\Delta(J_{i,j})\leq k - M_{i,j} + 2s^{2/3} \leq k - (M- M^{3/5})+ 2s^{2/3} \leq 3\xi M \le 3\xi k$.
\end{proof}


\section{Combining the packing results}\label{sec:concl}

Combining Theorem~\ref{main lemma} and Lemma~\ref{lem:section3newlemma}, we obtain the following result. It gives an approximate decomposition of $G$ into $L_1,\dots, L_s$ if the $L_i$ have a common vertex partition so that the densities of the resulting bipartite pairs reflect the densities of the corresponding pairs in $G$. It immediately implies Theorem~\ref{thm:main3} and generalises it to the setting where the pairs of $G$ have different densities. Properties (T1), (T3) and (T4) of Theorem~\ref{main lemma} will not be used in this section, so as remarked after Theorem~\ref{main lemma}, we can ignore the conditions in 
(S6)--(S8) whenever we apply Theorem~\ref{main lemma} in this section.

\begin{thm}\label{thm:main0}
Suppose $0<1/n\ll \epsilon \ll  \alpha, \eta, d, 1/(C+1), 1/\Delta, 1/\Delta_R$ and $1/n\ll 1/r$. Let $s\in \mathbb{N}$ be such that $s \leq \eta^{-1} n$. Suppose that $R$ is a graph on $[r]$ with $\Delta(R)\leq \Delta_R$.
Let $\vec{d}$ be a symmetric $r\times r$ matrix such that $d_{j,j'}\geq d$ for all $jj'\in E(R)$ and $d_{j,j'}=0$ if $jj'\notin E(R)$. 
Suppose that $L_1,\dots, L_s$ are graphs admitting a common vertex partition $(R,X_1,\dots, X_r)$ such that 
$\max_{j \in [r]} |X_j|=n$ and $n-C \leq |X_j|\leq n$ for all $j\in [r]$, $\Delta(L_\ell)\leq \Delta$ for all $\ell\in[s]$ and $\sum_{\ell=1}^s e(L_\ell[X_j,X_{j'}]) \leq (1-\alpha)d_{j,j'}n^2$ for all $jj'\in E(R)$. 
Suppose finally that $G$ is an $(\epsilon,\vec{d})$-super-regular graph with partition $(R,V_1,\dots, V_r)$ such that $|X_i|=|V_i|$ for all $i\in [r]$. Then $L_1,\dots, L_{s}$ pack into $G$. 
\end{thm}
\begin{proof}
Choose an integer $M$ such that $\epsilon \ll 1/M \ll \alpha,\eta,d,1/\Delta,1/\Delta_R$.  
Without loss of generality we may assume that $s> M^2$ and that $s$ is a multiple of $M$
(if not, add up to $M^2$ empty graphs $L_\ell$). 
Let $t:=s/M$.
Then we may take an arbitrary partition of $\{ L_1,\dots, L_{s}\}$ into $\mathcal{L}_1,\dots \mathcal{L}_{t}$ 
such that $|\mathcal{L}_i| =M$ for all $i \in [t]$.

Now for all $i\in [t]$ and $jj'\in E(R)$, let $M^i_{j,j'} := \sum_{L \in \mathcal{L}_i}\frac{e(L[X_j,X_{j'}])}{n}$. Then $M^i_{j,j'}\le \Delta M$ and
$\sum_{i=1}^{t} M^i_{j,j'} \leq (1-\alpha) d_{j,j'} n$.
If $jj'\in E(R)$, let $k^{i}_{j,j'}$ be the smallest integer satisfying 
$k^{i}_{j,j'} \geq M^{i}_{j,j'} + 2M^{2/3}$, and if $jj'\notin E(R)$ let $k^{i}_{j,j'}:=0$.
Apply Lemma~\ref{lem:section3newlemma} with $M$ playing the role of $s$ for all $i\in [t]$ to find an $(R,\vec{k}^i,C)$-near-equiregular graph $H_i$ with
partition $(R,V_1,\dots,V_r)$
such that the graphs in $\mathcal{L}_i$ pack into $H_i$.
Then for all $jj'\in E(R)$, 
\begin{align*}
\sum_{i=1}^{t} k^{i}_{j,j'} &\leq \sum_{i=1}^{t}( M^{i}_{j,j'} + 2M^{2/3}+1)  \leq (1-\alpha)d_{j,j'}n + t (2M^{2/3}+1) \\
&\leq (1-\alpha)d_{j,j'}n + \frac{3s}{M^{1/3}} \leq 
(1-\alpha) d_{j,j'}n + \alpha d n/2
\leq (1-\alpha/2) d_{j,j'} n.
\end{align*}
Here, we obtain the penultimate inequality since $1/M \ll \eta, \alpha, d$ and $s\leq \eta^{-1}n$. 
It is easy to see that $H_1,\dots, H_t$, $G$, $R$ satisfy the assumptions in Theorem~\ref{main lemma}, 
where $\alpha/2$ plays the role of $\alpha$. Thus $H_1,\dots, H_t$ pack into $G$, and so $L_1,\dots, L_s$ also pack into $G$. 
\end{proof}

In the next result, the reduced graph $R$ of the partition of $G$ is complete. In this case, we can drop the requirement that the $L_\ell$ and $G$ have a common partition.

\begin{thm}\label{main 1}
Suppose $0<{1}/{n} \ll \epsilon \ll \lambda, d, {1}/{\Delta}, {1}/{r}$ and that $r\geq 3\Delta+2$.  Let $c,\overline{n}$ be integers
such that $0\leq c\leq r-1$ and $\overline{n}=nr+c$. Let $a_1,a_2,a_3,n_1,n_2,n_3$ be non-negative integers such that
\begin{itemize}
\item[(i)] $a_2,a_3 \in \{0\}\cup\{\Delta+1,\Delta+2,\dots \}$,
\item[(ii)] $a_1+a_2+a_3=r$,
\item[(iii)] $n_1\in\{n-1,n\}$ and $n_3=n_2+1=n_1+2$, 
\item[(iv)] $a_1n_1 + a_2n_2 + a_3n_3 = \overline{n}$.
\end{itemize}
Let $I_1,I_2,I_3$ be a partition of $[r]$ with $|I_j|=a_j$.
Suppose that $L_1,\dots, L_s$ are graphs with $|L_\ell|=\overline{n}$ and $\Delta(L_\ell)\leq\Delta$ for all $\ell\in[s]$.
Suppose further that $G$ is an $\overline{n}$-vertex $(\epsilon,d)$-super-regular graph with respect to a partition $(V_1,\dots, V_r)$, where 
$|V_i|=n_j$ for all $i\in I_j$. Suppose finally that $(1-9\lambda)e(G)\leq \sum_{\ell=1}^{s} e(L_\ell) \leq (1-\lambda)e(G)$. Then $L_1,\dots, L_s$ pack into $G$.

Moreover, writing $\phi(L_\ell)$ for the copy of $L_\ell$ in this packing of $L_1,\dots,L_s$ in $H$, and writing $J':=G-(\phi(L_1)\cup\dots\cup \phi(L_s))$, we have $\Delta(J')\leq 10\lambda d \overline{n}$. 
\end{thm}
\begin{proof}
First, if there are graphs $L_{j_1}, \dots, L_{j_q}$, each having fewer than $n/4$ edges, then we can pack two of them into a single graph without increasing the maximum degree. We keep doing this until we obtain at most one graph having fewer than $n/4$ edges. Once we have at most one such graph, we add at most $n/4$ suitable new edges to assume that all graphs $L_\ell$ have at least
$n/4$ edges and still satisfy $\Delta(L_\ell)\le \Delta$. 
Choose $k\in \mathbb{N}$ and $\xi$ such that 
$\epsilon \ll {1}/{k}\ll \xi \ll \lambda, 1/r, 1/\Delta$.
Note that we may assume that $\sum_{\ell=1}^{s} e(L_\ell) \geq k^2 n$. (Indeed, since $\sum_{\ell=1}^{s} e(L_\ell)\ge (1-9\lambda)e(G)$
this is clearly true if $\lambda \le 1/10$. If $\lambda > 1/10$ we simply add additional graphs $L_\ell$ if necessary.)
Let \begin{align}\label{m' def}
m:=\frac{\sum_{\ell=1}^{s} e(L_\ell)}{k\binom{r}{2}n(1-3\xi/2)}.
\end{align}
Partition $\{L_1,\dots, L_s\}$ into $\mathcal{L}_1,\dots, \mathcal{L}_{m}$ such that for all $i\in[m]$ we have 
\begin{align}\label{kkkkk}
(1-7\xi/4)k\binom{r}{2}n\leq \sum_{L\in \mathcal{L}_i} e(L) \leq (1-\xi)k\binom{r}{2}n.
\end{align} 
Such a partition exists since $\sum_{\ell=1}^{s} e(L_\ell) \geq k^2 n$ and $1/k \ll \xi, 1/\Delta, 1/r$. 

Now for each $i\in [m]$ use Lemma~\ref{packing to regular graph}(ii) with $M_i:=\sum_{L\in \mathcal{L}_i} e(L)/(\binom{r}{2}n)$ playing
the role of $M$ to obtain an $(r,k,2)$-near-equiregular graph $H_i$ with partition $X_1,\dots, X_r$, where $|V_j|=|X_j|$ for all $j\in [r]$, and so
that the graphs in $\mathcal{L}_i$ can be packed into $H_i$. 
Moreover, writing $\tau(L)$ for the copy of $L \in \mathcal{L}_i$ in this packing of the graphs in $\mathcal{L}_i$ into $H_i$, and for each $i'\neq j' \in [r]$ writing $J^i_{i',j'}:=H_i[V_{i'},V_{j'}]-\bigcup_{L\in \mathcal{L}_i} \tau(L)$, we have $\Delta(J^i_{i',j'})\leq 3 \xi k$. 

Note that %
\COMMENT{recall we added up to $n/4$ artificial edges}
\begin{eqnarray*}
k m &\stackrel{\eqref{m' def}}{=}& \frac{\sum_{\ell=1}^{s} e(L_\ell)}{(1-3\xi/2)\binom{r}{2}n} \leq \frac{(1-\lambda)e(G)+n/4}{(1-3\xi/2)\binom{r}{2} n} \leq  (1-\lambda/2)(d+\epsilon) n
\leq (1-\lambda/3) dn.
\end{eqnarray*}
and
\begin{eqnarray}\label{eq: km}
k m &\stackrel{\eqref{m' def}}{=}& \frac{\sum_{\ell=1}^{s} e(L_\ell)}{(1-3\xi/2)\binom{r}{2}n} \geq \frac{(1-9\lambda)e(G)}{(1-3\xi/2)\binom{r}{2} n}
\geq (1-9\lambda) dn.
\end{eqnarray}
In particular, the conditions of Theorem \ref{main lemma} are satisfied with $K_r$ playing the role of $R$ and $\lambda /3$ playing the role of $\alpha$.
So we can apply Theorem~\ref{main lemma} to $G$ and $H_1,\dots, H_{m}$, to find a packing of $H_1,\dots, H_{m}$ into $G$.
This yields a packing of $L_1,\dots,L_s$ into $G$.

To prove the moreover part of Theorem~\ref{main 1}, note that since $\Delta(G)\le d\overline{n}$,%
\COMMENT{since $\Delta(G)\le (d+\epsilon)(r-1)n\le d\overline{n}$}
we may assume that
$\lambda \le 1/10$ (and thus we did not add any additional graphs $L_\ell$ at the beginning of the proof). For each $i',j'\in [r]$%
  \COMMENT{The extra $\Delta$ below comes from the additional $n/4$ edges that we (might have) added}
\begin{eqnarray*}
\Delta(J'[V_{i'},V_{j'}]) &\leq& \Delta(G[V_{i'},V_{j'}]) - km + \sum_{i=1}^{m}\Delta(J^i_{i',j'}) +\Delta \leq (d+\epsilon)n - km + 3\xi k m +\Delta\\
&\stackrel{\eqref{eq: km}}{\leq}& (d+\epsilon)n - (1-4\xi)(1-9\lambda)dn \leq 10\lambda dn.
\end{eqnarray*}
Thus, 
$\Delta(J') \leq 10\lambda d (r-1)n \leq 10\lambda d\overline{n}.$
\end{proof}

Now we can deduce Theorem~\ref{main 1b} from Theorem~\ref{main 1} (note that the conditions (1) and (2) needed 
in the proof are actually weaker than the quasi-randomness assumption of Theorem~\ref{main 1b}).

\begin{proof}[Proof of Theorem \ref{main 1b}.]
Choose $r,n_0\in\mathbb{N}$ and $\varepsilon>0$ such that $1/n_0 \ll \epsilon \ll 1/r \ll \alpha,p_0,1/\Delta$.
Let $n\ge n_0$, $p\ge p_0$ and assume that $G$ is an $n$-vertex graph satisfying the following two assumptions.
\begin{itemize}
\item[(1)] Each vertex $v$ in $G$ satisfies $d_G(v)= (1\pm \epsilon) p n$.
\item[(2)] There exists $D\subseteq \binom{V(G)}{2}$ such that $|D| \geq \binom{n}{2} - \epsilon n^2$ and all pairs $\{u,v\}\in D$ satisfy $|N_{G}(u,v)| = (1\pm \epsilon )p^2 n$. 
\end{itemize}
Note that an $(\epsilon,p)$-quasi-random graph satisfies both (1) and (2). By adding additional graphs~$L_\ell$ if necessary,
we may assume that $\sum_{\ell=1}^{s} e(L_\ell) \geq (1-2\alpha)\binom{n}{2}p$. Let $n':=\lfloor n/r\rfloor$ and let $c:=n-rn'$.
Choose $a_1,a_2,a_3,n_1,n_2,n_3$ satisfying (i)--(iv) of Theorem~\ref{main 1} (with $n$, $n'$ playing the roles of $\overline{n}$, $n$ in Theorem~\ref{main 1}).
Such a choice is possible by Proposition~\ref{arithm}. 
Let $I_1, I_2, I_3$ be a partition of $[r]$ with $|I_i| = a_i$.

Now choose a partition $\mathcal{V}=(V_1,\dots,V_r)$ of $V(G)$ such that $|V_i|=n_i$ for all $i\in I_j$ with $j\in [3]$ and
such that $G[V_i,V_{i'}]$ is $(\epsilon^{1/7},p)$-super-regular for all $i\neq i'\in [r]$.
(To see that such a partition exists, consider a partition $\mathcal{V}=(V_1,\dots, V_r)$ of $V(G)$ chosen uniformly at random among all the
partitions satisfying $|V_i|= n_i$ for all $i\in I_j$ and apply Lemma~\ref{Chernoff Bounds} as well as Theorem~\ref{codegree implies regularity}.)%
\COMMENT{
For $S\in D\cup V(G)$ and $i\in [r]$, $|N_{G}(S)\cap V_i|$ has hypergeometric distribution satisfying $\mathbb{E}[|N_{G}(S)\cap V_i|] = \frac{n_1\pm 2}{n}|N_{G}(S)| =
(1\pm 2\epsilon)p^{|S|}\frac{n}{r}.$ 
Thus by Lemma~\ref{Chernoff Bounds},
$$\mathbb{P}[ |N_{G}(S)\cap V_i| = (p^{|S|}\pm 3\epsilon)|V_i|] \geq 1 -(1-c)^n.$$

Thus with probability at least $1-r(\binom{n}{2}+n)(1-c)^n > 1/2$, we get a partition which satisfies $|N_{G}(S)\cap V_i| = (p^{|S|}\pm 3\epsilon)|V_i|$ for all $S\in D\cup V(G)$ and $i\in [r]$. We take a partition satisfying this. Then Theorem~\ref{codegree implies regularity} implies that $G$ is $(\epsilon^{1/7},p)$-super-regular with respect to the chosen partition $\mathcal{V}$.}%
\COMMENT{For $j$, $D\cap \binom{V_j}{2}$ has size at least $\binom{V_j}{2} - \epsilon n^2 \geq \binom{V_j}{2} - \epsilon r^2|V_j|^2$. And $r^2\epsilon < \epsilon^{6/7}$.}

Now we let $G'$ be the spanning subgraph of $G$ such that $G'[V_i]=\emptyset$ and $G'[V_i,V_j]=G[V_i,V_j]$ for $i\neq j \in [r]$. Then $G'$ is $(\epsilon^{1/7},p)$-super-regular with respect to $\mathcal{V}$. Moreover, since $1/r \ll \alpha, p_0$ we have
$e(G') \geq (1-\alpha/2) e(G)=(1-\alpha/2)(1\pm 2\epsilon)\binom{n}{2}p$. Thus
$$\left(1-3\alpha \right) e(G') \leq \left(1-2\alpha \right)\binom{n}{2}p \leq \sum_{\ell=1}^{s} e(H_\ell) \leq \left(1-\alpha\right)\binom{n}{2}p \leq \left(1-\alpha/3 \right) e(G').$$ 
Now we apply Theorem \ref{main 1} with $\epsilon^{1/7}$, $\alpha/3$, $p$, $n$, $n'$ playing the roles of $\epsilon,\lambda$, $d$, $\overline{n}$, $n$ to find the desired packing
of $H_1,\dots,H_s$ in $G'\subseteq G$. Moreover, writing $\phi(H_\ell)$ for the copy of $H_\ell$ in this packing of $H_1,\dots,H_s$ in $G$, and writing $J':=G'-(\phi(H_1)\cup\dots\cup \phi(H_s))$, we have $\Delta(J')\leq 10\alpha p n/3$.
Since $J= (G-G')\cup J'$, we conclude $\Delta(J)\leq \Delta(G-G')+ \Delta(J') \leq n/r + 10\alpha p n/3 \leq 4\alpha p n.$
\end{proof}

Similarly, we can now deduce Theorem~\ref{main 2} from Theorem~\ref{main lemma}.

\begin{proof}[Proof of Theorem \ref{main 2}.]
This follows from the same argument as in the proof of Theorem~\ref{main 1} with $\alpha$ playing the role of $\lambda$. Here the assumption $r\geq 3\Delta +2$ is not required since a common partition of the graphs $H_\ell$ and $G$ is given. In particular, this time we apply Lemma~\ref{packing to regular graph}(i) instead of Lemma~\ref{packing to regular graph}(ii).
\end{proof}

Next, we prove Corollary~\ref{Hamilton-Waterloo for many cycles}. 
For this, we use the following theorem by K\"uhn and Osthus.
It a special case of Theorem 1.2~in~\cite{KellyII}, which is stated for the more general class of robustly expanding graphs
(in~\cite{KellyII}, this result on Hamilton decompositions of robustly expanding graphs is in turn derived from a digraph version, which is the main result of~\cite{Kelly}).
 \begin{thm} \label{quasicor}\cite{Kelly, KellyII}
 For every $p_0>0$ there exist $\epsilon>0$ and $n_0\in\mathbb{N}$ such that the following holds for all $p\ge p_0$ and $n\ge n_0$ for which $pn$ is even.
 Let $G$ be any $pn$-regular $(\epsilon,p)$-quasi-random graph on $n$ vertices.
 Then $G$ has a Hamilton decomposition.
 \end{thm}

\begin{proof}[Proof of Corollary~\ref{Hamilton-Waterloo for many cycles}.]
Note that without loss of generality we may assume that $\beta\le p_0/3$.
Choose $\varepsilon,\alpha>0$ and $n_0\in\mathbb{N}$ such that $1/n_0\ll\epsilon \ll \alpha \ll p_0,\beta,1/\Delta$.
Let $n\ge n_0$, $p\ge p_0$, and let $G$ and $H_1,\dots,H_s$ be as given in Corollary~\ref{Hamilton-Waterloo for many cycles}.
Choose a subgraph $J$ of $G$ such that $G':=G-J$ is $(2\epsilon,p')$-quasi-random where $p'=p-2\beta+\alpha$
and $J$ is $(2\epsilon,2\beta-\alpha)$-quasi-random.
(Indeed, for each edge $e$ of $G$, we include $e$ in $E(J)$ with probability $(2\beta -\alpha)/ p$. Then it is straightforward to check that $J$ is $(2\epsilon,2\beta-\alpha)$-quasi-random and $G'=G-J$ is also $(2\epsilon,p-2\beta+\alpha)$-quasi-random with high probability.)

Since $\epsilon \ll \alpha, p',1/\Delta$, Theorem~\ref{main 1b} implies that $H_{\beta n+1},\dots, H_{s}$ pack into $G'$.
In other words, for each $\beta n+1\le \ell\le s$ there exists a copy $\phi_\ell(H_\ell)$ of $H_\ell$ in $G$ such that the $\phi_\ell(H_\ell)$ are pairwise edge-disjoint. Let $$G^*:= G - \bigcup_{\ell=\beta n+1}^{s} \phi_\ell(H_\ell).$$ Then $G^*$ is $2\beta n$-regular because $G$ is $pn$-regular and each $H_\ell$ is $r_\ell$-regular with $\sum_{\ell=\beta n+1}^{s} r_\ell = (p-2\beta)n$.
Let $G'':=G'-\bigcup_{\ell=\beta n+1}^{s} \phi_\ell(H_\ell)$. Then any vertex $v\in V(G)$ satisfies
\begin{equation} \label{hello}
d_{G''}(v) = (1\pm 2\epsilon) (p-2\beta+\alpha)n - \sum_{\ell=\beta n+1}^{s} r_\ell = (\alpha \pm 2\epsilon)n.
\end{equation}
Also recall that $J$ is $(2\epsilon,2\beta -\alpha)$-quasi-random. Thus for any $S\in \binom{V(G)}{2}$, we have
$$
N_{G^*}(S) = N_{J}(S) \pm 2\Delta(G'') \stackrel{(\ref{hello})}{=} 
(1\pm 2\epsilon) (2\beta -\alpha)^{2} n \pm 2(\alpha \pm 2\epsilon) n = (1\pm \alpha^{1/2})(2\beta)^2 n.
$$ 
Thus $G^*$ is a $(\alpha^{1/2},2\beta)$-quasi-random $2\beta n$-regular graph.  This means that we can apply Theorem~\ref{quasicor} to $G^*$
to obtain a decomposition of $G^*$ into $H_1,\dots, H_{\beta n}$. Together with the $\phi_\ell(H_\ell)$ this gives the required decomposition.
\end{proof}

The following bipartite version of the tree packing conjecture was formulated by Hobbs, Bourgeois and Kasiraj \cite{HB} in 1987. 
\begin{conj}\label{bipartite tree packing conjecture}
Let $T_1,\dots, T_n$ be trees such that, for each $i\in[n]$, $T_i$ has $i$ vertices. Let $G$ be the complete bipartite graph $K_{n-1,n/2}$ if $n$ is even, and $K_{n,(n-1)/2}$
if $n$ is odd. Then $G$ has a decomposition into copies of $T_1,\dots,T_n$.
\end{conj}
The following result implies an approximate version of this conjecture for bounded degree trees (see Corollary~\ref{bipartite version}). We derive Theorem~\ref{bipartite thm} from Theorem~\ref{thm:main0}.

\begin{thm}\label{bipartite thm}
Suppose $0<1/n \ll \epsilon \ll \alpha, 1/\Delta$.
Let $H_1,\dots, H_{s}$ be $n$-vertex bipartite graphs with $\Delta(H_i)\leq \Delta$ for all $i\in [s]$.
Let $G$ be an $(\epsilon,d)$-super-regular bipartite graph with bipartition $(A,B)$ such that $|A|=\lfloor n/2\rfloor$ and $|B|=n-1$.
If $\sum_{i=1}^{s} e(H_i) \leq (1-\alpha) e(G)$, then $H_1,\dots, H_s$ pack into $G$.
\end{thm}
\begin{proof}
If $H_i$ and $H_j$ both have at most $n/4$ edges, then we can pack them disjointly
into an $n$-vertex bipartite graph without increasing the maximum degree. 
So we may assume that $e(H_i)\geq n/4$ and $\Delta(H_i)\leq 2\Delta$ for all $i\in [s]$.\COMMENT{via packing,
we can make all graphs having at least $n/4$ edges except one. If there's one graph $H_i$ with less edges, then we take another graph $H_j$ and pack $H_i$ with $H_j$ into a bipartite graph with maximum degree at most $2\Delta$.}
Note that this implies that $s\leq 2n$. For each $H_i$, we take a bipartition $(A_i,B_i)$ such that $|A_i|\leq |B_i|$. Note that $|A_i|\leq \lfloor n/2 \rfloor, |B_i|\leq n-1$, so we may assume $|A_i|= \lfloor n/2 \rfloor, |B_i| =n-1$ by adding isolated vertices if needed. Let $X^i_1:= A_i$. 
Take an equipartition $(X^i_2,X^i_3)$ of $B_i$ with $|X^i_2|= \lceil (n-1)/2\rceil, |X^i_3| = \lfloor (n-1)/2 \rfloor$ such that $|e(H_i[X^i_1,X^i_2])-e(H_i[X^i_1,X^i_3])|\leq \epsilon n.$ 
(To see that such an equipartition exists, consider a random equipartition and apply  Theorem~\ref{Azuma}.)%
\COMMENT{Let $v_1,v_2,\dots,v_{n-1}$ be the vertices in $B_i$. Let $X_i$ be the random variable such that $X_i := d_{H_i}(v_i)$ if $v_i \in X^i_2$ and $X_i:= - d_{H_i}(v_i)$ otherwise. We also let $Y_i:= \mathbb{E}[\sum_{j=1}^{n-1} X_j : X_1,\dots, X_i ].$ Then $Y_i$ is an exposure $2\Delta$-Lipschitz martingale. Thus Theorem~\ref{Azuma} implies that 
$\mathbb{P}[Y_{n-1} = \pm \epsilon n ] \geq 1 - c^n.$ 
Also, $\sum_{j=1}^{n-1} X_j = e(H_i)[X^i_1,X^i_2] - e(H_i)[X^i_1,X^i_3]$.}

Let $X_1:=A$. We also choose a partition $(X_2,X_3)$ of $B$ with $|X_2|= \lceil (n-1)/2 \rceil, |X_3| = \lfloor (n-1)/2 \rfloor$ such that both $G[X_1,X_2]$ and $G[X_1,X_3]$ are $(4\epsilon,d)$-super-regular. 
(Again, consider a random equipartition for this.)%
\COMMENT{
Such equipartition exists because if we take a random equipartition, then Lemma~\ref{Chernoff Bounds} shows that $$\mathbb{P}[|N_G(v)\cap X_2|= \frac{1}{2}d_{G}(v) \pm \epsilon n = (d\pm 4\epsilon)|X_2| ] \geq 1-2e^{\frac{2\epsilon^2n^2}{2dn}} \geq 1- c^n.$$ 
Note that $d_{G}(v) = (d\pm \epsilon) n$ since $G$ is $(\epsilon,d)$-super-regular. 
Thus there exists a partition $B = X_2 \cup X_3$ such that $d_{G[X_1,X_2]}(v) = dn/2 \pm 4 \epsilon n$ and 
$d_{G[X_1,X_3]}(v) = dn/2 \pm 4 \epsilon n$ as $1 - nc^n > 0$. For this partition, both $G[X_1,X_2]$ and $G[X_1,X_3]$ are $(2\epsilon,d)$-regular by Proposition~\ref{restriction}. Thus the desired equipartition exists.}
Then it is easy to see that 
\COMMENT{$$\leq \sum_{i=1}^{s} (\frac{1}{2}|E(H_i)|+\epsilon n)  \leq \frac{1}{2}(1-\alpha)e(G) + \epsilon s n\\
 \leq \frac{1}{2}(1-\alpha/2) e(G) \leq (1-\alpha/2)(d+\epsilon)\frac{n^2}{4} \\
\leq (1-\alpha/3)(d-5\epsilon)\frac{n^2}{4}$$}
$$
\sum_{i=1}^{s} e(H_i[X^i_1,X^i_2]) \leq \sum_{i=1}^{s} (e(H_i)/2+\epsilon n) 
 \leq (1-\alpha/3)d|X_1||X_2|.
$$
Similarly $\sum_{i=1}^{s} e(H_i[X^i_1,X^i_3]) \leq (1-\alpha/3)d|X_1||X_3|$.
Let $R$ be a path on three vertices such that $V(R)=[3]$ and $d_R(1)=2$.
Then each $H_i$ admits vertex partition $(R,X^i_1, X^i_2, X^i_3)$. So we can now apply Theorem~\ref{thm:main0}
to obtain the desired packing.\end{proof}

The following corollary is immediate from Theorem~\ref{bipartite thm}. Note that we consider $K_{n-1,\lfloor n/2 \rfloor}$ here, which also implies the result for $K_{n,\lfloor n/2 \rfloor}$.

\begin{cor}\label{bipartite version}
Suppose $0<1/n\ll \alpha, 1/\Delta$. Let $T_{\alpha n},\dots, T_n$ be trees such that, for each $\alpha n\le i \le n$, $T_i$ has $i$ vertices and $\Delta(T_i)\leq \Delta$. Then $T_{\alpha n},\dots,T_n$ pack into $K_{n-1,\lfloor n/2\rfloor}$.
\end{cor}

\section*{Acknowledgements}
We are grateful to Felix Joos for suggesting Theorem~\ref{MM} and its proof which replaces a more complicated approach towards (M1)--(M3) based on a random greedy matching algorithm and an associated martingale analysis.

\end{document}